\documentclass[12pt]{article}
\usepackage[margin=1.0in]{geometry}
\usepackage{amsmath, amsthm}
\usepackage{times}
\usepackage{newtxmath}
\usepackage{hyperref}
\usepackage{cleveref}
\usepackage{enumitem}

\newtheorem{theorem}{Theorem}
\newtheorem{corollary}{Corollary}

\theoremstyle{definition}
\newtheorem{remark}{Remark}
\newtheorem{definition}{Definition}
\newtheorem{conjecture}{Conjecture}
\newtheorem{assumption}{Assumption}
\newtheorem{setting}{Setting}
\newtheorem{proposition}{Proposition}
\newtheorem{situation}{Situation}
\newtheorem{question}{Question}

\begin{document}

\title{$(\infty,\infty)$-Categorical Universal Motives}
\author{Xin Tong}
\date{}

\maketitle

\begin{abstract}
\noindent Ever since the introduction of motivic homotopy theory, as a well-proposed approximation of Grothendieck's dream, algebraic geometers then have the chance to study schemes via a homotopy theory. However topologists also found that lifting the usual homotopy theory over a sphere spectrum to the motivic homotopy category over a motivic bigraded sphere spectrum can make breakthroughs on elementary topology problems (such as computing homotopy groups of spheres, motivic Adams spectral sequences and so on). On the other hand, topological spaces can be all regarded as Grothendieck topoi, as in Lurie's work on ultracategories. Following Scholze, Lurie we systematically consider an $(\infty,\infty)$-ultracategorical universal motivic formalism, which directly fits into Lurie's framework on ultracategories, where we construct universal ultragestalten through motivicalization. The gestalten higher categorical six-functor formalism then allows us to consider the following classes of problems of different flavors: (I) Six-functor formalism for all Grothendieck sites and Grothendieck topoi; (II) Six-functor formalism for all topological spaces. We believe the current framework is a grand unified theory of many subjects in mathematics, such as an enhanced bridge between topology and algebraic geometry. And this framework can possibly lead to a more flexible framework on the study of general topological spaces. We then in this paper got the chance to apply this to many problems in $p$-adic geometry and $p$-adic functional analysis: (I) $(\infty,\infty)$-categoricalization of motivic $+$-de Rham prismatization approach to generalization of Colmez's Montr\'eal functor; (II) $(\infty,\infty)$-categoricalization of motivic generalized Riemann-Hilbert correspondence after Bhatt-Lurie; (III) $(\infty,\infty)$-categorical universal motives of derived algebraic stacks over $\mathbb{E}_\infty$-ring objects in the derived $\infty$-category of sphere spectrum, via universal motives of large Fargues-Fontaine gestalten, in families. 
\end{abstract}

\newpage
\tableofcontents

\newpage
\section{Introduction}

\subsection{Universal Motives Formalism}

Grothendieck topos is a significant concept in modern algebraic geometry, which is approximately some category of sheaves over some Grothendieck site. On the other hand any topological space can be regarded as a topos. Studying of topological spaces can be established by considering all the cohomologies, which is the main goal of algebraic topology. This means we should have motives theory for all topological spaces (like those very general $p$-adic analytic spaces in this paper) and Grothendieck topoi. Therefore the questions one may first want to ask are:

\begin{question}
Do we have a motivic six-functor formalism for all topological spaces? Then do we have a motivic six-functor formalism for all Grothendieck topoi? For instance for $p$-adic Lie groups. And then do we have the modern realization process as in \cite{5S4}?
\end{question}

These general questions usually have negative answer which motiviated the theory of condensed mathematics \cite{CS1} \cite{CS2} \cite{CS3}, where the category of topological spaces are enlarged to condensed sets. The corresponding idea is to replace the usual topological spaces with some sheaves over the Grothendieck site of totally disconnected compact Hausdorff spaces. This also has a direct relationship with the ultracategoricalization of Lurie \cite{5Lu3}. On the other hand in algebraic geometry one considers Grothendieck's notion of motives which is in general the following $(\infty,\infty)$-category of all the functors from a particular $\infty$-category over sphere spectrum $\mathbb{S}$ into category of presentable $(\infty,\infty)$-categories:
\begin{align}
\mathrm{Fun}(X_{/\mathbb{S}},\mathrm{Cat}_{\infty,\infty})^\mathrm{localization, localization}.
\end{align}
The localization is inspired by homotopy theory, where the first one is homotopy localization via bigraded motivic sphere spectrum, the second one is stabilization by using the infinite loop space functors. If one does not take stabilization, then the resulting category is called \textit{unstable motivic homotopy category}. This fits into Lurie's ultrafunctors formalism for instance when $X$ is a pretopos, then from taking models (in the sense of \cite{5Lu3}) we have a $(\infty,\infty)$-category of ultrafunctors:
\begin{align}
\mathrm{UltraFun}(\mathrm{Model}(X_{/\mathbb{S}}),\mathrm{Cat}_{\infty,\infty})^\mathrm{localization, localization}.
\end{align}
This paves a way to studying general topological spaces by embedding all topological spaces into the category of all the $(\infty,1)$-pretopoi. Lurie's Vollst\"andigkeitssatz then states that this is equivalent to the ultracategory of coherent topos associated to $X$. Then the question is:

\begin{question}
Throughout all topological spaces regarded as pretopoi, do we have a six-functor formalism for all such $(\infty,\infty)$-categories of ultrafunctors?  
\end{question}

\begin{definition}
With the notations above, we define univesal motivic gestalten formalism to be gestalten associated to the two $(\infty,\infty)$-categories above:
\begin{align}
\mathrm{Fun}(X_{/\mathbb{S}},\mathrm{Cat}_{\infty,\infty})^\mathrm{localization, localization},
\mathrm{UltraFun}(\mathrm{Model}(X_{/\mathbb{S}}),\mathrm{Cat}_{\infty,\infty})^\mathrm{localization, localization}.
\end{align}
When $X$ is changing we endow the gestalten six-functor formalism via countably presented mappings pullbacked from the six-functor formalism from \cite{5S5}.
\end{definition}

\begin{remark}\mbox{}\\
1. (Motives for Geometric Stacks) For instance if we consider $X$ to be just a Grothendieck site of the sphere spectrum, this immediately applies to algebraic geometry, where we have six-functor formalism on the universal gestalten above among Grothendieck sites in some inter-Grothendieck sites manner;\\
2. (Motives for Topological Spaces) On the other hand, if we consider $X$ to be a topological space (such as derived schemes over sphere spectrum, or just $p$-adic Lie groups), we will have six-functor formalism for our universal motivic gestalten attached to each such $X$. This is just \textit{ultracategorical motives theory for all topological spaces} including $p$-adic Lie groups, and so on, with well-defined six-functor formalism on the associated universal motivic gestalten to each such $X$;\\
3. (Motives For General Grothendieck Topoi) Then if $X$ is a Grothendieck topos, then we have again \textit{ultracategorical motives theory for all the Grothendieck topoi}. For instance for any $p$-adic adic space attached to some Banach ring $R$, one can regard it as a general topological space and then regard it as a Grothendieck topos, such as Fargues-Fontaine topos. 
\end{remark}

\subsection{Applications and Results}

\indent Emerton-Gee-Hellmann's \textit{$p$-adic categoricalization}, Fargues-Scholze's \textit{geometrization} as in \cite{1FS}, \cite{EGH1}, and Scholze's \textit{motivicalization} as in \cite{1S5}, \cite{1S6} by using symmetrical monoidal $\infty$-categoricalization, motivate us to consider the generalization we will present in this paper on Langlands program after \cite{L1}, \cite{D}, \cite{C}. We will consider essential \textit{de Rham prismatization} as in our previous work \cite{1T}. In both settings (one $p=\ell$ and considering $p$-adic Banach space representations, the other one $p$ is allowed to be away from $\ell$) even if we do not consider Banach representations, we consider however extensively $p$-adic coefficients. Therefore we call them all \textit{$p$-adic Local Langlands program}. But our ultimate goal is to parametrize all the corresponding possible condensed representations of all the possible \textit{motivic Galois groups} as in \cite{2A} by using motivic Hopf algebras over schemes, formal schemes, rigid analytic spaces, adic spaces, Berkovich spaces, small $v$-stacks, small arc stacks, derived small arc stacks, condensed analytic stacks and so on: 
\begin{itemize}\setlength{\itemsep}{-0.1cm}
\item[A] $p$-adic constructible condensed sheaves with the attached motivic Galois groups 
\begin{align}
\mathrm{Spec}\mathrm{Hopf}_{p-\text{adic},\blacksquare}
\end{align}
with the motivic 6-functor formalism ($F_1$,$F_2$,($F_3$,$F_4$,*-adjoint pairs),($F_5$,$F_6$,!-adjoint pairs));
\item[B] $\ell$-adic constructible condensed sheaves with the attached motivic Galois groups 
\begin{align}
\mathrm{Spec}\mathrm{Hopf}_{\ell-\text{adic},\blacksquare}
\end{align}
with the motivic 6-functor formalism ($F_1$,$F_2$,($F_3$,$F_4$,*-adjoint pairs),($F_5$,$F_6$,!-adjoint pairs));
\item[C] $p$-adic $+$-de Rham lattice deformed quasi-coherent sheaves with the attached motivic Galois groups $\mathrm{Spec}\mathrm{Hopf}_{\mathrm{dR},+,\text{prismatization}}$, $\mathrm{Spec}\mathrm{Hopf}_{\mathrm{dR},\mathrm{Robba},+,\text{stackification}}$ with the motivic 6-functor formalism ($F_1$,$F_2$,($F_3$,$F_4$,*-adjoint pairs),($F_5$,$F_6$,!-adjoint pairs));
\subitem[C1] Through de Rham prismatization;
\subsubitem[C11] Over small arc-stacks, under arc topology;
\subsubitem[C12] Over small $v$-stacks, under $v$-topology;
\subitem[C2] Through de Rham Robba stackification from pre-Fargues-Fontaine curves;
\item[D] Solid quasicoherent sheaves over Robba rings without Frobenius and the attached motivic Galois groups $\mathrm{Spec}\mathrm{Hopf}_{\mathrm{solid}, \mathrm{quasicoherent},\mathrm{Robba},\blacksquare}$ with the motivic 6-functor formalism ($F_1$,$F_2$,($F_3$,\\$F_4$,*-adjoint pairs),($F_5$,$F_6$,!-adjoint pairs));
\subitem[D1] Over small arc-stacks, under arc topology;
\subitem[D2] Over small $v$-stacks, under $v$-topology;
\item[E] ... and more motivic cohomology theories satisfying the formalism in \cite{2A} with the motivic 6-functor formalism ($F_1$,$F_2$,($F_3$,$F_4$,*-adjoint pairs),($F_5$,$F_6$,!-adjoint pairs))...
\end{itemize}
Also we believe that ultimately there will be some very well-defined generalizations of Colmez's Montr\'eal foncteur in \cite{C} by using the methods we are considering. We use the techniques of prismatization and its various de Rham stackified versions. Symmetrical monoidal $\infty$-categories in $p$-adic analytic geometry and $p$-adic functional analysis are very significant, though from a higher categorical perspective especially when we have the stability. Many problems emerge in $p$-adic analytic geometry and $p$-adic functional analysis if one sticks to the usual derived categories, such as the lack of enough projective objects and many tricky points of view in the modular representation theory even over $\mathbb{Z}_p/p^n$, let alone taking the inverse limit over the index $n\in \mathbb{Z}$. In this paper we discuss some specific aspects in the $p$-adic Hodge theory and the related $p$-adic functional analysis by using some well-defined $\infty$-categories which carry the corresponding symmetrical monoidal $\infty$-categorical structures. In \cite{T} the author in the derived category studied some generalization of the results in \cite{KPX} by using $p$-adic functional analysis method dated back to \cite{K1}. The resulting derived categorical consideration was of course following \cite{KPX}. However there should be many $\infty$-categorical considerations now which can be better illustrate the issues following \cite{BBBK}, \cite{BBKK}, \cite{CS1}, \cite{CS2}, \cite{CS3} by using certain symmetrical monoidal $\infty$-categories where more robust properties can hold. With the notations in \cref{definition1} and \cref{definition2} we have rings of analytic functions over those rigid affinoids as in \cite{T} after \cite{KPX}, \cite{CKZ}, \cite{PZ}, \cite{1Z}. Then over these rings we have the following well-defined categories, many of which are stable symmetrical monoidal $\infty$-categories.

\begin{remark}
In this paper we consider specific higher sheaves, higher stacks, higher gestalten and even stablization sequences of categories in the following most general way:
\begin{itemize}\setlength{\itemsep}{-0.1cm}
\item[1] $(\infty,\infty)$-sheaves means $(\infty,\infty)$-sheaves with values in $(\infty,\infty)$-categories. In this paper we do not make restrictions on the definition of $(\infty,\infty)$-categories as they are in general colimits of $(\infty,n)$-categories in reasonable sense;
\item[2] $(\infty,\infty)$-stacks means $(\infty,\infty)$-presheaves which satisfy certain descent for some topology;
\item[3] $(\infty,\infty)$-gestalten means gestalten in \cite{5S5}, i.e they are infinite stablization sequences $(A_0,A_1,...)$ of $(\infty,n),n=0,1,2,...$-categories such that for any $n$, $A_n$ is symmetrical monoidal module object over the symmetrical monoidal ring object $A_{n-1}$, in categories of presentable $\infty$-categories;
\item[4] $(\infty,\infty)$-$\mathrm{E}_1$-gestalten means a stablization sequence of $(\infty,n)$-monoidal categories, but we start from $A_0$ which can be a noncommutative monoidal object, in categories of presentable $\infty$-categories;
\item[5] $(\infty,\infty)$-stablization sequences of modules means such sequences of $(\infty,n)$-categories with action from some $(\infty,\infty)$-gestalten. This is needed when we consider $(\infty,\infty)$-sheaves with values in $(\infty,\infty)$-categories over some gestalten.
\end{itemize}

\end{remark}

\begin{definition}
Over $*$=
\begin{align}
L^X_{[a_I,b_I],K_I}(\pi_{K_I}),L^X_{b_I,K_I}(\pi_{K_I}):=L^X_{(0,b_I],K_I}(\pi_{K_I}),L^X_{[0,b_I],K_I}(\pi_{K_I}),L^X_{K_I}(\pi_{K_I}):=\bigcup_{b_I>0}\bigcap_{a_I>0}L^X_{[a_I,b_I],K_I}(\pi_{K_I}),
\end{align} 
we have the notion of $\Gamma$-Frobenius modules and bundles. Then we consider the derived $\infty$-categories 
\begin{align}
\underline{\mathrm{IndBanachModules}}^\sharp_*, \underline{\mathrm{Ind_{\mathrm{mono}}BanachModules}}^\sharp_*.
\end{align}
which are stable. And we have the condensed version:
\begin{align}
\blacksquare D_{*}, \blacksquare D_{*,\mathrm{bounded}},\blacksquare D_{*,\mathrm{perfect}}.
\end{align}
Therefore we have the corresponding Banach perfect complexes and condensed perfect complexes of the corresponding objects in our setting, namely we consider the perfect complexes of the $\Gamma$-Frobenius-Hodge modules in our setting, which again form certain stable $\infty$-categories with symmetrical monoidal structures. In the Banach setting we use the notations:
\begin{align}
&\underline{\mathrm{IndBanachModules}}^\sharp_{*,\mathrm{perfect},\Gamma,F}, \underline{\mathrm{Ind_{\mathrm{mono}}BanachModules}}^\sharp_{*,\mathrm{perfect},\Gamma,F},\\
&\underline{\mathrm{IndBanachModules}}^\sharp_{*,\mathrm{perfect},\mathrm{bounded},\Gamma,F}, \underline{\mathrm{Ind_{\mathrm{mono}}BanachModules}}^\sharp_{*,\mathrm{perfect},\mathrm{bounded},\Gamma,F},\\
&\underline{\mathrm{IndBanachModules}}^\sharp_{*,\mathrm{perfect},-,\Gamma,F}, \underline{\mathrm{Ind_{\mathrm{mono}}BanachModules}}^\sharp_{*,\mathrm{perfect},-,\Gamma,F},
\end{align}
to denote these complexes. And in the corresponding condensed setting we use:
\begin{align}
\blacksquare D_{*,\Gamma,F}, \blacksquare D_{*,\mathrm{bounded},\Gamma,F},\blacksquare D_{*,\mathrm{perfect},\mathrm{bounded},\Gamma,F},\blacksquare D_{*,\mathrm{perfect},\Gamma,F},\blacksquare D_{*,\mathrm{perfect},-,\Gamma,F}.
\end{align}
\end{definition}

\indent Then those derived $\infty$-sheaves in the categories above can have those certain derived cohomologies by using the $F$ and $\Gamma$-structure, by taking those iterated Yoneda groups, which are the corresponding derived $F,\Gamma$-cohomologies as in the following definition:

\begin{definition}
If $I$ is singleton, then we have the notion of $(F,\Gamma)$-complex of any $\Gamma$-Frobenius-Hodge module $F$: $C_{F,\Gamma}(F)$, where we have also the $C_{F,*}(F)$ and $C_{*,\Gamma}(F)$ complexes as well. Using them by induction we have the corresponding notion of $(F,\Gamma)$-complex of any $\Gamma$-Frobenius-Hodge module $F$: $C_{F,\Gamma}(F)$ when $I$ is not just a singleton. In our setting for each $i\in I$ we also have the corresponding $W_i=\varphi_i^{-1}$-operator. In such a way one can form the complex $C_{W}(F)$ directly.
\end{definition}

Following ideas in \cite{T}, \cite{KPX} we consider the symmetrical monoidal $\infty$-categories above instead to study the targeted categories where the cohomology groups are living, due to the fact that they are stable and usually having Grothendieck homotopy triangulated categories, many problems after \cite{T} and \cite{KPX} are solved, even over quite hard quasi-Stein spaces. Of course we are taking about condensed quasi-Stein spaces not the usual ones.

\begin{theorem}
$C_{F,\Gamma}(F)$ is in bounded $(\infty,1)$-derived category of complexes over $X$, restricting to perfect complexes:
\begin{align}
D_{\mathrm{perfect},\mathrm{bounded}}(\mathrm{Mod}_X).
\end{align}
$C_{F,\Gamma}(F)$ is in 
\begin{align}
\underline{\mathrm{IndBanachModules}}^\sharp_{X,\mathrm{perfect},\mathrm{bounded}}, \underline{\mathrm{Ind_{\mathrm{mono}}BanachModules}}^\sharp_{X, \mathrm{perfect},\mathrm{bounded}}.
\end{align} 
$C_{W}(F)$ is in 
\begin{align}
\underline{\mathrm{IndBanachModules}}^\sharp_{L^X_{\infty_I,K_I}(\Gamma_{K_I}),\mathrm{perfect},\mathrm{-}}, \underline{\mathrm{Ind_{\mathrm{mono}}BanachModules}}^\sharp_{L^X_{\infty_I,K_I}(\Gamma_{K_I}), \mathrm{perfect},\mathrm{-}},
\end{align}
where $X$ is just $\mathbb{Q}_p$.
\end{theorem}

\begin{theorem}
$C_{F,\Gamma}(F)$ is in bounded $(\infty,1)$-derived category of complexes over $X$, restricting to perfect complexes:
\begin{align}
D_{\mathrm{perfect},\mathrm{bounded}}(\mathrm{Mod}_X).
\end{align}
$C_{F,\Gamma}(F)$ is in 
\begin{align}
\blacksquare\underline{D}_{X,\mathrm{perfect},\mathrm{bounded}}.
\end{align} 
$C_{W}(F)$ is in 
\begin{align}
\blacksquare\underline{D}_{L^X_{\infty_I,K_I}(\Gamma_{K_I}),\mathrm{perfect},-},
\end{align}
where $X$ is just $\mathbb{Q}_p$.
\end{theorem}

\begin{corollary}
$C_{W}(.)$ induces a derived functor from the stable $\infty$-category 
\begin{align}
\underline{\mathrm{IndBanachModules}}^\sharp_{*,\mathrm{perfect},\mathrm{bounded},\Gamma,F}, \underline{\mathrm{Ind_{\mathrm{mono}}BanachModules}}^\sharp_{*,\mathrm{perfect},\mathrm{bounded},\Gamma,F},
\end{align}
and in the corresponding condensed setting:
\begin{align}
\blacksquare D_{*,\mathrm{perfect},\mathrm{bounded},\Gamma,F},
\end{align}
to the stable $\infty$-category 
\begin{align}
\blacksquare\underline{D}_{L^X_{\infty_I,K_I}(\Gamma_{K_I}),\mathrm{perfect},-},
\end{align}
where $X$ is just $\mathbb{Q}_p$. Here $*=L_{K_I}^X(\pi_{K_I})$. This also induces the morphism on the $K$-group spectra of $\mathbb{E}_\infty$-rings after applying \cite{BGT} to the corresponding stable $(\infty,1)$-categories, after \cite{G2}, \cite{A2}, \cite{BGT}.
\end{corollary}

\begin{remark}
The considerations over the level of \cite{T} and \cite{KPX} are quite deep, the reason is the following: although the original constructions of \cite{T} and \cite{KPX} are highly analytic using functional analysis and locally-convex topological vector spaces, the finiteness theorems are essentially around some algebra derived $(\infty,1)$-categories. This is in fact non-trivial, since even though we know the fininess in the condensed categories, functional analytic categories as above, we may not be able to derive the finiteness in the algebraic $(\infty,1)$-categories directly. This is a crucial point on the specific finiteness theorems in \cite{T} and \cite{KPX}. In the setting of algebraic derived categories, if we are working over some Fr\'echet-Stein algebra, then the result category is very complicated due to the lack of enough projective objects, where we cannot take the projective resolution directly. However in the current consideration in this paper, the problem disappears by using \cite{BBBK}, \cite{BBKK}, \cite{CS1}, \cite{CS2} , \cite{CS3}.
\end{remark}

\indent Here $X$ is some rigid affinoid algebra, which is actually making the context closely related to \cite{Z1}, \cite{ST}. However we find the integral model is also significant by taking the integral model of $X$, namely $X^+$ which is some first of all some $p$-adic $\mathbb{Z}_p$-algebra, then taking the reduction we have some $\mathbb{Z}_p/p^n$-algebra and finally we have some $\mathbb{F}_p$-algebra. Then the construction in such relative setting is closely related to \cite{Sc1}, \cite{So1}, \cite{SS1}, \cite{SS2}, \cite{HM} as well, where we consider some other significant related symmetrical monoidal $\infty$-categores as well as in \cite{Sc1}, \cite{So1}, \cite{SS1}, \cite{SS2}, \cite{HM} with general coefficients in $X^+$. In general over $G$ a reductive $p$-adic Lie group we have the smooth representations in the coefficient in $X^+$, such as the group $\Gamma_{K_I}$ in the above rigid analytic geometric consideration. The corresponding derived $\infty$-categories in coefficients $X^+$ are  Grothendieck symmetrical monoidal $\infty$-categories as in \cite{SS1}, \cite{SS2}, \cite{HM}. In \cite{Sc1} Schneider defined $E_1$-version of the usual Hecke algebra with respect to some compact open. The point is that the derived categories over $E_1$ Hecke algebras are actually related to the representations of the group $G$ in some direct manner. As in \cite{So1} we consider one step further, i.e. we try to find minimal $E_1$-model for the homology of the $E_1$-Hecke algebra. However we need to use derived \textit{$E_1$-rings} in \cite{Sa} to do this, which looks to go into some different direction from \cite{So1} due to the fact that we have $E_1$-rings over a general commutative ring $X^+$. After \cite{Sc1}, \cite{So1}, \cite{SS1}, \cite{SS2}, \cite{HM} we have the following:

\begin{conjecture}\mbox{\textbf{(After Sorensen, \cite[Theorem 1.1]{So1})}}
Let $G$ be compact as in \cite{So1}\footnote{Since \cite{So1} mentioned that one can generalize this to more general setting by using general dg Hecke algebras, this theorem can also be generalized to more general setting.}. Let $X^+$ be a $p$-adic $\mathbb{Z}_p$-formal algebra. $X^+$ can be written as a formal projective limit over $\mathrm{Z}_p/p^n$-algebras. Let $H$ be the dg Hecke algebra over $X^+$ defined as in \cite{Sc1}. Promoting $H$ to a derived $E_1$-ring $\mathbb{H}$ as in \cite{Sa}, we have the derived $\infty$-category of the minimal derived $E_1$-ring as in \cite[Theorem 1.2]{Sa}
\begin{align}
\underset{{\mathrm{totalized},\mathbb{H},\mathbb{H}}}{\mathrm{homomorphism}}
\end{align}
admits a functor from and a functor into the derived $\infty$-category
\begin{align}
D\mathrm{R}_{\mathrm{lisse},X^+,G}.
\end{align}
The latter is the derived $\infty$-category of all the $G$-$X^+$-smooth modules. We conjecture all the functors here are equivalences of symmetrical monoidal $\infty$-categories.
\end{conjecture}

We then consider the deformation functors from these well-defined functors. Then we work in the context of \cite{1S5}, \cite{1S6} where we consider the corresponding two different types of stackifications in the sense of prismatization, after in one situation \cite{1To1}, \cite{1To2}, \cite{1To3}, \cite{1To4}, \cite{1S4}, \cite{1ALBRCS}, \cite{1BS}, \cite{1D}, \cite{1BL}, and after in other situation \cite{1KL}, \cite{1KL1}, \cite{1S1}, \cite{1S2}, \cite{1S3}, \cite{1F1}, \cite{1F2}, \cite{1T1}, \cite{1CSA}, \cite{1CSB}, \cite{1CS}, \cite{1BS1}, \cite{1BL1}, \cite{1BL2}. Again the derived $\infty$-categories we consider in this topic are stable symmetrical monoidal ones. 

\begin{remark}
The $p$-adic aspects of the local Langlands program we presented here (though not very sure if this is the ultimate correct approach to generalize Colmez's work) are deeply inspired by the work of Emerton-Gee-Hellmann \cite{EGH1} where a categoricalization is conjectured. Emerton-Gee-Hellmann kept one side of the correspondence almost the same as in \cite{1FS}, i.e. the smooth representations with coefficient in at least $\mathbb{Z}_p$-coefficients, then one can consider locally analytic representations and ultimately consider Banach representations. However the other side of the fully-faithful functor conjecture in \cite{EGH1} used stacks related to \cite{KPX}, i.e. the arithmetic stacks of $(\varphi,\Gamma)$-modules over Robba rings. We consider the motivic Tannakina categorical consideration then from de Rham prismatization. 
\end{remark}

\begin{theorem}
Assume we are in our general setting by adding the element $b^{1/2}$. The de Rham-Robba stackification and the de Rham-prismatization stackification in our generalized setting by adding $b^{1/2}$ are equivalent, in both $p$-adic and $z$-adic settings, i.e. in the $p$-adic setting we consider the small $v$-stacks over $\mathrm{Spd}\mathbb{Q}_p$, and in the $z$-adic setting we consider the small $v$-stacks over $\mathrm{Spd}\mathbb{F}_p((t))$, in the $v$-topology. This applies immediately to rigid analytic varieties.
\end{theorem}

\begin{definition}
Assume we are in our general setting by adding the element $b^{1/2}$. Let $S$ be a small $v$-stack, which can be either over $\mathrm{Spd}\mathbb{Q}_p$ or $\mathrm{Spd}\mathbb{F}_p((u))$. We use the notation:
\begin{align}
\mathrm{deRhamRobba}_S
\end{align}
to denote the corresponding de Rham-Robba stackification from the FF stacks, in our generalized setting. And we use the notation
\begin{align}
\mathrm{deRhamPrismatization}_S
\end{align}
to denote the corresponding de Rham Prismatization stackification, in our generalized setting. And for $?= \mathrm{deRhamRobba}_S, \mathrm{deRhamPrismatization}_S$ we use the notation:
\begin{align}
\mathrm{SolidQuasiCoh}_?
\end{align}
to denote the corresponding condensed $\infty$-categories of the corresponding solid quasicoherent sheaves over $?$. Then we have a functor:
\begin{align}
\mathrm{SolidQuasiCoh}_{\mathrm{deRhamPrismatization}_S}\rightarrow \mathrm{SolidQuasiCoh}_{\mathrm{deRhamRobba}_S}
\end{align}
by taking the induced functor from identification of the de Rham functors on the perfectoids.
\end{definition}

\begin{theorem}
Assume we are in our general setting by adding the element $b^{1/2}$. The functor defined above:
\begin{align}
\mathrm{SolidQuasiCoh}_{\mathrm{deRhamPrismatization}_S}\rightarrow \mathrm{SolidQuasiCoh}_{\mathrm{deRhamRobba}_S}
\end{align}
is well-defined, and an equivalence of symmetrical monoidal $\infty$-categories which are stable.

\end{theorem}

\begin{definition}
Assume we are in our general setting by adding the element $b^{1/2}$. Let $S$ be a small arc-stack in \cite{1S5} over $\mathbb{Q}_p$ or $\mathbb{F}((u))$. We use the notation:
\begin{align}
\mathrm{deRhamRobba}_S
\end{align}
to denote the corresponding de Rham-Robba stackification from the FF stacks, in our generalized setting. And we use the notation
\begin{align}
\mathrm{deRhamPrismatization}_S
\end{align}
to denote the corresponding de Rham Prismatization stackification, in our generalized setting. And for $?= \mathrm{deRhamRobba}_S, \mathrm{deRhamPrismatization}_S$ we use the notation:
\begin{align}
\mathrm{SolidQuasiCoh}_?
\end{align}
to denote the corresponding condensed $\infty$-categories of the corresponding solid quasicoherent sheaves over $?$. When we consider the de-Rham Robba stackification we consider the corresponding $v$-stack associated to $S$, which is denoted by $\mathrm{Stack}_v(S)$ after \cite{1S5}. Then we have a functor:
\begin{align}
\mathrm{SolidQuasiCoh}_{\mathrm{deRhamPrismatization}_S}\rightarrow \mathrm{SolidQuasiCoh}_{\mathrm{deRhamRobba}_{\mathrm{Stack}_v(S)}}
\end{align}
by taking the induced functor from identification of the de Rham functors on the perfectoids, i.e. we set the Banach ring local chart to be perfectoid to reach the objects in the second $\infty$-category. 
\end{definition}

\begin{theorem}
Assume we are in our general setting by adding the element $b^{1/2}$. The functor defined above:
\begin{align}
\mathrm{SolidQuasiCoh}_{\mathrm{deRhamPrismatization}_S}\longrightarrow \mathrm{SolidQuasiCoh}_{\mathrm{deRhamRobba}_{\mathrm{Stack}_v(S)}}
\end{align}
is well-defined, as a symmetrical monoidal $\infty$-tensor functor.
\end{theorem}

\begin{theorem}
Assume we are in our general setting by adding the element $b^{1/2}$. The functor defined above:
\begin{align}
\mathrm{SolidQuasiCoh}_{\mathrm{deRhamPrismatization}_S}\longrightarrow \mathrm{SolidQuasiCoh}_{\mathrm{deRhamRobba}_{\mathrm{Stack}_v(S)}}
\end{align}
is fully faithful functor of symmetrical monoidal $\infty$-categories which are stable. 
\end{theorem}

\indent We then apply these consideration of de Rham lattice deformations of the $\overline{\mathbb{Q}}_p$-local systems to the local Langlands program after \cite{L1}, \cite{1FS}, \cite{1VL}, \cite{1GL}, \cite{D}, \cite{1LL}, \cite{D2}. Then we promote the discussion to the motivic level after \cite{1S5}, \cite{1S6}, \cite{2LH}, \cite{2A}, \cite{1RS} and construct very general $t$-adic local Langlands correspondence in family after \cite{2LH} by using the motivic Galois groups fibered over $\mathbb{N}^\wedge$. Then as more thorough discussion on motivic cohomology theories, we study in some uniform way many significant $p$-adic motivic theories in families after \cite{3G}, \cite{3A}, \cite{31A}, \cite{3V}, after the general framework and formalism of Ayoub \cite{3A}. We extend in some sense Ayoub's formalism after Scholze's theory of Berkovich motives \cite{3S}, \cite{3S2}. The $p$-adic motivic cohomology theories we are considering are: prismatizations (filtration, syntomification) and stackifications induced from them (de Rham, de Rham-Hodge-Tate, Laurent) in families after \cite{3BS}, \cite{3BL}, \cite{3D}.

\begin{remark}
We choose to work in families after \cite{3LH} due to some reason to extend certain mixed-characteristic constructions to function field, and \textit{vice versa} extend certain funtion field constructions to mixed-characteristic situations.
\end{remark}

\begin{remark}
We work in some combined and vast generalized manner after \cite{3A} and \cite{3S} in the following way. First \cite{3A} considered up to rigid analytic varieties, and found correspondence from this to the classical motivic theory for schemes. Second \cite{3S} obviously generalized \cite{3A}, but we try to uniformize certain consideration after Scholze in the fashion of \cite{3A} to use certain general Weil cohomology theories over Banach rings, in order to apply to the considerations we are considering. 
\end{remark}

\begin{setting}
We consider the setting as in \cite{3LH}. Let $f$ be some finite field, then we consider a bunch of local fields in mixed-characteristics away from $\mathbb{N}_\infty - \mathbb{N}= \{\infty\}$:
\begin{align}
L1,L2,...
\end{align}
such that we have a profinite version of local rings, where $L_n,n\in \mathbb{N}$ are all local fields over $f$ with characteristic $0$. We assume the growth on the ramification index. Then $L_\infty$ is a function field over $f$. We call the product $L$ throughout the union of $\mathbb{N}$ with the infinity and we call the ring of integer $A$. $z$ is a general uniformizer which will then turn to be $z_n$ for each $n\in \mathbb{N}_\infty$.
\end{setting}

Motives \textit{in families} in our current consideration will be the most generalized sense after \cite{3S}, \cite{3A}, i.e. the generalized motives with coefficients\footnote{Eventually this will be motives parametrizing universal 6-functor formalisms, i.e. motives are \textit{localization localization} of $(\infty,\infty)$-sheaves with values in $(\infty,\infty)$-categories, or less generally $(\infty,1)$-groupoids:
\begin{align}
\mathrm{Fun}^U:=\mathrm{Fun}^U(X_{\mathrm{site}}, \mathrm{CoeffCategory}_{\infty,\infty})^{\mathrm{twicelocalization}}.
\end{align}
We can regard this as a \textit{universal $(\infty,2)$-gestalt}
\begin{align}
\mathrm{Gest}(\mathrm{Fun}^U):= (\mathrm{End}_{\mathrm{End}_{\mathrm{Fun}^U}(1)}(1),\mathrm{End}_{\mathrm{Fun}^U}(1),\mathrm{Fun}^U,\mathrm{Mod}_{\mathrm{Fun}^U}(1Pr^L),\mathrm{Mod}_{\mathrm{Mod}_{\mathrm{Fun}^U}}(2Pr^L),...)
\end{align}
of all gestalten corresponding to all six-functor formalisms. If $X$ is just the sphere spectrum $\mathbb{S}$ or simply just $\mathbb{Z}$, then we have \textit{realization functor} acting on generators $[X]$, i.e. for any generator $[X]$ in this universal gestalt of motives, we have a corresponding associated gestalt $G_{[X]}$, call a \textit{realization}, i.e. a \textit{specific motivic cohomology theory via a gestalt} such as Fargues-Fontaine stackification gestalten, $(+,-, + \bigcup -)$-de Rham prismatization stackification gestalten, Nygaard prismatization stackification gestalten and syntomization prismatization stackification gestalten.}. We start from a category of stacks over some fixed stack $X$ in families over $A$. We consider analytic condensed stacks in families (i.e. with certain structure morphism to $\mathbb{N}_\infty$) in analytic Grothendieck toplogy, we consider small arc stacks in families in arc Grothendieck topology, we consider small $v$-stacks in families in $v$ Grothendieck topology, which form the corresponding big Grothendieck sites over $X$. Over these Grothendieck sites we consider construction of motives in the sense of \cite{3S}, \cite{3A} namely those $\infty$-presheaves valued in $(\infty,1)$-categories (or less general $\infty$-groupoid) with some further possible action from $A$. Usually we have 5 steps in constructing $\infty$-category (symmetrical monoidal) of motivic sheaves with coefficients (in families as well over $A$) following \cite{3G}, \cite{3V}, \cite{3A}, \cite{3S}, \cite{5S4}, \cite{5S5}:
\begin{itemize}\setlength{\itemsep}{-0.1cm}
\item[1] Fix an $(\infty,1)$-category of $(\infty,1)$-stacks over $X$ with a fixed Grothendieck topology $T$, which forms a Grothendieck site, $\mathrm{Site}_{G,T,X}$;
\item[2] Start from an $(\infty,\infty)$-ring sheaf with values in $(\infty,\infty)$-categories\footnote{This means we have such $(\infty,\infty)$-stack over the site $\mathrm{Site}_{G,T,X}$, for instance the prismatization stack $P_X$ over some other base stack $X$.} (or $(\infty,\infty)$-ring gestalten\footnote{Again some gestalten $P_X$ over some base stack $X$, for instance if $P'_X$ is prismatization associated with some $v$-stack/arc stack $X$, then we can just take the associated gestalt $P_X$ from $P'_X$.}, more precisely a family of $(\infty,n)$-ring categories $(C_0,C_1,C_2,...)$ as in \cite{5S4}, \cite{5S5}, where $C_i$ is a symmetrical monoidal module category over the symmetrical monoidal ring category $C_{i-1}$\footnote{However in correspondence with the notion of \textit{ring stacks}, we will regard \textit{ring categories} as ring objects in the category of sequences of $(\infty,n)$-categories such as $(D_0,D_1,...)$ where $D_k$ is module category over $D_{k-1}$.}) $\mathrm{Sh}_A$ with further action from $A$ over $\mathrm{Site}_{G,T,X}$, and consider the $(\infty,1)$-category of all the $(\infty,\infty)$-sheaves over $\mathrm{Sh}_A$\footnote{$(\infty,\infty)$-sheaves are with values in $(\infty,\infty)$-categories, in the most general sense. In the less general situation for instance the notation of $\infty$-sheaf means having values in some $(\infty,1)$-category. For instance if we only consider the first element in $\mathrm{Sh}$, then this is just the usual motivic homotopy $(\infty,1)$-category, then one considers higher elements to derive the usual motivic homotopy $(\infty,1)$-categories to $(\infty, n)$-categories. "Over $\mathrm{Sh}_A$" means that we consider the action from this ring object in infinite stablization sequences of $(\infty,n)$-categories, for instance in the analytic prismatization in families situation we can either use the ring gestalten  structure to define these stablization sequences, or we can consider the sheaves of modules over the structure sheaf.}, $\mathrm{Cat}_1$;
\item[3] Take the corresponding localization of $\mathrm{Cat}_1$ to get $\mathrm{Cat}_2$ which is usually the category of effective motives with coefficient in $\mathrm{Sh}_A$;
\item[4] Take the inverse of the Tate object, then make sure it is added to the category, which is the final category of motives $\mathrm{Cat}_3$;
\item[5] Take the corresponding homotopy suspension of $\mathrm{Sh}_A$ for the Tate object, then take the corresponding \v{C}ech conerve of this resulting suspension $\mathrm{Cechconerve}(\mathrm{Suspen}_*\mathrm{Sh}_A)$, which gives rise to the Hopf algebra sheaf over $X$ after we take the loop space functor:
\begin{align}
\mathrm{Loop}^\infty\mathrm{{Cechconerve}}(\mathrm{Suspen}_*\mathrm{Sh}_A).
\end{align} 
Take $\mathrm{Spec}$ we have the motivic Galois sheaf\footnote{Yes, this is exactly the suspension-loop space functor adjunction for gestalten, not just the usual structure sheaves as in \cite{31A}. The resulting coalgebras are Hopf cogestalten for each $E_\infty$-ring on each single degree.}.
\end{itemize}
For instance:
\begin{itemize}\setlength{\itemsep}{-0.1cm}
\item[M1] Voevodsky's A1 motives in \cite{3V}, one can use arc topology actually for the schemes over some base scheme $X$;
\item[M2] Ayoub's B1 motives with $\mathbb{Q}$-coefficient \cite{31A}, one considers smooth rigid spaces with \'etale topology for instance;
\item[M3] Scholze's Berkovich motives in \cite{3S} with $\mathbb{Z}$-coefficient, one considers all the small arc stacks and arc topology;
\item[M4] Generalized motives with general coefficients as in \cite{3A}, one considers smooth rigid analytic spaces in \'etale topology, then takes localization for $\infty$-category of sheaves in such \'etale site in any coefficient in well-defined Weil cohomology theories.
\end{itemize}
Following all these and the philosophy from Grothendieck \cite{3G}, we consider de Rham prismatization motives in families, de Rham filtration prismatization motives in families, and de Rham syntomization prismatization motives in families.

\begin{situation}\mbox{\textbf{(Motivic cohomology theories in families I)}}\label{situation1}
We consider all the small arc-stacks/small $v$-stacks over $A$ fibered over $\mathbb{N}_\infty$. This means we fix some small arc-stack or small $v$-stack $X$ and consider the corresponding $\infty$-categories of the arc site or $v$-site over $X$ using all the small arc stacks or small $v$-stacks:
\begin{align}
\mathrm{Cate}_\mathrm{arc}/X,\mathrm{Cate}_v/X.
\end{align}
A $z$-adic motivic cohomology theory over any small arc-stack/small $v$-stack $X$ consists of the datum as in the following after \cite{3A}:
\begin{itemize}\setlength{\itemsep}{-0.1cm}
\item[A1] An $(\infty,\infty)$-ring sheaf with values in $(\infty,\infty)$-category (or $(\infty,\infty)$-ring gestalten as in \cite{5S4}, \cite{5S5})\footnote{Note $(\infty,\infty)$-ring sheaves with values in $\infty$-categories carry natually $(\infty,\infty)$-ring stacks structure. And be more precise, $\mathrm{Sh}_A$ is some $(\infty,\infty)$-stack or some $(\infty,\infty)$-ring gestalt over $X_v,X_\mathrm{arc}$, for instance if we have prismatization stackification/analytic prismatization $P_X$ for such $X$, then we have the associated gestalt $P'_X$ defined by:
\begin{align}
P'_{X}:= (\mathcal{O}_{P_X}, D(P_X), 1Pr_{P_X}, 2Pr_{P_X},...).
\end{align}} $\mathrm{Sh}_A$, with possible action from $A$ over the arc stack $X$ in the arc site of small arc stacks over $X$ or over $v$-stack $X$ in the $v$-site of small $v$-stacks over $X$, in familes over $\mathbb{N}_\infty$;
\item[A2] A version of K\"unneth theorem in the derived sense holds for $\mathrm{Sh}_A$\footnote{For instance if $\mathrm{Sh}_A$ is anlaytic prismatization in families then we can regard this as an $(\infty,\infty)$-ring gestalten, then the corresponding K\"unneth theorem on the level of ring stacks can be derived to $(\infty,\infty)$-ring gestalten level.}, in familes over $\mathbb{N}_\infty$;
\subitem $\blacksquare$ Usually this will mean taking the derived tensor products in the derived $\infty$-categories if such theorem exists; 
\item[A3] A derived $(\infty,2)$-category of Weil sheaves $D\mathrm{Weil}(X)$\footnote{If we use the $(\infty,\infty)$-ring gestalten, we have two choices, one is that we can stop on gestalten and use itself for this, or the other one is to consider infinite stablization sequences of $(\infty,n)$-module categories over the ring gestalten $\mathrm{Sh}_A$.} with values in $(\infty,\infty)$-categories over again the arc site or $v$-site of $X$ satisfying the arc-descent or $v$-descent, in families over $\mathbb{N}_\infty$;
\item[A4] A motivic six-functor formalism $f_A,f_B,f_C,f_D,f_E,f_F$, in familes over $\mathbb{N}_\infty$;
\item[A5] A formalism of Hopf algebraic motivic fundamental groups as in \cite[4.7]{3A}, in families over $\mathbb{N}_\infty$.
\subitem $\blacksquare$ This is more general than just \textit{motivic Galois groups} of a point for $L$ or $A$ in families over $\mathbb{N}_\infty$. To be more precise for each $\mathrm{Sh}_A$, one takes the corresponding $\infty$-level \textbf{suspension} (with respect to the Tate element in \cite{3A}, \cite{3S}) of this ring to get the sheaf $\widetilde{\mathrm{Sh}}_A$, then takes the \v{C}ech conerve $\mathrm{CechConerve}(\widetilde{\mathrm{Sh}}_A)$ as in \cite[4.9]{3A}, then takes the \textbf{$\infty$-loop space} with respect to the Tate object as in \cite{3A}, \cite{3S}: $\mathrm{Loop}^\infty\mathrm{Cechconerve}(\widetilde{\mathrm{Sh}}_A)$, this is then the desired Hopf \textit{sheaf} of algebra over the arc site or $v$-site of $X$. One then takes the spectrum of this sheaf to get the motivic fundamental group \textit{sheaf} $\mathrm{Spec}\mathrm{Loop}^\infty\mathrm{Cechconerve}(\widetilde{\mathrm{Sh}}_A)$. One can also take the global section over $X$ to reach the algebra and group theoretic objects. This process is compatible with the definition of the Berkovich motives in families: i.e. one considers all the small arc stacks or $v$-stacks over $X$ in arc topology or $v$-topology:
\begin{align}
\mathrm{Cate}(\mathrm{Stack}_{\mathrm{*},X}),
\end{align}
where $*$ is arc or $v$, then takes the localization using the loop space functor at $\infty$-level from the Tate object (over pointed space for the projective space of dimension 1) to define the effective motives in families:
\begin{align}
\mathrm{Effe}\mathrm{Cate}(\mathrm{Stack}_{\mathrm{*},X}),
\end{align}
then one inverts the Tate object to finish the definition of the desired symmetrical monoidal $\infty$-category of arc or $v$-motives in families:
\begin{align}
\mathrm{Cate}(\mathrm{Stack}_{\mathrm{*},X})_{\mathrm{final},\times}.
\end{align}
\end{itemize}
(!) We allow $\mathrm{Sh}_A$ to be an $(\infty,\infty)$-$\mathrm{E}_1$-ring gestalten \footnote{i.e. a noncommutative gestalten $(N_0,N_1,...)$ starting from a \textit{noncommutative ring} object $N_0$ in the derived $\infty$-category $D(\mathbb{S})$ of the spherical element, then being constructed such that $N_j$ is a module monoidal category over ring monoidal category $N_{j-1}$. As in \cite{5S5} is we consider noncommutative stacks as colimits of noncommutative countably presentable noncommutative schemes, then the K\"unneth theorem and six-functor formalism are just diret consequence of this formalism after \cite{5S5} and \cite{5G}. This means in the noncommutative setting gestalten should be a very flexible framework to consider. This should have significant application to noncommutative Tamagawa number conjecture and noncommutative Iwasawa main conjecture, and the locally-analytic representations in $p$-adic local Langlands correspondence, see \cref{ntig}.}. We allow the six-functor formalism to be completely abstract in the sense of derived $\infty$-category (i.e., without really touching the spaces) by requiring the projection formula, smooth base change, and proper base change and so on in pure $\infty$-categorical sense as in \cite{3CS1}. Moreover there is a \textit{derived} version of this situation for small $(\infty,1)$-arc stacks in families by replacing arc stacks in families by $(\infty,1)$-arc stacks in families (which are fiber categories in families over the arc site of simplicial Banach rings in families). 
\end{situation}

\begin{situation}\mbox{\textbf{(Motivic cohomology theories in families II)}}\label{situation2}
We consider all the analytic stacks as in \cite{3CS1} over $A$ fibered over $\mathbb{N}_\infty$. A $z$-adic motivic cohomology theory over any analytic stack $X$ consists of the datum as in the following after \cite{3A}:
\begin{itemize}\setlength{\itemsep}{-0.1cm}
\item[A1] An $(\infty,\infty)$-ring sheaf with values in $\infty$-category (or an ($\infty,\infty$)-ring gestalten $(\mathrm{Sh}_0, \mathrm{Sh}_1,...)$) $\mathrm{Sh}_A$, with further action from $A$ over a fixed analytic stack $X$ over $A$ (again one consider all the analytic stacks over $X$ in analytic topology to form the analytic site) in analytic topology\footnote{$\mathrm{Sh}_A$ are $(\infty,\infty)$-stacks over the site $X_{an}$, or are $(\infty,\infty)$-ring gestalten over this site. This includes prismatization stack $P'_X$ associated with such $X$, then we have the gestalt $P'_X$:
\begin{align}
P'_X := (\mathcal{O}_{P_X}, D(P_X), 1Pr_{P_X},...).
\end{align}}, in familes over $\mathbb{N}_\infty$;
\item[A2] A version of K\"unneth theorem in the derived sense holds for $\mathrm{Sh}_A$, in familes over $\mathbb{N}_\infty$;
\item[A3] A derived $(\infty,2)$-category of Weil sheaves $D\mathrm{Weil}(X)$ over the analytic site of $X$ satisfying the descent under the !-functor formalism as in \cite{3CS1}, in families over $\mathbb{N}_\infty$; 
\item[A4] A motivic six-functor formalism $f_A,f_B,f_C,f_D,f_E,f_F$, in familes over $\mathbb{N}_\infty$;
\item[A5] A formalism of Hopf algebraic motivic fundamental groups as in \cite[4.7]{3A}, in families over $\mathbb{N}_\infty$.
\subitem $\blacksquare$ This is more general than just \textit{motivic Galois groups} of a point for $L$ or $A$ in families over $\mathbb{N}_\infty$. To be more precise for each $\mathrm{Sh}_A$, one takes the corresponding $\infty$-level suspension (with respect to the Tate element in \cite{3A}, \cite{3S}) of this ring to get the sheaf $\widetilde{\mathrm{Sh}}_A$, then takes the \v{C}ech conerve $\mathrm{Cechconerve}(\widetilde{\mathrm{Sh}}_A)$ as in \cite[4.9]{3A}, then takes the $\infty$-loop space with respect to the Tate object as in \cite{3A}, \cite{3S}: $\mathrm{Loop}^\infty\mathrm{Cechconerve}(\widetilde{\mathrm{Sh}}_A)$, this is then the desired Hopf \textit{sheaf} of algebra over the arc site or $v$-site of $X$. One then takes the spectrum of this sheaf to get the motivic fundamental group \textit{sheaf} $\mathrm{Spec}\mathrm{Loop}^\infty\mathrm{Cechconerve}(\widetilde{\mathrm{Sh}}_A)$. One can also take the global section over $X$ to reach the algebra and group theoretic objects. This process is compatible with the definition of the analytic  motives in families: i.e. one considers all the analytic condensed stacks over $X$ in analytic topology, then takes the localization using the loop space functor at $\infty$-level from the Tate object (over pointed space for the projective space of dimension 1) to define the effective motives in families, then one inverts the Tate object to finish the definition.
\end{itemize}
$(!)$ We allow $\mathrm{Sh}_A$ to be an $(\infty,\infty)$-$\mathrm{E}_1$-ring gestalten object. We allow the six-functor formalism to be completely abstract in the sense of derived $\infty$-category (i.e. without really touching the spaces) by requiring the projection formula, smooth base change, and proper base change and so on in pure $\infty$-categorical sense as in \cite{3CS1}.

\end{situation}

\begin{remark}
We consider \textit{schemes} (so then $z$-adic $A$-\textit{formal schemes} as well by taking the projective limit over schemes) over $A$ with fibration over $\mathbb{N}_\infty$ in these theories as well, as in \cite{3S} by taking the discrete norms. But we only regard schemes as small arc-stacks in arc topology. This meams we will regard derived formal stacks as derived arc stacks, and we will regard derived algebraic stacks as derived arc stacks as well.
\end{remark}

\begin{remark}
Our consideration is of course a more generalized version of the corresponding consideration in \cite{3A} by using the foundation from \cite{3S}. \cite{3A} considered the motives for rigid analytic varieties.
\end{remark}

\begin{remark}\mbox{(\textbf{Arithmetic $D$-modules in families from $\infty$})}
\cite{3A1} can be integrated into a motivic theory in the sense above by considering the corresponding algebraic stack setting. For this choose some identification from $\infty$ and some finite integer, and translate the theory of arithmetic $D$-module from $A_\infty/z_\infty$ to $A_i/z_i$ for some $i \in \mathbb{N}$, then take the corresponding inverse limit we can reach some analytic version of the arithmetic $D$-module theory in families over $\mathbb{N}$, where one can enlarge the consideration in \cite{3A} to the corresponding schemes over $A$ and then to small arc-stacks over $L$ for instance. One can see that the theory satisfies all the 5 conditions in the above considerations. We have no ideas about this construction on how it may be linked to other $p$-adic motivic cohomology theories, however over $A(\mathbb{N})$ the theory of $F$-isocrystals is not that far away from this. In such a way the existence of 6-functor formalism holds largely due to a tight correspondence between the $\infty\in \mathbb{N}$ and any finite number on the ring level, which leads to the corresponding 6-functor formalism of arithmetic $D$-modules in the category of rigid analytic spaces, again of course in families. However this process on the other hand provides the corresponding 6-functor formalism for arithmetic $D$-modules over rigid analytic spaces in some nontrivial way, in families.
\end{remark}

\begin{remark}
Here are the motivic cohomology theories in families relevant in the current consideration on motivic cohomology theories:
\begin{itemize}\setlength{\itemsep}{-0.1cm}
\item[1] Robba sheaves, and solid quasicoherent sheaves over them, the consideration is for $v$-stacks or arc-stacks over $A$ in families over $\mathbb{N}_\infty$;
\item[2] Prismatization over $A$ in families $\mathbb{N}_\infty$;
\subitem[2I] Prismatization over $A$ in families $\mathbb{N}_\infty$;
\subitem[2II] Filtration prismatization over $A$ in families $\mathbb{N}_\infty$;
\subitem[2III] Syntomization prismatization over $A$ in families $\mathbb{N}_\infty$;
\item[3] de Rham prismatization over $A$ in families $\mathbb{N}_\infty$;
\subitem[3I] de Rham prismatization over $A$ in families $\mathbb{N}_\infty$;
\subitem[3II] de Rham filtration prismatization over $A$ in families $\mathbb{N}_\infty$;
\subitem[3III] de Rham syntomization prismatization over $A$ in families $\mathbb{N}_\infty$;
\item[4] de Rham-Hodge-Tate prismatization over $A$ in families $\mathbb{N}_\infty$;
\subitem[4I] de Rham-Hodge-Tate prismatization over $A$ in families $\mathbb{N}_\infty$;
\subitem[4II] de Rham-Hodge-Tate filtration prismatization over $A$ in families $\mathbb{N}_\infty$;
\subitem[4III] de Rham-Hodge-Tate syntomization prismatization over $A$ in families $\mathbb{N}_\infty$;
\item[5] Laurent prismatization over $A$ in families $\mathbb{N}_\infty$;
\subitem[5I] Laurent prismatization over $A$ in families $\mathbb{N}_\infty$;
\subitem[5II] Laurent filtration prismatization over $A$ in families $\mathbb{N}_\infty$;
\subitem[5III] Laurent syntomization prismatization over $A$ in families $\mathbb{N}_\infty$;
\item[6] Analytic prismatization over $A$ in families $\mathbb{N}_\infty$;
\subitem[6I] Analytic prismatization over $A$ in families $\mathbb{N}_\infty$;
\subitem[6II] Analytic filtration prismatization over $A$ in families $\mathbb{N}_\infty$;
\subitem[6III] Analytic syntomization prismatization over $A$ in families $\mathbb{N}_\infty$;
\item[7] Analytic de Rham prismatization over $A$ in families $\mathbb{N}_\infty$;
\subitem[7I] Analytic de Rham prismatization over $A$ in families $\mathbb{N}_\infty$;
\subitem[7II] Analytic de Rham filtration prismatization over $A$ in families $\mathbb{N}_\infty$;
\subitem[7III] Analytic de Rham syntomization prismatization over $A$ in families $\mathbb{N}_\infty$;
\item[8] Analytic de Rham-Hodge-Tate prismatization over $A$ in families $\mathbb{N}_\infty$;
\subitem[8I] Analytic de Rham-Hodge-Tate prismatization over $A$ in families $\mathbb{N}_\infty$;
\subitem[8II] Analytic de Rham-Hodge-Tate filtration prismatization over $A$ in families $\mathbb{N}_\infty$;
\subitem[8III] Analytic de Rham-Hodge-Tate syntomization prismatization over $A$ in families $\mathbb{N}_\infty$;
\item[9] Analytic Laurent prismatization over $A$ in families $\mathbb{N}_\infty$;
\subitem[9I] Analytic Laurent prismatization over $A$ in families $\mathbb{N}_\infty$;
\subitem[9II] Analytic Laurent filtration prismatization over $A$ in families $\mathbb{N}_\infty$;
\subitem[9III] Analytic Laurent syntomization prismatization over $A$ in families $\mathbb{N}_\infty$;
\item[10] $B_{+,\mathrm{dR},\mathbb{N}_\infty,*}$-cohomology theory in families over $\mathbb{N}_\infty$;
\subitem[$\blacksquare$] This can be derived from the de Rham prismatizations in families over $\mathbb{N}_\infty$ after \cite{3Ta}, \cite{3F}, \cite{3S3};
\item[11] $B_{+,\mathrm{dR},\mathrm{Nygaard},\mathbb{N}_\infty,*}$-cohomology theory in families over $\mathbb{N}_\infty$;
\subitem[$\blacksquare$] This can be derived from the de Rham filtration prismatizations in families over $\mathbb{N}_\infty$ after \cite{3Ta}, \cite{3F}, \cite{3S3};
\item[12] $B_{+,\mathrm{dR},\mathrm{syntomization},\mathbb{N}_\infty,*}$-cohomology theory in families over $\mathbb{N}_\infty$;
\subitem[$\blacksquare$] This can be derived from the de Rham syntomization  prismatizations in families over $\mathbb{N}_\infty$ after \cite{3Ta}, \cite{3F}, \cite{3S3};
\item[13] and more motivic cohomology theories.
\end{itemize}
\end{remark}

\noindent We made further discussion in this paper on applications of corresponding motivic cohomology theories in families after \cite{5V}, \cite{5AI}, \cite{5AII}, \cite{5S2}, \cite{5G}, to Riemann-Hilbert correspondence in families. After one realizes things in the motivic homotopy categories, the resulting consideration will be then highly homotopical. This is a process of \textit{homotopicalization} in the motivic setting. In our current $p$-adic setting in fact there is a \textit{completely parallel correspondence} between the homotopy category using spectra and the motivic homotopy categories using presheaves over big Grothendieck categories, which basically indicate further the \textit{analogues} among the following \textit{four different situations}:
\begin{itemize}\setlength{\itemsep}{-0.1cm}
\item[1] (The derived $\infty$-category of abelian groups over $\mathbb{Z}_p$) Homological algebra and homotopical algebra in the sense of even Lurie \cite{5Lu1};
\item[2] (The derived $\infty$-category of modules spectra over $p$-adic sphere spectrum\footnote{For instance this can be the Witt vector sphere spectrum $\mathbb{S}^1_W$.}) Homotopy theory and homotopy categories;
\item[3] (The derived $\infty$-category of modules spectra over motivic bigraded $p$-adic  sphere spectrum\footnote{For instance this can be the bigraded Witt vector sphere specture $\mathbb{S}_W^{p,q}\wedge \mathbb{G}_{m,W}^q$ when $p=1, q=0$, taking the fibre product of $\mathbb{S}^{p,q}\wedge \mathbb{G}_m^q$ with the Witt vector sphere spectrum in [2].}) Motivic homotopy category and motivic homotopy theory;
\item[4] (The derived $\infty$-caetgory of modules spectra over prismatic motivic sphere spectrum\footnote{In general we have the prismatic lift instead of Witt lift of the bigraded sphere specture $\mathbb{S}^{p,q}\wedge \mathbb{G}_m^q$. Take the mod $p$ of the homotopy group of the usual sphere and lift the spectrum of the mod $p$ sphere specture by using prism associate to $\mathbb{F}_p$.}) Prismatic motivic homotopy category and prismatic motivic homotopy theory as in \cite{5Lu2}.
\end{itemize}

\begin{remark}
All these four settings have the gestalten versions, by all can be defined over $D(\mathbb{S})$ no matter how we construct the $p$-complete sphere spectrums. The idea in our discussion is that we can promote the results using homological algebra to the motivic homotopical setting, then promote the results to gestalten level, which will have very deep consequences following \cite{5S5} and \cite{5Ga}.
\end{remark}

We use again the following setting:

\begin{setting}
We now consider the setting in \cite{5LH}. Throughout we fix a finite field $F$, and choose $p$-adic local fields of characteristic zero with same residue field $F$:
\begin{align}
H_1,H_2,H_3,H_4,H_5,...,H_\infty
\end{align}
where the index is throughout all the natural numbers. There is one function field $F((x_\infty))$ at the infinity. Therefore we can now form one uniformizer for all of them: $x$ such that $x$ gives rise to $x_n$ at each point in $\mathbb{N}_c$. We use the notation $H$ to denote the product of these local fields, and we use $I_H$ to denote the integral subrings as in \cite{5LH}.
\end{setting}

\indent To have a well-defined understanding on Riemann-Hilbert correspondence through the main aspect we discussed in this paper thoroughly, in this paper we start to motivicalize and homotopicalize the Riemann-Hilbert correspondence, in some very generalized sense. It is well-known that quite original consideration was in \cite{5H}, i.e. the famous Hilbert $21^\mathrm{st}$ problem.
\begin{theorem}
Let $\mathrm{D}_{\mathrm{aHT},\mathrm{lisse},V}(S_{\mathrm{Robba},\mathrm{arc},*})$ be the derived $\infty$-category generated by locally finite free almost Hodge-Tate Frobenius sheaves over Robba rings as in our current family situation. Consider the category 
\begin{align}
\mathrm{Motives}_{S_{\mathrm{Robba},\mathrm{arc},*},\times}\overline{\mathrm{ArcStacks}}_{V,\mathrm{arc}}.
\end{align}
Here $\overline{\mathrm{ArcStacks}}$ denotes the restriction to smooth rigid spaces over $V$ from all the arc stacks over $V$. For each covering arc stack $W$ over $V$ consider the corresponding almost Hodge-Tate sheaves over the Robba sheaves over $W$:
\begin{align}
\mathrm{D}_{\mathrm{aHT},\mathrm{lisse},W}(S_{\mathrm{Robba},\mathrm{arc},W})
\end{align}
which can be arranged to the motivic $\infty$-category over the big site over $V$\footnote{Recall that the motivic homotopy of $\infty$-sheaves basically carries a coefficient $\mathrm{Ring}$ or just the most general $\infty$-groupoids $\infty\mathrm{groupoid}$ coefficient (or \textit{presentable $(\infty,1)$-category coefficient} as in \cite{5S4}, which then ends up with $(\infty,2)$-categories even under localizations), a big Grothendieck site $G$, and taking the Take object $\Sigma$ to form the loop space stablization, after a localization through $B_1$ localication, finally we have the symmetrical monoidal $\infty$-categorical structure which we use the notation $\mathrm{Motives}_{\mathrm{Ring},\times}(G)$ or $\mathrm{Motives}_{\infty\mathrm{groupoid},\times}(G)$ to denote that.}:
\begin{align}
\mathrm{Motives}_{S_{\mathrm{Robba},\mathrm{arc},*},\times}\overline{\mathrm{ArcStacks}}_{V,\mathrm{arc}},
\end{align}
taking the subcategory generated by almost Hodge-Tate sheaves we have:
\begin{align}
\mathrm{Motives}^{\mathrm{aHT},\mathrm{lisse}}_{S_{\mathrm{Robba},\mathrm{arc},*},\times}\overline{\mathrm{ArcStacks}}_{V,\mathrm{arc}}.
\end{align}
Consider the following differential module:
\begin{align}
\partial_W=\mathrm{Descent}_{W^\mathrm{rig}_\mathrm{arc},W,*}(S_{\mathrm{dR},-,\mathbb{N}_c}[\mathrm{log}(t)]).
\end{align}
Varying $W$ we have the sheaf over then the site of all the smooth rigid spaces over $V$:
\begin{align}
\partial=\mathrm{Descent}_{(\cdot/V)^\mathrm{rig}_\mathrm{arc},\cdot/V,*}(S_{\mathrm{dR},-,\mathbb{N}_c}[\mathrm{log}(t)]).
\end{align}
This is from the motivic cohomology theory over the analytic topology for smooth rigid spaces. Here we regard $S_{\mathrm{dR},H,\mathbb{N}_c}[\mathrm{log}(t)]$ also as the motivic cohomology $\infty$-sheaf of ring over the big site of all the arc stacks over $V$. Then we have a well-defined functor:
\begin{align}
\mathrm{Motives}^{\mathrm{aHT},\mathrm{lisse}}_{S_{\mathrm{Robba},\mathrm{arc},*},\times}\overline{\mathrm{ArcStacks}}_{V,\mathrm{arc}}\longrightarrow \mathrm{D}_{\mathrm{inductive},\mathrm{coherent}}(\mathrm{susp}_\Sigma\mathrm{Descent}_{(\cdot/V)^\mathrm{rig}_\mathrm{arc},(\cdot/V)^\mathrm{rig},*}(S_{\mathrm{dR},-,\mathbb{N}_c}[\mathrm{log}(t)]))
\end{align}
which is symmetrical monoidal $\infty$-categorical functor into the $D$-modules in families in our current setting. Here the target category is defined to be the derived $\infty$-category of inductive systems of the coherent $D$-modules as in \cite{5GR1}, \cite{5GR2}. By considering the corresponding filtration de Rham period rings and the corresponding syntomization de Rham period rings as in the following we have two another versions of motivic Riemann Hilbert correspondences. The construction above is well-defined. The non-motivic version in \cref{theorem3} of this functor is fully-faithful. These motivic homotopical symmetrical monoidal Riemann-Hilbert functors can be derived to be functors on gestalten associated to the both sites, which is compatible with the abstract K\"unneth theorem (for large enough degrees in the gestalten) for underlying spaces $V,V'$ on both sides, and the $(\infty,\infty)$-categorical gestalten six-functor formalisms on the both sides when $V$ is varying in algebraic stacks.
\end{theorem}

\newpage
\section{Symmetrical Monoidal $\infty$-Categories in $\mathbb{F}_1$-Analytic Geometry}

\subsection{Grothendieck Categoricalizations}

\indent \cite{BBBK}, \cite{CS1}, \cite{CS2}, \cite{CS3} enlarge the categories of the rings in order to study the analytic geometry over $\mathbb{F}_1$. For instance in \cite{BBBK} we have the following construction. Several constructions on symmetrical monoidal categories are constructed with certain default tensor product. For instance in \cite{BBBK}, over $\mathbb{F}_1$ it just considers all the Banach sets which forms the corresponding category $\underline{\mathrm{BanachSets}}_*$. This category is not the ideal one since it is not easy to construct some stable $\infty$-categories directly from it to do desired analytic geometry, i.e. to form well-defined ringed-spaces and ringed-stacks. The point of view is to take the corresponding inductive limits over the category to form the $\underline{\mathrm{IndBanachSets}}_*$. Another closely related construction is to look at the corresponding monomorphic morphisms when one forms the inductive limits: $\underline{\mathrm{Ind_{\mathrm{mono}}BanachSets}}_*$. The latter two are actually symmetrical monoidal tensor categories when can then form the corresponding stable $\infty$-categories from them, we use the notations as in the following to denote the corresponding stable $\infty$-categories:
\begin{align}
\underline{\mathrm{IndBanachSets}}^\sharp_*, \underline{\mathrm{Ind_{\mathrm{mono}}BanachSets}}^\sharp_*.
\end{align}

\begin{remark}
One can enlarge the categories by considering the corresponding seminormed and normed sets over $*$:
\begin{align}
&\underline{\mathrm{IndNSets}}^\sharp_*, \underline{\mathrm{Ind_{\mathrm{mono}}NSets}}^\sharp_*,\\
&\underline{\mathrm{IndSNSets}}^\sharp_*, \underline{\mathrm{Ind_{\mathrm{mono}}SNSets}}^\sharp_*.
\end{align}
\end{remark}

\subsection{Modules}

For the $A_\infty$-rings or $E_1$-rings or $E_\infty$-rings in the previous discussion the modules over them are actually quite complicated as in \cite{BBK}. On the other hand they are very significant in the quasicoherent sheaf theoretic consideration in \cite{BBK}. We replace the base by some Banach ring $*$ and we consider not just sets but instead the modules over $*$ carrying the topologizations as in the above. For instance in \cite{BBBK}, over $*$ it just considers all the Banach modules which forms the corresponding category $\underline{\mathrm{BanachModules}}_*$. This category is not the ideal one since it is not easy to construct some stable $\infty$-categories directly from it to do desired analytic geometry, i.e. to form well-defined ringed-spaces and ringed-stacks. The point of view is to take the corresponding inductive limits over the category to form the $\underline{\mathrm{IndBanachModules}}_*$. Another closely related construction is to look at the corresponding monomorphic morphisms (i.e., born\'e) when one forms the inductive limits: $\underline{\mathrm{Ind_{\mathrm{mono}}BanachModules}}_*$. The latter two are actually symmetrical monoidal tensor categories when can then form the corresponding stable $\infty$-categories from them, we use the notations as in the following to denote the corresponding stable $\infty$-categories:
\begin{align}
\underline{\mathrm{IndBanachModules}}^\sharp_*, \underline{\mathrm{Ind_{\mathrm{mono}}BanachModules}}^\sharp_*.
\end{align}

\begin{remark}
One can enlarge the categories by considering the corresponding seminormed and normed sets over $*$:
\begin{align}
&\underline{\mathrm{IndNModules}}^\sharp_*, \underline{\mathrm{Ind_{\mathrm{mono}}NModules}}^\sharp_*,\\
&\underline{\mathrm{IndSNModules}}^\sharp_*, \underline{\mathrm{Ind_{\mathrm{mono}}SNModules}}^\sharp_*.
\end{align}
\end{remark}

\subsection{Generalized Hodge modules}

For instance in \cite{BBKK} and \cite{CS3}, we have the well-defined structure sheaf $RX(.)$ over $X$ any $p$-adic analytic space. The structure sheaf provides the definition of quasicoherent sheaves. In \cite{T} we used the $p$-adic functional analytic method from \cite{K1} and \cite{KL} to have studied some spaces emerges after \cite{CKZ} and \cite{PZ}. Though spaces we looked at are those regular higher dimensional rigid analytic spaces, i.e. just multi-discs, the essential consideration and application in mind are quite new. Approximately we have the ring of analytic funcions over a multi-disc with parameter $[a_I,b_I]$ with $I$ some finite set. This ring is the higher dimensional version of the rings studied in \cite{KPX}. We use the notation $L_{[a_I,b_I]}$ to denote these rings. Recall the multi-rigid intervals take the form of:
\begin{align}
p^{-b_i{p/(p-1)}}<||x||_{\mathbb{Q}_p}<p^{-a_i{p/(p-1)}}, \forall i \in I.
\end{align}
Then taking inverse limit on $a_I$ we have the version of rings with just multi-radii $L_{b_I}$. These rings will give the full version of the ring we are considering if we take union on $b_I$ we then use the notation $L$ to denote this. So we have:
\begin{definition} \label{definition1}
We have defined as above the key rings in consideration:
\begin{align}
L_{[a_I,b_I]},L_{b_I}:=L_{(0,b_I]},L_{[0,b_I]},L:=\bigcup_{b_I>0}\bigcap_{a_I>0}L_{[a_I,b_I]}.
\end{align}
Here $b_I$ can be $+\infty$ towards all directions.
\end{definition}

We then enlarge the coeffcients to $I$ different finite extensions of $\mathbb{Q}_p$: $K_1$,...,$K_{|I|}$. Where we have the corresponding  
\begin{align}
L_{[a_I,b_I],K_I},L_{b_I,K_I}:=L_{(0,b_I],K_I},L_{[0,b_I],K_I},L_{K_I}:=\bigcup_{b_I>0}\bigcap_{a_I>0}L_{[a_I,b_I],K_I}.
\end{align}

One also has the version of these rings where we have the corresponding variable substitutions by using the uniformizer $\pi_{K_I}$ for different directions, as well as the group $\Gamma_{K_I}$ for different directions (by using the generator for each component in this product $\Gamma_{K_1}\times...\times\Gamma_{K_{|I|}}$):
\begin{align}
L_{[a_I,b_I],K_I}(\pi_{K_I}),L_{b_I,K_I}(\pi_{K_I}):=L_{(0,b_I],K_I}(\pi_{K_I}),L_{[0,b_I],K_I}(\pi_{K_I}),L_{K_I}(\pi_{K_I}):=\bigcup_{b_I>0}\bigcap_{a_I>0}L_{[a_I,b_I],K_I}(\pi_{K_I}),
\end{align}
\begin{align}
L_{[a_I,b_I],K_I}(\Gamma_{K_I}),L_{b_I,K_I}(\Gamma_{K_I}):=L_{(0,b_I],K_I}(\Gamma_{K_I}),L_{[0,b_I],K_I}(\Gamma_{K_I}),L_{K_I}(\Gamma_{K_I}):=\bigcup_{b_I>0}\bigcap_{a_I>0}L_{[a_I,b_I],K_I}(\Gamma_{K_I}).
\end{align}

\begin{definition} \label{definition2}
Recall from \cite{T} we have the multi Frobenius actions and the multi $\Gamma_{K_I}$ actions over these rings. The corresponding actions through different directions are mutually commutative and actions along different directions will be identity map. Therefore we have the corresponding Frobenius-Hodge modules over these rings (in further relativization by tensoring with some rigid affinoid $X/\mathbb{Q}_p$):
\begin{align}
L^X_{[a_I,b_I],K_I}(\pi_{K_I}),L^X_{b_I,K_I}(\pi_{K_I}):=L^X_{(0,b_I],K_I}(\pi_{K_I}),L^X_{[0,b_I],K_I}(\pi_{K_I}),L^X_{K_I}(\pi_{K_I}):=\bigcup_{b_I>0}\bigcap_{a_I>0}L^X_{[a_I,b_I],K_I}(\pi_{K_I}).
\end{align}
A Frobenius-Hodge module over $L^X_{K_I}(\pi_{K_I})$ is defined to be a Frobenius-Hodge module over some $L^X_{b_I,K_I}(\pi_{K_I})$ through the base change when we take the union over $b_I$, while the latter is defined in the following way. It is defined to be a finite projective module over $L^X_{b_I,K_I}(\pi_{K_I})$ carrying semilinear Frobenius action along each direction $i\in I$ such as:
\begin{align}
\varphi^*F \overset{\sim}{\rightarrow} F
\end{align}
holds true after base change to $L^X_{b_I/b,K_I}(\pi_{K_I})$. We have the notion of a Frobenius-Hodge module over $L^X_{[a_I,b_I],K_I}(\pi_{K_I})$, which is defined to be a finite projective module over  $L^X_{[a_I,b_I],K_I}(\pi_{K_I})$ such as:
\begin{align}
\varphi^*F \overset{\sim}{\rightarrow} F
\end{align}
holds true after base change to $L^X_{[a_I,b_I/p],K_I}(\pi_{K_I})$. Here we assume over each direction $i\in I$ we have that $[a_i,b_i]\cap[a_i/p,b_i/p]=[a_i,b_i/p]\neq \emptyset$. This also leads to the notion of Frobenius-Hodge bundles over all such multi-intervals with the requirement: over each direction $i\in I$ we have that $[a_i,b_i]\cap[a_i/p,b_i/p]=[a_i,b_i/p]\neq \emptyset$. It is defined to be a collection of Frobenius-Hodge modules over all $\{[a_I,b_I]|\forall i\in I, [a_i,b_i]\cap[a_i/p,b_i/p]=[a_i,b_i/p]\neq \emptyset\}$ in a compatible way such as over each $[a_I,b_I/p]$ the sections are assumed to be isomorphic. As in \cite{T} we have the $\Gamma_{K_I}$-actions over these rings which form the objects which are called $\Gamma$-Frobenius modules and bundles.
\end{definition}

\begin{definition}
Over $*$=
\begin{align}
L^X_{[a_I,b_I],K_I}(\pi_{K_I}),L^X_{b_I,K_I}(\pi_{K_I}):=L^X_{(0,b_I],K_I}(\pi_{K_I}),L^X_{[0,b_I],K_I}(\pi_{K_I}),L^X_{K_I}(\pi_{K_I}):=\bigcup_{b_I>0}\bigcap_{a_I>0}L^X_{[a_I,b_I],K_I}(\pi_{K_I}),
\end{align} 
we have the notion of $\Gamma$-Frobenius modules and bundles as above. Then we consider the derived $\infty$-categories 
\begin{align}
\underline{\mathrm{IndBanachModules}}^\sharp_*, \underline{\mathrm{Ind_{\mathrm{mono}}BanachModules}}^\sharp_*.
\end{align}
which are stable. And we have the condensed version:
\begin{align}
\blacksquare D_{*}, \blacksquare D_{*,\mathrm{bounded}},\blacksquare D_{*,\mathrm{perfect}}.
\end{align}
Therefore we have the corresponding Banach perfect complexes and condensed perfect complexes of the corresponding objects in our setting, namely we consider the perfect complexes of the $\Gamma$-Frobenius-Hodge modules in our setting, which again form certain stable $\infty$-categories with symmetrical monoidal structures. In the Banach setting we use the notations:
\begin{align}
&\underline{\mathrm{IndBanachModules}}^\sharp_{*,\mathrm{perfect},\Gamma,F}, \underline{\mathrm{Ind_{\mathrm{mono}}BanachModules}}^\sharp_{*,\mathrm{perfect},\Gamma,F},\\
&\underline{\mathrm{IndBanachModules}}^\sharp_{*,\mathrm{perfect},\mathrm{bounded},\Gamma,F}, \underline{\mathrm{Ind_{\mathrm{mono}}BanachModules}}^\sharp_{*,\mathrm{perfect},\mathrm{bounded},\Gamma,F},\\
&\underline{\mathrm{IndBanachModules}}^\sharp_{*,\mathrm{perfect},-,\Gamma,F}, \underline{\mathrm{Ind_{\mathrm{mono}}BanachModules}}^\sharp_{*,\mathrm{perfect},-,\Gamma,F},
\end{align}
to denote these complexes. And in the corresponding condensed setting we use:
\begin{align}
\blacksquare D_{*,\Gamma,F}, \blacksquare D_{*,\mathrm{bounded},\Gamma,F},\blacksquare D_{*,\mathrm{perfect},\mathrm{bounded},\Gamma,F},\blacksquare D_{*,\mathrm{perfect},\Gamma,F},\blacksquare D_{*,\mathrm{perfect},-,\Gamma,F}.
\end{align}
\end{definition}

\subsection{Derived $\infty$-Categories}

\begin{definition}
If $I$ is singleton, then we have the notion of $(F,\Gamma)$-complex of any $\Gamma$-Frobenius-Hodge module $F$: $C_{F,\Gamma}(F)$, where we have also the $C_{F,*}(F)$ and $C_{*,\Gamma}(F)$ complexes as well. Using them by induction we have the corresponding notion of $(F,\Gamma)$-complex of any $\Gamma$-Frobenius-Hodge module $F$: $C_{F,\Gamma}(F)$ when $I$ is not just a singleton. In our setting for each $i\in I$ we also have the corresponding $W_i=\varphi_i^{-1}$-operator. In such a way one can form the complex $C_{W}(F)$ directly.
\end{definition}

\begin{theorem}
$C_{F,\Gamma}(F)$ is in bounded $(\infty,1)$-derived category of complexes over $X$, restricting to perfect complexes:
\begin{align}
D_{\mathrm{perfect},\mathrm{bounded}}(\mathrm{Mod}_X).
\end{align}
$C_{F,\Gamma}(F)$ is in 
\begin{align}
\underline{\mathrm{IndBanachModules}}^\sharp_{X,\mathrm{perfect},\mathrm{bounded}}, \underline{\mathrm{Ind_{\mathrm{mono}}BanachModules}}^\sharp_{X, \mathrm{perfect},\mathrm{bounded}}.
\end{align} 
$C_{W}(F)$ is in 
\begin{align}
\underline{\mathrm{IndBanachModules}}^\sharp_{L^X_{\infty_I,K_I}(\Gamma_{K_I}),\mathrm{perfect},\mathrm{-}}, \underline{\mathrm{Ind_{\mathrm{mono}}BanachModules}}^\sharp_{L^X_{\infty_I,K_I}(\Gamma_{K_I}), \mathrm{perfect},\mathrm{-}},
\end{align}
where $X$ is just $\mathbb{Q}_p$.
\end{theorem}

\begin{proof}
The statements for $C_{F,\Gamma}(F)$ are \cite{T}. For $C_{W}(F)$, over $L^X_{\infty_I,K_I}(\Gamma_{K_I})$ the same argument in \cite{T} works here where one just takes the corresponding projective resolutions in our stable $\infty$-categories in the current setting.
\end{proof}

\begin{remark}
Here one can in fact work with the usual derived categories after Grothendieck categoricalization, which also produces stable $\infty$-categories. To be more precise we consider 
\begin{align}
\mathrm{Ind}\mathrm{Mod}_{L^X_{\infty_I,K_I}(\Gamma_{K_I})}
\end{align}
which is then Grothendieck. Then we have the following stable derived $\infty$-category of all the complexes formed from the objects in the above category:
\begin{align}
D_{\mathrm{bounded}}(\mathrm{Ind}\mathrm{Mod}_{L^X_{\infty_I,K_I}(\Gamma_{K_I})}), D_{\mathrm{bounded}}(\mathrm{Ind}\mathrm{Mod}_{L^X_{\infty_I,K_I}(\Gamma_{K_I})}).
\end{align}
Then one can further restrict to those complexes with cohomology groups in $\mathrm{Mod}_{L^X_{\infty_I,K_I}(\Gamma_{K_I})}$ to achive the desired $\infty$-categories.
\end{remark}

We have the following parallel result by using the foundation in \cite{CS1}, \cite{CS2}, \cite{CS3}:

\begin{theorem}
$C_{F,\Gamma}(F)$ is in bounded $(\infty,1)$-derived category of complexes over $X$, restricting to perfect complexes:
\begin{align}
D_{\mathrm{perfect},\mathrm{bounded}}(\mathrm{Mod}_X).
\end{align}
$C_{F,\Gamma}(F)$ is in 
\begin{align}
\blacksquare\underline{D}_{X,\mathrm{perfect},\mathrm{bounded}}.
\end{align} 
$C_{W}(F)$ is in 
\begin{align}
\blacksquare\underline{D}_{L^X_{\infty_I,K_I}(\Gamma_{K_I}),\mathrm{perfect},-},
\end{align}
where $X$ is just $\mathbb{Q}_p$.
\end{theorem}

\begin{proof}
The statements for $C_{F,\Gamma}(F)$ are \cite{T}. For $C_{W}(F)$, over $L^X_{\infty_I,K_I}(\Gamma_{K_I})$ again see \cite{T} the argument remains unchange as long as one works in the categories of condensed complexes in the current setting.
\end{proof}

\begin{corollary}
$C_{W}(.)$ induces a derived functor from the stable $\infty$-category 
\begin{align}
\underline{\mathrm{IndBanachModules}}^\sharp_{*,\mathrm{perfect},\mathrm{bounded},\Gamma,F}, \underline{\mathrm{Ind_{\mathrm{mono}}BanachModules}}^\sharp_{*,\mathrm{perfect},\mathrm{bounded},\Gamma,F},
\end{align}
and in the corresponding condensed setting:
\begin{align}
\blacksquare D_{*,\mathrm{perfect},\mathrm{bounded},\Gamma,F},
\end{align}
to the stable $\infty$-category 
\begin{align}
\blacksquare\underline{D}_{L^X_{\infty_I,K_I}(\Gamma_{K_I}),\mathrm{perfect},-},
\end{align}
where $X$ is just $\mathbb{Q}_p$. Here $*=L_{K_I}^X(\pi_{K_I})$. This also induces the morphism on the $K$-group spectra of $\mathbb{E}_\infty$-rings after applying \cite{BGT} to the corresponding stable $(\infty,1)$-categories, after \cite{G2}, \cite{A2}, \cite{BGT}.
\end{corollary}

\begin{proof}
Now we consider the hyper cohomological functor formed from $C_{W}(.)$ and the complexes $F^\ell$ in the categories, which produces the corresponding spectral sequence $E_k^{.,\ell}=E[C_{W}(.)|F^\ell]$ which realizes the derived functor as desired.
\end{proof}

\begin{conjecture}
$C_{W}(.)$ induces a derived functor from the stable $\infty$-category 
\begin{align}
\underline{\mathrm{IndBanachModules}}^\sharp_{*,\mathrm{perfect},\mathrm{bounded},\Gamma,F}, \underline{\mathrm{Ind_{\mathrm{mono}}BanachModules}}^\sharp_{*,\mathrm{perfect},\mathrm{bounded},\Gamma,F},
\end{align}
and in the corresponding condensed setting:
\begin{align}
\blacksquare D_{*,\mathrm{perfect},\mathrm{bounded},\Gamma,F},
\end{align}
to the stable $\infty$-category 
\begin{align}
\blacksquare\underline{D}_{L^X_{\infty_I,K_I}(\Gamma_{K_I}),\mathrm{perfect},\mathrm{bounded}},
\end{align}
where $X$ is just $\mathbb{Q}_p$. Here $*=L_{K_I}^X(\pi_{K_I})$. This also induces the morphism on the $K$-group spectra of $\mathbb{E}_\infty$-rings after applying \cite{BGT} to the corresponding stable $(\infty,1)$-categories, after \cite{G2}, \cite{A2}, \cite{BGT}.
\end{conjecture}

\subsection{Noncommutative Tamagawa-Iwasawa Gestalten}\label{ntig}

\noindent For the convenience of the readers, we recall and present the foundations here on the noncommutative gestalten following \cite{5S5} and \cite{5S4} on modern theory of motives\footnote{\cite{5S4} mainly considers the category $\mathrm{Sch}^\mathrm{fppf}_\mathbb{Z}$, i.e. \textit{la topologie fid\`elimente plate et present\'e fini}. Over the generality following \cite{5S5} this is replaced by the so-called countably presented maps for stacks, this is certainly a correct notion for gestalten.}. We start from the derived category of the sphere spectrum, $D(\mathbb{S})$, one can also work over some $p$-complete version of the sphere spectrum (Witt liftings, prismatic liftings and so on.). We consider the noncommutative associate ring objects $\mathrm{Alg}(D(\mathbb{S}))$\footnote{They are very general $\mathbb{E}_1$-ring objects as mentioned in \cite{5S5} which include consideration in \cite{CS1}. Note that we are considering rings over $\mathbb{S}$, the sphere spectrum whose $\pi_0$ is $\mathbb{Z}$.} in this category and we conside the ring gestalten by taking the colimits along the following morphisms:
\begin{align}
\mathrm{Alg}(D(\mathbb{S}))\hookrightarrow \mathrm{Alg}_{\mathrm{Alg}(D(\mathbb{S}))}(1Pr^L)\hookrightarrow ...
\end{align}

\noindent The following is the definition in \cite{5S5}:
\begin{definition}\mbox{(Scholze)}
$\mathrm{NoncommGest}(D(\mathbb{S}))$, the $\infty$-category of all the noncommutative gestalten is defined to be the colimit of the morphisms:
\begin{align}
\mathrm{Alg}(D(\mathbb{S}))\hookrightarrow \mathrm{Alg}(\mathrm{Mod}_{\mathrm{Alg}(D(\mathbb{S}))}(1Pr^L))\hookrightarrow ...
\end{align}
\end{definition}
Then as in \cite{5S5} we consider the noncommutative stacks which are simply $(\infty,1)$-sheaves over the opposite category of $\mathrm{Alg}(D(\mathbb{S}))$. Then we consider a possible suitable Grothendieck topology which we take to be the one as in \cite{5S5} consisting of all the presentable maps on noncommutative stacks, which forms a well-defined class of $!$-able maps and a well-defined six-functor formalism. Following \cite{5S5} and \cite{5Ga}, we establish some further foundation results on noncommutative Tamagawa-Iwasawa conjectures improving \cite{T}. The noncommutative motives are then the following universal gestalt:
\begin{align}
&GF_\mathrm{noncommutative}^U\\
&:=(\mathrm{End}_{\mathrm{End}_F(1)}(1),\mathrm{End}_F(1),F:=\mathrm{Fun}^U(\mathrm{AlgStack}_\mathbb{S})_{!-\mathrm{able}}, \mathrm{CoeffCategory}_{\infty,\infty})^{\mathrm{localization,localization}},\\
&\mathrm{Mod}_F(1Pr^L),\mathrm{Mod}_{\mathrm{Mod}_F},......).
\end{align}
The first localization we take the homotopy localization with respect to Tate element in the Berkovich setting, and then we take the loop operators to do the second localization to form the stablization. As in \cite{5S4} we have the generators which are all those noncommutative stacks $[X]$. We then have the realization functor which is a tensor morphism on gestalten from this univesal gestalt to some gestalt over some $E_\infty$-ring. Note that all the stacks here are over the sphere spectrum. If we change the base to some larger stack, then we simpy take the localized category over this stack.

\begin{definition}\mbox{(Scholze, \cite{5S5})}
For all the noncommutative gestalten taking a general form of:
\begin{align}
(G_0\in \mathrm{Alg}(D(\mathbb{S})),G_1,G_2,...),
\end{align}
the $!$-able mappings which realize a well-defined $(\infty,\infty)$-categorical six-functor formalism is defined to be all those mappings which are assumed to countably presented over the first $\mathbb{E}_1$-ring object on $D(\mathbb{S})$. To be more presice, for a map of two noncommutative gestalten $X \rightarrow X'$, we call it countably presented if it takes a form of $X \rightarrow {\lim}_{k\rightarrow \infty} X'_{k}$ where $X'$ is now a countably limit of $\mathrm{E}_1$-objects in $D(\mathbb{S})$ such that each $X\rightarrow X'_{k}$ realize $X'_{k}$ as a $X$-ring of countably presentation. For the category of all $\mathbb{E}_1$-ring objects in $D(\mathbb{S})$ (which is also a site with countably presented topology), we can form the opposite category by formally writing each element $-$ as $\mathrm{NSpec}(-)$. One then defines a corresponding \textit{noncommutative stack} simply as a countable colimit of $\mathrm{NSpec}(X'_k)\rightarrow \mathrm{NSpec}(\mathbb{S}),k\in K$ where each $X'_k$ is countably presented over the sphere spectrum $\mathbb{S}$. Then one puts certain big Grothendieck topology over them, simpy by using countably presented maps which are $!$-able, with the well-defined abstract 6-functor formalism. Then over the site $\mathrm{Alg}(D(\mathbb{S}))$ one defines the \textit{noncommutative $(\infty,\infty)$-stacks} as sheaves with values in $(\infty,\infty)$-categories. They have again $!$-able mappings from countably presented ones, which forms a well-defined abstract 6-functor formalism.
\end{definition}

\begin{remark}\label{remark14}
Again this is just over $\mathbb{S}$, so it is extremely general as long as one lifts things over $\mathbb{Z}$ to this sphere spectrum. For instance if one wants to study the $K$-theory of general nonarchimedean Banach algebras, then one immediately has the abstract 6-functor formalism. This applies directly to $p$-adic modular and Banach representations of $p$-adic Lie groups. To be more precise consider any $p$-adic Lie group $G$, one can form the $p$-adic Iwasawa algebra $\mathbb{Z}_p[[G]]$, then promote all such algebras to be gestalten. Then for all such $G$ one can then directly have a higher abstract 6-functor formalism.
\end{remark}

\begin{theorem}\mbox(Scholze, \cite[Theorem 5.2]{5S5})
For degrees sufficiently large for all noncommutative gestalten, we have $\infty$-categorical six-functor formalism. Where $f^*$ admits a left adjoint $f_\sharp$ which is automatically identical to $f_*$. Therefore we have the formal properness in the abstract six-functor formalism.
\end{theorem}

\noindent The following of course is a noncommutative direct analog of Clausen-Scholze's notion of analytic stacks. 

\begin{definition} (Clausen-Scholze, \cite{CS1})
A noncommutative pre-$\mathbb{E}_1$-analytic ring is defined to be an $\mathbb{E}_1$-ring object $R$ in the category of $\infty$-condensed abelian groups (i.e. sheaves with values in $(\infty,1)$-groupoids over profinite sets) together with a subcategory $D$ of the category of all $(\infty,1)$-$R$-left-module objects $D'$ in the same category of $\infty$-condensed abelian groups. It is then called a noncommutative $\mathbb{E}_1$-analytic ring if we have the closedness  among the following operations:
\begin{itemize}\setlength{\itemsep}{-0.1cm}
\item[1] Taking colimits in $D$;
\item[2] Taking limits in $D$;
\item[3] Taking extensions in $D$;
\item[4] Taking higher extension groups between objects in $D$ and $D'-D$.
\end{itemize}
We then denote the category of all the noncommutative $E_1$-analytic rings as:
\begin{align}
\mathrm{NoAnaRing}_{\mathbb{E}_1}.
\end{align}
Now we put a topology on this category (to be more precise we consider the opposite category) by embedding this category into noncommutative gestalten, and pullback the $!$-able topology for noncommutative gestalten generated by countably presented mappings. In such a way we have the resulting Grothendieck topology and site which are $!$-able:
\begin{align}
\mathrm{NoAnaRing}_{\mathbb{E}_1,!}
\end{align}
Now we define noncommutative analytic stacks as $\infty$-sheaves with values in $\infty$-groupoids over this site, which form the following category:
\begin{align}
\mathrm{NoAnaStacks}.
\end{align}
Then we embed this category into gestalten via forming the gestalten from the structure sheaves. We then simply pullback the $!$-able topology generated by countably presented mappings for gestalten to this category of noncommutative analytic stacks. Then we have a well-defined higher abstract six-functor formalism for noncommutative analytic stacks.
\end{definition}

\begin{definition}
Consider the setting of $\Gamma$-Frobenius \textbf{multivariate} module $M$ over the full multivariate Robba ring $L_{K_I,X}$. Here $X$ is allowed to be noncommutative noetherian Banach rigid affinoids as in \cite{T}. We take the quotient of $\mathrm{colimit}_{J\in (0,\infty)}\mathrm{AnSpec}L_{J,K_I,X}$ with respect to the $\Gamma$-Lie group and $Frobenius^\mathbb{Z}$ further, in the category of gestalten over the category of analytic stacks in the sense of Clausen-Scholze. We call the associate $E_1$-gestalten $\mathcal{R}_X = (R_{0,X}, R_{1,X}, R_{2,X},...)$ \textit{noncommutative (locally-analytic) Tamagawa-Iwasawa gestalten}, where $R_{0,X}\in \mathrm{Alg}(D(\mathbb{S}))$ a noncommutative algebra object in the derived category of the sphere spectrum (the deformed structure sheaf to $X$), $R_{i,X}$ is module category over the monoidal $R_{i-1,X}$. Then one considers the functor in \cite{5S5} from the category of Clausen-Scholze anlaytic stacks to that of all gestalten\footnote{Under the general philosophy from \cite{5S5}, in fact we do not have to really consider $X$ here as a noncommutaive analytic stack in the sense of Clausen-Scholze, this means we only have to consider $X$ as some $\mathbb{E}_1$-animated ring, which is in the large category of condensed abelian groups. This is also enough for applications in Iwasawa theory. To be more precise, what is really happening is for the $\infty$-category of all noncommutative animated ring objects in the condensed animated abelian groups, if we prescribe nothing on the $!$-maps, then the correct definition of such possible presciption of the $!$-maps comes from the pullback of $!$-able maps for noncommutative gestalten. This gives us the flexibility to consider condensed objects as if we are considering no topology at all.}:
\begin{align}
F: \mathrm{CSAnalyticStacks}\rightarrow \mathrm{NoncommGest}
\end{align}
where the pullback functor (taking the preimage) $F^{-1}$ will take these noncommutative Tamagawa-Iwasawa gestalten to the \textit{noncommutative Tamagawa-Iwasawa stacks} generalizing \cite{5M} to the multivariate setting, which gives the definition of these stacks now in our current setting.
\end{definition}

\begin{theorem}
Assume we are in the multivariate setting. $(\infty,\infty)$-categorical six functor formalism:
\begin{align}
f_A,f_B,f_C,f_D,\otimes,\mathrm{Hom}
\end{align}
holds for noncommutative Tamagawa-Iwasawa gestalten, for $X$ over $\mathbb{Q}_p$ varying in the category of noncommutative analytic stacks over $\mathbb{Q}_p$ in the sense of Clausen-Scholze, i.e. they are the sheaves values in $(\infty,\infty)$-categories over the category of all the the $E_1$-analytic rings, with descent under the specific $!$-able topology generated by countably presented maps. In particular, we have the abstract Poincar\'e duality for these gestalten. We also have the K\"unneth theorem for two such gestalten, for degrees sufficiently high as in \cite[Corollary 5.3, Corollary 5.4]{5S5} and \cite{5Ga}. 
\end{theorem}

\begin{proof}
The corresponding stacks in the current setting can be written as certain colimits of stacks coming from countably presented noncommutative algebras. Then we are in the correct noncommutative gestalten setting where we have the !-able maps as in \cite[Corollary 5.3, Corollary 5.4]{5S5}, which implies directly the result, also following \cite{5Ga}.
\end{proof}

\noindent Of course the following is already interesting enough:

\begin{theorem}
Assume we are in the multivariate setting. $(\infty,\infty)$-categorical six functor formalism:
\begin{align}
f_A,f_B,f_C,f_D,\otimes,\mathrm{Hom}
\end{align}
holds for commutative Tamagawa-Iwasawa gestalten, for $X$ over $\mathbb{Q}_p$ varying in the category of commutative analytic stacks over $\mathbb{Q}_p$ in the sense of Clausen-Scholze, i.e. they are the sheaves values in $(\infty,\infty)$-categories over the category of all the the $E_\infty$-analytic rings, with descent under the specific $!$-able topology generated by countably presented maps (induced from commutative gestalten). In particular, we have the abstract Poincar\'e duality for these gestalten. We also have the K\"unneth theorem for two such gestalten, for degrees sufficiently high as in \cite[Corollary 5.3, Corollary 5.4]{5S5} and \cite{5Ga}. 
\end{theorem}

\begin{proof}
The corresponding stacks in the current setting can be written as certain colimits of stacks coming from countably presented commutative algebras. Then we are in the correct commutative gestalten setting where we have the !-able maps as in \cite[Corollary 5.3, Corollary 5.4]{5S5}, which implies directly the result, also following \cite{5Ga}.
\end{proof}

\begin{corollary}
The Iwasawa functoriality holds in the current setting (multivariate $\Gamma$-Frobenius modules), i.e. for any sequence of $p$-adic Lie groups with morphisms $G_1\rightarrow G_2\rightarrow ...\rightarrow G_n$ we have the morphisms on the corresponding noncommutative Tamagawa-Iwasawa gestalten. And in the multivariate setting we have Iwasawa theoretic K\"unneth theorem for the product $G'_1\times G'_2\times ...G'_m$ (of course this will be about the products of noncommutative Tamagawa-Iwasawa gestalten). 
\end{corollary}

\begin{remark}
The K\"unneth theorem is closely after \cite{5Ga} where one has to consider higher categorical element in the gestalten. We currently have no idea on whether one can degenerate this to some lower degree elements in the gestalten, i.e. to solely the sheaves over modules over the the pull-back of the noncommutative Tamagawa-Iwasawa gestalten in the category of analytic stacks after Clausen-Scholze.
\end{remark}

\begin{corollary}
For multivariate $\Gamma$-Frobenius modules (either projective or coherent) we have the Poincar\'e duality on cohomology gestalten over $\mathrm{Spa}\mathbb{Q}_p$, i.e. we consider pushforward of the Tamagawa-Iwasawa gestalten (taking first element in the gestalten sequence to be the derived $\infty$-$\otimes$-category of pseudocoherent complexes) to pseudocoherent gestalten for $\mathrm{AnSpec}\mathbb{Q}_p$. We also have the finiteness of the cohomology groups over $\mathrm{Spa}A^\square$ for some rigid affinoid $A$ in the multivariate setting, from the fact that the associated Tamagawa-Iwasawa gestalten are proper.
\end{corollary}

\begin{remark}
This approach is largely inspired by \cite{5S5} and \cite{5M}, which is actually giving another \textit{higher categorical approach} to \textit{der Endlichkeitssatz} in \cite{T}. And the current consideration improves results in \cite{T}. Also following the philosophy of \cite{5S5}, gestalten are already encoding informations on the derived categories in the usual sense, noncommutative Tamagawa-Iwasawa gestalten are enough to generalize the original consideration on Iwasawa theory, where one stops taking the further module gestalten over these gestalten to study the determinant functors and etc. However by using the approach in \cite{T} it seems to us that it is not so straightforward to prove Poincar\'e duality, as compared to this current higher categorical consideration after \cite{5S5}. Also as mentioned above, it is also hard to prove K\"unneth theorem as well since we do not even know if any K\"unneth theorem will hold for sheaves over Robba rings if we do not consider higher categories after \cite{5S5} and essentially \cite{5Ga}.
\end{remark}

\newpage
\section{Symmetrical Monoidal $\infty$-Categories in $p$-adic Functional Analysis}

\subsection{Representation of $p$-adic Lie groups}

\indent The algebra $L^X_{\infty_I,K_I}(\Gamma_{K_I})$ is some abelian version of the more general algebras as in \cite{ST} and \cite{Z1}. Their integral version will be the algebra taking the following form:
\begin{align}
X^+[[G]]
\end{align}
where $G$ is some Lie group over $\mathbb{Z}_p$ and $X^+$ is some local integral model of $X$, i.e. commutative $p$-adic $\mathbb{Z}_p$-algebras which can be commutative $\mathbb{F}_p$-algebra. We start from $X^+$ to be commutative $\mathbb{F}_p$-algebra, then this includes those $\mathbb{Z}_p/p^n$-algebras. Then by taking the inverse limits along $n$ we have the $p$-adic $\mathbb{Z}_p$-situation and then by taking $p^{-1}$ we can discuss the $\mathbb{Q}_p$-algebra coefficients. It, the ring $X^+[[G]]$, obviously can be highly noncommutative. In the field coefficient situation the category:
\begin{align}
R_{\mathrm{lisse},X^+,G}
\end{align}
is studied in \cite{SS1}, \cite{SS2} where monoidal structure has been established. More general setting is also established in \cite{HM}, where the category is promoted to certain complete category in the condensed mathematics. This is defined to be smooth representation of $G$ over the coefficient in $X^+$-modules. What is related is the corresponding work in \cite{So1} where Schneider's $E_1$-algebra\footnote{Also called $A_\infty$-algebras.} style Hecke algebras defined in \cite{Sc1} are studied in some well-established context. Again \cite{HM} generalized the definition of dg Hecke algebras to the relative setting with general coefficients. One also has the following algebra:
\begin{align}
DE^{X^+[[G]]}_{X^+,X^+}
\end{align}
which is the derived extension of $X^+$ with itself as the representation in some trivial manner. This algebra is actually $E_1$-algebra as in \cite{So1} when $X^+$ is a finite field. Since now the $E_1$-algebraic consideration is over $X^+$ we expect this to be more complicated. What is happening is that this algebra taking this form might be complicatedly hard to be related to the de Hecke algebra (defined in the obvious generalized way after \cite{Sc1}, in \cite{HM}) when the base is not a field, instead the base is a commutative ring as in \cite{So1}. Though similar, the relationship might be possibly taking a more generalized form as we discussed in the following.

\subsection{Derived Categories over $X^+$}

\noindent After \cite{SS1}, \cite{SS2}, \cite{HM} we now consider the derived $\infty$-category of:
\begin{align}
R_{\mathrm{lisse},X^+,G}.
\end{align}
As in \cite{SS1}, \cite{SS2}, \cite{HM} we have the generalized monoidal structure by using the tensor product over $X^+$ (note here $G$ can be noncommutative but we assume that $X^+$ to be commutative). Therefore we then have the derived $\infty$-category associated to this category which we denote it as:
\begin{align}
DR_{\mathrm{lisse},X^+,G}.
\end{align}

\begin{theorem}\mbox{\textbf{(Schneider-Sorensen \cite{SS1} \cite{SS2}, Heyer-Mann \cite{HM})}}
The $\infty$-category 
\begin{align}
DR_{\mathrm{lisse},X^+,G}.
\end{align}
is symmetrical monoidal with respect to the tensor products $\times_{G}$ defined in the same fashion as in \cite{SS1} and \cite{SS2} over tensor products over $X^+$ of two complexes. 
\end{theorem}

\indent When we have that $X^+$ is defined over $\mathbb{Z}_p$ such that:
\begin{align}
X^+ = \varprojlim_{n} \overline{X}^+_{\mathbb{Z}_p/p^n},
\end{align}
we can then put:
\begin{align}
\varprojlim_n DR_{\mathrm{lisse},\overline{X}^+_{\mathbb{Z}_p/p^n},G}.
\end{align}

\subsection{Derived Categories over $E_1$-ring $DE^{X^+[[G]]}_{X^+,X^+}$ over $X^+$}

\noindent Here we consider the derived $\infty$-category of the ring:
\begin{align}
DE^{X^+[[G]]}_{X^+,X^+}.
\end{align}
As in \cite{So1} this is actually $E_1$-ring over $X^+$, which can also be regarded as come derived $E_1$-algebra in the sense of \cite{Sa}. Here we follow \cite{So1} to assume the compactness of the group $G$ with further assumptions as in \cite{So1}. $X^+$ is again a $\mathbb{F}_p$-ring which is commutative. This means that we are assuming that the dg Hecke algebras in \cite{Sc1} and more generally in \cite{HM} are just identical to the ones from $X^+$ directly since we are just taking the induction from $G$ to $G$ itself. When we have that $X^+$ is defined over $\mathbb{Z}_p$ such that:
\begin{align}
X^+ = \varprojlim_{n} \overline{X}^+_{\mathbb{Z}_p/p^n},
\end{align}
we can then put:
\begin{align}
\varprojlim_n DE^{\overline{X}^+_{\mathbb{Z}_p/p^n}[[G]]}_{\overline{X}^+_{\mathbb{Z}_p/p^n},\overline{X}^+_{\mathbb{Z}_p/p^n}}.
\end{align}
We are going to consider $X^+$ to be now taking such limit form which is $p$-adic $\mathbb{Z}_p$-algebra. We then use the notation $H$ to denote the corresponding dg Hecke algera which is defined in \cite{Sc1} and \cite{HM}. First we take the resolution of $X^+$ as the trivial representation, then we just take the derived homomorphism algebra from $X^+$ now with itself. So far the discussion is based on the situation where the generalization is directly through the corresponding generalization in \cite{So1}, i.e. the ring:
\begin{align}
DE^{X^+[[G]]}_{X^+,X^+}.
\end{align}
is just a relative version of the ring in \cite{So1}. \cite{So1} then makes the contact with the homology $E_1$-ring for the dg Hecke algebra from \cite{Sc1} in some transparent way. Then the homology $E_1$-ring can be directly related to the dg Hecke algebra in \cite{Sc1} and \cite{HM}. However in our general setting this is not possible since at least \cite{Sa} points out the issue where we need to take the derived homology instead to reach such similar relationship. Therefore we now change the point of view from the ring:
\begin{align}
DE^{X^+[[G]]}_{X^+,X^+}.
\end{align}
to the corresponding model in \cite{Sa}:
\begin{align}
\underset{{\mathrm{totalized},X^+,X^+}}{\mathrm{homomorphism}}
\end{align}
which is for instance the derived homology in \cite{Sa}. Here the point of view is that in order to reach certain homology in some derived sense which preserves the corresponding derived $E_1$-categories one has to work with the so-called derived $E_1$-algebras in \cite{Sa}. In such a way \cite{So1} can be generalized to the relative setting. First we recall that $H$ is defined by using projective resolution of $X^+$ (in our case this is the same as the induction from $G$ to $G$):
\begin{align}
\mathrm{proj}_\mathrm{resolution}(X^+)
\end{align}
which is then defined as the derived homomorphism of this resolution with itself. From \cite{Sc1}, \cite{HM} we have functors linking the two $\infty$-categories:
\begin{align}
DR_{\mathrm{lisse},X^+,G},
\end{align}
and the derived $\infty$-category of all the $H$-$E_1$-modules:
\begin{align}
D_H.
\end{align}
We use the following notations to denote the functors:
\begin{align}
F_{DR_{\mathrm{lisse},X^+,G}}, F_{D_H}.
\end{align}

\begin{conjecture}\mbox{\textbf{(After Sorensen, \cite[Theorem 1.1]{So1})}}
Promoting $H$ to a derived $E_1$-ring $\mathbb{H}$, we have the derived $\infty$-category of the minimal derived $E_1$-ring
\begin{align}
\underset{{\mathrm{totalized},\mathbb{H},\mathbb{H}}}{\mathrm{homomorphism}}
\end{align}
admits a functor from or a functor into the derived $\infty$-category
\begin{align}
DR_{\mathrm{lisse},X^+,G},
\end{align}
through the functors:
\begin{align}
F_{DR_{\mathrm{lisse},X^+,G}}, F_{D_H}.
\end{align}
We conjecture all functors here are equivalences of symmetrical monoidal $\infty$-categories.
\end{conjecture}

\begin{theorem}
Consider the two monoidal $(\infty,1)$-categories:
\begin{align}
D:=D(\underset{{\mathrm{totalized},\mathbb{H},\mathbb{H}}}{\mathrm{homomorphism}}),DR:=DR_{\mathrm{lisse},X^+,G}.
\end{align}
One then first promotes them to two gestalten via the following construction:
\begin{align}
&(\mathrm{End}_D(1),D,\mathrm{Mod}_{D}(2Pr^L),...),\\
&(\mathrm{End}_{DR}(1),DR,\mathrm{Mod}_{DR}(2Pr^L),...)
\end{align}
Suppose $G$ is varying in all the corresponding $p$-adic Lie groups. Then there are indeed well-defined functors between them which are compatible with well-defined abstract higher gestalten six-functor formalisms on the both sides and gestalten K\"unneth theorems on the both sides as in \cite[Theorem 5.2, Corollary 5.3, Corollary 5.4]{5S5}, pullbacked and induced from the ones for noncommutative gestalten.
\end{theorem}

\begin{proof}
This is just a direct application of general noncommutative gestalten, for instance see \cref{remark14}.
\end{proof}

\begin{remark}
We combine \cite[Theorem 1.2]{Sa} with the direct generalization of \cite{Sc1} from \cite{HM}. See construction of \cite[Theorem 9]{Sc1}. Then we do have the functors involved. However we don't know whether all the statements of this conjecture are correct or not.
\end{remark}

\begin{remark}
The goal here eventually will be consider certain Banach $\mathbb{Z}_p$-algebra and consider the corresponding Banach representations. One has to upgrade all the corresponding rings here to certain condensed setting such as in \cite{HM}, i.e. first the representation spaces have to be carrying the Banach topology, then the dg Hecke algebra will then be certain Banach dg algebra, then the corresponding derived $E_1$-ring in \cite{Sa} will also be taken to be the analytification version. For instance one works over the solid Banach $\mathbb{Z}_p$-modules. These consideration will be essentially certain $p$-adic local Langlands correspondence consideration after \cite{L1}, \cite{C}. 
\end{remark}

\begin{remark}
\indent Let us make some discussion on the ring
\begin{align}
\underset{{\mathrm{totalized},\mathrm{Ban},\mathbb{H},\mathbb{H}}}{\mathrm{homomorphism}}.
\end{align}
Recall that this ring is the totalization of the homomorphism group for the bigraded suspension of $\mathbb{H}$. In the Banach setting we need to take the homomorphisms in the solid Banach $\mathbb{Z}_p$-modules, which indicates that we use the notation $\mathrm{Ban}$.
\end{remark}

\indent Then as in \cite[Theorem 1.1]{So1} one can takes one step further to write the explicit formula for the derived homology for:
\begin{align}
\underset{{\mathrm{totalized},\mathbb{H},\mathbb{H}}}{\mathrm{homomorphism}}.
\end{align}
The resulting derived $E_1$-ring should then be following the definition from \cite{Sa}:
\begin{align}
\underset{{\mathrm{totalized},H,H}}{\mathrm{extension}}.
\end{align}

\subsection{Deformation of Modules over Derived $E_1$-Rings}

\indent Then one can use this to study the corresponding deformation functors in both Hecke algebraic setting and the corresponding smooth representation setting. Then this means the deformation happens in the following $\infty$-category of $\mathbb{Z}_p$ $p$-adic formal algebras with residue $\mathbb{F}_p$ by maximal ideals for $\pi_0$:
\begin{align}
\mathrm{LargeCoeff}_{\mathbb{Z}_p,\mathbb{F}_p,\mathrm{local}}
\end{align}
They have generic fibres as rigid analytic spaces, but we start from the corresponding integral picture. Recall the standard the deformation functors in the following sense. Deformation functor for the group $G$ is a functor fibered over:
\begin{align}
\mathrm{LargeCoeff}_{\mathbb{Z}_p,\mathbb{F}_p,\mathrm{local}}
\end{align}
by taking any ring $\blacksquare$ to a corresponding $\infty$-groupoid of all the representations over $G$ over $\blacksquare$. Deformation functor for the $E_1$-ring $H$ is a functor fibered over:
\begin{align}
\mathrm{LargeCoeff}_{\mathbb{Z}_p,\mathbb{F}_p,\mathrm{local}}
\end{align}
by taking any ring $\blacksquare$ to a corresponding $\infty$-groupoid of all the $E_1$-modules over $H$ with coefficient in $\blacksquare$. Deformation functor for the derived $E_1$-ring $\mathbb{H}$ is a functor fibered over:
\begin{align}
\mathrm{LargeCoeff}_{\mathbb{Z}_p,\mathbb{F}_p,\mathrm{local}}
\end{align}
by taking any ring $\blacksquare$ to a corresponding $\infty$-groupoid of all the $E_1$-modules over $\mathbb{H}$ with coefficient in $\blacksquare$. Deformation functor for the derived $E_1$-ring $\underset{{\mathrm{totalized},\mathbb{H},\mathbb{H}}}{\mathrm{homomorphism}}$ is a functor fibered over:
\begin{align}
\mathrm{LargeCoeff}_{\mathbb{Z}_p,\mathbb{F}_p,\mathrm{local}}
\end{align}
by taking any ring $\blacksquare$ to a corresponding $\infty$-groupoid of all the $E_1$-modules over 
\begin{align}
\underset{{\mathrm{totalized},\mathbb{H},\mathbb{H}}}{\mathrm{homomorphism}}
\end{align}
with coefficient in $\blacksquare$. We use the following notations to denote these $\infty$-deformation functors:
\begin{align}
\mathrm{Deform}_{G}, \mathrm{Deform}_H, \mathrm{Deform}_\mathbb{H}, \mathrm{Deform}_{\underset{{\mathrm{totalized},\mathbb{H},\mathbb{H}}}{\mathrm{homomorphism}}}.
\end{align}
Deformation functor for the group $G$ with respect to some representation $r/\mathbb{F}_p$ is a functor fibered over:
\begin{align}
\mathrm{LargeCoeff}_{\mathbb{Z}_p,\mathbb{F}_p,\mathrm{local}}
\end{align}
by taking any ring $\blacksquare$ to a corresponding $\infty$-groupoid of all the representations over $G$ over $\blacksquare$ such that the representations have quotients isomorphic to $r$ over $\blacksquare/m$. Deformation functor for the $E_1$-ring $H$ with respect to some module $r/H_{\mathrm{F}_p}$ is a functor fibered over:
\begin{align}
\mathrm{LargeCoeff}_{\mathbb{Z}_p,\mathbb{F}_p,\mathrm{local}}
\end{align}
by taking any ring $\blacksquare$ to a corresponding $\infty$-groupoid of all the $E_1$-modules over $H$ with coefficient in $\blacksquare$ such that the modules are isomorphic to $r$ after we take the quotient with respect to maximal ideals. Deformation functor for the derived $E_1$-ring $\mathbb{H}$ is similar lifting certain modules over $\mathbb{H}_{\mathbb{F}_p}$. We use the following notations to denote these $\infty$-deformation functors:
\begin{align}
\textit{Deform}_{G,r}, \textit{Deform}_{H,r}, \textit{Deform}_{\mathbb{H},r}, \textit{Deform}_{\underset{{\mathrm{totalized},\mathbb{H},\mathbb{H}}}{\mathrm{homomorphism}},r}.
\end{align}

\begin{conjecture}\mbox{\textbf{(After Sorensen, \cite[Theorem 1.1]{So1})}}
We have well-defined morphisms of deformation functors:
\begin{align}
\mathrm{Deform}_{G}\overset{\sim}{\rightarrow}\mathrm{Deform}_H\overset{\sim}{\rightarrow} \mathrm{Deform}_\mathbb{H}\overset{\sim}{\rightarrow} \mathrm{Deform}_{\underset{{\mathrm{totalized},\mathbb{H},\mathbb{H}}}{\mathrm{homomorphism}}}.
\end{align}
All functors are equivalent in this conjecture.
\end{conjecture}

\begin{conjecture}\mbox{\textbf{(After Sorensen, \cite[Theorem 1.1]{So1})}}
We have well-defined morphisms of deformation functors:
\begin{align}
\textit{Deform}_{G,r}\overset{\sim}{\rightarrow} \textit{Deform}_{H,r}\overset{\sim}{\rightarrow} \textit{Deform}_{\mathbb{H},r}\overset{\sim}{\rightarrow} \textit{Deform}_{\underset{{\mathrm{totalized},\mathbb{H},\mathbb{H}}}{\mathrm{homomorphism}},r}.
\end{align}
All functors are equivalent in this conjecture.
\end{conjecture}

\begin{conjecture}\mbox{\textbf{(After Sorensen, \cite[Theorem 1.1]{So1})}}
We have well-defined morphisms of derived deformation functors:
\begin{align}
\mathrm{Deform}_{G}\overset{\sim}{\rightarrow}\mathrm{Deform}_H\overset{\sim}{\rightarrow} \mathrm{Deform}_\mathbb{H}\overset{\sim}{\rightarrow} \mathrm{Deform}_{\underset{{\mathrm{totalized},\mathbb{H},\mathbb{H}}}{\mathrm{homomorphism}}},
\end{align}
over the animation:
\begin{align}
\underline{\mathrm{LargeCoeff}}_{\mathbb{Z}_p,\mathbb{F}_p,\mathrm{local}}.
\end{align}
All functors are equivalent in this conjecture.
\end{conjecture}

\begin{conjecture}\mbox{\textbf{(After Sorensen, \cite[Theorem 1.1]{So1})}}
We have well-defined morphisms of derived deformation functors:
\begin{align}
\textit{Deform}_{G,r}\overset{\sim}{\rightarrow} \textit{Deform}_{H,r}\overset{\sim}{\rightarrow} \textit{Deform}_{\mathbb{H},r}\overset{\sim}{\rightarrow} \textit{Deform}_{\underset{{\mathrm{totalized},\mathbb{H},\mathbb{H}}}{\mathrm{homomorphism}},r},
\end{align}
over the animation:
\begin{align}
\underline{\mathrm{LargeCoeff}}_{\mathbb{Z}_p,\mathbb{F}_p,\mathrm{local}}.
\end{align}
All functors are equivalent in this conjecture.
\end{conjecture}

\newpage
\section{Fundamental Comparison on Stackifications over Small Arc Stacks}\label{section5}

\subsection{Witt-Prisms for function field}
\indent We start from any $v$-stack in the work of Scholze in \cite{1S1}. We use a notation $K$ to denote such stack. We assume that $K$ is over some local field, namely we take the $\mathrm{Spd}$ of some local field $L$. This local field needs to be specified in different characteristics. In positive characteristic situation, we assume this takes the form being finite over some $\mathbb{F}_p((u))$ for some chosen uniformizer $u$. In the $p$-adic setting, we assume it to be finite over $\mathbb{Q}_p$. 
\begin{assumption}
In this \cref{section5}, $L$ is now assumed to be in the function field situation, namely over $\mathbb{F}_p$.
\end{assumption}
Now for such $K$ we have two different versions of the prismatizations, after \cite{1To1}, \cite{1To2}, \cite{1To3}, \cite{1To4}, \cite{1S4}, \cite{1ALBRCS}, \cite{1BS}, \cite{1D}, \cite{1BL}. We in the algebraic setting unveil the definition from \cite{1BS}, \cite{1D}, \cite{1BL} in the following for the convenience of the readers. We consider the following function fields:
\begin{align}
\mathbb{F}_p[[u]][u^{-1}].
\end{align}
Then we consider the moduli stacks of the Witt-line bundles in the following sense. The underlying sites will be chosen to be the categories of all 
\begin{align}
\mathbb{F}_p[[u]]
\end{align}
algebras where the element $u$ presents nilpotency. Then by using the valuations from the corresponding Witt vectors, we can consider the corresponding infinite product:
\begin{align}
\mathrm{Wi(\square)}=\prod_i \square_i = \square \times \square \times \square \times \square \times \square\times \square \times... 
\end{align} 
The Witt vector infinite product presentations present the Witt vectors as the corresponding infinite products. Then over any base ring $\square$ we consider the moduli of all the line bundles mapping to these infinite products, such as all such line bundles are generated principally locally through elements satisfying the following requirements: the first coordinates are nilpotent and the second coordinates are unital after we map the these elements to the corresponding infinite products. Furthermore when we are constructing the motives for some $u$-adic formal scheme we then require the corresponding spectrum of the quotients of $\mathrm{Wi}(\square)$ through the line bundles to be mapped to the formal scheme in the construction. For any such generalized prism pair (a line bundle and a Witt vector deformation under this line bundle) we can then put the structure sheaf
 \begin{align}
 \mathcal{T}
 \end{align}
 of the prestacks as just the ring $\square$ before the deformation. Then over $\mathcal{T}$ we have further the corresponding de Rham structure sheaf 
 \begin{align}
 \mathcal{T}[1/u]_{\mathrm{Li}}
  \end{align}
 for any line bundle $\mathrm{Li}$. The corresponding such pair can be called as the corresponding \textit{Witt-prisms}, in a very generlized fashion.

 We use the notation:
\begin{align}
\mathrm{CondPrismatization}_{K,v}
\end{align}
to denote the condensation of the prismatization of $K$ over the $v$-topology. To be more precise for each perfectoid $K_i$ living over $K$ we take the corresponding algebraic prismatization:
\begin{align}
\mathrm{AlgPrismatization}_{K_i}
\end{align}
Then we take the analytification from \cite{1CS}:
\begin{align}
\mathrm{AlgPrismatization}_{K_i,\square}
\end{align}
Then what we do is to change $K_i$ in the Grothendieck topology to rich the whole prismatization overall over $K$ as sheaves of categories:
\begin{align}
\mathrm{CondPrismatization}_{K,v}.
\end{align}
In \cite{1To1} we also considered the corresponding de Rham stacks and the cristalline ones, where we allow the untils to be parametrized through the prismatization:
\begin{align}
&\mathrm{ConddeRhamPrismatization}_{K,v},\\
&\mathrm{CondCristallinePrismatization}_{K,v}.
\end{align}

\begin{definition}
We have the corresponding generalized version after \cite{1BS1}, \cite{1F2}, \cite{1BL1}, \cite{1BL2}:
\begin{align}
&\mathrm{ConddeRhamPrismatization}_{K,v,2}\\
&\mathrm{CondCristallinePrismatization}_{K,v,2}\\
&\mathrm{CondPrismatization}_{K,v,2}.
\end{align}
Here $b$ is fixed to be the element in both situations as in \cite{1BL1}, \cite{1BL2}, \cite{1BS1}, in the $u$-adic setting we have the analog element as well where $u$-adic cyclotomic character is also defined on this element. Then we take the corresponding analytification:
 \begin{align}
&\mathrm{ConddeRhamPrismatization}_{K,v,2,\square}\\
&\mathrm{CondCristallinePrismatization}_{K,v,2,\square}\\
&\mathrm{CondPrismatization}_{K,v,2,\square}.
\end{align}
\end{definition}

Then we can consider the corresponding solid quasicoherent sheaves over the stacks. 

\begin{remark}
Here the solid analytification can be explicified in the following sense. Over some $K_i$ locally we have the stacks can be written as \textit{projective limits} of the corresponding formal spectrum of the Witt vector rings, de Rham period sheaves and the cristalline period sheaves. Then the corresponding condensation and solidified analytification will automatically transform the formal topology to the topology encoding the Banach norms from the underlying perfectoid rings after \cite{1KL} and \cite{1KL1}.
\end{remark}

\begin{definition}
Now over any $K_i$ as above, we have the corresponding local version of the $\infty$-category of solid quasicoherent sheaves over the stacks above:
 \begin{align}
&\mathrm{SolidModules}_{\mathrm{ConddeRhamPrismatization}_{K_i,v,2,\square}}\\
&\mathrm{SolidModules}_{\mathrm{CondCristallinePrismatization}_{K_i,v,2,\square}}\\
&\mathrm{SolidModules}_{\mathrm{CondPrismatization}_{K_i,v,2,\square}}.
\end{align}
Here before taking the corresponding condensation we have the corresponding only formal topology versions:
 \begin{align}
&\mathrm{Modules}_{\mathrm{ConddeRhamPrismatization}_{K_i,v,2}}\\
&\mathrm{Modules}_{\mathrm{CondCristallinePrismatization}_{K_i,v,2}}\\
&\mathrm{Modules}_{\mathrm{CondPrismatization}_{K_i,v,2}}.
\end{align}
\end{definition}

We remind the readers the internal structure of these definitions as in the following. First recall that we can have the chance to write the above as $\varprojlim$ of certain derived $\infty$-categories of prisms. In the $u$-adic circumstance we have the parallel definition of \textit{generalized prisms}, which are defined to be any form of pairs as in the following in the definition of Cartier divisor for the $u$-prismatization:
\begin{align}
I,WVL(A)
\end{align}
where $A$ is nilpotent for the element $u$, and we require that the Witt vector ring $WVL(A)$ is complete in the derived sense with respect to the element $u$ and $I$. Then we have:
\begin{align}
&\mathrm{Modules}_{\mathrm{ConddeRhamPrismatization}_{K_i,v,2}}=\varprojlim_{{I,WVL(A)}} \mathrm{Modules}_{\varprojlim_{I}WVL(A)[1/x][b^{1/2}]}\\
&\mathrm{Modules}_{\mathrm{CondCristallinePrismatization}_{K_i,v,2}}=\varprojlim_{{I,WVL(A)}} \mathrm{Modules}_{\ WVL(A)[1/x][b^{1/2}]}\\
&\mathrm{Modules}_{\mathrm{CondPrismatization}_{K_i,v,2}}=\varprojlim_{{I,WVL(A)}} \mathrm{Modules}_{WVL(A)[b^{1/2}]}.
\end{align}

\begin{theorem}
In the $u$-adic setting we have:
\begin{align}
&\mathrm{Modules}_{\mathrm{ConddeRhamPrismatization}_{K_i,v,2}}=\varprojlim_{{I,WVL(A)}} \mathrm{Modules}_{\varprojlim_{I}WVL(A)[1/x][b^{1/2}]}\\
&\mathrm{Modules}_{\mathrm{CondCristallinePrismatization}_{K_i,v,2}}=\varprojlim_{{I,WVL(A)}} \mathrm{Modules}_{\ WVL(A)[1/x][b^{1/2}]}\\
&\mathrm{Modules}_{\mathrm{CondPrismatization}_{K_i,v,2}}=\varprojlim_{{I,WVL(A)}} \mathrm{Modules}_{WVL(A)[b^{1/2}]}
\end{align}
can be further written as:
\begin{align}
&\mathrm{Modules}_{\mathrm{ConddeRhamPrismatization}_{K_i,v,2}}=\varprojlim_{{I\rightarrow WVL(K_i)}} \mathrm{Modules}_{\varprojlim_{I}WVL(K_i)[1/x][b^{1/2}]}\\
&\mathrm{Modules}_{\mathrm{CondCristallinePrismatization}_{K_i,v,2}}=\mathrm{Modules}_{ WVL(K_i)[1/x][b^{1/2}]}\\
&\mathrm{Modules}_{\mathrm{CondPrismatization}_{K_i,v,2}}=\mathrm{Modules}_{WVL(K_i)[b^{1/2}]}.
\end{align}
\end{theorem}

\begin{proof}
In the $u$-adic setting and in the perfectoid setting the corresponding Witt vectors ring can be written as the algebraic tensor product of the $p$-adic version with the ring $\mathcal{O}^L$. Then the last equality can be derived from the $p$-adic situation. Then the first two equations can be then derived.
\end{proof}

\begin{remark}
In the $p$-adic setting the corresponding construction are based on prisms not the Witt-prisms in \cite{1BL}, \cite{1D}. We then have:
\begin{align}
&\mathrm{Modules}_{\mathrm{ConddeRhamPrismatization}_{K_i,v,2}}=\varprojlim_{{I\rightarrow B}} \mathrm{Modules}_{\varprojlim_{I}B[1/x][b^{1/2}]}\\
&\mathrm{Modules}_{\mathrm{CondCristallinePrismatization}_{K_i,v,2}}=\mathrm{Modules}_{ WVL(K_i)[1/x][b^{1/2}]}\\
&\mathrm{Modules}_{\mathrm{CondPrismatization}_{K_i,v,2}}=\mathrm{Modules}_{WVL(K_i)[b^{1/2}]}.
\end{align}
\end{remark}

\indent In the $u$-adic setting the \textit{generalized prisms} are basically being regarded as more general since in the $p$-adic setting pairs like $I,WVL(A)$ are actually generalized the corresponding notion of the prisms, while the latter can be used to construct such pairs. Therefore in the $u$-adic setting one can consider such definitions after \cite{1BS}, \cite{1BL}, \cite{1D}. 

\begin{theorem}
In the $u$-adic setting we have:
\begin{align}
&\mathrm{Modules}_{\mathrm{ConddeRhamPrismatization}_{K_i,v,2}}=\varprojlim_{{I,WVL(A)}} \mathrm{Modules}_{\varprojlim_{I}WVL(A)[1/x][b^{1/2}]}\\
&\mathrm{Modules}_{\mathrm{CondCristallinePrismatization}_{K_i,v,2}}=\varprojlim_{{I,WVL(A)}} \mathrm{Modules}_{\ WVL(A)[1/x][b^{1/2}]}\\
&\mathrm{Modules}_{\mathrm{CondPrismatization}_{K_i,v,2}}=\varprojlim_{{I,WVL(A)}} \mathrm{Modules}_{WVL(A)[b^{1/2}]}
\end{align}
can be further written as:
\begin{align}
&\mathrm{Modules}_{\mathrm{ConddeRhamPrismatization}_{K_i,v,2}}=\varprojlim_{{I\rightarrow WVL(K_i)}} \mathrm{Modules}_{\varprojlim_{I}WVL(K_i)[1/x][b^{1/2}]}\\
&\mathrm{Modules}_{\mathrm{CondCristallinePrismatization}_{K_i,v,2}}=\mathrm{Modules}_{ WVL(K_i)[1/x][b^{1/2}]}\\
&\mathrm{Modules}_{\mathrm{CondPrismatization}_{K_i,v,2}}=\mathrm{Modules}_{WVL(K_i)[b^{1/2}]}.
\end{align}
Then obviously we can recover certain perfectoid picture by requiring the projection to component where $I$ is the one generated by $b$.
\end{theorem}

\subsection{Robba Stacks and Perfectoid Witt-Prisms}

Here we consider another stackification after \cite{1KL}, \cite{1KL1}, \cite{1S1}, \cite{1S2}, \cite{1S3}, \cite{1F1}, \cite{1F2}, \cite{1T1}. We consider foundation from \cite{1CSA}, \cite{1CSB}, \cite{1CS}. Also we consider generalization following  \cite{1BS1}, \cite{1BL1}, \cite{1BL2}. This also is towards some motivic generalization in the sense of \cite{1G}. 

\noindent Now we consider the parametrization in some other foundation, namely the corresponding Robba stacks. The story goes over some local perfectoid ring $K_i$ as above for the general $K$ over $L$. $L$ can be of mixed characteristic or equal characteristic. Over such $K_i$ as in \cite{1KL}, \cite{1KL1}, we have the parametrization space which is just defined as the adic spectrum of the perfect Robba rings defined with respect to $K_i$, where we do not take the corresponding Frobenius quotient:
\begin{align}
Y_{K_i}:=\mathrm{Union}_I\mathrm{SpecSpa}(P_{I,K_i},P_{I,K_i}^+).
\end{align}
One then consider the corresponding generalization in \cite{1BS1} and \cite{1F2} to contact with the context we consider here. Namely we have the following version generalization of the Fargues-Fontaine stacks (again in two different characteristic situations):
\begin{align}
&Y_{K_i,2}:=\mathrm{Union}_I\mathrm{SpecSpa}(P_{I,K_i}[b^{1/2}],P_{I,K_i}^+[b^{1/2}]).
\end{align}
Suppose we use the notation $Q$ to denote the structure sheaves of these spaces:
\begin{align}
Q_{Y_{K_i,2}}.
\end{align}
Now over $Q$ we have the corresponding solid quasicoherent sheaves which form $\infty$-categories:
\begin{align}
\mathrm{QC}^\mathrm{solid}Q_{Y_{K_i,2}}.
\end{align}

\begin{definition}
One can then define the corresponding de Rham version of the Robba stacks as in the corresponding prismatization in the following. Locally we consider the completion along all the untilts $\sharp$, which goes in the following way again we forget the corresponding underlying stacks:
\begin{align}
&\mathrm{QC}^\mathrm{solid}_{\mathrm{deRham}}Q_{Y_{K_i,2}}:=\varprojlim_{\sharp} \mathrm{QC}^\mathrm{solid}{Q_{Y_{K_i,2}}}^\wedge_\sharp.
\end{align}
Then let $K_i$ change in the $v$-topology we have the following:
\begin{align}
&\mathrm{QC}^\mathrm{solid}_{\mathrm{deRham}}Q_{Y_{K,2}}.
\end{align}
\end{definition}

\begin{definition}
One can then define the corresponding de Rham version of the Robba stacks as in the corresponding prismatization in the following. Locally we consider the completion along all the untilts $\sharp$, which goes in the following way again we forget the corresponding underlying stacks:
\begin{align}
&\mathrm{QC}^\mathrm{solid}_{\mathrm{deRham}}Q_{Y_{K_i,2}}:=\varprojlim_{\sharp} \mathrm{QC}^\mathrm{solid}{Q_{Y_{K_i,2}}}^\wedge_\sharp.
\end{align}
Then let $K_i$ change in the $v$-topology we have the following:
\begin{align}
&\mathrm{QC}^\mathrm{solid}_{\mathrm{deRham}}Q_{Y_{K,2}}.
\end{align}
We then have a well-defined functor $\mathcal{H}$ which is called generalized de Rhamization:
\begin{align}
\mathrm{QC}^\mathrm{solid}Q_{Y_{K_i,2}}
\longrightarrow
\mathrm{QC}^\mathrm{solid}_{\mathrm{deRham}}Q_{Y_{K,2}},
\end{align}
through the $\sharp$-completion through all the untilts in the coherent way as in the above.
\end{definition}

\begin{remark}
This definition is considering $u$-adic de Rham sheaves, though the internal structure of such sheaves can be simpler the construction is in a uniform framework as in the above. 
\end{remark}

\subsection{Discussion for a lisse chart over $\mathrm{Spd}L$}

\noindent We now assume that the stack is just a lisse chart $\mathcal{L}$ over $\mathrm{Spd}L$, where we have the geometrized generalized $\mathrm{Gamma}_{\mathcal{L},2}$-modules without the Frobenius actions:
\begin{align}
\mathrm{Gamma}_{\mathcal{L},2}\mathrm{QC}^\mathrm{solid}Q_{Y_{\mathrm{Spd}\mathcal{L},2}},
\end{align}
with
\begin{align}
\mathrm{Gamma}_{\mathcal{L},2}\mathrm{QC}_\mathrm{deRham}^\mathrm{solid}Q_{Y_{\mathrm{Spd}\mathcal{L},2}}.
\end{align}

\begin{theorem}
Assume we are in the $p$-adic setting. By projecting to $\sharp=b$ we have the generalized differential equations attached to finite-locally free sheaves in:
\begin{align}
\mathrm{Gamma}_{\mathcal{L},2}\mathrm{QC}^\mathrm{solid}Q_{Y_{\mathrm{Spd}\mathcal{L},2}},
\end{align}
which is further assumed to be generalized de Rham in the obvious generalized way. Here we \textbf{do not} assume the stability and compatibility of the rank throughout the noncompact Stein Robba stacks and we \textbf{do not} assume the finiteness of the rank when we reach the global sections of the Stein stacks. Namely for any such sheaf we can find a projective limit system $D=\varprojlim_w D_w$ to attach to this sheaf, over which we can have the structure of arithmetic $\mathcal{D}$-modules with the action from the group $\mathrm{Gamma}_{\mathcal{L},2}$. 
\end{theorem}

\begin{proof}
We only need to extend the corresponding map from the Robba rings (with respect to some radius in variable of $p^w$) to $L_w[[b]]\otimes L'$ ($L'$ large) into the corresponding situation where we have $b^{1/2}$, then the corresponding formation of the corresponding $w$-th level $p$-adic differential modules in \cite[See and follow the construction around 5.10, the Theorem]{1BA} can be applied directly in our setting, then after we have the construction the corresponding project limit will produce the mixed-parity differential modules. The current morphism here needs the further step of the corresponding deformation which maps the variables $*$ of the lisse chart to $[*^\flat]-1$, which is the difference we need to consider beyond the point situation here. Then the corresponding construction goes in a parallel way. There are two related modules attached to the original module over the Robba ring without Frobenius structure. One is the corresponding differential module as in \cite{1BA} and the other one is the corresponding de Rham module (as in the definition of the corresponding functor). The two modules can be reconctructed from each other by considering the infinite level:
\begin{align}
L_\infty[[b^{1/2}]]\{*^{\pm 1}\}\otimes L',
\end{align}
and by using the invariance of the group $\mathcal{L}$. Then each $D_w$ will be certain preimage (here we need to invert the element $b^{1/2}$) of the bundle $H_w$ for some radius $f(w)$ under the map from the Robba rings in the current setting. This needs to consider the tower:
\begin{align}
&L_w[[b^{1/2}]]\{*^{\pm 1}\}\otimes L',\\
&L_{w+1}[[b^{1/2}]]\{*^{\pm 1}\}\otimes L',\\
&L_{w+2}[[b^{1/2}]]\{*^{\pm 1}\}\otimes L',\\
&L_{w+3}[[b^{1/2}]]\{*^{\pm 1}\}\otimes L',\\
&L_{w+4}[[b^{1/2}]]\{*^{\pm 1}\}\otimes L',\\
&L_{w+5}[[b^{1/2}]]\{*^{\pm 1}\}\otimes L',\\
&...
\end{align}
Then for any such sheaf we can find a projective limit system $D=\varprojlim_w D_w$ to attach to this sheaf, over which we can have the structure of arithmetic $\mathcal{D}$-modules with the action from the group $\mathrm{Gamma}_{\mathcal{L},2}$. 
\end{proof}

\begin{remark}
Also see \cite{1AB}. However we do not have theorems after \cite{1M}, \cite{1K}, \cite{1A} in such generality.
\end{remark}

\subsection{Small Arc-Stacks via Small $v$-Stacks}

\begin{theorem}
Assume we are in our general setting by adding the element $b^{1/2}$. The de Rham-Robba stackification and the de Rham-prismatization stackification in our generalized setting by adding $b^{1/2}$ are equivalent, in both $p$-adic and $z$-adic settings, i.e. in the $p$-adic setting we consider the small $v$-stacks over $\mathrm{Spd}\mathbb{Q}_p$, and in the $z$-adic setting we consider the small $v$-stacks over $\mathrm{Spd}\mathbb{F}_p((t))$, in the $v$-topology. This applies immediately to rigid analytic varieties.
\end{theorem}

\begin{proof}
This is because in the local setting over the $v$-site of any small $v$-stack, what we have will be the corresponding stacks of all the untilts on the prismatization level. Then by taking the limit throughout all the de Rham period sheaves for all the untilts in our generalized setting, we reach the same $\infty$-categories. This construction is then in the local perfectoid setting identical in the both approaches.
\end{proof}

\noindent We can promote this equivalence to the categorical level. One can construct the following functor:

\begin{definition}
Assume we are in our general setting by adding the element $b^{1/2}$. Let $S$ be a small $v$-stack, which can be either over $\mathrm{Spd}\mathbb{Q}_p$ or $\mathrm{Spd}\mathbb{F}_p((u))$. We use the notation:
\begin{align}
\mathrm{deRhamRobba}_S
\end{align}
to denote the corresponding de Rham-Robba stackification from the FF stacks, in our generalized setting. And we use the notation
\begin{align}
\mathrm{deRhamPrismatization}_S
\end{align}
to denote the corresponding de Rham Prismatization stackification, in our generalized setting. And for $?= \mathrm{deRhamRobba}_S, \mathrm{deRhamPrismatization}_S$ we use the notation:
\begin{align}
\mathrm{SolidQuasiCoh}_?
\end{align}
to denote the corresponding condensed $\infty$-categories of the corresponding solid quasicoherent sheaves over $?$. Then we have a functor:
\begin{align}
\mathrm{SolidQuasiCoh}_{\mathrm{deRhamPrismatization}_S}\rightarrow \mathrm{SolidQuasiCoh}_{\mathrm{deRhamRobba}_S}
\end{align}
by taking the induced functor from identification of the de Rham functors on the perfectoids.
\end{definition}

\begin{theorem}
Assume we are in our general setting by adding the element $b^{1/2}$. The functor defined above:
\begin{align}
\mathrm{SolidQuasiCoh}_{\mathrm{deRhamPrismatization}_S}\rightarrow \mathrm{SolidQuasiCoh}_{\mathrm{deRhamRobba}_S}
\end{align}
is well-defined, and an equivalence of symmetrical monoidal $\infty$-categories which are stable.

\end{theorem}

\begin{proof}
Locally over perfectoid spaces the corresponding prismatization is actually the corresponding Witt vector in both $p$-adic and equal characteristic situations. Then this identification of the prismatization and the Witt vector rings in fact directly produces the desired equivalence, since after this identification the construction will be identical completely.
\end{proof}

\noindent We then contact Scholze's notion of small arc stacks in \cite{1S5} \cite{1S6} where local charts will be simply Banach rings, then we have the following definition:

\begin{definition}
Assume we are in our general setting by adding the element $b^{1/2}$. Let $S$ be a small arc-stack in \cite{1S5} over $\mathbb{Q}_p$ or $\mathbb{F}((u))$. We use the notation:
\begin{align}
\mathrm{deRhamRobba}_S
\end{align}
to denote the corresponding de Rham-Robba stackification from the FF stacks, in our generalized setting. And we use the notation
\begin{align}
\mathrm{deRhamPrismatization}_S
\end{align}
to denote the corresponding de Rham Prismatization stackification, in our generalized setting. And for $?= \mathrm{deRhamRobba}_S, \mathrm{deRhamPrismatization}_S$ we use the notation:
\begin{align}
\mathrm{SolidQuasiCoh}_?
\end{align}
to denote the corresponding condensed $\infty$-categories of the corresponding solid quasicoherent sheaves over $?$. When we consider the de-Rham Robba stackification we consider the corresponding $v$-stack associated to $S$, which is denoted by $\mathrm{Stack}_v(S)$ after \cite{1S5}. Then we have a functor:
\begin{align}
\mathrm{SolidQuasiCoh}_{\mathrm{deRhamPrismatization}_S}\rightarrow \mathrm{SolidQuasiCoh}_{\mathrm{deRhamRobba}_{\mathrm{Stack}_v(S)}}
\end{align}
by taking the induced functor from identification of the de Rham functors on the perfectoids, i.e. we set the Banach ring local chart to be perfectoid to reach the objects in the second $\infty$-category. 
\end{definition}

\begin{theorem}
Assume we are in our general setting by adding the element $b^{1/2}$. The functor defined above:
\begin{align}
\mathrm{SolidQuasiCoh}_{\mathrm{deRhamPrismatization}_S}\longrightarrow \mathrm{SolidQuasiCoh}_{\mathrm{deRhamRobba}_{\mathrm{Stack}_v(S)}}
\end{align}
is well-defined, as a symmetrical monoidal $\infty$-tensor functor.
\end{theorem}

The following theorem will then be highly nontrivial:

\begin{theorem}
Assume we are in our general setting by adding the element $b^{1/2}$. The functor defined above:
\begin{align}
\mathrm{SolidQuasiCoh}_{\mathrm{deRhamPrismatization}_S}\longrightarrow \mathrm{SolidQuasiCoh}_{\mathrm{deRhamRobba}_{\mathrm{Stack}_v(S)}}
\end{align}
is fully faithful functor of symmetrical monoidal $\infty$-categories which are stable.
\end{theorem}

\begin{proof}
Here we follow \cite{1S5} \cite{1S6}, where we pass the whole functor to the totally disconnected subspaces in our current nonarchimedean settings. On the both sides they are exactly perfectoid coverings. Then following \cite{1S5} \cite{1S6} we pass to spaces taking the forms of the adic spectrum of algebraically closed fields. Then there is nothing to prove then, since both sides have the same underlying spaces, then the theorem follows.
\end{proof}

\begin{remark}
However we want to mention that the de Rham Robba stackification can live over small arc-stacks by using the tiltings of local Banach subspaces, therefore we will have the corresponding de Rham Robba stackification over the small arc-stacks in some functorial way as well. This can also be achived by taking the perfectoidization of the corresponding de Rham prismatization over small arc-stacks directly.
\end{remark}

\subsection{6-Functor Formalism for Generalized Prismatization}\label{section7.2}

\begin{assumption}\label{assumption2}
In this \cref{section7.2}, all our considerations on the prismatization, de Rham Robba stackification and de Rham prismatization are assumed to be the generalized ones, i.e. we consider the corresponding prismatizations with $b^{1/2}$ added, de Rham Robba stackifications with $b^{1/2}$-added and de Rham prismatizations with $b^{1/2}$ added.
\end{assumption}

\begin{theorem}
Assume we are in the situation in \cref{assumption2}. $\mathrm{SolidQuasiCoh}_{\mathrm{deRhamPrismatization}_*}$ admits pullback functor and pushforward functor $F^\square$ and $F_\square$ in the 6-functor formalism where $*$ is varying in the category of all the small arc-stacks or all small $v$-stacks.
\end{theorem}

\begin{proof}
We follow the method of proof after Scholze in \cite{1S5} by passing to totally disconnected subspaces from the adic spectrum of corresponding algebraically closed fields. Then one can see that the corresponding $\infty$-category can be regarded as the corresponding generalized Galois representations with coefficients over those generalized de Rham period sheaves then the corresponding result will follow since the 6-functor will be reduced to 6-functors among condensed generalized Galois representations. But over the such geometrical points the Galois groups are trivial. Then we are done at least in the current situation where we only consider the pullback and pushforward functors.
\end{proof}

\begin{definition}
Assume we are in the situation in \cref{assumption2}. Since $\mathrm{SolidQuasiCoh}_{\mathrm{deRhamPrismatization}_*}$ are symmetrical monoidal $\infty$-categories, the corresponding motivic Galois group exists as the corresponding Tannakian groups we will use:
\begin{align}
\pi_{\mathrm{SolidQuasiCoh}_{\mathrm{deRhamPrismatization}_*}}
\end{align}
to denote the group in any particular situation over $*$.
\end{definition}

\begin{theorem}
Assume we are in the situation in \cref{assumption2}. $\mathrm{SolidQuasiCoh}_{\mathrm{deRhamRobba}_*}$ admits pullback functor and pushforward functor $F^\square$ and $F_\square$ in the 6-functor formalism where $*$ is varying in the category of all the small $v$-stacks.
\end{theorem}

\begin{proof}
We follow the method of proof after Scholze in \cite{1S5} by passing to totally disconnected subspaces from the adic spectrum of corresponding algebraically closed fields. Then one can see that the corresponding $\infty$-category can be regarded as the corresponding generalized Galois representations with coefficients over those generalized de Rham period sheaves then the corresponding result will follow since the 6-functor will be reduced to 6-functors among condensed generalized Galois representations with trivial Galois groups.
\end{proof}

\begin{definition}
Assume we are in the situation in \cref{assumption2}. Since $\mathrm{SolidQuasiCoh}_{\mathrm{deRhamRobba}_*}$ are symmetrical monoidal $\infty$-categories, the corresponding motivic Galois group exists as the corresponding Tannakian groups we will use:
\begin{align}
\pi_{\mathrm{SolidQuasiCoh}_{\mathrm{deRhamRobba}_*}}
\end{align}
to denote the group in any particular situation over $*$.
\end{definition}

\begin{theorem}
Assume we are in the situation in \cref{assumption2}. $\mathrm{PerfComplex}\mathrm{SolidQuasiCoh}_{\mathrm{deRhamRobba}_*}$ or $\mathrm{PerfComplex}\mathrm{SolidQuasiCoh}_{\mathrm{deRhamPrismatization}_*}$ admits pullback functor and pushforward functor $F^\square$ and $F_\square$ in the 6-functor formalism where $*$ is varying in the category of all the small $v$-stacks (while for the de Rham primatization $*$ can also be the corresponding small arc stacks). Here the notation means we consider all the perfect complexes in the solid quasicoherent sheaves. 
\end{theorem}

\begin{proof}
We follow the method of proof after Scholze in \cite{1S5} by passing to totally disconnected subspaces from the adic spectrum of corresponding algebraically closed fields. Then one can see that the corresponding category can be regarded as the corresponding generalized Galois representations with coefficients over those generalized de Rham period sheaves then the corresponding result will follow since the 6-functor will be reduced to 6-functors among condensed generalized Galois representations with trivial Galois groups. Then the corresponding functor will be either we take the base change over the genearalized de Rham rings with respect to different algebraically closed fields or we consider the compositions of homomorphisms from the Galois groups. But over the geometric points these are trivial. 
\end{proof}

\begin{remark} \mbox{\textbf{(Full 6-functor formalism)}}
The $!$-adjoint pairs in the current context are obviously those maps in \cite{1S5}, \cite{1S6} but the tricky part in the proof is to derived preservation of the corresponding local finiteness theorem for the perfect complexes. However one can easily prove this by using the corresponding idea presented above following \cite{1S5}, \cite{1S6} by consider the corresponding geometric points, which then eventually reduct to modules (with trivial Galois actions) over the integral generalized de Rham period rings for these geometric points. We use the notation $F_\blacksquare$ and $F^\blacksquare$ to denote the $!$-adjoint pairs as in \cite{1S5}, \cite{1S6}.
\end{remark}

\subsection{Application to Local Langlands} \label{subsection7.3}

\noindent We now follow \cite{1S5}, \cite{1S6}, \cite{L1}, \cite{1FS} to construct some generalized version of the local Langlands correspondence in \cite{1FS} by using the de Rham prismatization we constructed above in the generalized setting.

\begin{assumption} \label{assumption3}
All our considerations below in \cref{subsection7.3} on the prismatization, de Rham Robba stackification and de Rham prismatization are assumed to be the generalized ones, i.e. we consider the corresponding prismatizations with $b^{1/2}$ added, de Rham Robba stackifications with $b^{1/2}$ added and de Rham prismatizations with $b^{1/2}$ added. 
\end{assumption}

\begin{assumption}
The de Rham prismatizations and de Rham Robba stackifications in the following are all $p$-adic.
\end{assumption}

\noindent Recall the corresponding context in \cite{1FS} we have the corresponding $p$/$z$-adic group $G(F)$ with some local field $F$, this will provide the corresponding small arc stacks as in \cite{1S5}, \cite{1S6}. Since we have the Tannakian groups defined above, we can tranform a representation of the Tannakian group into the corresponding category on the other side. This process will define the following correponding functor.

\begin{theorem}
Assume we are in the situation of \cref{assumption3}. We have well-defined functors which are well-defined and isomorphisms:
\begin{align}
\mathrm{Repre}(?) \overset{\sim}{\rightarrow}  ! 
\end{align}
$?$ = $\pi_{\mathrm{SolidQuasiCoh}_{\mathrm{deRhamPrismatization}_*}}$, or $\pi_{\mathrm{SolidQuasiCoh}_{\mathrm{deRhamRobba}_*}}$, $!$ = $\mathrm{SolidQuasiCoh}_{\mathrm{deRhamPrismatization}_*}$ or $\mathrm{SolidQuasiCoh}_{\mathrm{deRhamRobba}_*}$.
\end{theorem}

\begin{proof}
By Tannakian formalism.
\end{proof}

\begin{corollary}
Assume we are in the situation of \cref{assumption3}. For $p$-adic representations of the Langlands dual group with coefficients in $\overline{\mathbb{Q}}_p$, by taking composition with the representation of the Tannakian group $\pi_{\mathrm{SolidQuasiCoh}_{\mathrm{deRhamPrismatization}_*}}$, or $\pi_{\mathrm{SolidQuasiCoh}_{\mathrm{deRhamRobba}_*}}$, we end up with certain complexes in $\mathrm{SolidQuasiCoh}_{\mathrm{deRhamPrismatization}_*}$ or $\mathrm{SolidQuasiCoh}_{\mathrm{deRhamRobba}_*}$.
\end{corollary}

\indent Now we consider the moduli $v$-stacks in \cite{1FS}, we denote it by $Y_{\mathrm{FS},G}$ which is the $v$-stack of $G$-bundles in both the equal characteristic and mixed characteristic situations for some local field $F$. We now consider the following following \cite{1FS}, \cite{1GL}. We actually relying on \cite{1FS}, \cite{1S5}, \cite{1S6}  can derive the Hecke operators:

\begin{definition}
Assume we are in the situation of \cref{assumption3}. Consider the map from the Hecke stack in \cite{1FS} which we denote that as $Y_{\mathrm{Hecke},G,I}$, and consider the map from this to fiber product of the Cartier stack $Y_\mathrm{Cartier}$ with the corresponding stack $Y_{\mathrm{FS},G}$, and consider the map from this Hecke stack to the $Y_{\mathrm{FS},G}$. Pulling back along the second and push-forward the product with $\square_O$ will define the Hecke operator, where $\square_O$ is defined for some representation of the Langlands full-dual group in the coefficient $\overline{\mathbb{Q}}_p$. By result in \cite{1S5}, \cite{1S6} we have the construction does not depend on the choice of the primes, so we can in some equivalent way to derive a corresponding $\overline{\mathbb{Q}}_p$-complex over the Hecke stacks with some finite set $I$. For instance after \cite{1FS} we have $\overline{\mathbb{Q}}_\ell$-adic complex with $\ell$ away from $p$. Take any motivic sheaf in \cite{1S5}, \cite{1S6} with $\ell$-adic realization which is isomorphic to this complex\footnote{Many things can now be defined over $\mathbb{Z}$ after \cite{1S5}, \cite{1S6}. First the Satake isomorphism can now be defined over $\mathbb{Z}$, one then just take the $p$-adic realization to reach our definition here by taking the base change from $\mathbb{Z}$ to $\mathbb{Z}_p$ which provides the desired complex here over the Hecke stack in order to finish the definition of the Hecke operators in our current setting.}. Then we consider the $p$-adic realization which provides a corresponding $p$-adic complex. Then one can take the base change to the corresponding $\varprojlim_\alpha B^+_\mathrm{dR,\alpha}[b^{1/2}]$-period ring to achieve finally an object in the category we are considering. This gives us desired $p$-adic complex over the Hecke stack, then one defines the corresponding morphisms from the Hecke stack for each finite set as above to $Y_{\mathrm{FS},G}$ and to $Y_{\mathrm{FS},G}\times Y'$. Here $Y'$ is defined to be the $v$-stack of all the solid quasicoherent sheaves over the two de Rham stackifications in our setting, over the Cartier stack $Y_{\mathrm{Cartier}}$ and those products of this Cartier stack. This will produce the desired Hecke operators.
\end{definition}

\begin{theorem}
Assume we are in the situation of \cref{assumption3}. The Hecke operator sends the complexes to those complexes carrying the action from the products of Tannakian group of the Cartier stack for $F$.
\end{theorem}

\begin{proof}
By our definition we have that the corresponding image complexes are those complexes over the corresponding fiber product of $Y_{\mathrm{FS},G}$ with the corresponding classifying stack of the product of the Tannakian group as in the statement of this theorem. For instance one can check this following the idea in \cite{1S5}, \cite{1S6} where we consider each totally disconneted subspace taking the form of the adic spectrum of some algebraically closed field. Over these algebraically closed geometric points we can see that we end up with purely perfect complexes of modules over:
\begin{align}
\varprojlim_\alpha B^+_\mathrm{dR,\alpha}[b^{1/2}],
\end{align}
but we do have the corresponding lattices then, which reduces to the correponding $\overline{\mathbb{Q}}_p$-situation. Here the action of Tannakian group for $F$ will then factors through the action of corresponding Weil group for $F$. Then we are in a situation parallel to \cite{1FS} and we only have to consider the action from the products of the Weil groups. Then the same proof as in \cite{1FS} will derive the result stated. In fact there is nothing to prove here once one follow the same ideas in \cite{1FS}, in particular the proposition IX.1.1.
\end{proof}

\begin{theorem}
Assume we are in the situation of \cref{assumption3}. One direction of the local Langlands holds true in this context: from Schur-irreducible objects to the corresponding $L$-parameters from the $?$ in our current context. Here $?$ = 
\begin{align}
\pi_{\mathrm{SolidQuasiCoh}_{\mathrm{deRhamPrismatization}_*}}, 
\end{align}
or 
\begin{align}
\pi_{\mathrm{SolidQuasiCoh}_{\mathrm{deRhamRobba}_*}}. 
\end{align}
$*$ is just the Cartier stack attached to $F$. Note that all the coefficients on the both sides are $p$-adic. 
\end{theorem}

\begin{proof}
By our construction we do have the mapping to the Bernstein centers in this current setting. Then as in VIII.4.1 and IV.4.1 of \cite{1FS} we can build up the corresponding mapping after \cite{1VL}. To be more precise for each finite set $I$ we can build up the corresponding symmetrical monoidal $\infty$-categories and the Hecke functors as in the above in our current setting, and we have the corresponding equivariant actions from the Tannakian groups on the Target symmetrical monoidal $\infty$-categories over the moduli $v$-stack. Then the excursion operators are generated automatically after \cite{1FS} and \cite{1VL}, where all these general abstract formalism will apply in our setting directly.
\end{proof}

\begin{remark}
Our consideration might be not a \textit{correct} generalization for ultimate $p$-adic local Langlands correspondence after \cite{C}. As in \cite{EGH1} our consideration at least (we don't know more about this) at this moment covers the $p$-adic smooth representations of the $p$-adic reductive groups, since the Hecke operators rely on Scholze's motivicalization with realization into $p$-adic sheaves over $v$-stacks. This might be serious thing to genuinely consider our consideration as correct generalization of Colmez's work and the context in Colmez's work. But our consideration does provide certain local Langlands correspondence in $p$-adic coefficients, and note that our consideration goes beyond the pro-\'etale local systems. However it is less clear to us if the corresponding \textit{$p$-adic spectral action} in such solid quasi-coherent sheaf level will produce the equivalence of the derived $\infty$-categories on the both sides as at least conjectured in \cite{1FS}. This is because we do not at this moment put more requirement on the motivic categories.
\end{remark}

\begin{remark}
This direction of Langlands correspondence is already many to one, however we want to mention that one can have a larger package if one changes the motivic Galois groups on the Weil side. For instance as in \cite{1T} on the Weil side we considered just the Weil groups (and their covering groups), this means that the Weil groups act as if they act through any full motivic Galois group related to them (but only through subgroups or quotient groups). This may enlarge the $L$-packages and may make the categoricalized functor on the $\infty$-categoricalized level fully faithful. However the correspondence still makes sense.
\end{remark}

\newpage
\section{Motivic $p$-adic Local Langlands Correspondence in Families}

\indent Function fields over finite fields share kind of arithmetical similarity to the local fields with mixed-characteristics. However due to the complication of positivity of the characteristics, many parallel obvious analogs of results might be extremely different and difficult. This is a problem for instance in prismatization as in our work in \cite{2T1}. Following \cite{2LH} we give a compactness method by taking the infinite induction from mixed-characteristic fibers to reach the function field by compactification of the natural numbers. We then study prismatization in families following \cite{2G}, \cite{2BS}, \cite{2BL}, \cite{2D}, \cite{2BL2}. Our construction rely on the corresponding construction in the following sense. For each $n$ one can pick some $p$-adic field with ramification index expanding, this process will define all family of the corresponding prismatization, de Rham prismatization and the corresponding Hodge-Tate prismatization. Then we consider the closure of the $\infty$-categories along the corresponding identification that $\overline{\mathbb{N}}= \mathbb{N}^\wedge$. One has a correspondence between categories over some finite $n_0$ and categories over $\infty$ through the corresponding identification on the ring level, which is the key idea in the proof of the generalized local Langlands in families in \cite{2LH}. We follow this idea significantly in this paper to reach the definitions of prismatizations at $\infty$, i.e. a function field. We then have a uniformed strategy on the prismatization. We then apply our construction to motivic $p$-adic local Langlands after \cite{1S5}, \cite{1S6}, \cite{1RS}, \cite{1FS}.

\subsection{Absolute prismatization}

\begin{setting}
We now consider the foundation in \cite{2LH} on the generalized topological ring $R$, which is defined in the following. We fix a finite field $\mathbb{F}_q$ as a universal residue field. Then we fix a family over $\mathbb{N}^\wedge$ (where we add the corresponding infinity as one point to the set $\mathbb{N}$) of local fields:
\begin{align}
K_1,K_2,...,K_\infty.
\end{align}
Then for each $n\in \mathbb{N}^\wedge$ we just fix some uniformizer $t_n\in K_n$. Then we just put $t$ to be the uniformizer for the topological ring $R$ to be defined in the following which evaluates at each $n$ to be $t(n)=t_n$. Then we just consider:
\begin{align}
R = \prod_{n\in \mathbb{N}^\wedge} \mathcal{O}_{K_n}
\end{align}
where we have that:
\begin{align}
R = \varprojlim_k R/t^k R.
\end{align}
Then we have the corresponding ring $K$ by simply inverting all the $t_n,n\in \mathbb{N}^\wedge$. The $\infty$-local field is then required to be a function field with uniformizer $t_\infty$ over the field $\mathbb{F}_q$. 
\end{setting}

\begin{setting}
As in \cite{2LH} we have the corresponding Witt vector $\mathrm{Witt}$ which is a functor from the category of $R$-algebras to the category of all the $R$-algebras which are $t$-complete.
\end{setting}

\begin{definition}
We now define prisms in our current setting. A prism in families is a pair over $\mathbb{N}$, which is denoted by $(L,A)$ where $L$ is a Cartier divisor in $\mathrm{Spec}A$ such that we have $A$ is derived $(t,L)$-complete and we have that the corresponding requirement that for each $n\in \mathbb{N}$ we require that $(L_n,A_n)$ is a prism in \cite{2BS}. This can be regarded as a corresponding family of prisms in \cite{2BS} such as if we put:
\begin{align}
\mathbb{N} = \varinjlim_i \{x\in \mathbb{N}|x\leq i\}.
\end{align} 
Then we can restrict the whole family of prisms in our current definition to each subspaces in this limit. Note that here $t$ is in our notation is just $t|\mathbb{N}$. We use the notation:
\begin{align}
\mathrm{Prism}_\mathbb{N}
\end{align}
to denote all the prisms fibered over $\mathbb{N}$, then we take the fiber product along $\mathbb{N}\hookrightarrow \mathbb{N}^\wedge$ to reach the definition of the prisms over $\mathbb{N}^\wedge$ in particular we have the definition of the prisms at the $\infty$. We use the notation:
\begin{align}
\widetilde{\mathrm{Prism}_\mathbb{N}}
\end{align}
to denote the category of all the prisms over $\mathbb{N}^\wedge$. Here we consider first the following category by taking the quotient by some power of $t$:
\begin{align}
\mathrm{Prism}_{\mathbb{N},R/t^a}
\end{align}
By the identification consider in \cite{2LH} we can take some $n_0\in \mathbb{N}$ and set the corresponding 
\begin{align}
\widetilde{\mathrm{Prism}_\mathbb{N}}(\infty)
\end{align}
to be just:
\begin{align}
\mathrm{Prism}_{\mathbb{N},\mathcal{O}_{K_{n_0}}/t_{n_0}^a}
\end{align}
which produces the corresponding mod $t^a$ category of all the prisms:
\begin{align}
\widetilde{\mathrm{Prism}_\mathbb{N}}_{R/t^a}
\end{align}
and then we define:
\begin{align}
\widetilde{\mathrm{Prism}_\mathbb{N}} = \varprojlim_a \widetilde{\mathrm{Prism}_\mathbb{N}}_{R/t^a}.
\end{align}
\end{definition}

\begin{definition}
We now define the corresponding Cartier stack which is the version of preprismatization in our setting. Consider first the Witt vector functor $\mathrm{Witt}$ fibered over the following category of all the $t$-nilpotent rings over $R$ in a family over $\mathbb{N}^\wedge$:
\begin{align}
\mathrm{Ring}_{\mathrm{Nil},\mathbb{N}^\wedge,t,R}.
\end{align}
Then consider another functor in the following sense which is called the Cartier stack in our setting:
\begin{align}
\mathrm{Cartier},
\end{align}
which parametrize over each $X\in \mathrm{Ring}_{\mathrm{Nil},\mathbb{N}^\wedge,t,R}$ the corresponding generalized Cartier ideal $(L,f)$ mapping under $f$ to $\mathrm{Witt}(X)$.
\end{definition}

\begin{definition}
We now define the corresponding Cartier-Witt stack which is the version of prismatization in our setting. Consider first the Witt vector functor $\mathrm{Witt}$ fibered over the following category of all the $t$-nilpotent rings over $R$ in a family over $\mathbb{N}^\wedge$:
\begin{align}
\mathrm{Ring}_{\mathrm{Nil},\mathbb{N}^\wedge,t,R}.
\end{align}
Then consider another functor in the following sense which is called the Cartier-Witt stack in our setting:
\begin{align}
\mathrm{Cartier}_W,
\end{align}
which parametrize over each $X\in \mathrm{Ring}_{\mathrm{Nil},\mathbb{N}^\wedge,t,R}$ the corresponding generalized Cartier ideal $(L,f)$ mapping under $f$ to $\mathrm{Witt}(X)$. We then furthermore require that the image of $(L,f)$ is principle generated by element which evaluates to a distinguished element for each $n\in \mathbb{N}^\wedge$, as in the usual situation in \cite{2BL}. As in \cite{2BL} we put the structure sheaf $\mathcal{F}$ as for each $X$ we set the section over $(L,f)$ with $f:L\rightarrow \mathrm{Witt}(X)$ to be just $X$ itself. We also consider the prismatization with filtration as in the usual situation, to do this we first set the notation $\mathrm{Cartier}_{W,\mathrm{NFil}}$ for the prismatization carrying the corresponding Nygaard filtration. We define this to be first the prismatization fibered over $\mathbb{N}$. Then we have the morphism:
\begin{align}
\mathrm{Cartier}_{W,\mathrm{NFil}} \rightarrow \mathrm{Cartier}_{W},
\end{align}
fibered over $\mathbb{N}$, by projection to the families over $\mathbb{N}$. Then we take the fibered product from $\mathrm{Cartier}_W|_{\mathbb{N}}$ to $\mathrm{Cartier}_{W}$ to reach the full stack: $\widetilde{\mathrm{Cartier}_{W,\mathrm{NFil}}}$ fibered over $\mathbb{N}^\wedge$.
\end{definition}

\begin{proposition}
We have a natural isomorphism between the following two formal stacks. The first one is the Cartier-Witt stack defined over the $\mathbb{N}^\wedge$ The second one is the compactification along $\mathbb{N}\hookrightarrow \mathbb{N}^\wedge$ of the Cartier-Witt stack over $\mathbb{N}$ after the restriction for instance. This means that we quotient out some $t^a$ and set the section of the stack at $\infty$ to be the section over some point $n'$, then take the inverse limit over $a$.
\end{proposition}

\begin{proof}
In fact we only have to consider what happens when we evaluate both stacks at $\infty$. In such a situation the first stack goes to be exactly the definition we have for the function field. Then for the corresponding compactification onto $\mathbb{N}^\wedge$, along the lifting of the ramification degrees, the corresponding distinguished element participate into the Witt wector to define each finite $n$ prismatization will eventual give rise to the need distinguished element as at $\infty$. This finishes the proof.
\end{proof}

\begin{remark}
By restricting to $\mathbb{N}$ of our prismatization from $\mathbb{N}^\wedge$ we have the following equivalence on the derived $\infty$-categories of complexes:
\begin{align}
\mathrm{Der}_{\mathcal{F}|_{\mathbb{N}}}\overset{\sim}{\rightarrow} \projlim_{(L,A)}\overline{\mathrm{Der}}_{(t,L)}(A)
\end{align}
where $(L,A)$ varies in the set of all the prisms in our setting. $\overline{Der}$ means the completion in the derived sense with respect to $t$ and $L$ over the complexes on the right of this equivalence.
\end{remark}

\begin{definition}
Now here we restrict to the space $\mathbb{N}$ is because that we have no relationship like this at the $\infty$ since that is a function field. But we can then take the corresponding compactification in the following sense. We call the above equivalence open prismatization equivalence $\mathrm{ope}_\mathbb{N}$:
\begin{align}
\mathrm{Der}_{\mathcal{F}|_{\mathbb{N}}}\overset{\sim}{\rightarrow} \projlim_{(L,A)}\overline{\mathrm{Der}}_{(t,L)}(A).
\end{align}
Then we define the following compact prismaization equivalence $\mathrm{cpe}$:
\begin{align}
\mathrm{Der}_{\mathcal{F}}\overset{\sim}{\rightarrow} \widetilde{\projlim_{(L,A)}\overline{\mathrm{Der}}_{(t,L)}(A)}.
\end{align}
Here the right hand side is the $\infty$-category which is the fiber product of $\mathrm{Der}_{\mathcal{F}}$ and $\projlim_{(L,A)}\overline{\mathrm{Der}}_{(t,L)}(A)$ along the map to the restriction over $\mathbb{N}$. 
\end{definition}

\begin{definition}
This also has a filtered version from Nygaard filtraion. Now here we restrict to the space $\mathbb{N}$ is because that we have no relationship like this at the $\infty$ since that is a function field. But we can then take the corresponding compactification in the following sense. We call the above equivalence open prismatization equivalence $\mathrm{ope}_\mathbb{N}$:
\begin{align}
\mathrm{Der}_{\mathcal{F}|_{\mathbb{N}}}\overset{\sim}{\rightarrow} \projlim_{(L,A,\mathrm{NFil})}\overline{\mathrm{Der}}_{(t,L)}(A).
\end{align}
Then we define the following compact prismaization equivalence $\mathrm{cpe}$:
\begin{align}
\mathrm{Der}_{\mathcal{F}}\overset{\sim}{\rightarrow} \widetilde{\projlim_{(L,A,\mathrm{NFil})}\overline{\mathrm{Der}}_{(t,L)}(A)}.
\end{align}
Here the right hand side is the $\infty$-category which is the fiber product of $\mathrm{Der}_{\mathcal{F}}$ and $\projlim_{(L,A)}\overline{\mathrm{Der}}_{(t,L)}(A)$ along the map to the restriction over $\mathbb{N}$. \end{definition}

\begin{proposition}
The fiber product resulted isomorphism:
\begin{align}
\mathrm{Der}_{\mathcal{F}}\overset{\sim}{\rightarrow} \widetilde{\projlim_{(L,A)}\overline{\mathrm{Der}}_{(t,L)}(A)},
\end{align}
is the same as the $\mathbb{N}\hookrightarrow \mathbb{N}^\wedge$ compactification of the restriction over $\mathbb{N}$ of this isomorphism.
\end{proposition}

\begin{proof}
This is because we have that the category $\mathrm{Der}_{\mathcal{F}}$ has the exactly the same behavior, though it can be defined directly using the Witt vector in \cite{2LH}.
\end{proof}

\subsection{Prismatization for $R$-formal schemes}

\begin{setting}
We now consider a $R$-formal ring which is $t$-adic complete, which is denoted by $H$. We assume that $H$ is fibered over $\mathbb{N}^\wedge$.
\end{setting}

\begin{definition}
We now define the corresponding Cartier stack which is the version of preprismatization in our setting. Consider first the Witt vector functor $\mathrm{Witt}$ fibered over the following category of all the $t$-nilpotent rings over $R$ in a family over $\mathbb{N}^\wedge$:
\begin{align}
\mathrm{Ring}_{\mathrm{Nil},\mathbb{N}^\wedge,t,R}.
\end{align}
Then consider another functor in the following sense which is called the Cartier stack in our setting:
\begin{align}
\mathrm{Cartier}_H,
\end{align}
which parametrize over each $X\in \mathrm{Ring}_{\mathrm{Nil},\mathbb{N}^\wedge,t,R}$ the corresponding generalized Cartier ideal $(L,f)$ mapping under $f$ to $\mathrm{Witt}(X)$. We then require further that there is open immersion from $\mathrm{Spec}\mathrm{Witt}(X)/L$. 
\end{definition}

\begin{definition}
We now define the corresponding Cartier-Witt stack which is the version of prismatization in our setting. Consider first the Witt vector functor $\mathrm{Witt}$ fibered over the following category of all the $t$-nilpotent rings over $R$ in a family over $\mathbb{N}^\wedge$:
\begin{align}
\mathrm{Ring}_{\mathrm{Nil},\mathbb{N}^\wedge,t,R}.
\end{align}
Then consider another functor in the following sense which is called the Cartier-Witt  stack in our setting:
\begin{align}
\mathrm{Cartier}_{W,H},
\end{align}
which parametrize over each $X\in \mathrm{Ring}_{\mathrm{Nil},\mathbb{N}^\wedge,t,R}$ the corresponding generalized Cartier ideal $(L,f)$ mapping under $f$ to $\mathrm{Witt}(X)$. We then furthermore require that the image of $(L,f)$ is principle generated by element which evaluates to a distinguished element for each $n\in \mathbb{N}^\wedge$, as in the usual situation in \cite{2BL}. As in \cite{2BL} we put the structure sheaf $\mathcal{F}$ as for each $X$ we set the section over $(L,f)$ with $f:L\rightarrow \mathrm{Witt}(X)$ to be just $X$ itself. We then require further that there is open immersion from $\mathrm{Spec}\mathrm{Witt}(X)/L$. We also consider the prismatization with filtration as in the usual situation, to do this we first set the notation $\mathrm{Cartier}_{W,H,\mathrm{NFil}}$ for the prismatization carrying the corresponding Nygaard filtration. We define this to be first the prismatization fibered over $\mathbb{N}$. Then we have the morphism:
\begin{align}
\mathrm{Cartier}_{W,H,\mathrm{NFil}} \rightarrow \mathrm{Cartier}_{W,H},
\end{align}
fibered over $\mathbb{N}$, by projection to the families over $\mathbb{N}$. Then we take the fibered product from $\mathrm{Cartier}_{W,H}|_{\mathbb{N}}$ to $\mathrm{Cartier}_{W,H}$ to reach the full stack: $\widetilde{\mathrm{Cartier}_{W,H,\mathrm{NFil}}}$ fibered over $\mathbb{N}^\wedge$.
\end{definition}

\begin{remark}
By restricting to $\mathbb{N}$ of our prismatization from $\mathbb{N}^\wedge$ we have the following equivalence on the derived $\infty$-categories of complexes:
\begin{align}
\mathrm{Der}_{\mathcal{F}|_{\mathbb{N}}}\overset{\sim}{\rightarrow} \projlim_{(L,A)/H}\overline{\mathrm{Der}}_{(t,L)}(A)
\end{align}
where $(L,A)$ varies in the set of all the prisms in our setting, over $H$. $\overline{Der}$ means the completion in the derived sense with respect to $t$ and $L$ over the complexes on the right of this equivalence.
\end{remark}

\begin{definition}
Now here we restrict to the space $\mathbb{N}$ is because that we have no relationship like this at the $\infty$ since that is a function field. But we can then take the corresponding compactification in the following sense. We call the above equivalence open prismatization equivalence $\mathrm{ope}_\mathbb{N}$:
\begin{align}
\mathrm{Der}_{\mathcal{F}|_{\mathbb{N}}}\overset{\sim}{\rightarrow} \projlim_{(L,A)/H}\overline{\mathrm{Der}}_{(t,L)}(A).
\end{align}
Then we define the following compact prismaization equivalence $\mathrm{cpe}$:
\begin{align}
\mathrm{Der}_{\mathcal{F}}\overset{\sim}{\rightarrow} \widetilde{\projlim_{(L,A)/H}\overline{\mathrm{Der}}_{(t,L)}(A)}.
\end{align}
Here the right hand side is the $\infty$-category which is the fiber product of $\mathrm{Der}_{\mathcal{F}}$ and $\projlim_{(L,A)}\overline{\mathrm{Der}}_{(t,L)}(A)$ along the map to the restriction over $\mathbb{N}$. \end{definition}

\begin{theorem}
The definitions above for $\mathrm{Cartier}_*$ and $\mathrm{Cartier}_{W,*}$ for $*$ a $R$ formal ring $t$-adic fibered over $\mathbb{N}^\wedge$ admit descent over $R$ formal scheme $S$ $t$-adic fibered over $\mathbb{N}^\wedge$. 
\end{theorem}

\begin{proof}
We reduce to \cite{2BL}.
\end{proof}

\begin{remark}
By restricting to $\mathbb{N}$ of our prismatization from $\mathbb{N}^\wedge$ we have the following equivalence on the derived $\infty$-categories of complexes:
\begin{align}
\mathrm{Der}_{\mathcal{F}|_{\mathbb{N}}}\overset{\sim}{\rightarrow} \projlim_{(L,A)/S}\overline{\mathrm{Der}}_{(t,L)}(A)
\end{align}
where $(L,A)$ varies in the set of all the prisms in our setting, over $H$. $\overline{Der}$ means the completion in the derived sense with respect to $t$ and $L$ over the complexes on the right of this equivalence.
\end{remark}

\begin{definition}
Now here we restrict to the space $\mathbb{N}$ is because that we have no relationship like this at the $\infty$ since that is a function field. But we can then take the corresponding compactification in the following sense. We call the above equivalence open prismatization equivalence $\mathrm{ope}_\mathbb{N}$:
\begin{align}
\mathrm{Der}_{\mathcal{F}|_{\mathbb{N}}}\overset{\sim}{\rightarrow} \projlim_{(L,A)/S}\overline{\mathrm{Der}}_{(t,L)}(A).
\end{align}
Then we define the following compact prismaization equivalence $\mathrm{cpe}$:
\begin{align}
\mathrm{Der}_{\mathcal{F}}\overset{\sim}{\rightarrow} \widetilde{\projlim_{(L,A)/S}\overline{\mathrm{Der}}_{(t,L)}(A)}.
\end{align}
Here the right hand side is the $\infty$-category which is the fiber product of $\mathrm{Der}_{\mathcal{F}}$ and $\projlim_{(L,A)}\overline{\mathrm{Der}}_{(t,L)}(A)$ along the map to the restriction over $\mathbb{N}$. 
\end{definition}

\begin{definition}
We have the definition in the Nygaard filtered situation. Now here we restrict to the space $\mathbb{N}$ is because that we have no relationship like this at the $\infty$ since that is a function field. But we can then take the corresponding compactification in the following sense. We call the above equivalence open prismatization equivalence $\mathrm{ope}_\mathbb{N}$:
\begin{align}
\mathrm{Der}_{\mathcal{F}|_{\mathbb{N}}}\overset{\sim}{\rightarrow} \projlim_{(L,A,\mathrm{NFil})/S}\overline{\mathrm{Der}}_{(t,L)}(A).
\end{align}
Then we define the following compact prismaization equivalence $\mathrm{cpe}$:
\begin{align}
\mathrm{Der}_{\mathcal{F}}\overset{\sim}{\rightarrow} \widetilde{\projlim_{(L,A,\mathrm{NFil})/S}\overline{\mathrm{Der}}_{(t,L)}(A)}.
\end{align}
Here the right hand side is the $\infty$-category which is the fiber product of $\mathrm{Der}_{\mathcal{F}}$ and $\projlim_{(L,A)}\overline{\mathrm{Der}}_{(t,L)}(A)$ along the map to the restriction over $\mathbb{N}$. 
\end{definition}

\begin{proposition}
We have a natural isomorphism between the following two formal stacks. The first one is the Cartier-Witt stack defined over the $\mathbb{N}^\wedge$ The second one is the compactification along $\mathbb{N}\hookrightarrow \mathbb{N}^\wedge$ of the Cartier-Witt stack over $\mathbb{N}$ after the restriction for instance.
\end{proposition}

\begin{proof}
In fact we only have to consider what happens when we evaluate both stacks at $\infty$. In such a situation the first stack goes to be exactly the definition we have for the function field. Then for the corresponding compactification onto $\mathbb{N}^\wedge$, along the lifting of the ramification degrees, the corresponding distinguished element participate into the Witt wector to define each finite $n$ prismatization will eventual give rise to the need distinguished element as at $\infty$. This finishes the proof.
\end{proof}

\subsection{Absolute de Rham prismatization in families}

By restricting to $\mathbb{N}$ of our prismatization from $\mathbb{N}^\wedge$ we have the following equivalence on the derived $\infty$-categories of complexes:
\begin{align}
\mathrm{Der}_{\mathcal{F}|_{\mathbb{N}}}\overset{\sim}{\rightarrow} \projlim_{(L,A)}\overline{\mathrm{Der}}_{(t,L)}(A)
\end{align}
where $(L,A)$ varies in the set of all the prisms in our setting. $\overline{\mathrm{Der}}$ means the completion in the derived sense with respect to $t$ and $L$ over the complexes on the right of this equivalence. Under this equivalence we define the corresponding de Rham prismatization to be the prismatization over $\mathbb{N}$ associated with $\projlim_{(L,A)}\overline{\mathrm{Der}}_{(t,L)}(\widehat{A[1/t]}^{L})$.

\begin{definition}
Now here we restrict to the space $\mathbb{N}$ is because that we have no relationship like this at the $\infty$ since that is a function field. But we can then take the corresponding compactification in the following sense. We call the above open de Rham prismatization $\mathrm{odrp}_\mathbb{N}$:
\begin{align}
\projlim_{(L,A)}\overline{\mathrm{Der}}_{(t,L)}(\widehat{A[1/t]}^{L}).
\end{align}
Then we define the following compact de Rham prismaization $\mathrm{cdrp}$:
\begin{align}
\widetilde{\projlim_{(L,A)}\overline{\mathrm{Der}}_{(t,L)}(\widehat{A[1/t]}^{L})}.
\end{align}
Here the write hand side is the $\infty$-category which is the fiber product of $\mathbb{N}^\wedge$ and $\projlim_{(L,A)}\overline{\mathrm{Der}}_{(t,L)}(\widehat{A[1/t]}^{L})$ along the map to the restriction over $\mathbb{N}$. 
This compactification means the process defined in the following. First consider the corresponding functors away from the infinity:
\begin{align}
\mathrm{Der}_{\mathcal{F}|_{\mathbb{N}}}\overset{\sim}{\rightarrow} \projlim_{(L,A)}\overline{\mathrm{Der}}_{(t,L)}(A) \overset{}{\rightarrow} \projlim_{(L,A)}\overline{\mathrm{Der}}_{(t,L)}(\widehat{A[1/t]}^L)
\end{align}
Then we just take the corresponding fiber product along this map from $\mathrm{Der}_{\mathcal{F}|_{\mathbb{N}}}$ to the full stack over $\mathbb{N}^\wedge$.
\end{definition}

\begin{definition}
We have the Nygaard filtered version. Now here we restrict to the space $\mathbb{N}$ is because that we have no relationship like this at the $\infty$ since that is a function field. But we can then take the corresponding compactification in the following sense. We call the above open de Rham prismatization $\mathrm{odrp}_\mathbb{N}$:
\begin{align}
\projlim_{(L,A,\mathrm{NFil})}\overline{\mathrm{Der}}_{(t,L)}(\widehat{A[1/t]}^{L}).
\end{align}
Then we define the following compact de Rham prismaization $\mathrm{cdrp}$:
\begin{align}
\widetilde{\projlim_{(L,A,\mathrm{NFil})}\overline{\mathrm{Der}}_{(t,L)}(\widehat{A[1/t]}^{L})}.
\end{align}
Here the write hand side is the $\infty$-category which is the fiber product of $\mathbb{N}^\wedge$ and $\projlim_{(L,A)}\overline{\mathrm{Der}}_{(t,L)}(\widehat{A[1/t]}^{L})$ along the map to the restriction over $\mathbb{N}$. 
This compactification means the process defined in the following. First consider the corresponding functors away from the infinity:
\begin{align}
\mathrm{Der}_{\mathcal{F}|_{\mathbb{N}}}\overset{\sim}{\rightarrow} \projlim_{(L,A,\mathrm{NFil})}\overline{\mathrm{Der}}_{(t,L)}(A) \overset{}{\rightarrow} \projlim_{(L,A,\mathrm{NFil})}\overline{\mathrm{Der}}_{(t,L,\mathrm{NFil})}(\widehat{A[1/t]}^L)
\end{align}
Then we just take the corresponding fiber product along this map from $\mathrm{Der}_{\mathcal{F}|_{\mathbb{N}}}$ to the full stack over $\mathbb{N}^\wedge$.
\end{definition}

\begin{proposition}
This definition is equivalent to the following definition over $\mathbb{N}^\wedge$. For each $k$, We consider the Cartier-Witt Stack over $R$ then we consider those ideals vanishes through the projection $\mathrm{Witt}(\square)\rightarrow \square$ after taking the power $k$. Then we consider the project limit over $k$.
\end{proposition}

\begin{proof}
This is obviously true away from $\infty$ by taking the corresponding restriction to each finite $n\in \mathbb{N}$. Then we have the result since for some $n_0$ the construction on the Witt vectors mode some power $t^a$ will establish some identification of Witt vectors from $n_0$ to $\infty$ along our construction above. 
\end{proof}

\subsection{de Rham Prismatization for $R$-formal schemes}

\begin{setting}
We now consider a $R$-formal ring which is $t$-adic complete, which is denoted by $H$. We assume that $H$ is fibered over $\mathbb{N}^\wedge$.
\end{setting}

By restricting to $\mathbb{N}$ of our prismatization from $\mathbb{N}^\wedge$ we have the following equivalence on the derived $\infty$-categories of complexes:
\begin{align}
\mathrm{Der}_{\mathcal{F}|_{\mathbb{N}}}\overset{\sim}{\rightarrow} \projlim_{(L,A)/H}\overline{\mathrm{Der}}_{(t,L)}(A)
\end{align}
where $(L,A)/H$ varies in the set of all the prisms in our setting. $\overline{\mathrm{Der}}$ means the completion in the derived sense with respect to $t$ and $L$ over the complexes on the right of this equivalence. Under this equivalence we define the corresponding de Rham prismatization to be the prismatization over $\mathbb{N}$ associated with $\projlim_{(L,A)}\overline{\mathrm{Der}}_{(t,L)}(\widehat{A[1/t]}^{L})$.

\begin{definition}
Now here we restrict to the space $\mathbb{N}$ is because that we have no relationship like this at the $\infty$ since that is a function field. But we can then take the corresponding compactification in the following sense. We call the above open de Rham prismatization $\mathrm{odrp}_\mathbb{N}$:
\begin{align}
\projlim_{(L,A)/H}\overline{\mathrm{Der}}_{(t,L)}(\widehat{A[1/t]}^{L}).
\end{align}
Then we define the following compact de Rham prismaization $\mathrm{cdrp}$:
\begin{align}
\widetilde{\projlim_{(L,A)/H}\overline{\mathrm{Der}}_{(t,L)}(\widehat{A[1/t]}^{L})}.
\end{align}
Here the right hand side is the $\infty$-category which is the fiber product of $\mathbb{N}^\wedge$ and $\projlim_{(L,A)/H}\overline{\mathrm{Der}}_{(t,L)}(\widehat{A[1/t]}^{L})$ along the map to the restriction over $\mathbb{N}$.  
One then glues over $S$ by reducing to \cite{2BL}. This compactification means the process defined in the following. First consider the corresponding functors away from the infinity:
\begin{align}
\mathrm{Der}_{\mathcal{F}|_{\mathbb{N}}}\overset{\sim}{\rightarrow} \projlim_{(L,A)/H}\overline{\mathrm{Der}}_{(t,L)}(A) \overset{}{\rightarrow} \projlim_{(L,A)/H}\overline{\mathrm{Der}}_{(t,L)}(\widehat{A[1/t]}^L)
\end{align}
Then we just take the corresponding fiber product along this map from $\mathrm{Der}_{\mathcal{F}|_{\mathbb{N}}}$ to the full stack over $\mathbb{N}^\wedge$.
\end{definition}

\begin{definition}
When we have the Nygaard filtration we can define the following by setting all objects to carry such filtrations. Now here we restrict to the space $\mathbb{N}$ is because that we have no relationship like this at the $\infty$ since that is a function field. But we can then take the corresponding compactification in the following sense. We call the above open de Rham prismatization $\mathrm{odrp}_\mathbb{N}$:
\begin{align}
\projlim_{(L,A,\mathrm{NFil})/H}\overline{\mathrm{Der}}_{(t,L)}(\widehat{A[1/t]}^{L}).
\end{align}
Then we define the following compact de Rham prismaization $\mathrm{cdrp}$:
\begin{align}
\widetilde{\projlim_{(L,A,\mathrm{NFil})/H}\overline{\mathrm{Der}}_{(t,L)}(\widehat{A[1/t]}^{L})}.
\end{align}
Here the right hand side is the $\infty$-category which is the fiber product of $\mathbb{N}^\wedge$ and $\projlim_{(L,A)/H}\overline{\mathrm{Der}}_{(t,L)}(\widehat{A[1/t]}^{L})$ along the map to the restriction over $\mathbb{N}$. 
One then glues over $S$ by reducing to \cite{2BL}. This compactification means the process defined in the following. First consider the corresponding functors away from the infinity:
\begin{align}
\mathrm{Der}_{\mathcal{F}|_{\mathbb{N}}}\overset{\sim}{\rightarrow} \projlim_{(L,A,\mathrm{NFil})/H}\overline{\mathrm{Der}}_{(t,L)}(A) \overset{}{\rightarrow} \projlim_{(L,A,\mathrm{NFil})/H}\overline{\mathrm{Der}}_{(t,L)}(\widehat{A[1/t]}^L)
\end{align}
Then we just take the corresponding fiber product along this map from $\mathrm{Der}_{\mathcal{F}|_{\mathbb{N}}}$ to the full stack over $\mathbb{N}^\wedge$.
\end{definition}

\begin{remark}
The strategy which is sort of indirect here we use here for constructing the de Rham prismatization over the whole space $\mathbb{N}^\wedge$ is kind of compact method. The idea is sort of infinite induction by taking the compactification of the construction along the compactification:
\begin{align}
\mathbb{N} \hookrightarrow \mathbb{N}^\wedge.
\end{align} 
This is non-trivial in many situation, where the above discussion on the function field prismatization and finding direct relationship with function field prisms is a very typical example, where we can reach some results on the $\infty$ function field at the $\infty$ by taking limit of results over the compactification, especially when the function field objects are hard to study. Even if they are not hard to study this strategy can provide many direct convience for the constructions for function fields by using the results in mixed-characteristic situations.
\end{remark}

\subsection{Absolute Hodge-Tate prismatization in families}

By restricting to $\mathbb{N}$ of our prismatization from $\mathbb{N}^\wedge$ we have the following equivalence on the derived $\infty$-categories of complexes:
\begin{align}
\mathrm{Der}_{\mathcal{F}|_{\mathbb{N}}}\overset{\sim}{\rightarrow} \projlim_{(L,A)}\overline{\mathrm{Der}}_{(t,L)}(A)
\end{align}
where $(L,A)$ varies in the set of all the prisms in our setting. $\overline{\mathrm{Der}}$ means the completion in the derived sense with respect to $t$ and $L$ over the complexes on the right of this equivalence. Under this equivalence we define the corresponding de Rham prismatization to be the prismatization over $\mathbb{N}$ associated with $\projlim_{(L,A)}\overline{\mathrm{Der}}_{(t,L)}(\widehat{A[1/t]}^{L})$. Then we take the corresponding first degree in the graded of the de Rham completion in the definition, which we call them Hodge-Tate prismatization in family over $\mathbb{N}$.

\begin{definition}
Now here we restrict to the space $\mathbb{N}$ is because that we have no relationship like this at the $\infty$ since that is a function field. But we can then take the corresponding compactification in the following sense. We call the above open de Rham prismatization $\mathrm{odrp}_\mathbb{N}$:
\begin{align}
\projlim_{(L,A)}\overline{\mathrm{Der}}_{(t,L)}(\widehat{A[1/t]}^{L}).
\end{align}
Then we define the following compact de Rham prismaization $\mathrm{cdrp}$:
\begin{align}
\widetilde{\projlim_{(L,A)}\overline{\mathrm{Der}}_{(t,L)}(\widehat{A[1/t]}^{L})}.
\end{align}
Here the write hand side is the $\infty$-category which is the fiber product of $\mathbb{N}^\wedge$ and $\projlim_{(L,A)}\overline{\mathrm{Der}}_{(t,L)}(\widehat{A[1/t]}^{L})$ along the map to the restriction over $\mathbb{N}$. Then we take the corresponding first degree in the graded of the de Rham completion in the definition, which we call them Hodge-Tate prismatization in family over $\mathbb{N}^\wedge$. This is equivalent to the following definition: one can define this by first consider the subspace $\mathbb{N}$ where one can take the product throughout the whole $\mathbb{N}$ from the Hodge-Tate prismatization in the usual situation for each $n\in \mathbb{N}$, then take the fiber product along the following functors:
\begin{align}
\mathrm{Der}_{\mathcal{F}|_{\mathbb{N}}}\overset{\sim}{\rightarrow} \projlim_{(L,A)}\overline{\mathrm{Der}}_{(t,L)}(A) \overset{}{\rightarrow} \projlim_{(L,A)}\overline{\mathrm{Der}}_{(t,L)}(\widehat{A[1/t]}^L) \overset{}{\rightarrow} \projlim_{(L,A)}\overline{\mathrm{Der}}_{(t,L)}(\widehat{A[1/t]}/L).
\end{align} 
\end{definition}

\begin{proposition}
One can equivalently define the corresponding following Hodge-Tate prismatization in the following different way, i.e. just consider those substacks of $\mathrm{Cartier}_{W}$ parametrizing those Cartier-Witt ideals which are Hodge-Tate ones over $\mathbb{N}^\wedge$ as in \cite{2BL}. That is to say those ideals in the kernel of the map $\mathrm{Witt}(.)\rightarrow .$ in our current setting over $\mathbb{N}^\wedge$.
\end{proposition}

\begin{proof}
The reason for this to be true is that this is just the underlying stackification for the corresponding Hodge-Tate quasi-coherent complexes over the Hodge-Tate stacks in our setting. Our definition makes the set of ideals implicit. Since the such stack will be limit of the one from $\mathbb{N}$ where we do have the corresponding equivalence, then we are done for the same limiting process.
\end{proof}

\subsection{Hodge-Tate Prismatization for $R$-formal schemes}

\begin{setting}
We now consider a $R$-formal ring which is $t$-adic complete, which is denoted by $H$. We assume that $H$ is fibered over $\mathbb{N}^\wedge$.
\end{setting}

By restricting to $\mathbb{N}$ of our prismatization from $\mathbb{N}^\wedge$ we have the following equivalence on the derived $\infty$-categories of complexes:
\begin{align}
\mathrm{Der}_{\mathcal{F}|_{\mathbb{N}}}\overset{\sim}{\rightarrow} \projlim_{(L,A)/H}\overline{\mathrm{Der}}_{(t,L)}(A)
\end{align}
where $(L,A)/H$ varies in the set of all the prisms in our setting. $\overline{Der}$ means the completion in the derived sense with respect to $t$ and $L$ over the complexes on the right of this equivalence. Under this equivalence we define the corresponding de Rham prismatization to be the prismatization over $\mathbb{N}$ associated with $\projlim_{(L,A)}\overline{\mathrm{Der}}_{(t,L)}(\widehat{A[1/t]}^{L})$. Then we take the graded pieces and take the first degree one to define Hodge-Tate prismatization fibered over $\mathbb{N}^\wedge$. Note that we have the following functors by base change:
\begin{align}
\mathrm{Der}_{\mathcal{F}|_{\mathbb{N}}}\overset{\sim}{\rightarrow} \projlim_{(L,A)/H}\overline{\mathrm{Der}}_{(t,L)}(A) \overset{}{\rightarrow} \projlim_{(L,A)/H}\overline{\mathrm{Der}}_{(t,L)}(\widehat{A[1/t]}^L) \overset{}{\rightarrow} \projlim_{(L,A)/H}\overline{\mathrm{Der}}_{(t,L)}(\widehat{A[1/t]}/L).
\end{align}

\begin{definition}
Now here we restrict to the space $\mathbb{N}$ is because that we have no relationship like this at the $\infty$ since that is a function field. But we can then take the corresponding compactification in the following sense. We call the above open de Rham prismatization $\mathrm{odrp}_\mathbb{N}$:
\begin{align}
\projlim_{(L,A)/H}\overline{\mathrm{Der}}_{(t,L)}(\widehat{A[1/t]}^{L}).
\end{align}
Then we define the following compact de Rham prismaization $\mathrm{cdrp}$:
\begin{align}
\widetilde{\projlim_{(L,A)/H}\overline{\mathrm{Der}}_{(t,L)}(\widehat{A[1/t]}^{L})}.
\end{align}
Here the right hand side is the $\infty$-category which is the fiber product of $\mathbb{N}^\wedge$ and $\projlim_{(L,A)/H}\overline{\mathrm{Der}}_{(t,L)}(\widehat{A[1/t]}^{L})$ along the map to the restriction over $\mathbb{N}$. 
One then glues over $S$ by reducing to \cite{2BL}. Then we take the graded pieces and take the first degree one to define Hodge-Tate prismatization fibered over $\mathbb{N}^\wedge$. This is equivalent to the following definition: one can define this by first consider the subspace $\mathbb{N}$ where one can take the product throughout the whole $\mathbb{N}$ from the Hodge-Tate prismatization in the usual situation for each $n\in \mathbb{N}$, then take the fiber product along the following functors:
\begin{align}
\mathrm{Der}_{\mathcal{F}|_{\mathbb{N}}}\overset{\sim}{\rightarrow} \projlim_{(L,A)/H}\overline{\mathrm{Der}}_{(t,L)}(A) \overset{}{\rightarrow} \projlim_{(L,A)/H}\overline{\mathrm{Der}}_{(t,L)}(\widehat{A[1/t]}^L) \overset{}{\rightarrow} \projlim_{(L,A)/H}\overline{\mathrm{Der}}_{(t,L)}(\widehat{A[1/t]}/L).
\end{align}

\end{definition}

\subsection{$p$-adic Motives in families and $6$-functor formalism}

\noindent In this section we consider the motivic point of view after \cite{2G} and \cite{2A}. The goal is to first consider the motives in our setting over families and then apply to the $B^+_{\mathrm{dR},\mathbb{N}^\wedge}$-cohomology theory dated back to \cite{2F} and the motivic theories for function fields under the foundation we are considering. Followig \cite{2A} we will consider for each $n\in \mathbb{N}$ a corresponding $p$-adic motivic theory and the corresponding Weil sheaves of any $t$-adic formal schemes $\mathbb{Y}$ over $R = \prod R_i$:
\begin{align}
(\mathrm{MotiveTheory}_{n}, \mathrm{WeilSheaves}_n)_{\mathrm{Ayoub}, R_n,\mathbb{Y}}
\end{align}
corresponding to the following three prismatizations:
\begin{itemize}
\item[1] Prismatization for each $n$, which is a motivic theory in the sense of \cite{2A};
\item[2] de Rham prismatization for each $n$, which is a motivic theory in the sense of \cite{2A};
\item[3] Hodge-Tate prismatization for each $n$, which is a motivic theory in the sense of \cite{2A};
\item[4] Other type prismatizaton such as crystalline ones, Laurent ones and so on...
\subitem[4A]  $\mathrm{Der}_{\mathcal{F}|_{\mathbb{N}}}\overset{\sim}{\rightarrow} \projlim_{(L,A)}\overline{\mathrm{Der}}_{(t,L)}(A) \overset{}{\rightarrow} \projlim_{(L,A)}\overline{\mathrm{Der}}_{(t,L)}(\widehat{A[1/L]}^t)$ through compactification at $\infty$; 
\subitem[4B]  $\mathrm{Der}_{\mathcal{F}|_{\mathbb{N}}}\overset{\sim}{\rightarrow} \projlim_{(L,A)}\overline{\mathrm{Der}}_{(t,L)}(A) \overset{}{\rightarrow} \projlim_{(L,A)}\overline{\mathrm{Der}}_{(t,L)}({A[1/t]})$ through compactification at $\infty$;
\subitem[4C] ...
\item[5] Prismatization carrying filtration such as Nygaard filtrations.
\end{itemize}

\begin{definition} \mbox{\textbf{(Motivic Prismatization in Families)}}
In our setting we consider the following motivic theory in families over $K$ by taking the product:
\begin{align}
\prod_{n\in \mathbb{N}} (\mathrm{MotiveTheory}_{n}, \mathrm{WeilSheaves}_n)_{\mathrm{Ayoub}, R_n,\mathbb{Y}}
\end{align}
Then we consider the corresponding compactifiction along the corresponding compactification of the Cartier-Witt stacks in families to get:
\begin{align}
\widetilde{\prod_{n\in \mathbb{N}} (\mathrm{MotiveTheory}_{n}, \mathrm{WeilSheaves}_n)_{\mathrm{Ayoub}, R_n,\mathbb{Y}}}.
\end{align}
This is a motivic theory over $R$ in the sense of \cite{2A}. We use the notation:
\begin{align}
\mathrm{Hopf}_{\widetilde{\prod_{n\in \mathbb{N}} (\mathrm{MotiveTheory}_{n}, \mathrm{WeilSheaves}_n)_{\mathrm{Ayoub}, R_n,\mathbb{Y}}}}
\end{align}
to denote the corresponding Hopf algebra sheaf over $\mathbb{Y}$ for this motivic theory over $R$ in families. Then we define the motivic Galois group sheaf to be the spectrum of this big Hopf algebra sheaf:
\begin{align}
\mathrm{Gal}_{\widetilde{\prod_{n\in \mathbb{N}} (\mathrm{MotiveTheory}_{n}, \mathrm{WeilSheaves}_n)_{\mathrm{Ayoub}, R_n,\mathbb{Y}}}}:= \mathrm{Spec}(\mathrm{Hopf}_{\widetilde{\prod_{n\in \mathbb{N}} (\mathrm{MotiveTheory}_{n}, \mathrm{WeilSheaves}_n)_{\mathrm{Ayoub}, R_n,\mathbb{Y}}}}).
\end{align}
\end{definition}

\begin{definition}  \mbox{\textbf{(Motivic de Rham Prismatization in Families)}}
In our setting we consider the following motivic theory in families over $K$ by taking the product:
\begin{align}
\prod_{n\in \mathbb{N}} (\mathrm{dRMotiveTheory}_{n}, \mathrm{dRWeilSheaves}_n)_{\mathrm{Ayoub}, R_n,\mathbb{Y}}
\end{align}
Then we consider the corresponding compactifiction along the corresponding compactification of the Cartier-Witt stacks in families to get:
\begin{align}
\widetilde{\prod_{n\in \mathbb{N}} (\mathrm{dRMotiveTheory}_{n}, \mathrm{dRWeilSheaves}_n)_{\mathrm{Ayoub}, R_n,\mathbb{Y}}}.
\end{align}
This is a motivic theory over $R$ in the sense of \cite{2A}. We use the notation:
\begin{align}
\mathrm{Hopf}_{\widetilde{\prod_{n\in \mathbb{N}} (\mathrm{dRMotiveTheory}_{n}, \mathrm{dRWeilSheaves}_n)_{\mathrm{Ayoub}, R_n,\mathbb{Y}}}}
\end{align}
to denote the corresponding Hopf algebra sheaf for this motivic theory over $R$ in families. Then we define the motivic Galois group sheaf to be the spectrum of this big Hopf algebra sheaf:
\begin{align}
\mathrm{Gal}_{\widetilde{\prod_{n\in \mathbb{N}} (\mathrm{dRMotiveTheory}_{n}, \mathrm{dRWeilSheaves}_n)_{\mathrm{Ayoub}, R_n,\mathbb{Y}}}}:= \mathrm{Spec}(\mathrm{Hopf}_{\widetilde{\prod_{n\in \mathbb{N}} (\mathrm{dRMotiveTheory}_{n}, \mathrm{dRWeilSheaves}_n)_{\mathrm{Ayoub}, R_n,\mathbb{Y}}}}).
\end{align}
\end{definition}

\begin{definition}  \mbox{\textbf{(Motivic Hodge-Tate Prismatization in Families)}}
In our setting we consider the following motivic theory in families over $K$ by taking the product:
\begin{align}
\prod_{n\in \mathbb{N}} (\mathrm{HTMotiveTheory}_{n}, \mathrm{HTWeilSheaves}_n)_{\mathrm{Ayoub}, R_n,\mathbb{Y}}
\end{align}
Then we consider the corresponding compactifiction along the corresponding compactification of the Cartier-Witt stacks in families to get:
\begin{align}
\widetilde{\prod_{n\in \mathbb{N}} (\mathrm{HTMotiveTheory}_{n}, \mathrm{HTWeilSheaves}_n)_{\mathrm{Ayoub}, R_n,\mathbb{Y}}}.
\end{align}
This is a motivic theory over $R$ in the sense of \cite{2A}. We use the notation:
\begin{align}
\mathrm{Hopf}_{\widetilde{\prod_{n\in \mathbb{N}} (\mathrm{HTMotiveTheory}_{n}, \mathrm{HTWeilSheaves}_n)_{\mathrm{Ayoub}, R_n,\mathbb{Y}}}}
\end{align}
to denote the corresponding Hopf algebra sheaf for this motivic theory over $R$ in families. Then we define the motivic Galois group sheaf to be the spectrum of this big Hopf algebra sheaf:
\begin{align}
\mathrm{Gal}_{\widetilde{\prod_{n\in \mathbb{N}} (\mathrm{HTMotiveTheory}_{n}, \mathrm{HTWeilSheaves}_n)_{\mathrm{Ayoub}, R_n,\mathbb{Y}}}}:= \mathrm{Spec}(\mathrm{Hopf}_{\widetilde{\prod_{n\in \mathbb{N}} (\mathrm{HTMotiveTheory}_{n}, \mathrm{HTWeilSheaves}_n)_{\mathrm{Ayoub}, R_n,\mathbb{Y}}}}).
\end{align}
\end{definition}

\begin{theorem}
Let $X$ be a $v$-stack over $\mathrm{Spd}R$ in family over $\mathbb{N}^\wedge$. Then the corresponding $\mathbb{B}_{\mathrm{HT},\mathbb{N}^\wedge}^+$-cohomology theory and the corresponding $\mathbb{B}_{\mathrm{dR},\mathbb{N}^\wedge}^+$-cohomology theory are motivic theory in the sense of \cite{2A}. 
\end{theorem}

\begin{proof}
Over the $v$-site of $X$, we apply the constructions above, which implies consequence directly away from $\infty$, then our construction will then imply the result through the whole family by taking the compactification.
\end{proof}

\begin{theorem}
Let $X$ be a $v$-stack over $\mathrm{Spd}R$ in family over $\mathbb{N}^\wedge$. Then the corresponding $\mathbb{B}_{\mathrm{HT},\mathbb{N}^\wedge}^+$-cohomology theory and the corresponding $\mathbb{B}_{\mathrm{dR},\mathbb{N}^\wedge}^+$-cohomology theory are motivic theory in the sense of \cite{2A}. We have th $6$-functor formalism in this motivic setting over families. 
\end{theorem}

\begin{proof}
Over the $v$-site of $X$, we apply the constructions above, which implies consequence directly away from $\infty$, then our construction will then imply the result through the whole family by taking the compactification. Then the 6-functor formalism reduces to formal schemes (locally we then reduce the $v$-topology to arc topology after \cite{1S5}, \cite{1S6}), then reduces to schemes. $*$-adjoint pairs are the obvious ones, and $!$-adjoint pairs are those $!$-able morphisms for schemes, for instance one consider the inductive coherent sheaves in \cite{GRI}, \cite{GRII} where a full 6-functor formalism is established. Then foundation from \cite{2A} will apply in this setting.
\end{proof}

\subsection{Application to motivic $p$-adic Local Langlands in families} 

\noindent We now follow \cite{1S5}, \cite{1S6}, \cite{L1}, \cite{1FS}, \cite{2LH} to construct some generalized version of the local Langlands correspondence in \cite{1FS} by using the motivic construction we constructed above in the generalized setting. Recall the corresponding context in \cite{2LH} we have the corresponding $p$/$z$-adic group $G(K)$ for our ring in family $K$, this will provide the corresponding small arc stacks as in \cite{1S5}, \cite{1S6}. Since we have the motivic groups defined above, we can tranform a representation of the motivic group into the corresponding category on the other side. This process will define the following correponding functor.

\begin{theorem}
We have well-defined functor which is well-defined isomorphism:
\begin{align}
\mathrm{Repre}(?) \overset{\sim}{\rightarrow}  ! 
\end{align}
$?$ = $\mathrm{Gal}_{\widetilde{\prod_{n\in \mathbb{N}} (\mathrm{dRMotiveTheory}_{n}, \mathrm{dRWeilSheaves}_n)_{\mathrm{Ayoub}, R_n,\mathbb{Y}}}}$, $!$ = $\widetilde{\prod_{n\in \mathbb{N}} (\mathrm{dRMotiveTheory}_{n}, \mathrm{dRWeilSheaves}_n)_{\mathrm{Ayoub}, R_n,\mathbb{Y}}}$. One also have the gestalten version by promoting to gestalten from the both side of this equivalence, where we use the same notation even in the gestalten context.
\end{theorem}

\begin{proof}
By motivic Hopf formalism in \cite{2A}.
\end{proof}

\indent Now we consider the moduli $v$-stacks in \cite{2LH}, we denote it by $Y_{\mathrm{FS},G,R}$ which is the $v$-stack of $G$-bundles for our family version local ring $K$ and $R$. We now consider the following following \cite{1FS}, \cite{1GL}. We actually relying on \cite{1FS}, \cite{1S5}, \cite{1S6}  can derive the Hecke operators:

\begin{definition}
Consider the map from the Hecke stack in \cite{2LH} which we denote that as $Y_{\mathrm{Hecke},G,R,I}$, and consider the map from this to fiber product of the Cartier stack $Y_{\mathrm{Cartier},R}$ with the corresponding stack $Y_{\mathrm{FS},G,R}$, and consider the map from this Hecke stack to the $Y_{\mathrm{FS},G,R}$. Pulling back along the second and push-forward the product with $\square_O$ will define the Hecke operator, where $\square_O$ is defined for some representation of the Langlands full-dual group in the coefficient $R=\prod R_n$. By result in \cite{1S5}, \cite{1S6} we have the construction does not depend on the choice of the primes, so we can in some equivalent way to derive a corresponding $R$-complex over the Hecke stacks with some finite set $I$. For instance after \cite{1FS} we have ${\mathbb{Q}}_\ell$-adic complex with $\ell$ away from $p$. Take any motivic sheaf in \cite{1S5}, \cite{1S6} with $\ell$-adic realization which is isomorphic to this complex, i.e. we can always work with $\mathbb{Z}$-algebra coefficients. Then we consider the $p$-adic realization which provides a corresponding $p$-adic complex. So in such a way we can first find a $\prod_{n\in \mathbb{N}}R_n$-complex over the Hecke stack by considering each fiber over $n\in \mathbb{N}$, then we can base change to $R$ to reach the function field at $\infty$. Then one can take the base change to the corresponding $\varprojlim_\alpha B^+_\mathrm{dR,\mathbb{N}^\wedge,\alpha}$-period ring to achieve finally an object in the category we are considering. This gives us desired $p$-adic complex over the Hecke stack, then one defines the corresponding morphisms from the Hecke stack for each finite set as above to $Y_{\mathrm{FS},G,R}$ and to $Y_{\mathrm{FS},G,R}\times Y'$. Here $Y'$ is defined to be the $v$-stack of all the solid quasicoherent sheaves over the de Rham stackifications in our family motivic setting, over the Cartier stack $Y_{\mathrm{Cartier},R}$ and those products of this Cartier stack\footnote{Here we do not have to consider $Y'$ actually, since considering the Cartier stack in families, the motivic Galois groupoid directly acts on the +-de Rham sheaves. Therefore one can $Y'$ to be just the Cartier stack in families.}. This will produce the desired Hecke operators.
\end{definition}

\begin{theorem}
The Hecke operator sends the complexes to those complexes carrying the action from the products of motivic Galois group (in the de Rham setting) of the Cartier stack for $K$ fibered over $\mathbb{N}^\wedge$.
\end{theorem}

\begin{proof}
By our definition we have that the corresponding image complexes are those complexes over the corresponding fiber product of $Y_{\mathrm{FS},G,R}$ with the corresponding classifying stack of the product of the motivic Galois group as in the statement of this theorem. For instance one can check this following the idea in \cite{1S5}, \cite{1S6} where we consider each totally disconneted subspace taking the form of the adic spectrum of some algebraically closed field. Over these algebraically closed geometric points we can see that we end up with purely perfect complexes of modules over:
\begin{align}
\varprojlim_\alpha B^+_\mathrm{dR,\mathbb{N}^\wedge,\alpha},
\end{align}
but we do have the corresponding lattices then, which reduces to the correponding $\overline{K}:=\prod_n \overline{K_n}^\wedge$-situation. Here the action of motivic de Rham Galois group for $K$ will then factors through the action of corresponding Weil group for $K$. Then we are in a situation parallel to \cite{2LH} and we only have to consider the action from the products of the Weil groups. Then the same proof as in \cite{1FS}, \cite{2LH} will derive the result stated. In fact there is nothing to prove here once one follows the same ideas in \cite{2LH}, in particular the proposition IX.1.1 in \cite{1FS}. Note that here by \cite{1S5}, \cite{1S6} for each $n\in \mathbb{N}^\wedge$ we can take the base change from $\mathbb{Z}$ to $R_n$ of the corresponding complexes on the smooth representation side.
\end{proof}

\begin{definition}
For quasi-split groups, fixing some $t$-adic character of the a chosen parabolic with the chosen unipotent radical, taking the induction from this radical of the character, one has the $t$-adic Whittaker sheaf $\mathrm{Whittaker}$, which can be regarded as a sheaf over the stack of $L$-parameter over the motivic Galois group in our current setting. Then we have from this theorem, as in \cite{1FS} the \textit{$t$-adic spectral action} from the perfect complexes over the stack of $L$-parameter\footnote{This stack is well-defined as in \cite{1FS}, where construction for any discrete group works.} of our motivic Galois group (discrete algebraic group scheme) to the corresponding derived $\infty$-category of the de Rham prismatized motivic sheaves over $Y_{\mathrm{FS},G,R}$. We denote this as the corresponding $\mathrm{SAction}_{\mathrm{Whittaker}}$. We then use the notation $D\mathrm{SAction}_{\mathrm{Whittaker}}$ to denote the derived $\infty$-categorical action induced as in \cite{HJ}. As in \cite{HJ}, we call the derived $\infty$-category of motivic sheaves in our setting in family:
\begin{align}
\mathrm{Repre}(\mathrm{Gal}_{\widetilde{\prod_{n\in \mathbb{N}} (\mathrm{dRMotiveTheory}_{n}, \mathrm{dRWeilSheaves}_n)_{\mathrm{Ayoub}, R_n,Y_{\mathrm{FS},G,R}}}})
\end{align}
an $\infty$-module over $D\mathrm{SAction}_{\mathrm{Whittaker}}$, in the higher categorical sense.
\end{definition}

\begin{theorem}
The $t$-adic motivic spectral action in our setting is well-defined. As in \cite{HJ}, we call the derived $\infty$-category of motivic sheaves in our setting in family:
\begin{align}
\mathrm{Repre}(\mathrm{Gal}_{\widetilde{\prod_{n\in \mathbb{N}} (\mathrm{dRMotiveTheory}_{n}, \mathrm{dRWeilSheaves}_n)_{\mathrm{Ayoub}, R_n,Y_{\mathrm{FS},G,R}}}})
\end{align}
an $\infty$-module over $D\mathrm{SAction}_{\mathrm{Whittaker}}$, in the higher categorical sense. This is also well-defined.
\end{theorem}

\begin{theorem}
One direction of the local Langlands in families holds true in this context: from Schur-irreducible objects to the corresponding $L$-parameters from the $?$ in our current context. Here $?$ is the motivic de Rham Galois group in familes in our setting over $\mathbb{N}^\wedge$. Note that all the coefficients on the both sides are $t$-adic. Each $t_n,n\in\mathbb{N}^\wedge-{\infty}$ comes from certain $p$-adic local fields, with $p$-fixed. 
\end{theorem}

\begin{proof}
By our construction we do have the mapping to the Bernstein centers in this current setting. Then as in VIII.4.1 and IV.4.1 of \cite{1FS} we can build up the corresponding mapping after \cite{1VL}. To be more precise for each finite set $I$ we can build up the corresponding symmetrical monoidal $\infty$-categories and the Hecke functors as in the above in our current setting, and we have the corresponding equivariant actions from the motivic Galois groups on the target symmetrical monoidal $\infty$-categories over the moduli $v$-stack. Then the excursion operators are generated automatically after \cite{1FS} and \cite{1VL}, where all these general abstract formalism will apply in our setting directly. This directly generalize the work \cite{2LH} as well. One follows the proof of VIII 4.1 in \cite{1FS} to derive the desired mapping in our current situation, with coefficient over $R=\prod_{n\in \mathbb{N}^\wedge} R_n$. The excursion operator space will then be the based change from $\mathbb{Z}$ to $R=\prod_{n\in \mathbb{N}^\wedge} R_n$ of the corresponding space over $\mathbb{Z}$ in \cite{1S5}, \cite{1S6}:
\begin{align}
\underline{S_{\mathrm{cocyc},G_\mathrm{dual},G_\mathrm{dual}-\mathrm{inv}}}\otimes_\mathbb{Z} \left(\prod_{n\in \mathbb{N}^\wedge} R_n\right).  
\end{align}
\end{proof}

\begin{remark}
This theorem considers not only the spaces in families (i.e. the corresponding moduli stacks of $G$-bundles in this setting is fibered over $\mathbb{N}^\wedge$) but also considers the corresponding motivic cohomology theory in families. For instance if the underlying ring is just a single local field, then there is no need to enlarge the motivic cohomological $\infty$-categories, since the motivic Galois group for this single motivic cohomological category will have then zero action on the other motivic cohomolgical components. On the other hand we consider also the motivic coefficient $\infty$-category to be a family version fibered over $\mathbb{N}^\wedge$. One may not have to do this but we choose to consider the family version of the motivic coefficient theory since the obvious reason from $p$-adic Hodge theory, i.e. if we consider Galois group $G_K$ of $K=\prod_{n\in \mathbb{N}^\wedge} K_n$, then the period rings which can be used to study the representation of $G_K$ have to be fibered over $\mathbb{N}^\wedge$, which includes the following: Robba rings in families, de Rham rings in families, cristalline rings in families. 
\end{remark}

\newpage
\section{Motivic Cohomology Theory for Prismatizations in Families}\label{section9}

\indent We now consider the corresponding prismatization in families in the three settings in \cite{3BL}, \cite{3D}. The first one is the prismatization in families, the second one is the filtration prismatization and finally we have the corresponding syntomization prismatization. They can be defined all in the corresponding families way as we did before. In the first two situations we recall our constrution before which will also be put into our current general consideration of motivic cohomology theories. As in \cite{3BSI} we have all the parallel definitions of $\delta$-rigns in function field situation which directly give rise to the \textit{prisms for function fields}.

\subsection{Prismatization, filtration prismatization and syntomization prismatization}

\begin{definition}\label{definition39}
We consider the corresponding prismatization in familes. Recall the definition goes in the following way. Eventually the definition is applied to $z$-adic $A$-formal schemes. In the absolute manner we consider the category of all $z$-nilpotent $A$-algebras as the underlying ring categorie $\mathrm{Nil}_{z,A}$, where all such algebras are assumed to be fibered over $\mathbb{N}_\infty$. Then we use the Witt vector functor $W_A$ from \cite{3LH}. Then the prismaization is the stackification over the ring category $\mathrm{Nil}_{z,A}$ where for each such ring we have the groupoid of all the Cartier-Witt ideals fibered over $\mathbb{N}_\infty$, i.e. those maps locally principally generated by distinguished elements in the big Witt vectors (for $w_0$ we require nilpotency and for $w_1$ we require unitality when we have the coordinates $(w_0,w_1,......)$)\footnote{

\begin{theorem}
Following \cite{3BL}, \cite{3D} one can also equivalently define Cartier-Witt ideals in families over $\mathbb{N}_\infty$ as the ideals in the following sense. Recall we have two maps on $W_A$: one is the map $W_A(\square)\rightarrow \square$, and the other one is the $\delta$-map for this generalized Witt vector functor (for instance see \cite{3BSI}), where we do have the $\delta$-ring structure for the Witt vector ring for function fields. Then one can define equivalently the Cartier-Witt ideals in families over $\mathbb{N}_\infty$ by giving requirement on the image ideals of Cartier ideals in families over $\mathbb{N}_\infty$: again for the second map we require the image ideals to have the property of being unit (unitality), and for the first map we require the image ideals to have the property of being nilpotent (nilpotency). Then this will lead to equivalent definitions in \cref{definition39} of all the prismatizations in families over $\mathbb{N}_\infty$, all the filtration prismatizations in families over $\mathbb{N}_\infty$, and all the syntomization prismatizations in families over $\mathbb{N}_\infty$.
\end{theorem}

\begin{proof}
After using general Witt vector $W_A$ we can see that this holds for all the fibers over $\mathbb{N}_\infty$, where in the function field situation we just use the corresponding $\delta$-ring structure in the function field situation (just replace $p$ by the uniformizater $z_\infty$, for instance see \cite{3BSI}). Then the results over these fibers will provide the result in the correspondig family over $\mathbb{N}_\infty$. Then the definition of the primatization in families over $\mathbb{N}_\infty$ can be given as in \cref{definition39}, equivalently.
\end{proof}

Then this will produce the equivalent definitions of all the \textit{prismatizations in families over $\mathbb{N}_\infty$/filtration prismatizations in families over $\mathbb{N}_\infty$/syntomization prismatizations in families over $\mathbb{N}_\infty$} and \textit{analytic prismatizations in families over $\mathbb{N}_\infty$/analytic filtration prismatizations in families over $\mathbb{N}_\infty$/analytic sytomization  prismatizations in families over $\mathbb{N}_\infty$} below in \cref{section9} and \cref{section10}, including those de Rham ones in families over $\mathbb{N}_\infty$, de Rham-Hodge-Tate ones in families over $\mathbb{N}_\infty$.

}. In this paper we use the notation 
\begin{align}
\underline{\mathrm{Prismatization}}_{\mathrm{abs},\mathbb{N}_\infty}
\end{align}
to denote the absolute prismatization in families in our current setting. Then we have the corresponding filtration prismatization:
\begin{align}
\overline{\underline{\mathrm{Prismatization}}}_{\mathrm{abs},\mathrm{Nygaard},\mathbb{N}}
\end{align}
which is defined to be compactification from $\mathbb{N}$ to $\mathbb{N}_\infty$ of the corresponding open version over $\mathbb{N}$:
\begin{align}
{\underline{\mathrm{Prismatization}}}_{\mathrm{abs},\mathrm{Nygaard},\mathbb{N}}
\end{align}
where this open version is defined to be taking the product over the finite fibers over $\mathbb{N}_\infty$ of the usual filtration prismatizations over $z_n$-nilpotent $A_n$-algebras. The compactification process needs to be using the following morphisms to reach our definition as the corresponding fiber product in families through the compactification:
\begin{align}
\prod_{n\in \mathbb{N}} {\underline{\mathrm{Prismatization}}}_{\mathrm{abs},\mathrm{Nygaard},n}\rightarrow  \prod_{n\in \mathbb{N}} {\underline{\mathrm{Prismatization}}}_{\mathrm{abs},n}
\end{align}
by taking the products along all $\mathbb{N}$ of the usual projection from the corresponding filtration prismatization to the prismatization, and:
\begin{align}
{\underline{\mathrm{Prismatization}}}_{\mathrm{abs},\mathbb{N}_\infty}\rightarrow  \prod_{n\in \mathbb{N}} {\underline{\mathrm{Prismatization}}}_{\mathrm{abs},n}.
\end{align}
Recall that the filtration prismatization is actually certain fibration version of the usual prismatization, then we take further fibration over $\mathbb{N}$. Recall how finally the syntomization prismatization is constructed, which is by taking the descent of the Hodge-Tate morphism and de Rham morphism in the filtration prismatization, along the diagonal morphism from the prismatization with itself. These maps are obviously admitting the corresponding compactification versions in our current setting. Therefore we consider the following three morphisms:
\begin{align}
f_\mathrm{dR}: {\underline{\mathrm{Prismatization}}}_{\mathrm{abs},\mathbb{N}_\infty}\rightarrow \overline{\underline{\mathrm{Prismatization}}}_{\mathrm{abs},\mathrm{Nygaard},\mathbb{N}},
\end{align}
\begin{align}
f_\mathrm{HT}: {\underline{\mathrm{Prismatization}}}_{\mathrm{abs},\mathbb{N}_\infty}\rightarrow \overline{\underline{\mathrm{Prismatization}}}_{\mathrm{abs},\mathrm{Nygaard},\mathbb{N}},
\end{align}
\begin{align}
d:  {\underline{\mathrm{Prismatization}}}_{\mathrm{abs},\mathbb{N}_\infty}\times {\underline{\mathrm{Prismatization}}}_{\mathrm{abs},\mathbb{N}_\infty} \rightarrow {\underline{\mathrm{Prismatization}}}_{\mathrm{abs},\mathbb{N}_\infty}.
\end{align}
All these three morphisms can be also constructed by taking the compactification from the morphism over $\mathbb{N}$ to $\mathbb{N}_\infty$. Then we take the product of $f_\mathrm{dR}$ and $f_\mathrm{HT}$ together and take the corresponding descent along $d$ we have the definition of the syntomization prismatization in families in our current setting over $\mathbb{N}_\infty$:
\begin{align}
\overline{\underline{\mathrm{Prismatization}}}_{\mathrm{abs},\mathrm{syntomization},\mathbb{N}}. 
\end{align}
This finishes the definition. 
\end{definition}

\begin{definition} 
We consider the corresponding prismatization in families. Recall the definition goes in the following way. Eventually the definition is applied to $z$-adic $A$-formal schemes. In the absolute manner we consider the category of all $z$-nilpotent $A$-algebras as the underlying ring categorie $\mathrm{Nil}_{z,A}$, where all such algebras are assumed to be fibered over $\mathbb{N}_\infty$. Then we use the Witt vector functor $W_A$ from \cite{3LH}. Then the prismaization is the stackification over the ring category $\mathrm{Nil}_{z,A}$ where for each such ring we have the groupoid of all the Cartier-Witt ideals fibered over $\mathbb{N}_\infty$, i.e. those maps locally principally generated by distinguished elements in the big Witt vectors (for $w_0$ we require nilpotency and for $w_1$ we require unitality). After the discussion above, we can now apply the whole definition for the 3 prismatization to any $z$-adic $A$-formal scheme $M$, by taking the immersion from the closure of the Cartier-Witt ideals into $M$ in the compatible way. In this paper we use the notation 
\begin{align}
\underline{\mathrm{Prismatization}}_{\mathrm{abs},\mathbb{N}_\infty,M}
\end{align}
to denote the absolute prismatization in families in our current setting. Then we have the corresponding filtration prismatization:
\begin{align}
\overline{\underline{\mathrm{Prismatization}}}_{\mathrm{abs},\mathrm{Nygaard},\mathbb{N},M}
\end{align}
which is defined to be compactification from $\mathbb{N}$ to $\mathbb{N}_\infty$ of the corresponding open version over $\mathbb{N}$:
\begin{align}
{\underline{\mathrm{Prismatization}}}_{\mathrm{abs},\mathrm{Nygaard},\mathbb{N},M}
\end{align}
where this open version is defined to be taking the product over the finite fibers over $\mathbb{N}_\infty$ of the usual filtration prismatizations over $z_n$-nilpotent $A_n$-algebras. The compactification process needs to be using the following morphisms to reach our definition as the corresponding fiber product in families through the compactification:
\begin{align}
\prod_{n\in \mathbb{N}} {\underline{\mathrm{Prismatization}}}_{\mathrm{abs},\mathrm{Nygaard},n,M}\rightarrow  \prod_{n\in \mathbb{N}} {\underline{\mathrm{Prismatization}}}_{\mathrm{abs},n,M}
\end{align}
by taking the products along all $\mathbb{N}$ of the usual projection from the corresponding filtration prismatization to the prismatization, and:
\begin{align}
{\underline{\mathrm{Prismatization}}}_{\mathrm{abs},\mathbb{N}_\infty,M}\rightarrow  \prod_{n\in \mathbb{N}} {\underline{\mathrm{Prismatization}}}_{\mathrm{abs},n,M}.
\end{align}
Recall that the filtration prismatization is actually certain fibration version of the usual prismatization, then we take further fibration over $\mathbb{N}$. Recall how finally the syntomization prismatization is constructed, which is by taking the descent of the Hodge-Tate morphism and de Rham morphism in the filtration prismatization, along the diagonal morphism from the prismatization with itself. These maps are obviously admitting the corresponding compactification versions in our current setting. Therefore we consider the following three morphisms:
\begin{align}
f_\mathrm{dR}: {\underline{\mathrm{Prismatization}}}_{\mathrm{abs},\mathbb{N}_\infty,M}\rightarrow \overline{\underline{\mathrm{Prismatization}}}_{\mathrm{abs},\mathrm{Nygaard},\mathbb{N},M},
\end{align}
\begin{align}
f_\mathrm{HT}: {\underline{\mathrm{Prismatization}}}_{\mathrm{abs},\mathbb{N}_\infty,M}\rightarrow \overline{\underline{\mathrm{Prismatization}}}_{\mathrm{abs},\mathrm{Nygaard},\mathbb{N},M},
\end{align}
\begin{align}
d:  {\underline{\mathrm{Prismatization}}}_{\mathrm{abs},\mathbb{N}_\infty,M}\times {\underline{\mathrm{Prismatization}}}_{\mathrm{abs},\mathbb{N}_\infty,M} \rightarrow {\underline{\mathrm{Prismatization}}}_{\mathrm{abs},\mathbb{N}_\infty,M}.
\end{align}
Then we take the product of $f_\mathrm{dR}$ and $f_\mathrm{HT}$ together and take the corresponding descent along $d$ we have the definition of the syntomization prismatization in families in our current setting over $\mathbb{N}_\infty$:
\begin{align}
\overline{\underline{\mathrm{Prismatization}}}_{\mathrm{abs},\mathrm{syntomization},\mathbb{N},M}.
\end{align}
\end{definition}

\begin{theorem}
There is a version of K\"unneth theorem in this context. To be more precise we have a K\"unneth theorem at least by passing to the corresponding quasicoherent sheaves over the corresponding prismatization for derived formal schemes (i.e. the corresponding stackifications). This holds for all three prismatization stackifications in the families over $\mathbb{N}_\infty$. 
\end{theorem}

\begin{proof}
We need to consider the compactification in this theorem. Any prismatization  stackification for a particular formal scheme $M$ is basically a stackification over $M$ and ringed we use the notation $\mathcal{P}$ to denote the structure sheaf. As in the usual situation we have that the derived $\infty$-category of all the $\mathcal{P}$-modules is equivalent to the prismatic site for $M$. However the derived prismatic cohomology theory in this setting does satisfy the derived K\"unneth theorem on the product of the corresponding $p$-adic formal ring locally (which can be glue over formal schemes like $M$). Then the result can be generalized to our setting over $\mathbb{N}$ away from $\infty$. Then we can take the corresponding compactification along:
\begin{align}
{\underline{\mathrm{Prismatization}}}_{\mathrm{abs},\mathbb{N}_\infty,M}\rightarrow  \prod_{n\in \mathbb{N}} {\underline{\mathrm{Prismatization}}}_{\mathrm{abs},n,M}.
\end{align}
The filtration prismatization and the syntomization prismatization have the same derived version of the K\"unneth theorems, which can be induced from the corresponding result for the prismatization. This finishes the proof. One can prove this directly by using the prisms in families over $\mathbb{N}_\infty$ to form the corresponding $z$-adic prismatic sites, where the $\mathrm{R}\Gamma$ will be quasi-isomorphic to the derived cohomology complex for $z$-adic prismatizations in families in our current consideration, where one just replaces $p$ by $z$. $R\Gamma$ for filtration prismatization in families is just derived base change of the prismatization in families over $\mathbb{N}_\infty$, then one considers the equilizers to reduce the proof for the syntomization prismatization to the the proof for the prismatization and the filtration prismatization.
 \end{proof}

\begin{theorem}
All three prismatization, filtration prismatization and syntomization prismatization over $\mathbb{N}_\infty$ satisfy our \cref{situation1} conditions: (A1), (A2), (A3), (A4), (A5).
\end{theorem}

\begin{proof}
We check this one by one. First for the first condition it is true after we consider the corresponding structure sheaves. The second condition is proved above. For the third condition we consider the corresponding derived $\infty$-categories of quasicoherent sheaves. A4 holds for pullbacks and pushforward. This is true over $\mathbb{N}$ then we consider the fiber product through:
\begin{align}
{\underline{\mathrm{Prismatization}}}_{\mathrm{abs},\mathbb{N}_\infty,M}\rightarrow  \prod_{n\in \mathbb{N}} {\underline{\mathrm{Prismatization}}}_{\mathrm{abs},n,M}
\end{align}
to reach the corresponding compactification. Finally A5 holds see \cite[Chapter 4, in particular 4.7, 4.8, 4.9, 4.10]{3A}. A5 in this setting is a stacky version where the base formal scheme will carry the stackifications over itself whose structure sheaf will provide the desired $\infty$-sheaf of ring after taking the $\infty$-level suspension, where we regard this base formal scheme as a small arc stack. One then finishes the proof as in \cref{theorem36} for the gestalten in the current setting.
\end{proof}

\begin{definition}
We in this context use the notations:
\begin{align}
&\mathrm{Galois}(\underline{\mathrm{Prismatization}}_{\mathrm{abs},\mathbb{N}_\infty,M})\\
&\mathrm{Galois}(\overline{\underline{\mathrm{Prismatization}}}_{\mathrm{abs},\mathrm{Nygaard},\mathbb{N}_\infty,M})\\
&\mathrm{Galois}(\overline{\underline{\mathrm{Prismatization}}}_{\mathrm{abs},\mathbb{N}_\infty,\mathrm{syntomization},M})
\end{align}
to denote the corresponding Hopf algebraic motivic Galois fundamental groups, which is the defined to be the spectra of the associated Hopf algebra.
\end{definition}

\begin{theorem}
There are morphisms from these motivic Galois fundamental groups to the corresponding motivic Galois groups of $A$, and moreover $L$: 
\begin{align}
&\mathrm{Galois}(\underline{\mathrm{Prismatization}}_{\mathrm{abs},\mathbb{N}_\infty,M})\rightarrow \mathrm{Gal}(\overline{A}/A),\\
&\mathrm{Galois}(\overline{\underline{\mathrm{Prismatization}}}_{\mathrm{abs},\mathrm{Nygaard},\mathbb{N}_\infty,M})\rightarrow \mathrm{Gal}(\overline{A}/A),\\
&\mathrm{Galois}(\overline{\underline{\mathrm{Prismatization}}}_{\mathrm{abs},\mathbb{N}_\infty,\mathrm{syntomization},M})\rightarrow \mathrm{Gal}(\overline{A}/A).
\end{align}
\end{theorem}

\begin{proof}
This is formal since we just apply the construction and definition to the point situation (regard the scheme $\mathrm{Spec}A$ or $\mathrm{Spec}L$ as the corresponding small arc-stacks), then the theorem follows by the functoriality of the construction of Hopf algebras.
\end{proof}

\subsection{Three prismatizations in the de Rham setting}

\begin{definition}
We consider the corresponding de Rham prismatization in families. Recall the definition goes in the following way. Eventually the definition is applied to $z$-adic $A$-formal schemes. In the absolute manner we consider the category of all $z$-nilpotent $A$-algebras as the underlying ring categorie $\mathrm{Nil}_{z,A}$, where all such algebras are assumed to be fibered over $\mathbb{N}_\infty$. Then we use the Witt vector functor $W_A$ from \cite{3LH}. Then the prismaization is the stackification over the ring category $\mathrm{Nil}_{z,A}$ where for each such ring we have the groupoid of all the Cartier-Witt ideals fibered over $\mathbb{N}_\infty$, i.e. those maps locally principally generated by distinguished elements in the big Witt vectors (for $w_0$ we require nilpotency and for $w_1$ we require unitality). In this paper we use the notation 
\begin{align}
\underline{\mathrm{Prismatization}}_{\mathrm{abs},\mathrm{dR},\mathbb{N}_\infty}
\end{align}
to denote the absolute de Rham prismatization in families in our current setting. The definition for this is through the corresponding compactification from the stackification over $\mathbb{N}$, since then we have the morphism:
\begin{align}
\mathcal{P}_{\underline{\mathrm{Prismatization}}_{\mathrm{abs},\mathbb{N}_\infty}(\mathbb{N})} \overset{\sim}{\rightarrow} \varprojlim_{(Q_P,P)} D\mathrm{Cat}(P)^\mathrm{comp}\rightarrow \varprojlim_{(Q_P,P)} D\mathrm{Cat}(P[1/z]_{Q_P})^\mathrm{comp}
\end{align}
where $\mathrm{comp}$ is the completion with respect to the natural toplogy induced from the prisms involved along the inverse limit. Then we take the compactification to reach the de Rham prismatization in families $\mathbb{N}_\infty$:
\begin{align}
\mathcal{P}_{\underline{\mathrm{Prismatization}}_{\mathrm{abs},\mathbb{N}_\infty}} \overset{\sim}{\rightarrow} \overline{\varprojlim_{(Q_P,P)} D\mathrm{Cat}(P)^\mathrm{comp}}\rightarrow \overline{\varprojlim_{(Q_P,P)} D\mathrm{Cat}(P[1/z]_{Q_P})^\mathrm{comp}}.
\end{align}
Then we have the corresponding filtration de Rham  prismatization:
\begin{align}
\overline{\underline{\mathrm{Prismatization}}}_{\mathrm{abs},\mathrm{Nygaard},\mathrm{dR},\mathbb{N}}
\end{align}
which is defined to be compactification from $\mathbb{N}$ to $\mathbb{N}_\infty$ of the corresponding open version over $\mathbb{N}$:
\begin{align}
{\underline{\mathrm{Prismatization}}}_{\mathrm{abs},\mathrm{Nygaard},\mathrm{dR},\mathbb{N}}
\end{align}
where this open version is defined to be taking the product over the finite fibers over $\mathbb{N}_\infty$ of the usual filtration de Rham  prismatizations over $z_n$-nilpotent $A_n$-algebras. The compactification process needs to be using the following morphisms to reach our definition as the corresponding fiber product in families through the compactification:
\begin{align}
\prod_{n\in \mathbb{N}} {\underline{\mathrm{Prismatization}}}_{\mathrm{abs},\mathrm{Nygaard},\mathrm{dR},n}\rightarrow  \prod_{n\in \mathbb{N}} {\underline{\mathrm{Prismatization}}}_{\mathrm{abs},\mathrm{dR},n}
\end{align}
by taking the products along all $\mathbb{N}$ of the usual projection from the corresponding filtration de Rham prismatization to the de Rham prismatization, and:
\begin{align}
{\underline{\mathrm{Prismatization}}}_{\mathrm{abs},\mathrm{dR},\mathbb{N}_\infty}\rightarrow  \prod_{n\in \mathbb{N}} {\underline{\mathrm{Prismatization}}}_{\mathrm{abs},\mathrm{dR},n}.
\end{align}
Here the morphism
\begin{align}
\prod_{n\in \mathbb{N}} {\underline{\mathrm{Prismatization}}}_{\mathrm{abs},\mathrm{Nygaard},\mathrm{dR},n}\rightarrow  \prod_{n\in \mathbb{N}} {\underline{\mathrm{Prismatization}}}_{\mathrm{abs},\mathrm{dR},n}
\end{align}
is the base change of the morphism:
\begin{align}
\prod_{n\in \mathbb{N}} {\underline{\mathrm{Prismatization}}}_{\mathrm{abs},\mathrm{Nygaard},n}\rightarrow  \prod_{n\in \mathbb{N}} {\underline{\mathrm{Prismatization}}}_{\mathrm{abs},n}
\end{align}
along the following morphism defining the de Rham prismatizatio in families:
\begin{align}
\prod_{n\in \mathbb{N}} {\underline{\mathrm{Prismatization}}}_{\mathrm{abs},\mathrm{dR},n}\rightarrow  \prod_{n\in \mathbb{N}} {\underline{\mathrm{Prismatization}}}_{\mathrm{abs},n}.
\end{align}
Recall from \cite{3D} that the filtration de Rham prismatization is actually certain fibration version of the usual prismatization, then we take further fibration over $\mathbb{N}$. Recall how finally the syntomization de Rham prismatization is constructed, which is by taking the descent of the Hodge-Tate morphism and de Rham morphism in the filtration prismatization, along the diagonal morphism from the de Rham prismatization with itself. These maps are obviously admitting the corresponding compactification versions in our current setting. Therefore we consider the following three morphisms:
\begin{align}
f_\mathrm{dR}: {\underline{\mathrm{Prismatization}}}_{\mathrm{abs},\mathrm{dR},\mathbb{N}_\infty}\rightarrow \overline{\underline{\mathrm{Prismatization}}}_{\mathrm{abs},\mathrm{Nygaard},\mathrm{dR},\mathbb{N}_\infty},
\end{align}
\begin{align}
f_\mathrm{HT}: {\underline{\mathrm{Prismatization}}}_{\mathrm{abs},\mathrm{dR},\mathbb{N}_\infty}\rightarrow \overline{\underline{\mathrm{Prismatization}}}_{\mathrm{abs},\mathrm{Nygaard},\mathrm{dR},\mathbb{N}_\infty},
\end{align}
\begin{align}
d:  {\underline{\mathrm{Prismatization}}}_{\mathrm{abs},\mathrm{dR},\mathbb{N}_\infty}\times {\underline{\mathrm{Prismatization}}}_{\mathrm{abs},\mathrm{dR},\mathbb{N}_\infty} \rightarrow {\underline{\mathrm{Prismatization}}}_{\mathrm{abs},\mathrm{dR},\mathbb{N}_\infty}.
\end{align}
Then we take the product of $f_\mathrm{dR}$ and $f_\mathrm{HT}$ together and take the corresponding descent along $d$ we have the definition of the syntomization prismatization in families in our current setting over $\mathbb{N}_\infty$:
\begin{align}
\overline{\underline{\mathrm{Prismatization}}}_{\mathrm{abs},\mathrm{dR},\mathrm{syntomization},\mathbb{N}}. 
\end{align}
This finishes the definition. $\square$
\end{definition}

\begin{remark}
The de Rham prismatization, de Rham filtration prismatization, and de Rham syntomization prismatization can all be defined by taking limit along certain substacks of Cartier-Witt stacks, in families over $\mathbb{N}_\infty$, which is parametrized by integer $k$. 
\end{remark}

\begin{definition}
We consider the corresponding de Rham prismatization in families. Recall the definition goes in the following way. Eventually the definition is applied to $z$-adic $A$-formal schemes. In the absolute manner we consider the category of all $z$-nilpotent $A$-algebras as the underlying ring categorie $\mathrm{Nil}_{z,A}$, where all such algebras are assumed to be fibered over $\mathbb{N}_\infty$. Then we use the Witt vector functor $W_A$ from \cite{3LH}. Then the prismaization is the stackification over the ring category $\mathrm{Nil}_{z,A}$ where for each such ring we have the groupoid of all the Cartier-Witt ideals fibered over $\mathbb{N}_\infty$, i.e. those maps locally principally generated by distinguished elements in the big Witt vectors (for $w_0$ we require nilpotency and for $w_1$ we require unitality). Now map all the construction and definitions in the de Rham setting to $z$-adic $A$-formal scheme $M$. In this paper we use the notation 
\begin{align}
\underline{\mathrm{Prismatization}}_{\mathrm{abs},\mathrm{dR},\mathbb{N}_\infty,M}
\end{align}
to denote the absolute de Rham prismatization in families in our current setting. The definition for this is through the corresponding compactification from the stackification over $\mathbb{N}$, since then we have the morphism:
\begin{align}
\mathcal{P}_{\underline{\mathrm{Prismatization}}_{\mathrm{abs},\mathbb{N}_\infty,M}(\mathbb{N})} \overset{\sim}{\rightarrow} \varprojlim_{(Q_P,P)/M} D\mathrm{Cat}(P)^\mathrm{comp}\rightarrow \varprojlim_{(Q_P,P)/M} D\mathrm{Cat}(P[1/z]_{Q_P})^\mathrm{comp}
\end{align}
where $\mathrm{comp}$ is the completion with respect to the natural toplogy induced from the prisms involved along the inverse limit. Then we take the compactification to reach the de Rham prismatization in families $\mathbb{N}_\infty$:
\begin{align}
\mathcal{P}_{\underline{\mathrm{Prismatization}}_{\mathrm{abs},\mathbb{N}_\infty},M} \overset{\sim}{\rightarrow} \overline{\varprojlim_{(Q_P,P)/M} D\mathrm{Cat}(P)^\mathrm{comp}}\rightarrow \overline{\varprojlim_{(Q_P,P)/M} D\mathrm{Cat}(P[1/z]_{Q_P})^\mathrm{comp}}.
\end{align}
Then we have the corresponding filtration de Rham prismatization:
\begin{align}
\overline{\underline{\mathrm{Prismatization}}}_{\mathrm{abs},\mathrm{Nygaard},\mathrm{dR},\mathbb{N},M}
\end{align}
which is defined to be compactification from $\mathbb{N}$ to $\mathbb{N}_\infty$ of the corresponding open version over $\mathbb{N}$:
\begin{align}
{\underline{\mathrm{Prismatization}}}_{\mathrm{abs},\mathrm{Nygaard},\mathrm{dR},\mathbb{N},M}
\end{align}
where this open version is defined to be taking the product over the finite fibers over $\mathbb{N}_\infty$ of the usual filtration de Rham prismatizations over $z_n$-nilpotent $A_n$-algebras. The compactification process needs to be using the following morphisms to reach our definition as the corresponding fiber product in families through the compactification:
\begin{align}
\prod_{n\in \mathbb{N}} {\underline{\mathrm{Prismatization}}}_{\mathrm{abs},\mathrm{Nygaard},\mathrm{dR},n,M}\rightarrow  \prod_{n\in \mathbb{N}} {\underline{\mathrm{Prismatization}}}_{\mathrm{abs},\mathrm{dR},n,M}
\end{align}
by taking the products along all $\mathbb{N}$ of the usual projection from the corresponding filtration de Rham prismatization to the de Rham prismatization, and:
\begin{align}
{\underline{\mathrm{Prismatization}}}_{\mathrm{abs},\mathrm{dR},\mathbb{N}_\infty,M}\rightarrow  \prod_{n\in \mathbb{N}} {\underline{\mathrm{Prismatization}}}_{\mathrm{abs},\mathrm{dR},n,M}.
\end{align}
Recall that the filtration de Rham prismatization is actually certain fibration version of the usual de Rham prismatization, then we take further fibration over $\mathbb{N}$. Recall how finally the syntomization prismatization is constructed, which is by taking the descent of the Hodge-Tate morphism and de Rham morphism in the filtration prismatization, along the diagonal morphism from the prismatization with itself. These maps are obviously admitting the corresponding compactification versions in our current setting. Therefore we consider the following three morphisms:
\begin{align}
f_\mathrm{dR}: {\underline{\mathrm{Prismatization}}}_{\mathrm{abs},\mathrm{dR},\mathbb{N}_\infty,M}\rightarrow \overline{\underline{\mathrm{Prismatization}}}_{\mathrm{abs},\mathrm{Nygaard},\mathrm{dR},\mathbb{N}_\infty,M},
\end{align}
\begin{align}
f_\mathrm{HT}: {\underline{\mathrm{Prismatization}}}_{\mathrm{abs},\mathrm{dR},\mathbb{N}_\infty,M}\rightarrow \overline{\underline{\mathrm{Prismatization}}}_{\mathrm{abs},\mathrm{Nygaard},\mathrm{dR},\mathbb{N}_\infty,M},
\end{align}
\begin{align}
d:  {\underline{\mathrm{Prismatization}}}_{\mathrm{abs},\mathrm{dR},\mathbb{N}_\infty,M}\times {\underline{\mathrm{Prismatization}}}_{\mathrm{abs},\mathrm{dR},\mathbb{N}_\infty,M} \rightarrow {\underline{\mathrm{Prismatization}}}_{\mathrm{abs},\mathrm{dR},\mathbb{N}_\infty,M}.
\end{align}
Then we take the product of $f_\mathrm{dR}$ and $f_\mathrm{HT}$ together and take the corresponding descent along $d$ we have the definition of the syntomization de Rham prismatization in families in our current setting over $\mathbb{N}_\infty$:
\begin{align}
\overline{\underline{\mathrm{Prismatization}}}_{\mathrm{abs},\mathrm{dR},\mathrm{syntomization},\mathbb{N},M}. 
\end{align}
This finishes the definition. $\square$
\end{definition}

\begin{theorem}
There is a version of K\"unneth theorem in this context. To be more precise we have a K\"unneth theorem at least by passing to the corresponding quasicoherent sheaves over the corresponding prismatization for derived formal schemes (i.e. the corresponding stackifications). This holds for all three stackifications in the families over $\mathbb{N}_\infty$. 
\end{theorem}

\begin{proof}
We need to consider the compactification in this theorem. Any prismatization   stackification for a particular formal scheme $M$ is basically a stackification over $M$ and ringed we use the notation $\mathcal{P}$ to denote the structure sheaf. As in the usual situation we have that the derived $\infty$-category of all the $\mathcal{P}$-modules is equivalent to the prismatic site for $M$. However the derived prismatic cohomology theory in this setting does satisfy the K\"unneth theorem in the derived version on the product of the corresponding $p$-adic formal ring locally (which can be glue over formal schemes like $M$). Then the result can be generalized to our setting over $\mathbb{N}$ away from $\infty$. Then we can take the corresponding compactification along:
\begin{align}
{\underline{\mathrm{Prismatization}}}_{\mathrm{abs},\mathbb{N}_\infty,\mathrm{dR},M}\rightarrow  \prod_{n\in \mathbb{N}} {\underline{\mathrm{Prismatization}}}_{\mathrm{abs},n,\mathrm{dR},M}.
\end{align}
The filtration prismatization and the syntomization prismatization have the same derived version of the K\"unneth theorems, which can be induced from the corresponding result for the prismatization. In the current de Rham situation, we just consider the restriction to the de Rham prismatization. This finishes the proof. This amounts to saying that as in \cite{2BL2} one takes the base change to some prism $P$, which then reduces the proof for derived cohomology complexes for the three de Rham prismatizations to the proof for the derived base change of the derived cohomology complexes for three prismatizations to some de Rham prism from this prism $P$, i.e. $P[1/z]_{I_P}$. Then we just have the result from the result for the three prismatizations. Note that the three prismatizations are integral versions, but the three de Rham prismatizations can be derived from them through inverting $z$, or one can extract the integral versions of the de Rham prismatization directly first (that is to say without inverting $z$), then derive the result for the three integral de Rham prismatizations from the three integral prismatizations, then invert $z$.
\end{proof}

\begin{theorem}\label{theorem36}
All three prismatization, filtration prismatization and syntomization prismatization over $\mathbb{N}_\infty$ satisfy our \cref{situation1} conditions: (A1), (A2), (A3), the formalism pullback and pushforward in the category of all the formal schemes in (A4), (A5). In particular, promoting these three prismatization in the current setting to gestalten, we have that they satisfy all conditions (A1), (A2), (A3), (A4), (A5) for situation 1 in our consideration, and moreover for (A2) the K\"unneth holds for each element in the gestalten sequences when the degree $n$ is sufficiently large.
\end{theorem}

\begin{proof}
We check this one by one. First for the first condition it is true after we consider the corresponding structure sheaves. The second condition is proved above. For the third condition we consider the corresponding derived $\infty$-categories of quasicoherent sheaves. A4 holds for pullbacks and pushforward. This is true over $\mathbb{N}$ then we consider the fiber product through:
\begin{align}
{\underline{\mathrm{Prismatization}}}_{\mathrm{abs},\mathbb{N}_\infty,\mathrm{dR},M}\rightarrow  \prod_{n\in \mathbb{N}} {\underline{\mathrm{Prismatization}}}_{\mathrm{abs},n,\mathrm{dR},M}
\end{align}
to reach the corresponding compactification. Finally A5 holds see \cite[Chapter 4, in particular 4.7, 4.8, 4.9, 4.10]{3A}. For A5 we consider the big Grothendieck site over the base derived formal stack (regarded as a derived small arc stack) in families, with coverings from derived small arc stacks in families. After the promotion of the prismatizations to gestalten we have the result from \cite{5S5}, \cite{5G}. The corresponding 6-functor formalism is then existing for all those presentable maps as in \cite{5S5} after we use these maps as the corresponding !-able mappings. 
\end{proof}

\subsection{Three prismatization in the de Rham-Hodge-Tate setting}

\begin{definition}
We consider the corresponding de Rham-Hodge-Tate prismatization in families. Recall the definition goes in the following way. Eventually the definition is applied to $z$-adic $A$-formal schemes. In the absolute manner we consider the category of all $z$-nilpotent $A$-algebras as the underlying ring categorie $\mathrm{Nil}_{z,A}$, where all such algebras are assumed to be fibered over $\mathbb{N}_\infty$. Then we use the Witt vector functor $W_A$ from \cite{3LH}. Then the prismaization is the stackification over the ring category $\mathrm{Nil}_{z,A}$ where for each such ring we have the groupoid of all the Cartier-Witt ideals fibered over $\mathbb{N}_\infty$, i.e. those maps locally principally generated by distinguished elements in the big Witt vectors (for $w_0$ we require nilpotency and for $w_1$ we require unitality). In this paper we use the notation 
\begin{align}
\underline{\mathrm{Prismatization}}_{\mathrm{abs},\mathrm{dRHT},\mathbb{N}_\infty}
\end{align}
to denote the absolute de Rham-Hodge-Tate prismatization in families in our current setting. The definition for this is through the corresponding compactification from the stackification over $\mathbb{N}$, since then we have the morphism:
\begin{align}
\mathcal{P}_{\underline{\mathrm{Prismatization}}_{\mathrm{abs},\mathbb{N}_\infty}(\mathbb{N})} \overset{\sim}{\rightarrow} \varprojlim_{(Q_P,P)} D\mathrm{Cat}(P)^\mathrm{comp}\rightarrow \varprojlim_{(Q_P,P)} D\mathrm{Cat}(P[1/z]_{Q_P})^\mathrm{comp}\rightarrow \varprojlim_{(Q_P,P)} D\mathrm{Cat}(P[1/z]/{Q_P})^\mathrm{comp}
\end{align}
where $\mathrm{comp}$ is the completion with respect to the natural toplogy induced from the prisms involved along the inverse limit. Then we take the compactification to reach the de Rham-Hodge-Tate prismatization in families $\mathbb{N}_\infty$:
\begin{align}
\mathcal{P}_{\underline{\mathrm{Prismatization}}_{\mathrm{abs},\mathbb{N}_\infty}} \overset{\sim}{\rightarrow} \overline{\varprojlim_{(Q_P,P)} D\mathrm{Cat}(P)^\mathrm{comp}}\rightarrow \overline{\varprojlim_{(Q_P,P)} D\mathrm{Cat}(P[1/z]_{Q_P})^\mathrm{comp}}\rightarrow \overline{\varprojlim_{(Q_P,P)} D\mathrm{Cat}(P[1/z]/{Q_P})^\mathrm{comp}}.
\end{align}
Then we have the corresponding filtration de Rham-Hodge-Tate prismatization:
\begin{align}
\overline{\underline{\mathrm{Prismatization}}}_{\mathrm{abs},\mathrm{Nygaard},\mathrm{dRHT},\mathbb{N}}
\end{align}
which is defined to be compactification from $\mathbb{N}$ to $\mathbb{N}_\infty$ of the corresponding open version over $\mathbb{N}$:
\begin{align}
{\underline{\mathrm{Prismatization}}}_{\mathrm{abs},\mathrm{Nygaard},\mathrm{dRHT},\mathbb{N}}
\end{align}
where this open version is defined to be taking the product over the finite fibers over $\mathbb{N}_\infty$ of the usual filtration prismatizations over $z_n$-nilpotent $A_n$-algebras. The compactification process needs to be using the following morphisms to reach our definition as the corresponding fiber product in families through the compactification:
\begin{align}
\prod_{n\in \mathbb{N}} {\underline{\mathrm{Prismatization}}}_{\mathrm{abs},\mathrm{Nygaard},\mathrm{dRHT},n}\rightarrow  \prod_{n\in \mathbb{N}} {\underline{\mathrm{Prismatization}}}_{\mathrm{abs},\mathrm{dRHT},n}
\end{align}
by taking the products along all $\mathbb{N}$ of the usual projection from the corresponding filtration prismatization to the prismatization in the de Rham-Hodge-Tate situation, and:
\begin{align}
{\underline{\mathrm{Prismatization}}}_{\mathrm{abs},\mathrm{dRHT},\mathbb{N}_\infty}\rightarrow  \prod_{n\in \mathbb{N}} {\underline{\mathrm{Prismatization}}}_{\mathrm{abs},\mathrm{dRHT},n}.
\end{align}
Recall that the filtration prismatization is actually certain fibration version of the usual prismatization, then we take further fibration over $\mathbb{N}$, in the de Rham-Hodge-Tate situation. Recall how finally the syntomization prismatization is constructed, which is by taking the descent of the Hodge-Tate morphism and de Rham morphism in the filtration prismatization, along the diagonal morphism from the prismatization with itself. These maps are obviously admitting the corresponding compactification versions in our current setting. Therefore we consider the following three morphisms:
\begin{align}
f_\mathrm{dR}: {\underline{\mathrm{Prismatization}}}_{\mathrm{abs},\mathrm{dRHT},\mathbb{N}_\infty}\rightarrow \overline{\underline{\mathrm{Prismatization}}}_{\mathrm{abs},\mathrm{Nygaard},\mathrm{dRHT},\mathbb{N}_\infty},
\end{align}
\begin{align}
f_\mathrm{HT}: {\underline{\mathrm{Prismatization}}}_{\mathrm{abs},\mathrm{dRHT},\mathbb{N}_\infty}\rightarrow \overline{\underline{\mathrm{Prismatization}}}_{\mathrm{abs},\mathrm{Nygaard},\mathrm{dRHT},\mathbb{N}_\infty},
\end{align}
\begin{align}
d:  {\underline{\mathrm{Prismatization}}}_{\mathrm{abs},\mathrm{dRHT},\mathbb{N}_\infty}\times {\underline{\mathrm{Prismatization}}}_{\mathrm{abs},\mathrm{dRHT},\mathbb{N}_\infty} \rightarrow {\underline{\mathrm{Prismatization}}}_{\mathrm{abs},\mathrm{dRHT},\mathbb{N}_\infty}.
\end{align}
Then we take the product of $f_\mathrm{dR}$ and $f_\mathrm{HT}$ together and take the corresponding descent along $d$ we have the definition of the syntomization prismatization in families in our current setting over $\mathbb{N}_\infty$:
\begin{align}
\overline{\underline{\mathrm{Prismatization}}}_{\mathrm{abs},\mathrm{dRHT},\mathrm{syntomization},\mathbb{N}}. 
\end{align}
This finishes the definition. $\square$
\end{definition}

\begin{definition}
We consider the corresponding de Rham-Hodge-Tate prismatization in families. Recall the definition goes in the following way. Eventually the definition is applied to $z$-adic $A$-formal schemes. In the absolute manner we consider the category of all $z$-nilpotent $A$-algebras as the underlying ring categorie $\mathrm{Nil}_{z,A}$, where all such algebras are assumed to be fibered over $\mathbb{N}_\infty$. Then we use the Witt vector functor $W_A$ from \cite{3LH}. Then the prismaization is the stackification over the ring category $\mathrm{Nil}_{z,A}$ where for each such ring we have the groupoid of all the Cartier-Witt ideals fibered over $\mathbb{N}_\infty$, i.e. those maps locally principally generated by distinguished elements in the big Witt vectors (for $w_0$ we require nilpotency and for $w_1$ we require unitality). Now map all the construction and definitions in the de Rham-Hodge-Tate setting to $z$-adic $A$-formal scheme $M$. In this paper we use the notation 
\begin{align}
\underline{\mathrm{Prismatization}}_{\mathrm{abs},\mathrm{dRHT},\mathbb{N}_\infty,M}
\end{align}
to denote the absolute prismatization in families in our current setting. The definition for this is through the corresponding compactification from the stackification over $\mathbb{N}$, since then we have the morphism:
\begin{align}
&\mathcal{P}_{\underline{\mathrm{Prismatization}}_{\mathrm{abs},\mathbb{N}_\infty,M}(\mathbb{N})} \overset{\sim}{\rightarrow} \varprojlim_{(Q_P,P)/M} D\mathrm{Cat}(P)^\mathrm{comp}\\
&\rightarrow \varprojlim_{(Q_P,P)/M} D\mathrm{Cat}(P[1/z]/{Q_P})^\mathrm{comp}\rightarrow \varprojlim_{(Q_P,P)/M} D\mathrm{Cat}(P[1/z]/{Q_P})^\mathrm{comp}
\end{align}
where $\mathrm{comp}$ is the completion with respect to the natural toplogy induced from the prisms involved along the inverse limit. Then we take the compactification to reach the de Rham-Hodge-Tate prismatization in families $\mathbb{N}_\infty$:
\begin{align}
\mathcal{P}_{\underline{\mathrm{Prismatization}}_{\mathrm{abs},\mathbb{N}_\infty},M} \overset{\sim}{\rightarrow} \overline{\varprojlim_{(Q_P,P)/M} D\mathrm{Cat}(P)^\mathrm{comp}}
\rightarrow \overline{\varprojlim_{(Q_P,P)/M} D\mathrm{Cat}(P[1/z]/{Q_P})^\mathrm{comp}}\rightarrow \overline{\varprojlim_{(Q_P,P)/M} D\mathrm{Cat}(P[1/z]/{Q_P})^\mathrm{comp}}.
\end{align}
Then we have the corresponding filtration prismatization:
\begin{align}
\overline{\underline{\mathrm{Prismatization}}}_{\mathrm{abs},\mathrm{Nygaard},\mathrm{dRHT},\mathbb{N},M}
\end{align}
which is defined to be compactification from $\mathbb{N}$ to $\mathbb{N}_\infty$ of the corresponding open version over $\mathbb{N}$:
\begin{align}
{\underline{\mathrm{Prismatization}}}_{\mathrm{abs},\mathrm{Nygaard},\mathrm{dRHT},\mathbb{N},M}
\end{align}
where this open version is defined to be taking the product over the finite fibers over $\mathbb{N}_\infty$ of the usual filtration prismatizations over $z_n$-nilpotent $A_n$-algebras. The compactification process needs to be using the following morphisms to reach our definition as the corresponding fiber product in families through the compactification:
\begin{align}
\prod_{n\in \mathbb{N}} {\underline{\mathrm{Prismatization}}}_{\mathrm{abs},\mathrm{Nygaard},\mathrm{dRHT},n,M}\rightarrow  \prod_{n\in \mathbb{N}} {\underline{\mathrm{Prismatization}}}_{\mathrm{abs},\mathrm{dRHT},n,M}
\end{align}
by taking the products along all $\mathbb{N}$ of the usual projection from the corresponding filtration prismatization to the prismatization, and:
\begin{align}
{\underline{\mathrm{Prismatization}}}_{\mathrm{abs},\mathrm{dRHT},\mathbb{N}_\infty,M}\rightarrow  \prod_{n\in \mathbb{N}} {\underline{\mathrm{Prismatization}}}_{\mathrm{abs},\mathrm{dRHT},n,M}.
\end{align}
Recall that the filtration prismatization is actually certain fibration version of the usual prismatization in the de Rham-Hodge-Tate situation, then we take further fibration over $\mathbb{N}$. Recall how finally the syntomization prismatization is constructed, which is by taking the descent of the Hodge-Tate morphism and de Rham morphism in the filtration prismatization, along the diagonal morphism from the prismatization with itself. These maps are obviously admitting the corresponding compactification versions in our current setting. Therefore we consider the following three morphisms:
\begin{align}
f_\mathrm{dR}: {\underline{\mathrm{Prismatization}}}_{\mathrm{abs},\mathrm{dRHT},\mathbb{N}_\infty,M}\rightarrow \overline{\underline{\mathrm{Prismatization}}}_{\mathrm{abs},\mathrm{Nygaard},\mathrm{dRHT},\mathbb{N}_\infty,M},
\end{align}
\begin{align}
f_\mathrm{HT}: {\underline{\mathrm{Prismatization}}}_{\mathrm{abs},\mathrm{dRHT},\mathbb{N}_\infty,M}\rightarrow \overline{\underline{\mathrm{Prismatization}}}_{\mathrm{abs},\mathrm{Nygaard},\mathrm{dRHT},\mathbb{N}_\infty,M},
\end{align}
\begin{align}
d:  {\underline{\mathrm{Prismatization}}}_{\mathrm{abs},\mathrm{dRHT},\mathbb{N}_\infty,M}\times {\underline{\mathrm{Prismatization}}}_{\mathrm{abs},\mathrm{dRHT},\mathbb{N}_\infty,M} \rightarrow {\underline{\mathrm{Prismatization}}}_{\mathrm{abs},\mathrm{dRHT},\mathbb{N}_\infty,M}.
\end{align}
Then we take the product of $f_\mathrm{dR}$ and $f_\mathrm{HT}$ together and take the corresponding descent along $d$ we have the definition of the syntomization prismatization in families in our current setting over $\mathbb{N}_\infty$:
\begin{align}
\overline{\underline{\mathrm{Prismatization}}}_{\mathrm{abs},\mathrm{dRHT},\mathrm{syntomization},\mathbb{N},M}. 
\end{align}
This finishes the definition. $\square$
\end{definition}

\begin{theorem}
There is a version of K\"unneth theorem in this context. To be more precise we have a K\"unneth theorem at least by passing to the corresponding quasicoherent sheaves over the corresponding prismatization for derived formal schemes (i.e. the corresponding stackifications). This holds for all three stackifications in the families fibered over $\mathbb{N}_\infty$. 
\end{theorem}

\begin{proof}
We need to consider the compactification in this theorem. Any prismatization  stackification for a particular formal scheme $M$ is basically a stackification over $M$ and ringed we use the notation $\mathcal{P}$ to denote the structure sheaf. As in the usual situation we have that the derived $\infty$-category of all the $\mathcal{P}$-modules is equivalent to the prismatic site for $M$. However the derived prismatic cohomology theory in this setting does satisfy the K\"unneth theorem in the derived version on the product of the corresponding $p$-adic formal ring locally (which can be glue over formal schemes like $M$). Then the result can be generalized to our setting over $\mathbb{N}$ away from $\infty$. Then we can take the corresponding compactification along:
\begin{align}
{\underline{\mathrm{Prismatization}}}_{\mathrm{abs},\mathbb{N}_\infty,\mathrm{dRHT},M}\rightarrow  \prod_{n\in \mathbb{N}} {\underline{\mathrm{Prismatization}}}_{\mathrm{abs},n,\mathrm{dRHT},M}.
\end{align}
The derived version of the K\"unneth theorem then holds for the filtration  prismatization and syntomization prismatization. Then the derived K\"unneth holds in the 3 de Rham prismatizations situation which implies the result in the current setting. This finishes the proof.
\end{proof}

\begin{theorem}
All three prismatization, filtration prismatization and syntomization prismatization over $\mathbb{N}_\infty$ satisfy our \cref{situation1} conditions: (A1), (A2), (A3), (A4), (A5). 
\end{theorem}

\begin{proof}
We check this one by one. First for the first condition it is true after we consider the corresponding structure sheaves. The second condition is proved above. For the third condition we consider the corresponding derived $\infty$-categories of quasicoherent sheaves. A4 holds for pullbacks and pushforward. This is true over $\mathbb{N}$ then we consider the fiber product through:
\begin{align}
{\underline{\mathrm{Prismatization}}}_{\mathrm{abs},\mathbb{N}_\infty,\mathrm{dRHT},M}\rightarrow  \prod_{n\in \mathbb{N}} {\underline{\mathrm{Prismatization}}}_{\mathrm{abs},n,\mathrm{dRHT},M}
\end{align}
to reach the corresponding compactification. Finally A5 holds see \cite[Chapter 4, in particular 4.7, 4.8, 4.9, 4.10]{3A}. For A5 we consider the derived small arc stacks over the base formal scheme (regarded as a small arc stack) to form the corresponding Grothendieck site in families. Then one finishes the proof as in \cref{theorem36}.
\end{proof}

\subsection{Three prismatizations in the Laurent setting}

\begin{definition}
We consider the corresponding Laurent prismatization in families. Recall the definition goes in the following way. Eventually the definition is applied to $z$-adic $A$-formal schemes. In the absolute manner we consider the category of all $z$-nilpotent $A$-algebras as the underlying ring categorie $\mathrm{Nil}_{z,A}$, where all such algebras are assumed to be fibered over $\mathbb{N}_\infty$. Then we use the Witt vector functor $W_A$ from \cite{3LH}. Then the prismaization is the stackification over the ring category $\mathrm{Nil}_{z,A}$ where for each such ring we have the groupoid of all the Cartier-Witt ideals fibered over $\mathbb{N}_\infty$, i.e. those maps locally principally generated by distinguished elements in the big Witt vectors (for $w_0$ we require nilpotency and for $w_1$ we require unitality). In this paper we use the notation 
\begin{align}
\underline{\mathrm{Prismatization}}_{\mathrm{abs},\mathrm{Laurent},\mathbb{N}_\infty}
\end{align}
to denote the absolute prismatization in families in our current setting. The definition for this is through the corresponding compactification from the stackification over $\mathbb{N}$, since then we have the morphism:
\begin{align}
\mathcal{P}_{\underline{\mathrm{Prismatization}}_{\mathrm{abs},\mathbb{N}_\infty}(\mathbb{N})} \overset{\sim}{\rightarrow} \varprojlim_{(Q_P,P)} D\mathrm{Cat}(P)^\mathrm{comp}\rightarrow \varprojlim_{(Q_P,P)} D\mathrm{Cat}(P[1/Q_P]_z)^\mathrm{comp}
\end{align}
where $\mathrm{comp}$ is the completion with respect to the natural toplogy induced from the prisms involved along the inverse limit. Then we take the compactification to reach the Laurent prismatization in families $\mathbb{N}_\infty$:
\begin{align}
\mathcal{P}_{\underline{\mathrm{Prismatization}}_{\mathrm{abs},\mathbb{N}_\infty}} \overset{\sim}{\rightarrow} \overline{\varprojlim_{(Q_P,P)} D\mathrm{Cat}(P)^\mathrm{comp}}\rightarrow \overline{\varprojlim_{(Q_P,P)} D\mathrm{Cat}(P[1/Q_P]_z)^\mathrm{comp}}.
\end{align}
Then we have the corresponding filtration prismatization:
\begin{align}
\overline{\underline{\mathrm{Prismatization}}}_{\mathrm{abs},\mathrm{Nygaard},\mathrm{Laurent},\mathbb{N}}
\end{align}
which is defined to be compactification from $\mathbb{N}$ to $\mathbb{N}_\infty$ of the corresponding open version over $\mathbb{N}$:
\begin{align}
{\underline{\mathrm{Prismatization}}}_{\mathrm{abs},\mathrm{Nygaard},\mathrm{Laurent},\mathbb{N}}
\end{align}
where this open version is defined to be taking the product over the finite fibers over $\mathbb{N}_\infty$ of the usual filtration prismatizations over $z_n$-nilpotent $A_n$-algebras. The compactification process needs to be using the following morphisms to reach our definition as the corresponding fiber product in families through the compactification:
\begin{align}
\prod_{n\in \mathbb{N}} {\underline{\mathrm{Prismatization}}}_{\mathrm{abs},\mathrm{Nygaard},\mathrm{Laurent},n}\rightarrow  \prod_{n\in \mathbb{N}} {\underline{\mathrm{Prismatization}}}_{\mathrm{abs},\mathrm{Laurent},n}
\end{align}
by taking the products along all $\mathbb{N}$ of the usual projection from the corresponding filtration prismatization to the prismatization, and:
\begin{align}
{\underline{\mathrm{Prismatization}}}_{\mathrm{abs},\mathrm{Laurent},\mathbb{N}_\infty}\rightarrow  \prod_{n\in \mathbb{N}} {\underline{\mathrm{Prismatization}}}_{\mathrm{abs},\mathrm{Laurent},n}.
\end{align}
Recall that the filtration prismatization is actually certain fibration version of the usual prismatization in the current setting, then we take further fibration over $\mathbb{N}$. Recall how finally the syntomization prismatization is constructed, which is by taking the descent of the Hodge-Tate morphism and de Rham morphism in the filtration prismatization, along the diagonal morphism from the prismatization with itself. These maps are obviously admitting the corresponding compactification versions in our current setting. Therefore we consider the following three morphisms:
\begin{align}
f_\mathrm{dR}: {\underline{\mathrm{Prismatization}}}_{\mathrm{abs},\mathrm{Laurent},\mathbb{N}_\infty}\rightarrow \overline{\underline{\mathrm{Prismatization}}}_{\mathrm{abs},\mathrm{Nygaard},\mathrm{Laurent},\mathbb{N}_\infty},
\end{align}
\begin{align}
f_\mathrm{HT}: {\underline{\mathrm{Prismatization}}}_{\mathrm{abs},\mathrm{Laurent},\mathbb{N}_\infty}\rightarrow \overline{\underline{\mathrm{Prismatization}}}_{\mathrm{abs},\mathrm{Nygaard},\mathrm{Laurent},\mathbb{N}_\infty},
\end{align}
\begin{align}
d:  {\underline{\mathrm{Prismatization}}}_{\mathrm{abs},\mathrm{Laurent},\mathbb{N}_\infty}\times {\underline{\mathrm{Prismatization}}}_{\mathrm{abs},\mathrm{Laurent},\mathbb{N}_\infty} \rightarrow {\underline{\mathrm{Prismatization}}}_{\mathrm{abs},\mathrm{Laurent},\mathbb{N}_\infty}.
\end{align}
Then we take the product of $f_\mathrm{dR}$ and $f_\mathrm{HT}$ together and take the corresponding descent along $d$ we have the definition of the syntomization prismatization in families in our current setting over $\mathbb{N}_\infty$:
\begin{align}
\overline{\underline{\mathrm{Prismatization}}}_{\mathrm{abs},\mathrm{Laurent},\mathrm{syntomization},\mathbb{N}}. 
\end{align}
This finishes the definition. $\square$
\end{definition}

\begin{definition}
We consider the corresponding Laurent prismatization in families. Recall the definition goes in the following way. Eventually the definition is applied to $z$-adic $A$-formal schemes. In the absolute manner we consider the category of all $z$-nilpotent $A$-algebras as the underlying ring categorie $\mathrm{Nil}_{z,A}$, where all such algebras are assumed to be fibered over $\mathbb{N}_\infty$. Then we use the Witt vector functor $W_A$ from \cite{3LH}. Then the prismaization is the stackification over the ring category $\mathrm{Nil}_{z,A}$ where for each such ring we have the groupoid of all the Cartier-Witt ideals fibered over $\mathbb{N}_\infty$, i.e. those maps locally principally generated by distinguished elements in the big Witt vectors (for $w_0$ we require nilpotency and for $w_1$ we require unitality). Now map all the construction and definitions in the de Rham setting to $z$-adic $A$-formal scheme $M$. In this paper we use the notation 
\begin{align}
\underline{\mathrm{Prismatization}}_{\mathrm{abs},\mathrm{Laurent},\mathbb{N}_\infty,M}
\end{align}
to denote the absolute Laurent prismatization in families in our current setting. The definition for this is through the corresponding compactification from the stackification over $\mathbb{N}$, since then we have the morphism:
\begin{align}
\mathcal{P}_{\underline{\mathrm{Prismatization}}_{\mathrm{abs},\mathbb{N}_\infty,M}(\mathbb{N})} \overset{\sim}{\rightarrow} \varprojlim_{(Q_P,P)/M} D\mathrm{Cat}(P)^\mathrm{comp}\rightarrow \varprojlim_{(Q_P,P)/M} D\mathrm{Cat}(P[1/Q_P]_z)^\mathrm{comp}
\end{align}
where $\mathrm{comp}$ is the completion with respect to the natural toplogy induced from the prisms involved along the inverse limit. Then we take the compactification to reach the Laurent prismatization in families $\mathbb{N}_\infty$:
\begin{align}
\mathcal{P}_{\underline{\mathrm{Prismatization}}_{\mathrm{abs},\mathbb{N}_\infty},M} \overset{\sim}{\rightarrow} \overline{\varprojlim_{(Q_P,P)/M} D\mathrm{Cat}(P)^\mathrm{comp}}\rightarrow \overline{\varprojlim_{(Q_P,P)/M} D\mathrm{Cat}(P[1/Q_P]_z)^\mathrm{comp}}.
\end{align}
Then we have the corresponding filtration prismatization:
\begin{align}
\overline{\underline{\mathrm{Prismatization}}}_{\mathrm{abs},\mathrm{Nygaard},\mathrm{Laurent},\mathbb{N},M}
\end{align}
which is defined to be compactification from $\mathbb{N}$ to $\mathbb{N}_\infty$ of the corresponding open version over $\mathbb{N}$:
\begin{align}
{\underline{\mathrm{Prismatization}}}_{\mathrm{abs},\mathrm{Nygaard},\mathrm{Laurent},\mathbb{N},M}
\end{align}
where this open version is defined to be taking the product over the finite fibers over $\mathbb{N}_\infty$ of the usual filtration prismatizations over $z_n$-nilpotent $A_n$-algebras. The compactification process needs to be using the following morphisms to reach our definition as the corresponding fiber product in families through the compactification:
\begin{align}
\prod_{n\in \mathbb{N}} {\underline{\mathrm{Prismatization}}}_{\mathrm{abs},\mathrm{Nygaard},\mathrm{Laurent},n,M}\rightarrow  \prod_{n\in \mathbb{N}} {\underline{\mathrm{Prismatization}}}_{\mathrm{abs},\mathrm{Laurent},n,M}
\end{align}
by taking the products along all $\mathbb{N}$ of the usual projection from the corresponding filtration prismatization to the prismatization, and:
\begin{align}
{\underline{\mathrm{Prismatization}}}_{\mathrm{abs},\mathrm{Laurent},\mathbb{N}_\infty,M}\rightarrow  \prod_{n\in \mathbb{N}} {\underline{\mathrm{Prismatization}}}_{\mathrm{abs},\mathrm{Laurent},n,M}.
\end{align}
Recall that the filtration prismatization is actually certain fibration version of the usual prismatization, then we take further fibration over $\mathbb{N}$ in the current setting. Recall how finally the syntomization prismatization is constructed, which is by taking the descent of the Hodge-Tate morphism and de Rham morphism in the filtration prismatization, along the diagonal morphism from the prismatization with itself. These maps are obviously admitting the corresponding compactification versions in our current setting. Therefore we consider the following three morphisms:
\begin{align}
f_\mathrm{dR}: {\underline{\mathrm{Prismatization}}}_{\mathrm{abs},\mathrm{Laurent},\mathbb{N}_\infty,M}\rightarrow \overline{\underline{\mathrm{Prismatization}}}_{\mathrm{abs},\mathrm{Nygaard},\mathrm{Laurent},\mathbb{N}_\infty,M},
\end{align}
\begin{align}
f_\mathrm{HT}: {\underline{\mathrm{Prismatization}}}_{\mathrm{abs},\mathrm{Laurent},\mathbb{N}_\infty,M}\rightarrow \overline{\underline{\mathrm{Prismatization}}}_{\mathrm{abs},\mathrm{Nygaard},\mathrm{Laurent},\mathbb{N}_\infty,M},
\end{align}
\begin{align}
d:  {\underline{\mathrm{Prismatization}}}_{\mathrm{abs},\mathrm{Laurent},\mathbb{N}_\infty,M}\times {\underline{\mathrm{Prismatization}}}_{\mathrm{abs},\mathrm{Laurent},\mathbb{N}_\infty,M} \rightarrow {\underline{\mathrm{Prismatization}}}_{\mathrm{abs},\mathrm{Laurent},\mathbb{N}_\infty,M}.
\end{align}
Then we take the product of $f_\mathrm{dR}$ and $f_\mathrm{HT}$ together and take the corresponding descent along $d$ we have the definition of the syntomization prismatization in families in our current setting over $\mathbb{N}_\infty$:
\begin{align}
\overline{\underline{\mathrm{Prismatization}}}_{\mathrm{abs},\mathrm{Laurent},\mathrm{syntomization},\mathbb{N},M}. 
\end{align}
This finishes the definition. $\square$
\end{definition}

\begin{theorem}
There is a version of K\"unneth theorem in this context. To be more precise we have a K\"unneth theorem at least by passing to the corresponding quasicoherent sheaves over the corresponding prismatization for derived formal schemes (i.e. the corresponding stackifications). This holds for all three stackifications in the families over $\mathbb{N}_\infty$. 
\end{theorem}

\begin{proof}
We need to consider the compactification in this theorem. Any prismatization  stackification for a particular formal scheme $M$ is basically a stackification over $M$ and ringed we use the notation $\mathcal{P}$ to denote the structure sheaf. As in the usual situation we have that the derived $\infty$-category of all the $\mathcal{P}$-modules is equivalent to the prismatic site for $M$. However the derived prismatic cohomology theory in this setting does satisfy the K\"unneth theorem in the derived version on the product of the corresponding $p$-adic formal ring locally (which can be glue over formal schemes like $M$). Then the result can be generalized to our setting over $\mathbb{N}$ away from $\infty$. Then we can take the corresponding compactification along:
\begin{align}
{\underline{\mathrm{Prismatization}}}_{\mathrm{abs},\mathbb{N}_\infty,\mathrm{Laurent},M}\rightarrow  \prod_{n\in \mathbb{N}} {\underline{\mathrm{Prismatization}}}_{\mathrm{abs},n,\mathrm{Laurent},M}.
\end{align}
Then this implies derived version of K\"unneth theorem holds for filtration prismatization and syntomization prismatization. One then have the K\"unneth theorem in a derived version holds for the Laurent prismatizations in the three situations. This finishes the proof.
\end{proof}

\begin{theorem}
All three prismatization, filtration prismatization and syntomization prismatization over $\mathbb{N}_\infty$ satisfy our \cref{situation1} conditions: (A1), (A2), (A3), (A4), (A5).
\end{theorem}

\begin{proof}
We check this one by one. First for the first condition it is true after we consider the corresponding structure sheaves. The second condition is proved above. For the third condition we consider the corresponding derived $\infty$-categories of quasicoherent sheaves. A4 holds for pullbacks and pushforward. This is true over $\mathbb{N}$ then we consider the fiber product through:
\begin{align}
{\underline{\mathrm{Prismatization}}}_{\mathrm{abs},\mathbb{N}_\infty,\mathrm{Laurent},M}\rightarrow  \prod_{n\in \mathbb{N}} {\underline{\mathrm{Prismatization}}}_{\mathrm{abs},n,\mathrm{Laurent},M}
\end{align}
to reach the corresponding compactification. Finally A5 holds see \cite[Chapter 4, in particular 4.7, 4.8, 4.9, 4.10]{3A}. Then one finishes the proof as in \cref{theorem36}.
\end{proof}

\newpage
\section{Motivic Cohomology Theory for Analytic Prismatizations in Families}\label{section10}

\indent We now consider the corresponding prismatization in families in the three settings in \cite{3BL}, \cite{3D}, with further analytification as in \cite{3ALBRCS} and \cite{3S1}. The idea is to construct there analytic versions of the prismatization in families $\mathbb{N}_\infty$, then apply our previous results above to derive the corresponding motivic cohomology theories. The first one is the prismatization in families, the second one is the filtration prismatization and finally we have the corresponding syntomization prismatization. They can be defined all in the corresponding families way as we did before. In the first two situations we recall our construtions before which will also be put into the general consideration we are current considering on motivic cohomology theories.

\subsection{Prismatization, filtration prismatization and syntomization prismatization}

\begin{definition}
We consider the corresponding prismatization in families. Recall the definition goes in the following way. Eventually the definition is applied to $z$-adic $A$-formal schemes. In the absolute manner we consider the category of all $z$-nilpotent $A$-algebras as the underlying ring categorie $\mathrm{Nil}_{z,A}$, where all such algebras are assumed to be fibered over $\mathbb{N}_\infty$. Then we use the Witt vector functor $W_A$ from \cite{3LH}. Then the prismaization is the stackification over the ring category $\mathrm{Nil}_{z,A}$ where for each such ring we have the groupoid of all the Cartier-Witt ideals fibered over $\mathbb{N}_\infty$, i.e. those maps locally principally generated by distinguished elements in the big Witt vectors (for $w_0$ we require nilpotency and for $w_1$ we require unitality). In this paper we use the notation 
\begin{align}
\underline{\mathrm{Prismatization}}^L_{\mathrm{abs},\mathbb{N}_\infty}
\end{align}
to denote the absolute analytic prismatization in families in our current setting. We follow \cite{3S1}, \cite{3ALBRCS} to the following three step for any $L$-rigid affioid with formal local model:
\begin{itemize}
\item[S1] Over $\mathbb{N}$ we have the analytic prismatization from local formal model from \cite{3S1}, \cite{3ALBRCS}, which provide the analytic prismatization over the generic fiber over $L$;
\item[S2] We take the corresponding compactification of this to $\mathbb{N}_\infty$, by using the morphism to take the fiber product:
\begin{align}
\underline{\mathrm{Prismatization}}_{\mathrm{abs},\mathbb{N}_\infty}\rightarrow \underline{\mathrm{Prismatization}}_{\mathrm{abs},\mathbb{N}_\infty}(\mathbb{N}).
\end{align}
\item[3S3] Take condensed analytification from \cite{3CS},\cite{3CS1},\cite{3CS2}.
\end{itemize}
Then we have the corresponding filtration analytic  prismatization:
\begin{align}
\overline{\underline{\mathrm{Prismatization}}}^L_{\mathrm{abs},\mathrm{Nygaard},\mathbb{N}}
\end{align}
which is defined to be compactification from $\mathbb{N}$ to $\mathbb{N}_\infty$ of the corresponding open version over $\mathbb{N}$:
\begin{align}
{\underline{\mathrm{Prismatization}}}^L_{\mathrm{abs},\mathrm{Nygaard},\mathbb{N}}
\end{align}
where this open version is defined to be taking the product over the finite fibers over $\mathbb{N}_\infty$ of the usual filtration prismatizations over $z_n$-nilpotent $A_n$-algebras. The compactification process needs to be using the following morphisms to reach our definition as the corresponding fiber product in families through the compactification:
\begin{align}
\prod_{n\in \mathbb{N}} {\underline{\mathrm{Prismatization}}}^L_{\mathrm{abs},\mathrm{Nygaard},n}\rightarrow  \prod_{n\in \mathbb{N}} {\underline{\mathrm{Prismatization}}}^L_{\mathrm{abs},n}
\end{align}
by taking the products along all $\mathbb{N}$ of the usual projection from the corresponding filtration analytic prismatization to the prismatization, and:
\begin{align}
{\underline{\mathrm{Prismatization}}}^L_{\mathrm{abs},\mathbb{N}_\infty}\rightarrow  \prod_{n\in \mathbb{N}} {\underline{\mathrm{Prismatization}}}^L_{\mathrm{abs},n}.
\end{align}
Recall that the filtration prismatization is actually certain fibration version of the usual prismatization, then we take further fibration over $\mathbb{N}$. Recall how finally the syntomization prismatization is constructed, which is by taking the descent of the Hodge-Tate morphism and de Rham morphism in the filtration prismatization, along the diagonal morphism from the prismatization with itself. These maps are obviously admitting the corresponding compactification versions in our current setting. Therefore we consider the following three morphisms:
\begin{align}
f_\mathrm{dR}: {\underline{\mathrm{Prismatization}}}^L_{\mathrm{abs},\mathbb{N}_\infty}\rightarrow \overline{\underline{\mathrm{Prismatization}}}^L_{\mathrm{abs},\mathrm{Nygaard},\mathbb{N}_\infty},
\end{align}
\begin{align}
f_\mathrm{HT}: {\underline{\mathrm{Prismatization}}}^L_{\mathrm{abs},\mathbb{N}_\infty}\rightarrow \overline{\underline{\mathrm{Prismatization}}}^L_{\mathrm{abs},\mathrm{Nygaard},\mathbb{N}_\infty},
\end{align}
\begin{align}
d:  {\underline{\mathrm{Prismatization}}}^L_{\mathrm{abs},\mathbb{N}_\infty}\times {\underline{\mathrm{Prismatization}}}^L_{\mathrm{abs},\mathbb{N}_\infty} \rightarrow {\underline{\mathrm{Prismatization}}}^L_{\mathrm{abs},\mathbb{N}_\infty}.
\end{align}
Then we take the product of $f_\mathrm{dR}$ and $f_\mathrm{HT}$ together and take the corresponding descent along $d$ we have the definition of the syntomization analytic prismatization in families in our current setting over $\mathbb{N}_\infty$:
\begin{align}
\overline{\underline{\mathrm{Prismatization}}}^L_{\mathrm{abs},\mathrm{syntomization},\mathbb{N}}. 
\end{align}
This finishes the definition. $\square$
\end{definition}

\begin{definition} 
We consider the corresponding prismatization in families. Recall the definition goes in the following way. Eventually the definition is applied to $z$-adic $A$-formal schemes. In the absolute manner we consider the category of all $z$-nilpotent $A$-algebras as the underlying ring categorie $\mathrm{Nil}_{z,A}$, where all such algebras are assumed to be fibered over $\mathbb{N}_\infty$. Then we use the Witt vector functor $W_A$ from \cite{3LH}. Then the prismaization is the stackification over the ring category $\mathrm{Nil}_{z,A}$ where for each such ring we have the groupoid of all the Cartier-Witt ideals fibered over $\mathbb{N}_\infty$, i.e. those maps locally principally generated by distinguished elements in the big Witt vectors (for $w_0$ we require nilpotency and for $w_1$ we require unitality). After the discussion above, we can now apply the whole definition for the 3 prismatization to any $z$-adic $A$-formal scheme $M$, by taking the immersion from the closure of the Cartier-Witt ideals into $M$ in the compatible way. In this paper we use the notation 
\begin{align}
\underline{\mathrm{Prismatization}}^L_{\mathrm{abs},\mathbb{N}_\infty,M}
\end{align}
to denote the absolute analytic prismatization in families in our current setting. Then we have the corresponding filtration analytic prismatization:
\begin{align}
\overline{\underline{\mathrm{Prismatization}}}^L_{\mathrm{abs},\mathrm{Nygaard},\mathbb{N},M}
\end{align}
which is defined to be compactification from $\mathbb{N}$ to $\mathbb{N}_\infty$ of the corresponding open version over $\mathbb{N}$:
\begin{align}
{\underline{\mathrm{Prismatization}}}^L_{\mathrm{abs},\mathrm{Nygaard},\mathbb{N},M}
\end{align}
where this open version is defined to be taking the product over the finite fibers over $\mathbb{N}_\infty$ of the usual filtration prismatizations over $z_n$-nilpotent $A_n$-algebras. The compactification process needs to be using the following morphisms to reach our definition as the corresponding fiber product in families through the compactification:
\begin{align}
\prod_{n\in \mathbb{N}} {\underline{\mathrm{Prismatization}}}^L_{\mathrm{abs},\mathrm{Nygaard},n,M}\rightarrow  \prod_{n\in \mathbb{N}} {\underline{\mathrm{Prismatization}}}^L_{\mathrm{abs},n,M}
\end{align}
by taking the products along all $\mathbb{N}$ of the usual projection from the corresponding filtration analytic prismatization to the analytic  prismatization, and:
\begin{align}
{\underline{\mathrm{Prismatization}}}^L_{\mathrm{abs},\mathbb{N}_\infty,M}\rightarrow  \prod_{n\in \mathbb{N}} {\underline{\mathrm{Prismatization}}}^L_{\mathrm{abs},n,M}.
\end{align}
Recall that the filtration analytic prismatization is actually certain fibration version of the usual analytic prismatization, then we take further fibration over $\mathbb{N}$. Recall how finally the syntomization prismatization is constructed, which is by taking the descent of the Hodge-Tate morphism and de Rham morphism in the filtration prismatization, along the diagonal morphism from the prismatization with itself. These maps are obviously admitting the corresponding compactification versions in our current setting. Therefore we consider the following three morphisms:
\begin{align}
f_\mathrm{dR}: {\underline{\mathrm{Prismatization}}}^L_{\mathrm{abs},\mathbb{N}_\infty,M}\rightarrow \overline{\underline{\mathrm{Prismatization}}}^L_{\mathrm{abs},\mathrm{Nygaard},\mathbb{N}_\infty,M},
\end{align}
\begin{align}
f_\mathrm{HT}: {\underline{\mathrm{Prismatization}}}^L_{\mathrm{abs},\mathbb{N}_\infty,M}\rightarrow \overline{\underline{\mathrm{Prismatization}}}^L_{\mathrm{abs},\mathrm{Nygaard},\mathbb{N}_\infty,M},
\end{align}
\begin{align}
d:  {\underline{\mathrm{Prismatization}}}^L_{\mathrm{abs},\mathbb{N}_\infty,M}\times {\underline{\mathrm{Prismatization}}}^L_{\mathrm{abs},\mathbb{N}_\infty,M} \rightarrow {\underline{\mathrm{Prismatization}}}^L_{\mathrm{abs},\mathbb{N}_\infty,M}.
\end{align}
Then we take the product of $f_\mathrm{dR}$ and $f_\mathrm{HT}$ together and take the corresponding descent along $d$ we have the definition of the syntomization analytic prismatization in families in our current setting over $\mathbb{N}_\infty$:
\begin{align}
\overline{\underline{\mathrm{Prismatization}}}^L_{\mathrm{abs},\mathrm{syntomization},\mathbb{N},M}.
\end{align}
\end{definition}

\begin{remark}
One can also define the three analytic prismatizations by taking the direct analytification of the family verion of the three prismatizations before as in \cite{3S1}, \cite{3ALBRCS}.
\end{remark}

\begin{theorem}
There is a version of K\"unneth theorem in this context. To be more precise we have a K\"unneth theorem at least by passing to the corresponding quasicoherent sheaves over the corresponding prismatization for derived rigid analytic spaces in families (i.e. the corresponding stackifications). This holds for all three stackifications in the families over $\mathbb{N}_\infty$. 
\end{theorem}

\begin{proof}
One considers formal models in this setting. We need to consider the compactification in this theorem. Any prismatization stackification for a particular formal scheme $M$ is basically a stackification over $M$ and ringed we use the notation $\mathcal{P}$ to denote the structure sheaf. As in the usual situation we have that the derived $\infty$-category of all the $\mathcal{P}$-modules is equivalent to the prismatic site for $M$. However the derived prismatic cohomology theory in this setting does satisfy the K\"unneth theorem in the derived version on the product of the corresponding $p$-adic formal ring locally (which can be glue over formal schemes like $M$). Then the result can be generalized to our setting over $\mathbb{N}$ away from $\infty$. Then we can take the corresponding compactification along:
\begin{align}
{\underline{\mathrm{Prismatization}}}^L_{\mathrm{abs},\mathbb{N}_\infty,M}\rightarrow  \prod_{n\in \mathbb{N}} {\underline{\mathrm{Prismatization}}}^L_{\mathrm{abs},n,M}.
\end{align}
This implies the derived version of K\"unneth theorem holds for analytic filtration prismatization and anlytic syntomization prismatization. This finishes the proof.
\end{proof}

\begin{theorem}
All three analytic prismatization, analytic filtration prismatization and analytic syntomization prismatization over $\mathbb{N}_\infty$ satisfy our \cref{situation2} conditions: (A1), (A2), (A3), (A4), (A5).
\end{theorem}

\begin{proof}
We check this one by one. First for the first condition it is true after we consider the corresponding structure sheaves. The second condition is proved above. For the third condition we consider the corresponding derived $\infty$-categories of quasicoherent sheaves. A4 holds for pullbacks and pushforward. This is true over $\mathbb{N}$ then we consider the fiber product through:
\begin{align}
{\underline{\mathrm{Prismatization}}}^L_{\mathrm{abs},\mathbb{N}_\infty,M}\rightarrow  \prod_{n\in \mathbb{N}} {\underline{\mathrm{Prismatization}}}^L_{\mathrm{abs},n,M}
\end{align}
to reach the corresponding compactification. Finally A5 holds see \cite[Chapter 4, in particular 4.7, 4.8, 4.9, 4.10]{3A}. However since we are talking about rigid analytic spaces, \cite{3A} can be directly applied, where one can derive all the results from the formalism in \cite{3A}, which allows us to reduce the corresponding proof to the formal scheme situation before. Then one finishes the full proof as in \cref{theorem36}.
\end{proof}

\begin{definition}
We in this context use the notations:
\begin{align}
&\mathrm{Galois}(\underline{\mathrm{Prismatization}}^L_{\mathrm{abs},\mathbb{N}_\infty,M})\\
&\mathrm{Galois}(\overline{\underline{\mathrm{Prismatization}}}^L_{\mathrm{abs},\mathrm{Nygaard},\mathbb{N}_\infty,M})\\
&\mathrm{Galois}(\overline{\underline{\mathrm{Prismatization}}}^L_{\mathrm{abs},\mathbb{N}_\infty,\mathrm{syntomization},M})
\end{align}
to denote the corresponding Hopf algebraic motivic Galois fundamental groups, which is the defined to be the spectra of the associated Hopf algebra.
\end{definition}

\begin{theorem}
There are morphisms from these motivic Galois fundamental groups to the corresponding motivic Galois groups of $A$, and moreover $L$: 
\begin{align}
&\mathrm{Galois}(\underline{\mathrm{Prismatization}}^L_{\mathrm{abs},\mathbb{N}_\infty,M})\rightarrow \mathrm{Gal}(\overline{A}/A),\\
&\mathrm{Galois}(\overline{\underline{\mathrm{Prismatization}}}^L_{\mathrm{abs},\mathrm{Nygaard},\mathbb{N}_\infty,M})\rightarrow \mathrm{Gal}(\overline{A}/A),\\
&\mathrm{Galois}(\overline{\underline{\mathrm{Prismatization}}}^L_{\mathrm{abs},\mathbb{N}_\infty,\mathrm{syntomization},M})\rightarrow \mathrm{Gal}(\overline{A}/A).
\end{align}
\end{theorem}

\begin{proof}
This is formal since we just apply the construction and definition to the point situation (regard the scheme $\mathrm{Spec}A$ or $\mathrm{Spec}L$ as the corresponding formal scheme or rigid analytic space), then the theorem follows by the functoriality of the construction of Hopf algebras. Then we reduce to the results before for formal local model before following \cite{3A}. 
\end{proof}

\begin{definition}
The three analytic prismatizations:
\begin{align}
&\underline{\mathrm{Prismatization}}^L_{\mathrm{abs},\mathbb{N}_\infty,M}\\
&\overline{\underline{\mathrm{Prismatization}}}^L_{\mathrm{abs},\mathrm{Nygaard},\mathbb{N}_\infty,M}\\
&\overline{\underline{\mathrm{Prismatization}}}^L_{\mathrm{abs},\mathbb{N}_\infty,\mathrm{syntomization},M}
\end{align}
have a small arc-stack version as well as a small v-stack version, where the condensed structure sheaves are pull-back along the strictly totally disconnected coverings. For such switching of categories and consideration, we use the then the following different notation:
\begin{align}
&\underline{\mathrm{Prismatization}}^L_{\mathrm{abs},\mathbb{N}_\infty,-}\\
&\overline{\underline{\mathrm{Prismatization}}}^L_{\mathrm{abs},\mathrm{Nygaard},\mathbb{N}_\infty,-}\\
&\overline{\underline{\mathrm{Prismatization}}}^L_{\mathrm{abs},\mathbb{N}_\infty,\mathrm{syntomization},-}
\end{align}
where $-$ is a small arc stack or a small v-stack attached to a rigid analytic space over $L$ over $\mathbb{N}_\infty$, however locally the topology is arc-topology or $v$-topology.
\end{definition}

\begin{theorem}
There is a version of K\"unneth theorem in this context
\begin{align}
&\underline{\mathrm{Prismatization}}^L_{\mathrm{abs},\mathbb{N}_\infty,-}\\
&\overline{\underline{\mathrm{Prismatization}}}^L_{\mathrm{abs},\mathrm{Nygaard},\mathbb{N}_\infty,-}\\
&\overline{\underline{\mathrm{Prismatization}}}^L_{\mathrm{abs},\mathbb{N}_\infty,\mathrm{syntomization},-}.
\end{align}
To be more precise we have a K\"unneth theorem at least by passing to the corresponding quasicoherent sheaves over the corresponding prismatization (i.e. the corresponding stackifications). This holds for all three stackifications in the families over $\mathbb{N}_\infty$. 
\end{theorem}

\begin{proof}
We have to consider formal models in this setting. We need to consider the compactification in this theorem. Any prismatization stackification for a particular formal scheme $-$ is basically a stackification over $-$ and ringed we use the notation $\mathcal{P}$ to denote the structure sheaf. As in the usual situation we have that the derived $\infty$-category of all the $\mathcal{P}$-modules is equivalent to the prismatic site for $-$. However the derived prismatic cohomology theory in this setting does satisfy the K\"unneth theorem in the derived version on the product of the corresponding $p$-adic formal ring locally (which can be glue over formal schemes like $-$). Then the result can be generalized to our setting over $\mathbb{N}$ away from $\infty$. Then we can take the corresponding compactification along:
\begin{align}
{\underline{\mathrm{Prismatization}}}^L_{\mathrm{abs},\mathbb{N}_\infty,-}\rightarrow  \prod_{n\in \mathbb{N}} {\underline{\mathrm{Prismatization}}}^L_{\mathrm{abs},n,-}.
\end{align}
This implies the derived version of the K\"unneth theorem holds for analytic filtration prismatization and analytic syntomization prismatization. This finishes the proof.
\end{proof}

\begin{theorem}
All three analytic prismatization, analytic filtration prismatization and analytic syntomization prismatization over $\mathbb{N}_\infty$ satisfy our \cref{situation1} conditions: (A1), (A2), (A3), (A4), (A5).
\end{theorem}

\begin{proof}
We check this one by one. First for the first condition it is true after we consider the corresponding structure sheaves. The second condition is proved above. For the third condition we consider the corresponding derived $\infty$-categories of quasicoherent sheaves. A4 holds for pullbacks and pushforward. This is true over $\mathbb{N}$ then we consider the fiber product through:
\begin{align}
{\underline{\mathrm{Prismatization}}}^L_{\mathrm{abs},\mathbb{N}_\infty,-}\rightarrow  \prod_{n\in \mathbb{N}} {\underline{\mathrm{Prismatization}}}^L_{\mathrm{abs},n,-}
\end{align}
to reach the corresponding compactification. Finally A5 holds see \cite[Chapter 4, in particular 4.7, 4.8, 4.9, 4.10]{3A}. However since we are talking about rigid analytic spaces by using perfectoid coverings, \cite{3A} can be directly applied, where one can derive all the results from the formalism in \cite{3A}, which allows us to reduce the corresponding proof to the formal scheme situation before. Again we need to consider the generalized version of \cite{3A} along \cite{3S}, i.e. regard the corresponding rigid analytic motivic cohomology theories as ones in our \cref{situation1}. One then finishes the full proof as in \cref{theorem36}.
\end{proof}

\begin{definition}
We in this context use the notations:
\begin{align}
&\mathrm{Galois}(\underline{\mathrm{Prismatization}}^L_{\mathrm{abs},\mathbb{N}_\infty,-})\\
&\mathrm{Galois}(\overline{\underline{\mathrm{Prismatization}}}^L_{\mathrm{abs},\mathrm{Nygaard},\mathbb{N}_\infty,-})\\
&\mathrm{Galois}(\overline{\underline{\mathrm{Prismatization}}}^L_{\mathrm{abs},\mathbb{N}_\infty,\mathrm{syntomization},-})
\end{align}
to denote the corresponding Hopf algebraic motivic Galois fundamental groups, which is the defined to be the spectra of the associated Hopf algebra.
\end{definition}

\begin{theorem}
There are morphisms from these motivic Galois fundamental groups to the corresponding motivic Galois groups of $A$, and moreover $L$: 
\begin{align}
&\mathrm{Galois}(\underline{\mathrm{Prismatization}}^L_{\mathrm{abs},\mathbb{N}_\infty,-})\rightarrow \mathrm{Gal}(\overline{L}/L),\\
&\mathrm{Galois}(\overline{\underline{\mathrm{Prismatization}}}^L_{\mathrm{abs},\mathrm{Nygaard},\mathbb{N}_\infty,-})\rightarrow \mathrm{Gal}(\overline{L}/L),\\
&\mathrm{Galois}(\overline{\underline{\mathrm{Prismatization}}}^L_{\mathrm{abs},\mathbb{N}_\infty,\mathrm{syntomization},-})\rightarrow \mathrm{Gal}(\overline{L}/L).
\end{align}
\end{theorem}

\begin{proof}
This is formal since we just apply the construction and definition to the point situation (regard the scheme $\mathrm{Spec}A$ or $\mathrm{Spec}L$ as the corresponding formal scheme or rigid analytic space), then the theorem follows by the functoriality of the construction of Hopf algebras. Then we reduce to the results before for formal local model before following \cite{3A}. 
\end{proof}

\subsection{Three analytic prismatizations in the de Rham setting}

\begin{definition}
We consider the corresponding de Rham prismatization in families. Recall the definition goes in the following way. Eventually the definition is applied to $z$-adic $A$-formal schemes. In the absolute manner we consider the category of all $z$-nilpotent $A$-algebras as the underlying ring categorie $\mathrm{Nil}_{z,A}$, where all such algebras are assumed to be fibered over $\mathbb{N}_\infty$. Then we use the Witt vector functor $W_A$ from \cite{3LH}. Then the prismaization is the stackification over the ring category $\mathrm{Nil}_{z,A}$ where for each such ring we have the groupoid of all the Cartier-Witt ideals fibered over $\mathbb{N}_\infty$, i.e. those maps locally principally generated by distinguished elements in the big Witt vectors (for $w_0$ we require nilpotency and for $w_1$ we require unitality). We follow \cite{3S1}, \cite{3ALBRCS} to the following three step for any $L$-rigid affioid with formal local model:
\begin{itemize}
\item[S1] Over $\mathbb{N}$ we have the analytic prismatization from local formal model from \cite{3S1}, \cite{3ALBRCS}, which provide the analytic prismatization over the generic fiber over $L$;
\item[S2] We take the corresponding compactification of this to $\mathbb{N}_\infty$, by using the morphism to take the fiber product:
\begin{align}
\underline{\mathrm{Prismatization}}_{\mathrm{abs},\mathbb{N}_\infty}\rightarrow \underline{\mathrm{Prismatization}}_{\mathrm{abs},\mathbb{N}_\infty}(\mathbb{N}).
\end{align}
\item[S3] Take condensed analytification from \cite{3CS}, \cite{3CS1}, \cite{3CS2}.
\end{itemize} In this paper we use the notation 
\begin{align}
\underline{\mathrm{Prismatization}}^L_{\mathrm{abs},\mathrm{dR},\mathbb{N}_\infty}
\end{align}
to denote the absolute analytic prismatization in families in our current setting. The definition for this is through the corresponding compactification from the stackification over $\mathbb{N}$, since then we have the morphism:
\begin{align}
\mathcal{P}_{\underline{\mathrm{Prismatization}}_{\mathrm{abs},\mathbb{N}_\infty}(\mathbb{N})} \overset{\sim}{\rightarrow} \varprojlim_{(Q_P,P)} D\mathrm{Cat}(P)^\mathrm{comp}\rightarrow \varprojlim_{(Q_P,P)} D\mathrm{Cat}(P[1/z]_{Q_P})^\mathrm{comp}
\end{align}
where $\mathrm{comp}$ is the completion with respect to the natural toplogy induced from the prisms involved along the inverse limit. Then we take the compactification to reach the de Rham prismatization in families $\mathbb{N}_\infty$:
\begin{align}
\mathcal{P}_{\underline{\mathrm{Prismatization}}_{\mathrm{abs},\mathbb{N}_\infty}} \overset{\sim}{\rightarrow} \overline{\varprojlim_{(Q_P,P)} D\mathrm{Cat}(P)^\mathrm{comp}}\rightarrow \overline{\varprojlim_{(Q_P,P)} D\mathrm{Cat}(P[1/z]_{Q_P})^\mathrm{comp}}.
\end{align}
Then we have to take analytification when needed. Then we have the corresponding filtration analytic prismatization:
\begin{align}
\overline{\underline{\mathrm{Prismatization}}}^L_{\mathrm{abs},\mathrm{Nygaard},\mathrm{dR},\mathbb{N}}
\end{align}
which is defined to be compactification from $\mathbb{N}$ to $\mathbb{N}_\infty$ of the corresponding open version over $\mathbb{N}$:
\begin{align}
{\underline{\mathrm{Prismatization}}}^L_{\mathrm{abs},\mathrm{Nygaard},\mathrm{dR},\mathbb{N}}
\end{align}
where this open version is defined to be taking the product over the finite fibers over $\mathbb{N}_\infty$ of the usual filtration prismatizations over $z_n$-nilpotent $A_n$-algebras. The compactification process needs to be using the following morphisms to reach our definition as the corresponding fiber product in families through the compactification:
\begin{align}
\prod_{n\in \mathbb{N}} {\underline{\mathrm{Prismatization}}}^L_{\mathrm{abs},\mathrm{Nygaard},\mathrm{dR},n}\rightarrow  \prod_{n\in \mathbb{N}} {\underline{\mathrm{Prismatization}}}^L_{\mathrm{abs},\mathrm{dR},n}
\end{align}
by taking the products along all $\mathbb{N}$ of the usual projection from the corresponding filtration analytic prismatization to the prismatization, and:
\begin{align}
{\underline{\mathrm{Prismatization}}}^L_{\mathrm{abs},\mathrm{dR},\mathbb{N}_\infty}\rightarrow  \prod_{n\in \mathbb{N}} {\underline{\mathrm{Prismatization}}}^L_{\mathrm{abs},\mathrm{dR},n}.
\end{align}
Recall that the filtration analytic prismatization is actually certain fibration version of the usual analytic prismatization, then we take further fibration over $\mathbb{N}$, in our current setting. Recall how finally the syntomization prismatization is constructed, which is by taking the descent of the Hodge-Tate morphism and de Rham morphism in the filtration prismatization, along the diagonal morphism from the prismatization with itself. These maps are obviously admitting the corresponding compactification versions in our current setting. Therefore we consider the following three morphisms:
\begin{align}
f_\mathrm{dR}: {\underline{\mathrm{Prismatization}}}^L_{\mathrm{abs},\mathrm{dR},\mathbb{N}_\infty}\rightarrow \overline{\underline{\mathrm{Prismatization}}}^L_{\mathrm{abs},\mathrm{Nygaard},\mathrm{dR},\mathbb{N}},
\end{align}
\begin{align}
f_\mathrm{HT}: {\underline{\mathrm{Prismatization}}}^L_{\mathrm{abs},\mathrm{dR},\mathbb{N}_\infty}\rightarrow \overline{\underline{\mathrm{Prismatization}}}^L_{\mathrm{abs},\mathrm{Nygaard},\mathrm{dR},\mathbb{N}},
\end{align}
\begin{align}
d:  {\underline{\mathrm{Prismatization}}}^L_{\mathrm{abs},\mathrm{dR},\mathbb{N}_\infty}\times {\underline{\mathrm{Prismatization}}}^L_{\mathrm{abs},\mathrm{dR},\mathbb{N}_\infty} \rightarrow {\underline{\mathrm{Prismatization}}}^L_{\mathrm{abs},\mathrm{dR},\mathbb{N}_\infty}.
\end{align}
Then we take the product of $f_\mathrm{dR}$ and $f_\mathrm{HT}$ together and take the corresponding descent along $d$ we have the definition of the syntomization analytic prismatization in families in our current setting over $\mathbb{N}_\infty$:
\begin{align}
\overline{\underline{\mathrm{Prismatization}}}^L_{\mathrm{abs},\mathrm{dR},\mathrm{syntomization},\mathbb{N}}. 
\end{align}
This finishes the definition. $\square$
\end{definition}

\begin{definition}
We consider the corresponding de Rham prismatization in families. Recall the definition goes in the following way. Eventually the definition is applied to $z$-adic $A$-formal schemes. In the absolute manner we consider the category of all $z$-nilpotent $A$-algebras as the underlying ring categorie $\mathrm{Nil}_{z,A}$, where all such algebras are assumed to be fibered over $\mathbb{N}_\infty$. Then we use the Witt vector functor $W_A$ from \cite{3LH}. Then the prismaization is the stackification over the ring category $\mathrm{Nil}_{z,A}$ where for each such ring we have the groupoid of all the Cartier-Witt ideals fibered over $\mathbb{N}_\infty$, i.e. those maps locally principally generated by distinguished elements in the big Witt vectors (for $w_0$ we require nilpotency and for $w_1$ we require unitality). Now map all the construction and definitions in the de Rham setting to $z$-adic $A$-formal scheme $M$. In this paper we use the notation 
\begin{align}
\underline{\mathrm{Prismatization}}^L_{\mathrm{abs},\mathrm{dR},\mathbb{N}_\infty,M}
\end{align}
to denote the absolute analytic de Rham prismatization in families in our current setting. The definition for this is through the corresponding compactification from the stackification over $\mathbb{N}$, since then we have the morphism:
\begin{align}
\mathcal{P}_{\underline{\mathrm{Prismatization}}_{\mathrm{abs},\mathrm{dR},\mathbb{N}_\infty,M}(\mathbb{N})} \overset{\sim}{\rightarrow} \varprojlim_{(Q_P,P)/M} D\mathrm{Cat}(P)^\mathrm{comp}\rightarrow \varprojlim_{(Q_P,P)/M} D\mathrm{Cat}(P[1/z]_{Q_P})^\mathrm{comp}
\end{align}
where $\mathrm{comp}$ is the completion with respect to the natural toplogy induced from the prisms involved along the inverse limit. Then we take the compactification to reach the de Rham prismatization in families $\mathbb{N}_\infty$:
\begin{align}
\mathcal{P}_{\underline{\mathrm{Prismatization}}_{\mathrm{abs},\mathrm{dR},\mathbb{N}_\infty},M} \overset{\sim}{\rightarrow} \overline{\varprojlim_{(Q_P,P)/M} D\mathrm{Cat}(P)^\mathrm{comp}}\rightarrow \overline{\varprojlim_{(Q_P,P)/M} D\mathrm{Cat}(P[1/z]_{Q_P})^\mathrm{comp}}.
\end{align}
Then we have the corresponding filtration prismatization:
\begin{align}
\overline{\underline{\mathrm{Prismatization}}}^L_{\mathrm{abs},\mathrm{Nygaard},\mathrm{dR},\mathbb{N},M}
\end{align}
which is defined to be compactification from $\mathbb{N}$ to $\mathbb{N}_\infty$ of the corresponding open version over $\mathbb{N}$:
\begin{align}
{\underline{\mathrm{Prismatization}}}^L_{\mathrm{abs},\mathrm{Nygaard},\mathrm{dR},\mathbb{N},M}
\end{align}
where this open version is defined to be taking the product over the finite fibers over $\mathbb{N}_\infty$ of the usual filtration prismatizations over $z_n$-nilpotent $A_n$-algebras. The compactification process needs to be using the following morphisms to reach our definition as the corresponding fiber product in families through the compactification:
\begin{align}
\prod_{n\in \mathbb{N}} {\underline{\mathrm{Prismatization}}}^L_{\mathrm{abs},\mathrm{Nygaard},\mathrm{dR},n,M}\rightarrow  \prod_{n\in \mathbb{N}} {\underline{\mathrm{Prismatization}}}^L_{\mathrm{abs},\mathrm{dR},n,M}
\end{align}
by taking the products along all $\mathbb{N}$ of the usual projection from the corresponding filtration prismatization to the prismatization, and:
\begin{align}
{\underline{\mathrm{Prismatization}}}^L_{\mathrm{abs},\mathrm{dR},\mathbb{N}_\infty,M}\rightarrow  \prod_{n\in \mathbb{N}} {\underline{\mathrm{Prismatization}}}^L_{\mathrm{abs},\mathrm{dR},n,M}.
\end{align}
Recall that the filtration analytic prismatization is actually certain fibration version of the usual analytic prismatization, then we take further fibration over $\mathbb{N}$, in our current setting. Recall how finally the syntomization prismatization is constructed, which is by taking the descent of the Hodge-Tate morphism and de Rham morphism in the filtration prismatization, along the diagonal morphism from the prismatization with itself. These maps are obviously admitting the corresponding compactification versions in our current setting. Therefore we consider the following three morphisms:
\begin{align}
f_\mathrm{dR}: {\underline{\mathrm{Prismatization}}}^L_{\mathrm{abs},\mathrm{dR},\mathbb{N}_\infty,M}\rightarrow \overline{\underline{\mathrm{Prismatization}}}^L_{\mathrm{abs},\mathrm{Nygaard},\mathrm{dR},\mathbb{N}_\infty,M},
\end{align}
\begin{align}
f_\mathrm{HT}: {\underline{\mathrm{Prismatization}}}^L_{\mathrm{abs},\mathrm{dR},\mathbb{N}_\infty,M}\rightarrow \overline{\underline{\mathrm{Prismatization}}}^L_{\mathrm{abs},\mathrm{Nygaard},\mathrm{dR},\mathbb{N},M},
\end{align}
\begin{align}
d:  {\underline{\mathrm{Prismatization}}}^L_{\mathrm{abs},\mathrm{dR},\mathbb{N}_\infty,M}\times {\underline{\mathrm{Prismatization}}}^L_{\mathrm{abs},\mathrm{dR},\mathbb{N}_\infty,M} \rightarrow {\underline{\mathrm{Prismatization}}}^L_{\mathrm{abs},\mathrm{dR},\mathbb{N}_\infty,M}.
\end{align}
Then we take the product of $f_\mathrm{dR}$ and $f_\mathrm{HT}$ together and take the corresponding descent along $d$ we have the definition of the syntomization analytic prismatization in families in our current setting over $\mathbb{N}_\infty$:
\begin{align}
\overline{\underline{\mathrm{Prismatization}}}^L_{\mathrm{abs},\mathrm{dR},\mathrm{syntomization},\mathbb{N},M}. 
\end{align}
This finishes the definition. $\square$
\end{definition}

\begin{theorem}\label{theorem20}
There is a version of K\"unneth theorem in this context. To be more precise we have a K\"unneth theorem at least by passing to the corresponding quasicoherent sheaves over the corresponding prismatization (i.e. the corresponding stackifications). This holds for all three stackifications in the families over $\mathbb{N}_\infty$. 
\end{theorem}

\begin{proof}
After we reduce this to formal models, we need to consider the compactification in this theorem. Any prismatization stackification for a particular formal scheme $M$ is basically a stackification over $M$ and ringed we use the notation $\mathcal{P}$ to denote the structure sheaf. As in the usual situation we have that the derived $\infty$-category of all the $\mathcal{P}$-modules is equivalent to the prismatic site for $M$. However the derived prismatic cohomology theory in this setting does satisfy the derived K\"unneth theorem on the product of the corresponding $p$-adic formal ring locally (which can be glue over formal schemes like $M$). Then the result can be generalized to our setting over $\mathbb{N}$ away from $\infty$. Then we can take the corresponding compactification along:
\begin{align}
{\underline{\mathrm{Prismatization}}}^L_{\mathrm{abs},\mathbb{N}_\infty,\mathrm{dR},M}\rightarrow  \prod_{n\in \mathbb{N}} {\underline{\mathrm{Prismatization}}}^L_{\mathrm{abs},n,\mathrm{dR},M}.
\end{align}
This implies the derived K\"unneth theorem holds for analytic filtration prismatization and analytic syntomization prismatization. We then get the derived version of the K\"unneth theorem in the current setting by restricting to the 3 de Rham prismatizations. This finishes the proof.
\end{proof}

\begin{theorem}
All three prismatization, filtration prismatization and syntomization prismatization fibered over $\mathbb{N}_\infty$ satisfy our \cref{situation2} conditions: (A1), (A2), (A3), (A4), (A5).
\end{theorem}

\begin{proof}
We check this one by one. First for the first condition it is true after we consider the corresponding structure sheaves. The second condition is proved above. For the third condition we consider the corresponding derived $\infty$-categories of quasicoherent sheaves. A4 holds for pullbacks and pushforward. This is true over $\mathbb{N}$ then we consider the fiber product through:
\begin{align}
{\underline{\mathrm{Prismatization}}}_{\mathrm{abs},\mathbb{N}_\infty,\mathrm{dR},M}\rightarrow  \prod_{n\in \mathbb{N}} {\underline{\mathrm{Prismatization}}}_{\mathrm{abs},n,\mathrm{dR},M}
\end{align}
to reach the corresponding compactification. Finally A5 holds see \cite[Chapter 4, in particular 4.7, 4.8, 4.9, 4.10]{3A}. One then finishes the full proof as in \cref{theorem36}.
\end{proof}

\begin{definition}
The three analytic prismatizations:
\begin{align}
&\underline{\mathrm{Prismatization}}^L_{\mathrm{abs},\mathbb{N}_\infty,\mathrm{dR},M}\\
&\overline{\underline{\mathrm{Prismatization}}}^L_{\mathrm{abs},\mathrm{Nygaard},\mathbb{N}_\infty,\mathrm{dR},M}\\
&\overline{\underline{\mathrm{Prismatization}}}^L_{\mathrm{abs},\mathbb{N}_\infty,\mathrm{syntomization},\mathrm{dR},M}
\end{align}
have a small arc-stack version as well as a small v-stack version, where the condensed structure sheaves are pull-back along the strictly totally disconnected coverings. For such switching of categories and consideration, we use the then the following different notation:
\begin{align}
&\underline{\mathrm{Prismatization}}^L_{\mathrm{abs},\mathbb{N}_\infty,\mathrm{dR},-}\\
&\overline{\underline{\mathrm{Prismatization}}}^L_{\mathrm{abs},\mathrm{Nygaard},\mathbb{N}_\infty,\mathrm{dR},-}\\
&\overline{\underline{\mathrm{Prismatization}}}^L_{\mathrm{abs},\mathbb{N}_\infty,\mathrm{syntomization},\mathrm{dR},-}
\end{align}
where $-$ is a small arc stack or a small v-stack attached to a rigid analytic space over $L$ over $\mathbb{N}_\infty$, however locally the topology is arc-topology or $v$-topology.
\end{definition}

\begin{theorem}
There is a version of K\"unneth theorem in this context
\begin{align}
&\underline{\mathrm{Prismatization}}^L_{\mathrm{abs},\mathbb{N}_\infty,\mathrm{dR},-}\\
&\overline{\underline{\mathrm{Prismatization}}}^L_{\mathrm{abs},\mathrm{Nygaard},\mathbb{N}_\infty,\mathrm{dR},-}\\
&\overline{\underline{\mathrm{Prismatization}}}^L_{\mathrm{abs},\mathbb{N}_\infty,\mathrm{syntomization},\mathrm{dR},-}.
\end{align}
To be more precise we have a K\"unneth theorem at least by passing to the corresponding quasicoherent sheaves over the corresponding prismatization (i.e. the corresponding stackifications). This holds for all three stackifications in the families over $\mathbb{N}_\infty$. 
\end{theorem}

\begin{proof}
We need to consider the compactification in this theorem, after we reduce the consideration to formal models. Any prismatization stackification for a particular formal scheme $-$ is basically a stackification over $-$ and ringed we use the notation $\mathcal{P}$ to denote the structure sheaf. As in the usual situation we have that the derived $\infty$-category of all the $\mathcal{P}$-modules is equivalent to the prismatic site for $-$. However the derived prismatic cohomology theory in this setting does satisfy the K\"unneth theorem in the derived version on the product of the corresponding $p$-adic formal ring locally (which can be glue over formal schemes like $-$). Then the result can be generalized to our setting over $\mathbb{N}$ away from $\infty$. Then we can take the corresponding compactification along:
\begin{align}
{\underline{\mathrm{Prismatization}}}^L_{\mathrm{abs},\mathbb{N}_\infty,\mathrm{dR},-}\rightarrow  \prod_{n\in \mathbb{N}} {\underline{\mathrm{Prismatization}}}^L_{\mathrm{abs},n,\mathrm{dR},-}.
\end{align}
This implies the derived K\"unneth theorem holds for analytic filtration prismatization and analytic syntomization prismatization. We then get the derived version of the K\"unneth theorem in the current setting by restricting to the 3 de Rham prismatizations. This finishes the proof.
\end{proof}

\begin{theorem}
All three analytic prismatization, analytic filtration prismatization and analytic syntomization prismatization over $\mathbb{N}_\infty$ satisfy our \cref{situation1} conditions: (A1), (A2), (A3), (A4), (A5).
\end{theorem}

\begin{proof}
We check this one by one. First for the first condition it is true after we consider the corresponding structure sheaves. The second condition is proved above. For the third condition we consider the corresponding derived $\infty$-categories of quasicoherent sheaves. A4 holds for pullbacks and pushforward. This is true over $\mathbb{N}$ then we consider the fiber product through:
\begin{align}
{\underline{\mathrm{Prismatization}}}^L_{\mathrm{abs},\mathbb{N}_\infty,\mathrm{dR},-}\rightarrow  \prod_{n\in \mathbb{N}} {\underline{\mathrm{Prismatization}}}^L_{\mathrm{abs},n,\mathrm{dR},-}
\end{align}
to reach the corresponding compactification. Finally A5 holds see \cite[Chapter 4, in particular 4.7, 4.8, 4.9, 4.10]{3A}. However since we are talking about rigid analytic spaces by using perfectoid coverings, \cite{3A} can be directly applied, where one can derive all the results from the formalism in \cite{3A}, which allows us to reduce the corresponding proof to the formal scheme situation before. Again we need to consider the generalized version of \cite{3A} along \cite{3S}, i.e. regard the corresponding rigid analytic motivic cohomology theories as ones in our \cref{situation1}. One then finishes the proof as in \cref{theorem36}.
\end{proof}

\begin{definition}
We in this context use the notations:
\begin{align}
&\mathrm{Galois}(\underline{\mathrm{Prismatization}}^L_{\mathrm{abs},\mathbb{N}_\infty,\mathrm{dR},-})\\
&\mathrm{Galois}(\overline{\underline{\mathrm{Prismatization}}}^L_{\mathrm{abs},\mathrm{Nygaard},\mathbb{N}_\infty,\mathrm{dR},-})\\
&\mathrm{Galois}(\overline{\underline{\mathrm{Prismatization}}}^L_{\mathrm{abs},\mathbb{N}_\infty,\mathrm{syntomization},\mathrm{dR},-})
\end{align}
to denote the corresponding Hopf algebraic motivic Galois fundamental groups, which is the defined to be the spectra of the associated Hopf algebra.
\end{definition}

\begin{theorem}
There are morphisms from these motivic Galois fundamental groups to the corresponding motivic Galois groups of $A$, and moreover $L$: 
\begin{align}
&\mathrm{Galois}(\underline{\mathrm{Prismatization}}^L_{\mathrm{abs},\mathbb{N}_\infty,\mathrm{dR},-})\rightarrow \mathrm{Gal}(\overline{L}/L),\\
&\mathrm{Galois}(\overline{\underline{\mathrm{Prismatization}}}^L_{\mathrm{abs},\mathrm{Nygaard},\mathbb{N}_\infty,\mathrm{dR},-})\rightarrow \mathrm{Gal}(\overline{L}/L),\\
&\mathrm{Galois}(\overline{\underline{\mathrm{Prismatization}}}^L_{\mathrm{abs},\mathbb{N}_\infty,\mathrm{syntomization},\mathrm{dR},-})\rightarrow \mathrm{Gal}(\overline{L}/L).
\end{align}
\end{theorem}

\begin{proof}
This is formal since we just apply the construction and definition to the point situation (regard the scheme $\mathrm{Spec}A$ or $\mathrm{Spec}L$ as the corresponding formal scheme or rigid analytic space), then the theorem follows by the functoriality of the construction of Hopf algebras. Then we reduce to the results before for formal local model before following \cite{3A}. 
\end{proof}

\subsection{Three analytic prismatizations in the de Rham-Hodge-Tate setting}

\begin{definition}
We consider the corresponding de Rham-Hodge-Tate prismatization in families. Recall the definition goes in the following way. Eventually the definition is applied to $z$-adic $A$-formal schemes. In the absolute manner we consider the category of all $z$-nilpotent $A$-algebras as the underlying ring categorie $\mathrm{Nil}_{z,A}$, where all such algebras are assumed to be fibered over $\mathbb{N}_\infty$. Then we use the Witt vector functor $W_A$ from \cite{3LH}. Then the prismaization is the stackification over the ring category $\mathrm{Nil}_{z,A}$ where for each such ring we have the groupoid of all the Cartier-Witt ideals fibered over $\mathbb{N}_\infty$, i.e. those maps locally principally generated by distinguished elements in the big Witt vectors (for $w_0$ we require nilpotency and for $w_1$ we require unitality). In this paper we use the notation 
\begin{align}
\underline{\mathrm{Prismatization}}^L_{\mathrm{abs},\mathrm{dRHT},\mathbb{N}_\infty}
\end{align}
to denote the absolute analytic de Rham-Hodge-Tate prismatization in families in our current setting. The definition for this is through the corresponding compactification from the stackification over $\mathbb{N}$, since then we have the morphism:
\begin{align}
\mathcal{P}_{\underline{\mathrm{Prismatization}}_{\mathrm{abs},\mathrm{dRHT},\mathbb{N}_\infty}(\mathbb{N})} \overset{\sim}{\rightarrow} \varprojlim_{(Q_P,P)} D\mathrm{Cat}(P)^\mathrm{comp}\rightarrow \varprojlim_{(Q_P,P)} D\mathrm{Cat}(P[1/z]_{Q_P})^\mathrm{comp}\rightarrow \varprojlim_{(Q_P,P)} D\mathrm{Cat}(P[1/z]/{Q_P})^\mathrm{comp}
\end{align}
where $\mathrm{comp}$ is the completion with respect to the natural toplogy induced from the prisms involved along the inverse limit. Then we take the compactification to reach the de Rham-Hodge-Tate prismatization in families $\mathbb{N}_\infty$:
\begin{align}
\mathcal{P}_{\underline{\mathrm{Prismatization}}_{\mathrm{abs},\mathrm{dRHT},\mathbb{N}_\infty}} \overset{\sim}{\rightarrow} \overline{\varprojlim_{(Q_P,P)} D\mathrm{Cat}(P)^\mathrm{comp}}\rightarrow \overline{\varprojlim_{(Q_P,P)} D\mathrm{Cat}(P[1/z]_{Q_P})^\mathrm{comp}}\rightarrow \overline{\varprojlim_{(Q_P,P)} D\mathrm{Cat}(P[1/z]/{Q_P})^\mathrm{comp}}.
\end{align}
Then we have the corresponding filtration prismatization:
\begin{align}
\overline{\underline{\mathrm{Prismatization}}}^L_{\mathrm{abs},\mathrm{Nygaard},\mathrm{dRHT},\mathbb{N}}
\end{align}
which is defined to be compactification from $\mathbb{N}$ to $\mathbb{N}_\infty$ of the corresponding open version over $\mathbb{N}$:
\begin{align}
{\underline{\mathrm{Prismatization}}}^L_{\mathrm{abs},\mathrm{Nygaard},\mathrm{dRHT},\mathbb{N}}
\end{align}
where this open version is defined to be taking the product over the finite fibers over $\mathbb{N}_\infty$ of the usual filtration prismatizations over $z_n$-nilpotent $A_n$-algebras. The compactification process needs to be using the following morphisms to reach our definition as the corresponding fiber product in families through the compactification:
\begin{align}
\prod_{n\in \mathbb{N}} {\underline{\mathrm{Prismatization}}}^L_{\mathrm{abs},\mathrm{Nygaard},\mathrm{dRHT},n}\rightarrow  \prod_{n\in \mathbb{N}} {\underline{\mathrm{Prismatization}}}^L_{\mathrm{abs},\mathrm{dRHT},n}
\end{align}
by taking the products along all $\mathbb{N}$ of the usual projection from the corresponding filtration prismatization to the prismatization, and:
\begin{align}
{\underline{\mathrm{Prismatization}}}^L_{\mathrm{abs},\mathrm{dRHT},\mathbb{N}_\infty}\rightarrow  \prod_{n\in \mathbb{N}} {\underline{\mathrm{Prismatization}}}^L_{\mathrm{abs},\mathrm{dRHT},n}.
\end{align}
Recall that the filtration analytic prismatization is actually certain fibration version of the usual analytic prismatization, then we take further fibration over $\mathbb{N}$, in our current setting. Recall how finally the syntomization prismatization is constructed, which is by taking the descent of the Hodge-Tate morphism and de Rham morphism in the filtration prismatization, along the diagonal morphism from the prismatization with itself. These maps are obviously admitting the corresponding compactification versions in our current setting. Therefore we consider the following three morphisms:
\begin{align}
f_\mathrm{dR}: {\underline{\mathrm{Prismatization}}}^L_{\mathrm{abs},\mathrm{dRHT},\mathbb{N}_\infty}\rightarrow \overline{\underline{\mathrm{Prismatization}}}^L_{\mathrm{abs},\mathrm{Nygaard},\mathrm{dRHT},\mathbb{N}_\infty},
\end{align}
\begin{align}
f_\mathrm{HT}: {\underline{\mathrm{Prismatization}}}^L_{\mathrm{abs},\mathrm{dRHT},\mathbb{N}_\infty}\rightarrow \overline{\underline{\mathrm{Prismatization}}}^L_{\mathrm{abs},\mathrm{Nygaard},\mathrm{dRHT},\mathbb{N}_\infty},
\end{align}
\begin{align}
d:  {\underline{\mathrm{Prismatization}}}^L_{\mathrm{abs},\mathrm{dRHT},\mathbb{N}_\infty}\times {\underline{\mathrm{Prismatization}}}^L_{\mathrm{abs},\mathrm{dRHT},\mathbb{N}_\infty} \rightarrow {\underline{\mathrm{Prismatization}}}^L_{\mathrm{abs},\mathrm{dRHT},\mathbb{N}_\infty}.
\end{align}
Then we take the product of $f_\mathrm{dR}$ and $f_\mathrm{HT}$ together and take the corresponding descent along $d$ we have the definition of the analytic syntomization prismatization in families in our current setting over $\mathbb{N}_\infty$:
\begin{align}
\overline{\underline{\mathrm{Prismatization}}}^L_{\mathrm{abs},\mathrm{dRHT},\mathrm{syntomization},\mathbb{N}}. 
\end{align}
This finishes the definition. $\square$
\end{definition}

\begin{definition}
We consider the corresponding de Rham-Hodge-Tate prismatization in families. Recall the definition goes in the following way. Eventually the definition is applied to $z$-adic $A$-formal schemes. In the absolute manner we consider the category of all $z$-nilpotent $A$-algebras as the underlying ring categorie $\mathrm{Nil}_{z,A}$, where all such algebras are assumed to be fibered over $\mathbb{N}_\infty$. Then we use the Witt vector functor $W_A$ from \cite{3LH}. Then the prismaization is the stackification over the ring category $\mathrm{Nil}_{z,A}$ where for each such ring we have the groupoid of all the Cartier-Witt ideals fibered over $\mathbb{N}_\infty$, i.e. those maps locally principally generated by distinguished elements in the big Witt vectors (for $w_0$ we require nilpotency and for $w_1$ we require unitality). Now map all the construction and definitions in the de Rham-Hodge-Tate setting to $z$-adic $A$-formal scheme $M$. In this paper we use the notation 
\begin{align}
\underline{\mathrm{Prismatization}}^L_{\mathrm{abs},\mathrm{dRHT},\mathbb{N}_\infty,M}
\end{align}
to denote the absolute analytic prismatization in families in our current setting. The definition for this is through the corresponding compactification from the stackification over $\mathbb{N}$, since then we have the morphism:
\begin{align}
&\mathcal{P}_{\underline{\mathrm{Prismatization}}_{\mathrm{abs},\mathrm{dRHT},\mathbb{N}_\infty,M}(\mathbb{N})} \overset{\sim}{\rightarrow} \varprojlim_{(Q_P,P)/M} D\mathrm{Cat}(P)^\mathrm{comp}\\
&\rightarrow \varprojlim_{(Q_P,P)/M} D\mathrm{Cat}(P[1/z]/{Q_P})^\mathrm{comp}\rightarrow \varprojlim_{(Q_P,P)/M} D\mathrm{Cat}(P[1/z]/{Q_P})^\mathrm{comp}
\end{align}
where $\mathrm{comp}$ is the completion with respect to the natural toplogy induced from the prisms involved along the inverse limit. Then we take the compactification to reach the de Rham prismatization in families $\mathbb{N}_\infty$:
\begin{align}
&\mathcal{P}_{\underline{\mathrm{Prismatization}}_{\mathrm{abs},\mathrm{dRHT},\mathbb{N}_\infty},M} \overset{\sim}{\rightarrow} \overline{\varprojlim_{(Q_P,P)/M} D\mathrm{Cat}(P)^\mathrm{comp}}\\
&\rightarrow \overline{\varprojlim_{(Q_P,P)/M} D\mathrm{Cat}(P[1/z]/{Q_P})^\mathrm{comp}}\rightarrow \overline{\varprojlim_{(Q_P,P)/M} D\mathrm{Cat}(P[1/z]/{Q_P})^\mathrm{comp}}.
\end{align}
Then we have the corresponding filtration prismatization:
\begin{align}
\overline{\underline{\mathrm{Prismatization}}}^L_{\mathrm{abs},\mathrm{Nygaard},\mathrm{dRHT},\mathbb{N},M}
\end{align}
which is defined to be compactification from $\mathbb{N}$ to $\mathbb{N}_\infty$ of the corresponding open version over $\mathbb{N}$:
\begin{align}
{\underline{\mathrm{Prismatization}}}^L_{\mathrm{abs},\mathrm{Nygaard},\mathrm{dRHT},\mathbb{N},M}
\end{align}
where this open version is defined to be taking the product over the finite fibers over $\mathbb{N}_\infty$ of the usual filtration prismatizations over $z_n$-nilpotent $A_n$-algebras. The compactification process needs to be using the following morphisms to reach our definition as the corresponding fiber product in families through the compactification:
\begin{align}
\prod_{n\in \mathbb{N}} {\underline{\mathrm{Prismatization}}}^L_{\mathrm{abs},\mathrm{Nygaard},\mathrm{dRHT},n,M}\rightarrow  \prod_{n\in \mathbb{N}} {\underline{\mathrm{Prismatization}}}^L_{\mathrm{abs},\mathrm{dRHT},n,M}
\end{align}
by taking the products along all $\mathbb{N}$ of the usual projection from the corresponding filtration prismatization to the prismatization, and:
\begin{align}
{\underline{\mathrm{Prismatization}}}^L_{\mathrm{abs},\mathrm{dRHT},\mathbb{N}_\infty,M}\rightarrow  \prod_{n\in \mathbb{N}} {\underline{\mathrm{Prismatization}}}^L_{\mathrm{abs},\mathrm{dRHT},n,M}.
\end{align}
Recall that the analytic filtration prismatization is actually certain fibration version of the usual analytic prismatization, then we take further fibration over $\mathbb{N}$. Recall how finally the syntomization prismatization is constructed, which is by taking the descent of the Hodge-Tate morphism and de Rham morphism in the filtration prismatization, along the diagonal morphism from the prismatization with itself. These maps are obviously admitting the corresponding compactification versions in our current setting. Therefore we consider the following three morphisms:
\begin{align}
f_\mathrm{dR}: {\underline{\mathrm{Prismatization}}}^L_{\mathrm{abs},\mathrm{dRHT},\mathbb{N}_\infty,M}\rightarrow \overline{\underline{\mathrm{Prismatization}}}^L_{\mathrm{abs},\mathrm{Nygaard},\mathrm{dRHT},\mathbb{N}_\infty,M},
\end{align}
\begin{align}
f_\mathrm{HT}: {\underline{\mathrm{Prismatization}}}^L_{\mathrm{abs},\mathrm{dRHT},\mathbb{N}_\infty,M}\rightarrow \overline{\underline{\mathrm{Prismatization}}}^L_{\mathrm{abs},\mathrm{Nygaard},\mathrm{dRHT},\mathbb{N}_\infty,M},
\end{align}
\begin{align}
d:  {\underline{\mathrm{Prismatization}}}^L_{\mathrm{abs},\mathrm{dRHT},\mathbb{N}_\infty,M}\times {\underline{\mathrm{Prismatization}}}^L_{\mathrm{abs},\mathrm{dRHT},\mathbb{N}_\infty,M} \rightarrow {\underline{\mathrm{Prismatization}}}^L_{\mathrm{abs},\mathrm{dRHT},\mathbb{N}_\infty,M}.
\end{align}
Then we take the product of $f_\mathrm{dR}$ and $f_\mathrm{HT}$ together and take the corresponding descent along $d$ we have the definition of the analytic syntomization prismatization in families in our current setting over $\mathbb{N}_\infty$:
\begin{align}
\overline{\underline{\mathrm{Prismatization}}}^L_{\mathrm{abs},\mathrm{dRHT},\mathrm{syntomization},\mathbb{N},M}. 
\end{align}
This finishes the definition. $\square$
\end{definition}

\begin{theorem}
There is a version of K\"unneth theorem in this context. To be more precise we have a K\"unneth theorem at least by passing to the corresponding quasicoherent sheaves over the corresponding prismatization (i.e. the corresponding stackifications). This holds for all three stackifications in the families over $\mathbb{N}_\infty$. 
\end{theorem}

\begin{proof}
See the proof of \cref{theorem20}.
\end{proof}

\begin{theorem}
All three prismatization, filtration prismatization and syntomization prismatization over $\mathbb{N}_\infty$ satisfy our \cref{situation2} conditions: (A1), (A2), (A3), (A4), (A5).
\end{theorem}

\begin{proof}
We check this one by one. First for the first condition it is true after we consider the corresponding structure sheaves. The second condition is proved above. For the third condition we consider the corresponding derived $\infty$-categories of quasicoherent sheaves. A4 holds for pullbacks and pushforward. This is true over $\mathbb{N}$ then we consider the fiber product through:
\begin{align}
{\underline{\mathrm{Prismatization}}}^L_{\mathrm{abs},\mathbb{N}_\infty,\mathrm{dRHT},M}\rightarrow  \prod_{n\in \mathbb{N}} {\underline{\mathrm{Prismatization}}}^L_{\mathrm{abs},n,\mathrm{dRHT},M}
\end{align}
to reach the corresponding compactification. Finally A5 holds see \cite[Chapter 4, in particular 4.7, 4.8, 4.9, 4.10]{3A}. One then finishes the proof as in \cref{theorem36}.
\end{proof}

\newpage
\section{Motivic Riemann-Hilbert Correspondence in Families}

\subsection{Generalized Riemann-Hilbert Correspondence}

\begin{setting}
Throughout this paper, by a $D$-module we mean the module over the wedge algebra from the differential forms. We will then call a differential module to mean the $D$-modules \textit{in the usual sense}. 
\end{setting}

\begin{setting}
We now consider the setting in \cite{5LH}. Throughout we fix a finite field $F$, and choose $p$-adic local fields of characteristic zero with same residue field $F$:
\begin{align}
H_1,H_2,...
\end{align}
where the index is throughout all the natural numbers. There is one function field $F((x_\infty))$ at the infinity. Therefore we can now form a common uniformizer for all of them: $x$ such that $x$ gives rise to $x_n$ at each point in $\mathbb{N}_c$. We use the notation $H$ to denote the product of these local fields, and we use $I_H$ to denote the integral subrings as in \cite{5LH}.
\end{setting}

\begin{setting}
With inteval in $(0,\infty)$ we have families of Robba rings as in \cite{5LH} by taking the sections over the Fargues-Fontaine curves in families. We use then the notation:
\begin{align}
S_{\mathrm{Robba},I,\mathrm{perfectoid},P}
\end{align}
where $P$ is perfectoid Banach ring over $\mathbb{N}_c$. These perfect Robba rings are certain family version of the rings in \cite{5KLA}, \cite{5KLB}. We consider the Frobenius operator $k$ as in \cite{5KLA}, \cite{5KLB}.
\end{setting}

\begin{definition} \label{definition51}
Now we define certain almost Hodge-Tate Galois representation of Galois group of $H$. Let $L$ be such representation which is free of finite rank $r$, we consider the corresponding de Rham period ring:
\begin{align}
S_{\mathrm{dR},H,\mathbb{N}_c} := \prod_{n\in \mathbb{N}} S_{\mathrm{dR},H_n}\times S_{\mathrm{dR},H_\infty}
\end{align}
which is just the product of all the usual de Rham period rings at each $n\in \mathbb{N}_c$. This ring is filtred and gives rise to the Hodge-Tate ring by taking the associated graded ring:
\begin{align}
S_{\mathrm{HT},H,\mathbb{N}_c}:=gr(S_{\mathrm{dR},H,\mathbb{N}_c}).
\end{align}
After \cite{5F2}, \cite{5BLA}, \cite{5BLB} we say $L$ is almost Hodge-Tate in families if the rank preserves after we take the following invariance:
\begin{align}
(L\otimes gr(S_{\mathrm{dR},H,\mathbb{N}_c}[\mathrm{log}(t)]))^{G_H},
\end{align}
here $t$ is the product of all the Fontaine's '2$\pi$i' element for each $n\in \mathbb{N}_c$. 
\end{definition}

\begin{definition}
Now we consider a smooth algebraic variety $V$ over $H$. We consider the rigid generic fiber $V^\mathrm{rig}$ over $H$, over which we consider the pro-\'etale topology as in \cite{5L}, \cite{5S1}. We call any local system $M$ under this toplogy almost Hodge-Tate if each localization of $M_x$ at any point $x\in V^\mathrm{rig}$ is almost Hodge-Tate representation as in \cref{definition51}. Then we generalize the definition to the corresponding locally finite free Frobenius sheaves over Robba sheaves in our current family setting. If $M$ is such a module then after \cite{5KLA}, \cite{5KLB} we say $L$ is almost Hodge-Tate in families if the whole module preserves after we take the following invariance in the sense of taking pro-\'etale descent:
\begin{align}
\mathrm{Descent}_{V^\mathrm{rig}_\text{pro\'etale},V^\mathrm{rig},*}(M_\infty\otimes gr(S_{\mathrm{dR},-,\mathbb{N}_c}[\mathrm{log}(t)])),
\end{align}
here $t$ is the product of all the Fontaine's '2$\pi$i' element for each $n\in \mathbb{N}_c$, after we take further base change to the whole Hodge-Tate period sheaf in families in our current setting. Here $M_\infty$ is the section over the Robba ring $\varprojlim_{r\rightarrow \infty}S_{\mathrm{Robba},r,\mathrm{perfectoid},P}$.
\end{definition}

\begin{theorem}\label{theorem3}
Let $\mathrm{D}_{\mathrm{aHT},\mathrm{lisse},V}(S_{\mathrm{Robba},\mathrm{perfectoid},*})$ be the derived $\infty$-category generated by locally finite free almost Hodge-Tate Frobenius sheaves over Robba rings as in our current family situation. We then have the following differential module:
\begin{align}
\partial=\mathrm{Descent}_{V^\mathrm{rig}_\text{pro\'etale},H,*}(S_{\mathrm{dR},-,\mathbb{N}_c}[\mathrm{log}(t)]).
\end{align}
Then we have a well-defined functor:
\begin{align}
\mathrm{D}_{\mathrm{aHT},\mathrm{lisse},V}(S_{\mathrm{Robba},\mathrm{perfectoid},*})\longrightarrow \mathrm{D}_{\mathrm{inductive},\mathrm{coherent}}(\mathrm{Descent}_{V^\mathrm{rig}_\text{pro\'etale},V^\mathrm{rig},*}(S_{\mathrm{dR},-,\mathbb{N}_c}[\mathrm{log}(t)]))
\end{align}
which is symmetrical monoidal $\infty$-categorical functor into the $D$-modules in families in our current setting. Here the target category is defined to be the derived $\infty$-category of inductive systems of the coherent $D$-modules as in \cite{5GR1}, \cite{5GR2}. This functor is fully faithful.
\end{theorem}

\begin{proof}
As in the usual situation in \cite{5L}, \cite{5BLA}, \cite{5BLB} the functor is defined to be just the derived homomorphism from the structure sheaf of the rigid space $V^\mathrm{rig}$ to:
\begin{align}
\mathrm{Descent}_{V^\mathrm{rig}_\text{pro\'etale},V^\mathrm{rig},*}(M_\infty\otimes gr(S_{\mathrm{dR},-,\mathbb{N}_c}[\mathrm{log}(t)])).
\end{align}
The fully-faithfulness of the functor can be proved after \cite{5BLA} and \cite{5BLB} by first considering the perfectoid subdomains generated by the corresponding algebraic closed field. Note that at $\infty$ the local function field situation is actually reduced to the situation where we have the local fields of mixed-characteristics since in the \textit{perfect setting} the Witt vectors for the local function field can be just the base change of the Witt vectors for local fields of mixed-characteristics. So for those $p$-adic fields by using definition of Frobenius sheaves in this particular situation we have can see that the functor preserves the homomorphism groups as in the usual de Rham situation as in \cite{5KLB}, since by the definition of the almost Hodge-Tate sheaves the mappings on the homomorphism groups can be seen to be surjective and injective simultaneously, which finishes the proof.
\end{proof}

Pulling things to arc-topology directly we have then the following motivic version of this theorem:

\begin{definition}
Let $\mathrm{D}_{\mathrm{aHT},\mathrm{lisse},V}(S_{\mathrm{Robba},\mathrm{arc},*})$ be the derived $\infty$-category generated by locally finite free almost Hodge-Tate Frobenius sheaves over Robba rings as in our current family situation. Here we consider the category 
\begin{align}
\mathrm{Motives}_{S_{\mathrm{Robba},\mathrm{arc},*},\times}\overline{\mathrm{ArcStacks}}_{V,\mathrm{arc}}.
\end{align}
Here $\overline{\mathrm{ArcStacks}}$ denotes the restriction to smooth rigid spaces over $V$ from all the arc stacks over $V$. For each covering arc stack $W$ over $V$ we consider the corresponding almost Hodge-Tate sheaves over the Robba sheaves over $W$:
\begin{align}
\mathrm{D}_{\mathrm{aHT},\mathrm{lisse},W}(S_{\mathrm{Robba},\mathrm{arc},W})
\end{align}
which can be arranged to the motivic $\infty$-category over the big site over $V$:
\begin{align}
\mathrm{Motives}_{S_{\mathrm{Robba},\mathrm{arc},*},\times}\overline{\mathrm{ArcStacks}}_{V,\mathrm{arc}},
\end{align}
take the sub category generated by almost Hodge-Tate sheaves we have:
\begin{align}
\mathrm{Motives}^{\mathrm{aHT},\mathrm{lisse}}_{S_{\mathrm{Robba},\mathrm{arc},*},\times}\overline{\mathrm{ArcStacks}}_{V,\mathrm{arc}}.
\end{align}
We consider the following differential module:
\begin{align}
\partial_W=\mathrm{Descent}_{W^\mathrm{rig}_\mathrm{arc},W,*}(S_{\mathrm{dR},-,\mathbb{N}_c}[\mathrm{log}(t)]).
\end{align}
Varying $W$ we have the sheaf over then the site of all the smooth rigid spaces over $V$:
\begin{align}
\partial=\mathrm{Descent}_{(\cdot/V)^\mathrm{rig}_\mathrm{arc},\cdot/V,*}(S_{\mathrm{dR},-,\mathbb{N}_c}[\mathrm{log}(t)]).
\end{align}
This is from the motivic cohomology theory over the analytic topology for smooth rigid spaces. Here we regard $S_{\mathrm{dR},H,\mathbb{N}_c}[\mathrm{log}(t)]$ also as the motivic cohomology $\infty$-sheaf of ring over the big site of all the arc stacks over $V$. Then we have a well-defined functor:
\begin{align}
\mathrm{Motives}^{\mathrm{aHT},\mathrm{lisse}}_{S_{\mathrm{Robba},\mathrm{arc},*},\times}\overline{\mathrm{ArcStacks}}_{V,\mathrm{arc}}\longrightarrow \mathrm{D}_{\mathrm{inductive},\mathrm{coherent}}(\mathrm{susp}_\Sigma\mathrm{Descent}_{(\cdot/V)^\mathrm{rig}_\mathrm{arc},(\cdot/V)^\mathrm{rig},*}(S_{\mathrm{dR},-,\mathbb{N}_c}[\mathrm{log}(t)]))
\end{align}
which is symmetrical monoidal $\infty$-categorical functor into the $D$-modules in families in our current setting. Here the target category is defined to be the derived $\infty$-category of inductive systems of the coherent $D$-modules as in \cite{5GR1}, \cite{5GR2}. By consider the corresponding filtration de Rham period rings and the corresponding syntomization de Rham period rings as in the following we have two another versions of motivic Riemann Hilbert correspondences.
\end{definition}

\begin{remark}
Because we are considering the almost Hodge-Tate sheaves, the resulting image of this functor should be those coherent $D$-modules.
\end{remark}

We have from \cite{5GR1}, \cite{5GR2} the 6-functor formalism for the target of:
\begin{align}
\mathrm{D}_{\mathrm{aHT},\mathrm{lisse},V}(S_{\mathrm{Robba},\mathrm{perfectoid},*})\longrightarrow \mathrm{D}_{\mathrm{inductive},\mathrm{coherent}}(\mathrm{Descent}_{V^\mathrm{rig}_\text{pro\'etale},V^\mathrm{rig},*}(S_{\mathrm{dR},-,\mathbb{N}_c}[\mathrm{log}(t)])).
\end{align}

\indent We now consider more generalized context as in \cite{5BS1}, \cite{5BLC} and \cite{5D} where we consider the corresponding de Rham prismatization. In our current pro-\'etale setting we have the following generalized definition of de Rham period rings with Nygaard filtration and syntomification.

\begin{definition} \label{definition52}
Now we define certain filtration almost Hodge-Tate Galois representation of Galois group of $H$. Let $L$ be such representation which is free of finite rank $r$, we consider the corresponding de Rham period ring:
\begin{align}
S_{\mathrm{Nygaard},\mathrm{dR},H,\mathbb{N}_c}
\end{align}
where we consider the corresponding evaluation of the filtration Nygaard prismatization over $\prod_{n\in \mathbb{N}_c}\widehat{\overline{H}}_n$. This ring is filtred and gives rise to the Hodge-Tate ring by taking the associated graded ring:
\begin{align}
S_{\mathrm{Nygaard},\mathrm{HT},H,\mathbb{N}_c}:=gr(S_{\mathrm{Nygaard},\mathrm{dR},H,\mathbb{N}_c}).
\end{align}
After \cite{5F2}, \cite{5BLA}, \cite{5BLB} we say $L$ is almost Hodge-Tate in families if the rank preserves after we take the following invariance:
\begin{align}
(L\otimes gr(S_{\mathrm{Nygaard},\mathrm{dR},H,\mathbb{N}_c}[\mathrm{log}(t)]))^{G_H},
\end{align}
here $t$ is the product of all the Fontaine's '2$\pi$i' element for each $n\in \mathbb{N}_c$. One then has the syntomification version of the definition by considering syntomification prismatization.
\end{definition}

\begin{definition}
Now we consider a smooth algebraic variety $V$ over $H$. We consider the rigid generic fiber $V^\mathrm{rig}$ over $H$, over which we consider the pro-\'etale topology as in \cite{5L}, \cite{5S1}. We call any local system $M$ under this topology filtration almost Hodge-Tate if each localization of $M_x$ at any point $x\in V^\mathrm{rig}$ is filtration almost Hodge-Tate representation as in \cref{definition52}. Then we generalize the definition to the corresponding locally finite free Frobenius sheaves over Robba sheaves in our current family setting. If $M$ is such a module then after \cite{5KLA}, \cite{5KLB} we say $L$ is filtration almost Hodge-Tate in families if the whole module preserves after we take the following invariance in the sense of taking pro-\'etale descent:
\begin{align}
\mathrm{Descent}_{V^\mathrm{rig}_\text{pro\'etale},V^\mathrm{rig},*}(M_\infty\otimes gr(S_{\mathrm{Nygaard},\mathrm{dR},-,\mathbb{N}_c}[\mathrm{log}(t)])),
\end{align}
here $t$ is the product of all the Fontaine's '2$\pi$i' element for each $n\in \mathbb{N}_c$, after we take further base change to the whole filtration Hodge-Tate period sheaf in families in our current setting. Here $M_\infty$ is the section over the Robba ring $\varprojlim_{r\rightarrow \infty}S_{\mathrm{Robba},r,\mathrm{perfectoid},P}$. One then has the definition in the syntomification situation by using the syntomification prismatization.
\end{definition}

\begin{theorem}
Let $\mathrm{D}_{\mathrm{aHT},\mathrm{lisse},V}(S_{\mathrm{Robba},\mathrm{perfectoid},*})$ be the derived $\infty$-category generated by locally finite free almost Hodge-Tate Frobenius sheaves over Robba rings as in our current family situation. We then have the following differential module:
\begin{align}
\partial=\mathrm{Descent}_{V^\mathrm{rig}_\text{pro\'etale},V^\mathrm{rig},*}(S_{\sharp,\mathrm{dR},-,\mathbb{N}_c}[\mathrm{log}(t)]).
\end{align}
Then we have a well-defined functor:
\begin{align}
\mathrm{D}_{\mathrm{aHT},\mathrm{lisse},V}(S_{\mathrm{Robba},\mathrm{perfectoid},*})\longrightarrow \mathrm{D}_{\mathrm{inductive},\mathrm{coherent}}(\mathrm{Descent}_{V^\mathrm{rig}_\text{pro\'etale},V^\mathrm{rig},*}(S_{\sharp,\mathrm{dR},-,\mathbb{N}_c}[\mathrm{log}(t)]))
\end{align}
which is symmetrical monoidal $\infty$-categorical functor into the $D$-modules in families in our current setting. Here the target category is defined to be the derived $\infty$-category of inductive systems of the coherent $D$-modules as in \cite{5GR1}, \cite{5GR2}. Here $\sharp=\mathrm{Nygaard},\mathrm{syntomification}$.
\end{theorem}

\begin{proof}
As in the usual situation in \cite{5L}, \cite{5BLA}, \cite{5BLB} the functor is defined to be just the derived homomorphism from the structure sheaf of the rigid space $V^\mathrm{rig}$ to:
\begin{align}
\mathrm{Descent}_{V^\mathrm{rig}_\text{pro\'etale},V^\mathrm{rig},*}(M_\infty\otimes gr(S_{\sharp,\mathrm{dR},-,\mathbb{N}_c}[\mathrm{log}(t)])).
\end{align}
Here $\sharp=\mathrm{Nygaard},\mathrm{syntomification}$.
\end{proof}

We have from \cite{5GR1}, \cite{5GR2} the 6-functor formalism for the target of:
\begin{align}
\mathrm{D}_{\mathrm{aHT},\mathrm{lisse},V}(S_{\mathrm{Robba},\mathrm{perfectoid},*})\longrightarrow \mathrm{D}_{\mathrm{inductive},\mathrm{coherent}}(\mathrm{Descent}_{V^\mathrm{rig}_\text{pro\'etale},V^\mathrm{rig},*}(S_{\sharp,\mathrm{dR},-,\mathbb{N}_c}[\mathrm{log}(t)])).
\end{align}

\begin{remark}
We have the $D$-module theory in positive characteristic. In the situation where $F$ is a finite field for instance, we then have the resulting $D$-modules.
\end{remark}

\subsection{$\infty$-Category over Differential Forms}

Our consideration on the $\infty$-categorical $D$-module is largely inspired by \cite{5GR1} and \cite{5GR2}. We consider essentially the 6-functor formalism in \cite{5GR1} and \cite{5GR2} on the inductive category of the coherent $D$-modules where in our setting are the following three derived $\infty$-categories over the rigid generic fiber $V^\mathrm{rig}$:
\begin{align}
\mathrm{D}_{\mathrm{inductive},\mathrm{coherent}}(\mathrm{Descent}_{V^\mathrm{rig}_\text{pro\'etale},V^\mathrm{rig},*}(S_{\sharp,\mathrm{dR},-,\mathbb{N}_c}[\mathrm{log}(t)]))
\end{align}
where $\sharp=\emptyset,\mathrm{Nygaard},\mathrm{syntomification}$. They carry different $D$-module structures due to the fact that they come from different prismatizations. In fact this can be seen by using the corresponding functor from the corresponding prismatizations. We first recall from the constructions in family of the de Rham primatizations, we have the de Rham prismatization over $\mathbb{N}_c$:
\begin{align}
\underline{\mathfrak{X}}_{\mathrm{prism},\mathrm{dR},\mathbb{N}_c,\sharp}
\end{align}
where $\sharp=\emptyset,\mathrm{Nygaard},\mathrm{syntomification}$. Following \cite{5GR1}, \cite{5GR2} we consider the corresponding inductive coherent sheaves over them in order to gain automatically the 6-functor formalism from \cite{5GR1} and \cite{5GR2}:
\begin{align}
D_{\mathrm{inductive},\mathrm{coherent}}(\underline{\mathfrak{X}}_{\mathrm{prism},\mathrm{dR},\mathbb{N}_c,\sharp}).
\end{align}
We now add $\mathrm{log}(t)$ to them so we just have the following:
\begin{align}
D_{\mathrm{inductive},\mathrm{coherent}}(\underline{\mathfrak{X}}_{\mathrm{prism},\mathrm{dR},\mathbb{N}_c,\sharp}[\mathrm{log}(t)]).
\end{align}
These three de Rham prismatizations over $\mathbb{N}_c$ actually turn out to be (after we adding $\mathrm{log}(t)$) the corresponding de Rham period sheaves over $\mathbb{N}_c$ we are considering over the pro-\'etale site, then after we take the corresponding pushforward along the structure morphism we get the categories of $D$-modules over $\mathbb{N}_c$. They can carry certain filtrations, however in the almost setting the filtration is parametrized by $\mathbb{Q}$. Therefore we now consider $\mathbb{Q}$-filteration such as the Hodge-Pink ones as in \cite{5GL1}, \cite{5GL2}, \cite{5HK}. In our family setting, we combine the consideration in \cite{5GL1}, \cite{5GL2} and \cite{5HK} to define the corresponding $\mathbb{Q}$-filtration on the de Rham primatizations and the derived categories.

\begin{definition}
Over pro-\'etale site, Hodge-Pink $\mathbb{Q}$-filtration $\mathrm{HPink}^\mathbb{Q}$ over $\mathbb{N}_c$ over the $D$-modules in 
\begin{align}
D_{\mathrm{inductive},\mathrm{coherent}}(\underline{\mathfrak{X}}_{\mathrm{prism},\mathrm{dR},\mathbb{N}_c,\sharp}[\mathrm{log}(t)]),
\end{align}
are defined to be the corresponding filtration parametrized by those Hodge-Pink weights in \cite{5GL1}, \cite{5GL2}, \cite{5HK} for each fiber over $\mathbb{N}_c$, where we take the corresponding product over the family. Therefore now the Hodge-Pink filtration in such de Rham setting is a family version, where for each fiber over $\mathbb{N}_c$ we then end up with the corresponding de Rham version in \cite{5GL1}, \cite{5GL2}, \cite{5HK} respectively.
\end{definition} 

\begin{theorem}
Let $\mathrm{D}_{\mathrm{aHT},\mathrm{lisse},V}(S_{\mathrm{Robba},\mathrm{perfectoid},*})$ be the derived $\infty$-category generated by locally finite free almost Hodge-Tate Frobenius sheaves over Robba rings as in our current family situation. We then have the following differential module:
\begin{align}
\partial=\mathrm{Descent}_{V^\mathrm{rig}_\text{pro\'etale},V^\mathrm{rig},*}(S_{\sharp,\mathrm{dR},-,\mathbb{N}_c}[\mathrm{log}(t)]).
\end{align}
Then we have a well-defined functor:
\begin{align}
\mathrm{D}_{\mathrm{aHT},\mathrm{lisse},V}(S_{\mathrm{Robba},\mathrm{perfectoid},*})\longrightarrow \mathrm{D}^{\mathrm{HPink}^\mathbb{Q}}_{\mathrm{inductive},\mathrm{coherent}}(\mathrm{Descent}_{V^\mathrm{rig}_\text{pro\'etale},V^\mathrm{rig},*}(S_{\sharp,\mathrm{dR},-,\mathbb{N}_c}[\mathrm{log}(t)]))
\end{align}
which is symmetrical monoidal $\infty$-categorical functor into the $D$-modules in families in our current setting, carrying the $\mathbb{Q}$-Hodge-Pink filtration. Here the target category is defined to be the derived $\infty$-category of inductive systems of the coherent $D$-modules as in \cite{5GR1}, \cite{5GR2}. Here $\sharp=\mathrm{Nygaard},\mathrm{syntomification}$.
\end{theorem}

\begin{remark}
This theorem generalizes consideration in \cite{5HK} to a family version on the de Rham representations in function field situation. The generalization along crystalline representation situation is also covered in this theorem in the following sense: as in \cite{5BLA}, \cite{5BLB} syntomification prismatization though itself motivic is a consideration completely for crystalline representations. Almost de Rham covers the almost crystalline sheaves in our setting over $\mathbb{N}_c$. 
\end{remark}

\indent We now consider the corresponding generalization of \cite{5HK} in our current setting, i.e. we classify the representations of the Galois group of $H$ in families. As in \cite{5HK} since we are consider function field and $p$-adic local fields together, we need to at least be careful enough for the crystalline representations. But we consider the following category for the cystalline representations as in \cite{5HK} where primatization plays a key role in such generalized setting.

\begin{definition}
\textit{Generalized almost crystalline representations} of the Galois group of $H$ are defined to be Frobenius-coherent sheaves over the prismatization in families
\begin{align}
\underline{\mathfrak{X}}_{\mathrm{prism},\mathrm{spf}I_H}[\log(t)]
\end{align}
over $I_H$ by regarding $\mathrm{spec}I_H$ as a formal scheme. The Frobenius operator over $\mathbb{N}_c$ is the product ones for each fiber as in \cite{5HK} after inverting $x=(x_n)$, note that at $\infty$ in our family setting we need to consider the structure in \cite{5HK} for the function field, where we are going to invert $x_n-y_n$ for some other variable $y_n$. We assume the sheaves are locally finite free.
\end{definition}

\begin{theorem}
The family version of the main theorem in chapter 6 of \cite{5GL2} holds true in our current setting by combining function field and $p$-adic local field situation together. Namely let $\mathrm{BK}_{\mathbb{N}_c}$ be the corresponding Breuil-Kisin prism in family, and use the variable $x$ to form the divided power closure $\overline{\mathrm{BK}_{\mathbb{N}_c}}$, then the categories of Frobenius modules over these two different rings are equivalent. Fix the amplitude below $\sharp-2$ where $\sharp$ is the cardinality of the fixed residue field in our setting i.e. $F$.
\end{theorem}

\begin{proof}
The conjecture is true for each fiber in our setting for Breuil-Kisin modules in families over $\mathbb{N}_c$. Note that in this setting the Hodge-Pink structures in family are defined different from the de Rham setting in \cite{5HK}. Now we have a family version at each fiber the Genestier-Lafforgue correspondence for weakly-admissible and admissible sheaves through Hodge-Pink structural construction as in \cite[Chapitre 6]{5GL2}. Then by the corresponding over $\mathbb{N}_c$ from $\infty$ to some finite $n$ modulo some power of $x$ one can put these Genestier-Lafforgue style correspondence into a profinite family, then the projective limit will produce the corresponding equivalence on the categories.
\end{proof}

\begin{definition}
\textit{Generalized almost de Rham representations} of the Galois group of $H$ are defined to be Frobenius-coherent sheaves over the prismatization in families
\begin{align}
\underline{\mathfrak{X}}_{\mathrm{prism},\mathrm{dR},\mathrm{spf}I_H}[\log(t)]
\end{align}
over $I_H$ by regarding $\mathrm{spec}I_H$ as a formal scheme. The Frobenius operator over $\mathbb{N}_c$ is the product ones for each fiber as in \cite{5HK} after inverting $x=(x_n)$, note that at $\infty$ in our family setting we need to consider the structure in \cite{5HK} for the function field, where we are going to invert $x_n-y_n$ for some other variable $y_n$. We assume the sheaves are locally finite free. Over the pro\'etale site of $\mathrm{Spa}H$, this turns out to be the category of almost de Rham sheaves as in our family setting.
\end{definition}

\subsection{Differential Modules in Families}

To have a genuine Riemann-Hilbert correspondence in families, we need to find the correct category of differential modules to consider, by regarding the differential forms as module over differential operators. However since we are working in a mixed manner on the local fields, function field may make this problem much harder. For instance we consider the differential modules throughout the whole family the one generated by divided powers as in \cite{5A}, which will be some sort of the only choice for us since we have the function field situation in our families. Away from $\infty$ the thing is fine since we just get the categories in the style of \cite{5BLA}, \cite{5BLB}.

\begin{definition}
Taking product throughout the whole families over $H$ of the differential operators with divided powers, for a general smooth rigid analytic space, we define the category of $\infty$-category of differential modules to be the $\infty$-category of coherent modules over the ring of differential operators defined in families. We now use the notation:
\begin{align}
D_{\mathrm{differential},\mathbb{N}_c}(*)
\end{align}
to denote this category for any smooth rigid space $*$ in families over $\mathbb{N}_c$.
\end{definition}

\begin{remark}
At $\infty$ we are not taking a mixed-characteristic rigid analytic covering in the theory of arithmetic $D$-modules. That is to say the resulting sheaf of rings of differential operators is of positive characteristic.
\end{remark}

\begin{definition}
For the derived $\infty$-category of differential modules in our setting:
\begin{align}
D_{\mathrm{differential},\mathbb{N}_c}(V)
\end{align}
we have the Koszul duality in families, i.e. there is a functor into this $\infty$-category from the corresponding $\infty$-category of modules over differential forms as before. To be more precise for each fiber over $\mathbb{N}_c$ we take the Koszul duality as in \cite{5BLA}, \cite{5BLB}. Here we consider the motivic setting. To be more precise for each fiber we have:
\begin{align}
\mathrm{D}_{\mathrm{inductive},\mathrm{coherent}}(\mathrm{Descent}_{(\cdot/V)^\mathrm{rig}_\mathrm{arc},(\cdot/V)^\mathrm{rig},*}(S_{\mathrm{dR},-,\mathbb{N}_c}[\mathrm{log}(t)]))(n)\rightarrow D_{\mathrm{differential},n}(\cdot/V). 
\end{align}
The category on the left of this map is the one over the analytic Grothendieck site of $V$ using all smooth rigid analytic varieties. Then we can derive the version in families:
\begin{align}
D_{\mathrm{differential},\mathbb{N}_c}(\cdot/V)\leftarrow \mathrm{D}_{\mathrm{inductive},\mathrm{coherent}}(\mathrm{Descent}_{(\cdot/V)^\mathrm{rig}_\mathrm{arc},(\cdot/V)^\mathrm{rig},*}(S_{\mathrm{dR},-,\mathbb{N}_c}[\mathrm{log}(t)])). 
\end{align}
\end{definition}

\begin{theorem}\label{theorem59}
Let $\mathrm{D}_{\mathrm{aHT},\mathrm{lisse},V}(S_{\mathrm{Robba},\mathrm{arc},*})$ be the derived $\infty$-category generated by locally finite free almost Hodge-Tate Frobenius sheaves over Robba rings as in our current family situation. Here we consider the category 
\begin{align}
\mathrm{Motives}_{S_{\mathrm{Robba},\mathrm{arc},*},\times}\overline{\mathrm{ArcStacks}}_{V,\mathrm{arc}}.
\end{align}
Where $\overline{\mathrm{ArcStacks}}$ denotes the restriction to smooth rigid spaces over $V$ from all the arc stacks over $V$. Then for each covering arc stack $W$ over $V$ we consider the corresponding almost Hodge-Tate sheaves over the Robba sheaves over $W$:
\begin{align}
\mathrm{D}_{\mathrm{aHT},\mathrm{lisse},W}(S_{\mathrm{Robba},\mathrm{arc},W})
\end{align}
which can be arranged to the motivic $\infty$-category over the big site over $V$:
\begin{align}
\mathrm{Motives}_{S_{\mathrm{Robba},\mathrm{arc},*},\times}\overline{\mathrm{ArcStacks}}_{V,\mathrm{arc}},
\end{align}
take the sub category generated by almost Hodge-Tate sheaves we have:
\begin{align}
\mathrm{Motives}^{\mathrm{aHT},\mathrm{lisse}}_{S_{\mathrm{Robba},\mathrm{arc},*},\times}\overline{\mathrm{ArcStacks}}_{V,\mathrm{arc}}.
\end{align}
We then have the following differential module:
\begin{align}
\partial_W=\mathrm{Descent}_{W^\mathrm{rig}_\mathrm{arc},W,*}(S_{\mathrm{dR},-,\mathbb{N}_c}[\mathrm{log}(t)]).
\end{align}
Varying $W$ we have the sheaf over then the site of all the smooth rigid spaces over $V$:
\begin{align}
\partial=\mathrm{Descent}_{(\cdot/V)^\mathrm{rig}_\mathrm{arc},\cdot/V,*}(S_{\mathrm{dR},-,\mathbb{N}_c}[\mathrm{log}(t)]).
\end{align}
This is from the motivic cohomology theory over the analytic topology for smooth rigid spaces. Here we regard $S_{\mathrm{dR},H,\mathbb{N}_c}[\mathrm{log}(t)]$ also as the motivic cohomology $\infty$-sheaf of ring over the big site of all the arc stacks over $V$. Then we have a well-defined functor:
\begin{align}
\mathrm{Motives}^{\mathrm{aHT},\mathrm{lisse}}_{S_{\mathrm{Robba},\mathrm{arc},*},\times}\overline{\mathrm{ArcStacks}}_{V,\mathrm{arc}}\longrightarrow \mathrm{D}_{\mathrm{inductive},\mathrm{coherent}}(\mathrm{susp}_\Sigma\mathrm{Descent}_{(\cdot/V)^\mathrm{rig}_\mathrm{arc},(\cdot/V)^\mathrm{rig},*}(S_{\mathrm{dR},-,\mathbb{N}_c}[\mathrm{log}(t)]))
\end{align}
which is symmetrical monoidal $\infty$-categorical functor into the $D$-modules in families in our current setting. Here the target category is defined to be the derived $\infty$-category of inductive systems of the coherent $D$-modules as in \cite{5GR1}, \cite{5GR2}. By consider the corresponding filtration de Rham period rings and the corresponding syntomization de Rham period rings as in the following we have two another versions of motivic Riemann Hilbert correspondences. One then have the map into the motivic $\infty$-category of differential modules in families:
\begin{align}
\mathrm{Motives}^{\mathrm{aHT},\mathrm{lisse}}_{S_{\mathrm{Robba},\mathrm{arc},*},\times}\overline{\mathrm{ArcStacks}}_{V,\mathrm{arc}}&\longrightarrow \mathrm{D}_{\mathrm{inductive},\mathrm{coherent}}(\mathrm{susp}_\Sigma\mathrm{Descent}_{(\cdot/V)^\mathrm{rig}_\mathrm{arc},(\cdot/V)^\mathrm{rig},*}(S_{\mathrm{dR},-,\mathbb{N}_c}[\mathrm{log}(t)]))\\
&\longrightarrow D_{\mathrm{differential},\mathbb{N}_c}(\cdot/V).
\end{align}
\end{theorem}

\indent Immediately after \cite{5S5}, \cite{5Ga} we have the following theorems:

\begin{theorem}
Throughout this theorem we assume we are in the family setting in \cref{theorem59}. The symmetrical monoidal $\infty$-categorical motivic Riemann-Hilbert correspondence functor in the previous theorem has a gestalten version by taking the gestalten associated directly on the two categories, which is compatible with the K\"unneth theorem (for degrees higher enough in the gestalten sequences) on the both sides, and is compatible with the $(\infty,\infty)$-categorial gestalten six-functor formalisms (as in \cite{5S5}) on the both sides, when $V$ is varying (we use the result that smooth algebras are dense in commutative algebras after we take the colimits), in the category of all the algebraic stacks over the fpqc site of all the commutative algebra objects in $D(\mathbb{S})$.
\end{theorem}

\begin{proof}
After \cite{5S5} and \cite{5Ga} the K\"unneth theorem is then formal as long as one considers higher enough degrees. Then we taking the limit of the functors over the colimit of smooth varieties to define the functors over stacks. Then when $V$ is varying in the specified categories, the resulting structure sheaves on the both side preserves the countably-presentable maps as those $!$-able maps for gestalten six-functor formalism, since they are just Robba rings and differential forms, or noncommutative differential operators. Then we have the well-defined gestalten $(\infty,\infty)$-categorical six functor formalism as in \cite{5S5}.
\end{proof}

\begin{remark}
The functor is compatible with the Poincar\'e duality in the following sense, if we unveil the definition. On the left it is perfect almost Hodge-Tate Robba sheaves, the duality is essentially the pseudocoherent Poincar\'e duality for Fargues-Fontaine gestalten, namely in current family version we have the Poincar\'e duality for gestalten associated to pseudocoherent complexes over Farges-Fontaine stacks, for different base algebraic stack $X$\footnote{One associates to each derived algebra stack $X$ in the current setting the Fargues-Fontaine curve $\mathrm{FF}_{X^\mathrm{analytification}}$ by taking solid analytification of $X$ and then regarding $X^\mathrm{analytification}$ as a derived arc-stack.}:
\begin{align}
(\mathrm{End}_{\mathrm{PsedoCoh}(\mathrm{FF}_{X^\mathrm{analytification}})}(1),\mathrm{PsedoCoh}(\mathrm{FF}_{X^\mathrm{analytification}}), \mathrm{Mod}_{\mathrm{PsedoCoh}(\mathrm{FF}_{X^\mathrm{analytification}})}(2Pr^L),...).
\end{align}
On the right it is just relative differential form algebra, the Poincar\'e duality is just the duality in that setting. 
\end{remark}

\newpage
\section{Abstract Universal Motivic Gestalten Formalism}

\subsection{Abstract Universal Motivic Gestalten Formalism}

\noindent We now amplify some generalization closely after \cite{5S4}, \cite{5S5}. Since we have a new modern contemporary theory of motives after \cite{5S4}, we can now give the following abstract formalism:

\begin{definition}(\text{Abstract Universal Motivic Gestalten Formalism})
We now systematically working over the sphere spectrum $\mathbb{S}$. For a general $(\infty,1)$-category\footnote{Or one can simply consider symmetrical monoidal ones.} $X_\mathrm{object}$ over the sphere spectrum $\mathbb{S}$ (for instance a big Grothendieck site of the sphere spectrum $\mathbb{S}$). For such $X_\mathrm{object}$ we consider the coefficient category:
\begin{align}
\mathrm{CoeffCat}_{\infty,\infty}
\end{align}
as above discussion which is the $(\infty,1)$-category of all the presentable $(\infty,\infty)$-categories, which can be degenerated to $(\infty,1)$-groupoids. Now we consider the following category of all the functors from $X$ into the corresponding categories:
\begin{align}
\mathrm{CoeffCat}_{\infty,\infty}.
\end{align}
We denote this category as:
\begin{align}
\mathrm{Functor}^U(X_\mathrm{object},\mathrm{CoeffCat}_{\infty,\infty})
\end{align}
since we are working over the sphere spectrum, we have then a promotion lifting to a bigraded motivic spectrum $\mathrm{S}^{1,0}$, in such a way we allow a twice localization (usually the second localization is the loop space functor stabilization with respect to the Tate element)\footnote{We regard this as some $(\infty,2)$-category.}:
\begin{align}
\mathrm{Fun}^U:=\mathrm{Functor}^U(X_\mathrm{object},\mathrm{CoeffCat}_{\infty,\infty})^{\mathrm{localization}, \mathrm{localization}}.
\end{align}
Then for general $X_\mathrm{object}$, we have the following universal motivic gestalten\footnote{Since we have the shifting property of gestalten as in \cite{5S5} we do not have to worry about which degree $\mathrm{Fun}^U$ has since we can always stabilize towards the negative integer direction via $\mathrm{End}_{-}(1)$, until we reach $D(\mathbb{S})$.}:
\begin{align}
(...,\mathrm{End}_{\mathrm{End}_{\mathrm{Fun}^U}(1)}(1),\mathrm{End}_{\mathrm{Fun}^U}(1),\mathrm{Fun}^U,\mathrm{Mod}_{\mathrm{Fun}^U}(3Pr^L),\mathrm{Mod}_{\mathrm{Mod}_{\mathrm{Fun}^U}(3Pr^L)}(4Pr^L),...).
\end{align}
Now after \cite{5S5} we have the corresponding abstract 6-functors formalism across these universal gelstalten, for $!$-able maps generated from countably presented ones. In the current formalism we consider the corresponding functoriality over $X:=\mathrm{X}_\mathrm{object}$, where a general form of such map takes the form of a countable limit:
\begin{align}
\lim_k X'_k \rightarrow X
\end{align}
where $X'_k$ in the opposite category is countably presented over $X$, in the category of all presentable $(\infty,1)$-categories. For this we consider all the relavant such categories $X$ which are assumed to be simply $\infty$-category we denote them as:
\begin{align}
\mathrm{Mon}_\infty.
\end{align}
And we denote the corresponding six functors as:
\begin{align}
f_\flat, f^\sharp, g_\flat, g^\sharp, \boxtimes, \mathrm{Hom}.  
\end{align}
However over the higher category level the difference between $f,g$ is disappearing as in \cite{5S5}. 
\end{definition}

The second formalism goes along some generalization along Lurie's work on ultracategories in \cite{5Lu3}, where we have the following extension of the previous formalism in a parallel manner:

\begin{definition}(\text{Motivicalization Ultracategoricalization Gestalten Formalism and the Integration Theory})
We now systematically working over the sphere spectrum $\mathbb{S}$. For a general $(\infty,1)$-ultracategory $X_\mathrm{object}$ over the sphere spectrum $\mathbb{S}$ (for instance a big Grothendieck site of the sphere spectrum $\mathbb{S}$, a big Grothendieck topos, or just a pretopos). Note that any such ultracategory can be regarded as a $\infty$-sheaf values in $\infty$-categories over the site of all the compact Hausdorff spaces:
\begin{align}
\mathrm{CHSpace}.
\end{align}
For such $X_\mathrm{object}$ we consider the coefficient ultracategory:
\begin{align}
\mathrm{CoeffCat}_{\infty,\infty}
\end{align}
as above discussion which is the $(\infty,1)$-ultracategory of all the presentable $(\infty,\infty)$-ultracategories, which can be degenerated to $(\infty,1)$-ultragroupoids. Now we consider the following ultracategory of all the ultrafunctors from $X$ into the corresponding ultracategories:
\begin{align}
\mathrm{CoeffCat}_{\infty,\infty}.
\end{align}
We denote this ultracategory as:
\begin{align}
\mathrm{UltraFunctor}^U(X_\mathrm{object},\mathrm{CoeffCat}_{\infty,\infty})
\end{align}
since we are working over the sphere spectrum, we have then a promotion lifting to a bigraded motivic spectrum $\mathrm{S}^{1,0}$, in such a way we allow a twice localization (usually the second localization is the loop space functor stabilization with respect to the Tate element)\footnote{We regard this as some $(\infty,2)$-category.}:
\begin{align}
\mathrm{UltraFun}^U:=\mathrm{UltraFunctor}^U(X_\mathrm{object},\mathrm{CoeffCat}_{\infty,\infty})^{\mathrm{localization}, \mathrm{localization}}.
\end{align}
Then for general $X_\mathrm{object}$, we have the following universal motivic gestalten:
\begin{align}
(\mathrm{End}_{\mathrm{End}_{\mathrm{UltraFun}^U}(1)}(1),\mathrm{End}_{\mathrm{UltraFun}^U}(1),\mathrm{UltraFun}^U,\mathrm{Mod}_{\mathrm{UltraFun}^U}(3Pr^L),\mathrm{Mod}_{\mathrm{Mod}_{\mathrm{UltraFun}^U}(3Pr^L)}(4Pr^L),...).
\end{align}
As in \cite{5Lu3} we can consider the ultramodels (i.e. the functors from $X_\mathrm{object}$) for $X_\mathrm{object}$:
\begin{align}
\mathrm{UltraModel}(X_\mathrm{object}).
\end{align}
In such a way we also have the corresponding ultracategory of ultrafunctors from this ultracategory to the coefficient ultracategories via further motivicalization through double localizations:
\begin{align}
\mathrm{UltraFun}_M^U:=\mathrm{UltraFunctor}^U(\mathrm{UltraModel}(X_\mathrm{object}),\mathrm{CoeffCat}_{\infty,\infty})^{\mathrm{localization}, \mathrm{localization}}.
\end{align}
Therefore after this we then have an associated ultragestalt:
\begin{align}
(\mathrm{End}_{\mathrm{End}_{\mathrm{UltraFun}_M^U}(1)}(1),\mathrm{End}_{\mathrm{UltraFun}_M^U}(1),\mathrm{UltraFun}_M^U,\mathrm{Mod}_{\mathrm{UltraFun}_M^U}(3Pr^L),\mathrm{Mod}_{\mathrm{Mod}_{\mathrm{UltraFun}_M^U}(3Pr^L)}(4Pr^L),...).
\end{align}
Now after \cite{5S5} we have the corresponding abstract 6-functors formalism across these universal gelstalten, for $!$-able maps generated from countably presented ones. In the current formalism we consider the corresponding functoriality over $X:=\mathrm{X}_\mathrm{object}$, where a general form of such map takes the form of a countable limit:
\begin{align}
\lim_k X'_k \rightarrow X
\end{align}
where $X'_k$ in the opposite ultracategory is countably presented over $X$, in the ultracategory of all presentable $(\infty,1)$-ultracategories. For this we consider all the relavant such ultracategories $X$ which are assumed to be simply $\infty$-ultracategory we denote them as:
\begin{align}
\mathrm{Mon}_\infty.
\end{align}
And we denote the corresponding six functors as:
\begin{align}
f_\flat, f^\sharp, g_\flat, g^\sharp, \boxtimes, \mathrm{Hom}.  
\end{align}
However over the higher category level the difference between $f,g$ is disappearing as in \cite{5S5}. Recall that in this setting the corresponding ultracategories are defined to be carrying integration theory which satisfies \textit{ultracategorical Fubini natural transformation theorem}:
\begin{align}
&\int_S (\mathrm{End}_{\mathrm{End}_{\mathrm{UltraFun}^U}(1)}(1),\mathrm{End}_{\mathrm{UltraFun}^U}(1),\mathrm{UltraFun}^U,\\
&\mathrm{Mod}_{\mathrm{UltraFun}^U}(3Pr^L),\mathrm{Mod}_{\mathrm{Mod}_{\mathrm{UltraFun}^U}(3Pr^L)}(4Pr^L),...)_s d(\int_T \mu_t) =\\ 
&\int_T\int_S (\mathrm{End}_{\mathrm{End}_{\mathrm{UltraFun}^U}(1)}(1),\mathrm{End}_{\mathrm{UltraFun}^U}(1),\mathrm{UltraFun}^U,\\
&\mathrm{Mod}_{\mathrm{UltraFun}^U}(3Pr^L),\mathrm{Mod}_{\mathrm{Mod}_{\mathrm{UltraFun}^U}(3Pr^L)}(4Pr^L),...)_s d\mu_t,\\
&\int_S(\mathrm{End}_{\mathrm{End}_{\mathrm{UltraFun}_M^U}(1)}(1),\mathrm{End}_{\mathrm{UltraFun}_M^U}(1),\mathrm{UltraFun}_M^U,\\
&\mathrm{Mod}_{\mathrm{UltraFun}_M^U}(3Pr^L),\mathrm{Mod}_{\mathrm{Mod}_{\mathrm{UltraFun}_M^U}(3Pr^L)}(4Pr^L),...)_s d(\int_T \mu_t) =\\ 
&\int_T\int_S (\mathrm{End}_{\mathrm{End}_{\mathrm{UltraFun}_M^U}(1)}(1),\mathrm{End}_{\mathrm{UltraFun}_M^U}(1),\mathrm{UltraFun}_M^U,\\
&\mathrm{Mod}_{\mathrm{UltraFun}_M^U}(3Pr^L),\mathrm{Mod}_{\mathrm{Mod}_{\mathrm{UltraFun}_M^U}(3Pr^L)}(4Pr^L),...)_s d\mu_t.
\end{align}
The measure is $(0,1)$-valued and the corresponding $S,T$ are sets and the $\sigma$-algebras are defined to be the ones generated by power sets. As in the usual abstract measure theory here the product measures are the tensor product measures.
\end{definition}

\begin{remark}
Let us basically explain our notation. We take the ultraproducts over ultracategorical motivic gestalten, which \textit{really} means we consider a huge ultracategory 
\begin{align}
\mathrm{UltraMotGestalten}
\end{align}
of all the ultracategorical motivic gestalten. For any set $S$ we can then do the ultraproduct/ultraintegration:
\begin{align}
\int_S d\mu (-) = \varinjlim_{\mu}\prod_{\mathrm{Ultra},\mu} (-) 
\end{align}
as functors on product of $S$-copies of this huge category:
\begin{align}
\prod_S \mathrm{UltraMotGestalen}.
\end{align}
One then put family of measure on $S$, indexed by another set $S'$, where we have a measure transformation:
\begin{align}
d\mu = d(\int_T \mu_t).
\end{align}
This then satisfies Fubini transformation:
\begin{align}
\int_S - d(\int_T \mu_t) \rightarrow \int_T\int_S - d\mu_t.
\end{align}
Where we will call it Fubini natural equivalence when it is natural equivalence: i.e. the categorical Fubini theorem.
\end{remark}

\begin{definition}
A ultragestalt is defined to be an $(\infty,\infty)$-sheaf valued in $(\infty,\infty)$-ring gestalten over the category of compact Hausdorff spaces. A noncommutative ultragestalt is defined to be an $(\infty,\infty)$-sheaf valued in $(\infty,\infty)$-ring noncommutative gestalten over the category of compact Hausdorff spaces. We have the parallel definition in condensed mathematics: a condensed gestalt is defined to be an $(\infty,\infty)$-sheaf valued in $(\infty,\infty)$-ring gestalten over the category of totally disconnected compact Hausdorff spaces. A noncommutative condensed gestalt is defined to be an $(\infty,\infty)$-sheaf valued in $(\infty,\infty)$-ring noncommutative gestalten over the category of totally disconnected compact Hausdorff spaces. we further endow these $(\infty,\infty)$-stacks with the ultrastructure by using the integration, together with Fubini transformations, and with the requirements as in \cite[Definition 1.3.1, (A), (B) (C)]{5Lu3} namely we have the Fubini equivalence with respect to twice integration over integration set, Fubini equivalence with point measure, Fubini equivalence with respect to double twice integration over the integration set of the integration set. 
\end{definition}

\begin{remark}
What happens in the current setting is that we can consider the situation where $X_\mathrm{object}$ is the category of all the Grothendieck sites. In such a way let the base be $\mathbb{S}$, we have then functors existing between Grothendieck sites, which can be applied in the situation where we consider the Fargues-Fontaine gestalten over different sites: la topologie $v$, la topologie fid\`limente plate et quasicompacte, la topologie fid\`elimente plate et present\'e fini, la topologie arc, la topologie analytique,..., all over $\mathbb{S}$, ie. we have a corresponding six-functor formalism for all Grothendieck sites. For instance now have an algebraic stack $X$ fibered over all the $\mathbb{E}_\infty$-ring objects in $D(\mathbb{S})$, then consider the \textit{category of all the Grothendieck sites and topologies} over $X$ (namely we allow arc topology for instance which can be realized over $X$ through solid analytification), then for two different Grothendieck sites if one has the map which is countably presented between them then we can basically have $f^*,f_*,f_\sharp,f^\sharp,\otimes, \mathrm{Hom}$ and we can take the corresponding K\"unneth style product for higher enough degrees in the gestalten. 
\end{remark}

\begin{remark}
The universal motivic gestalten framework is a very general modern theory of abstract motivicalization (i.e. the process promoting things over $\mathbb{Z}$ to motivic bigraded sphere spectrum), in the context of \cite{5S4}, we consider the fpqc-site of all the schemes over $\mathbb{Z}$, then we consider presheaves over them to form this formalism. For instance if we consider the base to be $\mathbb{S}$, and we consider all other sites of all the derived algebraic stacks with the corresponding simply countably presented topology we have a new site for $\mathbb{S}$, as compared to la fpqc-topologie. The question is do we have a universal site to have a universal category $\mathrm{Fun}^U$, or the question is can we take K\"unneth product of two such categories? However \cite{5S5} solves this in large generality. 
\end{remark}

\subsection{Motivic Higher-Ultracategorical Lurie's Vollst\"andigkeitssatz}

Since the corresponding mathematical framework of universal motivic gestalten carry the corresponding ultracategorical structure, even when we consider the functoriality over the model categories. This is non-trivial, since it fits into the mathematical foundation of Vollst\"andigkeitssatz after \cite{5Lu3}, namely for instance if $X_\mathrm{object}$ is a Grothendieck topos, then we have a directly concrete description via the framework after \cite{5Lu3}. \cite{5Lu3} is a well-defined mathematical foundational framework to work with some general pretopoi via some hyper-general topos theory. Of course this is the deep foundation of modern mathmatics namely many mathematical theories can be built based on topos theory, such as topology, we have many model spaces to consider such as in topology we have \textit{sober} topological spaces, compact Hausdorff spaces, totally disconnected compact Hausdorff spaces. What we are generalizing here after \cite{5Lu3} here have two-folds. The first we consider is ultra mathematics in a way like condensed mathematics, after Clausen-Scholze, where one can also rely on the corresponding framework of \cite{5S5} to study the ultra analytic stacks. Other consideration secondly will be the corresponding motivicalized version of the approach of \cite{5Lu3}, where we do motivicalization to ultracategories over sphere spectrum. 

\begin{question} (Further Ultracategoricalization of Motivic Method)
Suppose we consider some $\infty$-category $C$ of $(\infty,1)$-pretopoi (for instance $\infty$-big Grothendieck topoi) in correspondence to certain topological spaces over the sphere spectrum in algebraic geometry, algebraic topology and point-set topology (for instance \textit{derived totally-disconnected quasi-compact algebraic stacks over $p$-adic Witt vectorial sphere spectrum} through perfectoidization). We promote this to be some motivicalization (with further stabilization) $C_m$ over bigraded motivic sphere spectrum. One asks the motivation questions: do we have a six-functor formalism for motivic pretopoi in $C_m$? Do we have the chance to motivicalize \cite{5Lu3}? Is a motivicalized ultracategoricalization approach after \cite{5Lu3} to algebraic topology (for instance motivic method used in this paper to many problems on general topological spaces and analytic stacks in algebraic geometry and analytic geometry) really helpful?
\end{question}

\begin{remark} \mbox{(Six-Functor Formalism for All Topological Spaces)}
Now suppose we consider the category all the topological spaces over the sphere spectrum, which can be embedded into the category of all the $\infty$-pretopoi. For instance we can consider schemes over $\mathbb{Z}$, adic spaces, $p$-adic Lie groups, rigid analytic spaces. Sometimes we hope to consider phenomenon when we go across these classes, that is to say we hope to have a six-functor formalism for all the topological spaces. Our current framework via universal motivic higher-ultracategorical gestalten will provide a general framework on this, relying on the topos philosophy and aspect from \cite{5Lu3}.
\end{remark}

\begin{remark}
Philosophyically speaking, in the current paper, our main approach overall is motivicalization. We hope to rely on this to tackle problems in algebraic geometry and analytic geometry. Motivic method is always applied with homotopicalization as a lift from the usual sphere spectrum, which is something one can borrows from algebraic topology (i.e. the abstract homotopy theory). However since \cite{5Lu3} has application to algebraic topology via most general Grothendieck topos theory, one may obviously hope vice versa the using of ultracategories and topos theory can enhance the considered motivic method in the current paper, which is a motivation to now relatively systematically consider the framework and mathematical foundation of \cite{5Lu3}.
\end{remark}

\begin{remark}
The questions have already some answer in modern algebraic topology (for instance one studies the motivic Adams spectral sequence in application to homotopy groups of spheres, since this is really a very old and tough question in mathematics) where one uses motivic method to study cohomologies, such as what we did above in this paper. But \cite{5Lu3} obviously may provide potential insight on another motivic method via ultracategoricalization since pretopoi can always be related to ultracategories via the main theorem of \cite{5Lu3}. The main theorem of \cite{5Lu3} gives rise to a certain chance to relate certain topological spaces to the corresponding topoi of sheaves over 'pro-\'etale' categories/sites.
\end{remark}

We first consider Lurie's theorem over $(\infty,\infty)$-ultracategorical level, and then we extend Lurie's theorem to motivic level through twice localizations. The following theorem is just due to Lurie. However we believe the proof is also significant since it uses the 'pro-\'etale topology' as a corresponding modern consideration, which gives a connection to Stone spaces.

\begin{definition}\mbox{(Lurie, \cite{5Lu3})}
We recall some basic constructions from \cite{5Lu3}. The first thing is that we have to lift up the construction over to the sphere spectrum $\mathbb{S}$. This allows us to further lift to motivic bigraded sphere spectrum, in order to do the localization localization to achieve stable categories. Therefore we start from $X_\mathrm{object}$ over $\mathbb{S}$. Now we assume that $X_\mathrm{object}$ is a $(\infty,n)$-pretopos. Then we can define the corresponding coherent $(\infty,n)$-topos over this pretopos:
\begin{align}
\mathrm{Coh}_{\mathrm{Shv}}(X_\mathrm{object})
\end{align}
which can be further motivicalized to the motivic version via the twice localizations from the bigraded motivic sphere spectrum $\mathbb{S}^{1,0}$:
\begin{align}
\mathrm{Coh}_\mathrm{Shv}:=\mathrm{Coh}_{\mathrm{Shv}}(X_\mathrm{object})^\mathrm{localization, localization}.
\end{align}
Then we get the corresponding gestalt via the stablization along two directions\footnote{In such a manner, we do not have to fix a degree and specify what is that degree for this.}:
\begin{align}
(...,\mathrm{End}_{\mathrm{End}_{\mathrm{Coh}_\mathrm{Shv}}(1)}(1),\mathrm{End}_{\mathrm{Coh}_\mathrm{Shv}}(1),\mathrm{Coh}_\mathrm{Shv},\mathrm{Mod}_{\mathrm{Coh}_\mathrm{Shv}}(nPr^L),\mathrm{Mod}_{\mathrm{Mod}_{\mathrm{Coh}_\mathrm{Shv}}(nPr^L)}((n+1)Pr^L),...).
\end{align}
This motivic ultragestalt has the corresponding Fubini theorem:
\begin{align}
&\int_S(...,\mathrm{End}_{\mathrm{End}_{\mathrm{Coh}_\mathrm{Shv}}(1)}(1),\mathrm{End}_{\mathrm{Coh}_\mathrm{Shv}}(1),\mathrm{Coh}_\mathrm{Shv},\mathrm{Mod}_{\mathrm{Coh}_\mathrm{Shv}}(nPr^L),\mathrm{Mod}_{\mathrm{Mod}_{\mathrm{Coh}_\mathrm{Shv}}(nPr^L)}((n+1)Pr^L),...)_{s} d(\int_T d\mu_t) = \\
&\int_T\int_S (...,\mathrm{End}_{\mathrm{End}_{\mathrm{Coh}_\mathrm{Shv}}(1)}(1),\mathrm{End}_{\mathrm{Coh}_\mathrm{Shv}}(1),\mathrm{Coh}_\mathrm{Shv},\mathrm{Mod}_{\mathrm{Coh}_\mathrm{Shv}}(nPr^L),\mathrm{Mod}_{\mathrm{Mod}_{\mathrm{Coh}_\mathrm{Shv}}(nPr^L)}((n+1)Pr^L),...)_{s} d\mu_t.
\end{align}
Going from here we can take the pro-category of $X_\mathrm{object}$, and consider the $(\infty,1)$-category of all the continuous sheaves with values in $(\infty,\infty)$-categories over this site:
\begin{align}
\mathrm{ConShvFunctor}(\mathrm{Pro}(X_\mathrm{object}),\mathrm{CoeffCat}_{\infty,\infty}).
\end{align}
Again we consider the corresponding motivicalization:
\begin{align}
\mathrm{ConShvFun}:=\mathrm{ConShvFunctor}(\mathrm{Pro}(X_\mathrm{object}),\mathrm{CoeffCat}_{\infty,\infty})^\mathrm{localization,localization}
\end{align}
which is then further gives rise to the corresponding gestalt:
\begin{align}
(...,\mathrm{End}_{\mathrm{End}_{\mathrm{ConShvFun}}(1)}(1),\mathrm{End}_{\mathrm{ConShvFun}}(1),\mathrm{ConShvFun},\mathrm{Mod}_{\mathrm{ConShvFun}}(nPr^L),\\
\mathrm{Mod}_{\mathrm{Mod}_{\mathrm{ConShvFun}}(nPr^L)}((n+1)Pr^L),...).
\end{align}
This motivic ultragestalt also has the integration formula:
\begin{align}
&\int_X(...,\mathrm{End}_{\mathrm{End}_{\mathrm{ConShvFun}}(1)}(1),\mathrm{End}_{\mathrm{ConShvFun}}(1),\mathrm{ConShvFun},\mathrm{Mod}_{\mathrm{ConShvFun}}(nPr^L),\\
&\mathrm{Mod}_{\mathrm{Mod}_{\mathrm{ConShvFun}}(nPr^L)}((n+1)Pr^L),...)d(\int_Y \mu_y)=\\
&\int_Y\int_X (...,\mathrm{End}_{\mathrm{End}_{\mathrm{ConShvFun}}(1)}(1),\mathrm{End}_{\mathrm{ConShvFun}}(1),\mathrm{ConShvFun},\mathrm{Mod}_{\mathrm{ConShvFun}}(nPr^L),\\
&\mathrm{Mod}_{\mathrm{Mod}_{\mathrm{ConShvFun}}(nPr^L)}((n+1)Pr^L),...)_{x} d\mu_y.
\end{align}
Then we consider the ringed Stone spaces enriched over the Model category of $X_\mathrm{object}$, which is subcategory of all the ringed compact Housdorff spaces in $\mathrm{CHSpace}$ taking a general form of $(A,L_A)$ where $L_A$ is a functor from $A$ to $\mathrm{UltraModel}(X_\mathrm{object})$. We denote the category of these Stone spaces as:
\begin{align}
\mathrm{RingedStone}^{\mathrm{UltraModel}(X_\mathrm{object})}.
\end{align}  
Then as in \cite{5Lu3} we have the following $(\infty,2)$-ultracategory of all the $(\infty,\infty)$-functors from this category to our chosen coefficient category:
\begin{align}
\mathrm{UltraFunctor}^U(\mathrm{UltraModel}(\mathrm{RingedStone}^{\mathrm{UltraModel}(X_\mathrm{object})},\mathrm{CoeffCat}_{\infty,\infty})^{\mathrm{localization}, \mathrm{localization}}
\end{align}
In such a way we also have the corresponding ultracategory of ultrafunctors from this ultracategory to the coefficient ultracategories via further motivicalization through double localizations:
\begin{align}
\mathrm{UltraFun}_S^U:=\mathrm{UltraFunctor}^U(\mathrm{RingedStone}^{\mathrm{UltraModel}(X_\mathrm{object})},\mathrm{CoeffCat}_{\infty,\infty})^{\mathrm{localization}, \mathrm{localization}}.
\end{align}
Therefore after this we then have an associated ultragestalt:
\begin{align}
(\mathrm{End}_{\mathrm{End}_{\mathrm{UltraFun}_S^U}(1)}(1),\mathrm{End}_{\mathrm{UltraFun}_S^U}(1),\mathrm{UltraFun}_S^U,\mathrm{Mod}_{\mathrm{UltraFun}_S^U}(3Pr^L),\mathrm{Mod}_{\mathrm{Mod}_{\mathrm{UltraFun}_S^U}(3Pr^L)}(4Pr^L),...).
\end{align}
We then have the integration formula:
\begin{align}
&\int_S(\mathrm{End}_{\mathrm{End}_{\mathrm{UltraFun}_S^U}(1)}(1),\mathrm{End}_{\mathrm{UltraFun}_S^U}(1),\mathrm{UltraFun}_S^U,\mathrm{Mod}_{\mathrm{UltraFun}_S^U}(3Pr^L),\\
&\mathrm{Mod}_{\mathrm{Mod}_{\mathrm{UltraFun}_S^U}(3Pr^L)}(4Pr^L),...)_{s}d(\int_T \mu_t)=\\
&\int\int_{T,S}(\mathrm{End}_{\mathrm{End}_{\mathrm{UltraFun}_S^U}(1)}(1),\mathrm{End}_{\mathrm{UltraFun}_S^U}(1),\mathrm{UltraFun}_S^U,\mathrm{Mod}_{\mathrm{UltraFun}_S^U}(3Pr^L),\\
&\mathrm{Mod}_{\mathrm{Mod}_{\mathrm{UltraFun}_S^U}(3Pr^L)}(4Pr^L),...)_{s}d\mu_t.
\end{align}
Then as in \cite{5Lu3} we have the following $(\infty,2)$-ultracategory of all the $(\infty,\infty)$-functors from the model category for $X_\mathrm{object}$ to our chosen coefficient category:
\begin{align}
\mathrm{UltraFunctor}^U(\mathrm{UltraModel}(X_\mathrm{object}),\mathrm{CoeffCat}_{\infty,\infty})^{\mathrm{localization}, \mathrm{localization}}
\end{align}
In such a way we also have the corresponding ultracategory of ultrafunctors from this ultracategory to the coefficient ultracategories via further motivicalization through double localizations:
\begin{align}
\mathrm{UltraFun}_M^U:=\mathrm{UltraFunctor}^U({\mathrm{UltraModel}(X_\mathrm{object})},\mathrm{CoeffCat}_{\infty,\infty})^{\mathrm{localization}, \mathrm{localization}}.
\end{align}
Therefore after this we then have an associated ultragestalt:
\begin{align}
(\mathrm{End}_{\mathrm{End}_{\mathrm{UltraFun}_M^U}(1)}(1),\mathrm{End}_{\mathrm{UltraFun}_M^U}(1),\mathrm{UltraFun}_M^U,\mathrm{Mod}_{\mathrm{UltraFun}_M^U}(3Pr^L),\mathrm{Mod}_{\mathrm{Mod}_{\mathrm{UltraFun}_M^U}(3Pr^L)}(4Pr^L),...).
\end{align}
This satisfies the corresponding Fubini theorem:
\begin{align}
&\int_{S}(\mathrm{End}_{\mathrm{End}_{\mathrm{UltraFun}_M^U}(1)}(1),\mathrm{End}_{\mathrm{UltraFun}_M^U}(1),\mathrm{UltraFun}_M^U,\mathrm{Mod}_{\mathrm{UltraFun}_M^U}(3Pr^L),\\
&\mathrm{Mod}_{\mathrm{Mod}_{\mathrm{UltraFun}_M^U}(3Pr^L)}(4Pr^L),...)_{s}d(\int_T \mu_t)=\\
&\int_T\int_{S}(\mathrm{End}_{\mathrm{End}_{\mathrm{UltraFun}_M^U}(1)}(1),\mathrm{End}_{\mathrm{UltraFun}_M^U}(1),\mathrm{UltraFun}_M^U,\mathrm{Mod}_{\mathrm{UltraFun}_M^U}(3Pr^L),\\
&\mathrm{Mod}_{\mathrm{Mod}_{\mathrm{UltraFun}_M^U}(3Pr^L)}(4Pr^L),...)_{s}d\mu_t.
\end{align}
\end{definition}

\begin{theorem}{\mbox{(Lurie,\cite{5Lu3})}}
We have an equivalence of motivic $(\infty,\infty)$-ultragestalten:
\begin{align}
&(...,\mathrm{End}_{\mathrm{End}_{\mathrm{CohShvFun}}(1)}(1),\mathrm{End}_{\mathrm{ConShvFun}}(1),\mathrm{ConShvFun},\mathrm{Mod}_{\mathrm{ConShvFun}}(nPr^L),\mathrm{Mod}_{\mathrm{Mod}_{\mathrm{ConShvFun}}(nPr^L)}((n+1)Pr^L),...)\\
&\overset{\sim}{\rightarrow}\\
&(\mathrm{End}_{\mathrm{End}_{\mathrm{UltraFun}_M^U}(1)}(1),\mathrm{End}_{\mathrm{UltraFun}_M^U}(1),\mathrm{UltraFun}_M^U,\mathrm{Mod}_{\mathrm{UltraFun}_M^U}(3Pr^L),\mathrm{Mod}_{\mathrm{Mod}_{\mathrm{UltraFun}_M^U}(3Pr^L)}(4Pr^L),...).
\end{align}
Now consider the $!$-able topology (induced from countably presented mappings) over the category of all the $(\infty,1)$-pretopoi, we have that this equivalence is compatible with the abstract gestalten six-functor formalisms on the both sides.
\end{theorem}

\begin{proof}
For the convenience of the readers and for a clear presentation of the main modern essence of Lurie's proof via pro-\'etale topology we present the key steps in Lurie's proof. The original main theorem of \cite{5Lu3} is a theorem on the level of category, but it obviously generalizes to $(\infty,n)$-categories. The twice localization motivicalization is also fine, since we in general take first the homotopy localization with respect to the bigraded sphere spectrum. Then we do second localization via infinite loop space functor $\Omega_\infty$ with respect to this spectrum as well. Therefore we start from $X_\mathrm{object}$ over $\mathbb{S}$. Now we assume that $X_\mathrm{object}$ is a $(\infty,n)$-pretopos. Then we can define the corresponding coherent $(\infty,n)$-topos over this pretopos:
\begin{align}
\mathrm{Coh}_{\mathrm{Shv}}(X_\mathrm{object})
\end{align}
which can be further motivicalized to the motivic version via the twice localizations from the bigraded motivic sphere spectrum $\mathbb{S}^{1,0}$:
\begin{align}
\mathrm{Coh}_\mathrm{Shv}:=\mathrm{Coh}_{\mathrm{Shv}}(X_\mathrm{object})^\mathrm{localization, localization}.
\end{align}
Then we get the corresponding gestalt via the stablization along two directions\footnote{In such a manner, we do not have to fix a degree and specify what is that degree for this.}:
\begin{align}
(...,\mathrm{End}_{\mathrm{End}_{\mathrm{Coh}_\mathrm{Shv}}(1)}(1),\mathrm{End}_{\mathrm{Coh}_\mathrm{Shv}}(1),\mathrm{Coh}_\mathrm{Shv},\mathrm{Mod}_{\mathrm{Coh}_\mathrm{Shv}}(nPr^L),\mathrm{Mod}_{\mathrm{Mod}_{\mathrm{Coh}_\mathrm{Shv}}(nPr^L)}((n+1)Pr^L),...).
\end{align}
Going from here we can take the pro-category of $X_\mathrm{object}$, and consider the $(\infty,1)$-category of all the continuous sheaves with values in $(\infty,\infty)$-categories over this site:
\begin{align}
\mathrm{ConShvFunctor}(\mathrm{Pro}(X_\mathrm{object}),\mathrm{CoeffCat}_{\infty,\infty}).
\end{align}
Again we consider the corresponding motivicalization:
\begin{align}
\mathrm{ConShvFun}:=\mathrm{ConShvFunctor}(\mathrm{Pro}(X_\mathrm{object}),\mathrm{CoeffCat}_{\infty,\infty})^\mathrm{localization,localization}
\end{align}
which is then further gives rise to the corresponding gestalt:
\begin{align}
(...,\mathrm{End}_{\mathrm{End}_{\mathrm{CohShvFun}}(1)}(1),\mathrm{End}_{\mathrm{ConShvFun}}(1),\mathrm{ConShvFun},\mathrm{Mod}_{\mathrm{ConShvFun}}(nPr^L),\\
\mathrm{Mod}_{\mathrm{Mod}_{\mathrm{ConShvFun}}(nPr^L)}((n+1)Pr^L),...).
\end{align}
Then we consider the ringed Stone spaces enriched over the Model category of $X_\mathrm{object}$, which is subcategory of all the ringed compact Housdorff spaces in $\mathrm{CHSpace}$ taking a general form of $(A,L_A)$ where $L_A$ is a functor from $A$ to $\mathrm{UltraModel}(X_\mathrm{object})$. We denote the category of these Stone spaces as:
\begin{align}
\mathrm{RingedStone}^{\mathrm{UltraModel}(X_\mathrm{object})}.
\end{align}  
Then as in \cite{5Lu3} we have the following $(\infty,2)$-ultracategory of all the $(\infty,\infty)$-functors from this category to our chosen coefficient category:
\begin{align}
\mathrm{UltraFunctor}^U(\mathrm{UltraModel}(\mathrm{RingedStone}^{\mathrm{UltraModel}(X_\mathrm{object})},\mathrm{CoeffCat}_{\infty,\infty})^{\mathrm{localization}, \mathrm{localization}}
\end{align}
In such a way we also have the corresponding ultracategory of ultrafunctors from this ultracategory to the coefficient ultracategories via further motivicalization through double localizations:
\begin{align}
\mathrm{UltraFun}_S^U:=\mathrm{UltraFunctor}^U(\mathrm{RingedStone}^{\mathrm{UltraModel}(X_\mathrm{object})},\mathrm{CoeffCat}_{\infty,\infty})^{\mathrm{localization}, \mathrm{localization}}.
\end{align}
Therefore after this we then have an associated ultragestalt:
\begin{align}
(\mathrm{End}_{\mathrm{End}_{\mathrm{UltraFun}_S^U}(1)}(1),\mathrm{End}_{\mathrm{UltraFun}_S^U}(1),\mathrm{UltraFun}_S^U,\mathrm{Mod}_{\mathrm{UltraFun}_S^U}(3Pr^L),\mathrm{Mod}_{\mathrm{Mod}_{\mathrm{UltraFun}_S^U}(3Pr^L)}(4Pr^L),...).
\end{align}
Then as in \cite{5Lu3} we have the following $(\infty,2)$-ultracategory of all the $(\infty,\infty)$-functors from the model category for $X_\mathrm{object}$ to our chosen coefficient category:
\begin{align}
\mathrm{UltraFunctor}^U(\mathrm{UltraModel}(X_\mathrm{object}),\mathrm{CoeffCat}_{\infty,\infty})^{\mathrm{localization}, \mathrm{localization}}
\end{align}
In such a way we also have the corresponding ultracategory of ultrafunctors from this ultracategory to the coefficient ultracategories via further motivicalization through double localizations:
\begin{align}
\mathrm{UltraFun}_M^U:=\mathrm{UltraFunctor}^U({\mathrm{UltraModel}(X_\mathrm{object})},\mathrm{CoeffCat}_{\infty,\infty})^{\mathrm{localization}, \mathrm{localization}}.
\end{align}
Therefore after this we then have an associated ultragestalt:
\begin{align}
(\mathrm{End}_{\mathrm{End}_{\mathrm{UltraFun}_M^U}(1)}(1),\mathrm{End}_{\mathrm{UltraFun}_M^U}(1),\mathrm{UltraFun}_M^U,\mathrm{Mod}_{\mathrm{UltraFun}_M^U}(3Pr^L),\mathrm{Mod}_{\mathrm{Mod}_{\mathrm{UltraFun}_M^U}(3Pr^L)}(4Pr^L),...).
\end{align} 
Then this gives the key steps in Lurie's proof by equivalenting all such four $(\infty,\infty)$-ultragestalten above, which proves the main theorem of \cite{5Lu3} on the level of motivic $(\infty,\infty)$-ultragestalten.
\end{proof}

\begin{remark}
When we put ultrastructure over gestalten, it is obviously an enrichment of the story. This is deeply inspired by condensed mathematics. Since ultracategories are sheaves over compact Hausdorff spaces, while condensed categories are sheaves over totally disconnected compact Hausdorff spaces, which is the so-called profinite sets. Following the philosophy of \cite{5S5}, if one does mathmatics through the foundation of ultracategories and ultragestalten, then the category of ultragestalten embedds nicely into the category of gestalten, just as the condensed rings are embedded nicely into the category of gestalten. What happens is that we then have a well-defined six-functor formalism for ultragestalten induced from gestalten six-functor formalism.

\end{remark}

This means we can do ultra-mathematics in the same fashion as in condensed mathematics after \cite{CS1}. 

\begin{definition} \mbox{(Ultra Mathematics)}
We define a pre-ultra analytic ring $(A,M_A)$ as in \cite{CS1} to be some $E_\infty$-ring object $A$ in the category of ultra-abelian groups (i.e. $\infty$-sheaves with values in $\infty$-abelian groups over compact Hausdorff spaces), with $M_A$ which is some sub-$\infty$-category of the category of all ultramodules $U$ over $A$ stable under limits, colimits and extensions. We then have the notion of ultra analytic rings by assuming $M_A$ have good property which is closed under the derived extension up to any order as in \cite{CS1} (between $M_A$ and objects outside $M_A$ but within $U$). We can then collect the corresponding ultra analytic rings as some $(\infty,1)$-category which is $\mathrm{UltraAnRing}$, where we can embed this category into gestalten:
\begin{align}
(A,M_A)\mapsto (A,D(A),\mathrm{Mod}_{D(A)}(1Pr^L),...),
\end{align}
then we pullback the gestalten $!$-able topology back to this $(\infty,1)$-category. We then use this $!$-able topology to define a $!$-able site for ultra analytic rings to form the Grothendieck site:
\begin{align}
\mathrm{UltraAnRing}_!.
\end{align} 
Now we define the corresponding notion of ultra anlaytic stacks to be the $\infty$-sheaves with values in $\infty$-groupoids over this site. Then the correseponding category of ultra analytic stacks $\mathrm{UltraAnStacks}$ have an abstract six-functor formalism in the category of gestalten via the obvious embeding as we embed ultra analytic stacks. In completely parallel fashion we consider $\mathbb{E}_1$-ring object to define noncommutative ultra analytic stacks.
\end{definition}

\subsection{The Mathematical Theory of Large Fargues-Fontaine Ultramodel Gestalten}

\noindent Now we apply the construction above to Fargues-Fontaine Gestalten, where we expand the definition of Fargues-Fontaine gestalten.

\begin{definition}(\text{General Large Fargues-Fontaine Gestalten})
We now systematically working over the sphere spectrum. Consider an algebraic stack $Z$ in fpqc topology over $\mathrm{CAlg}(D(\mathbb{S}))$ and consider the associated FF-stack $\mathrm{FF}_Z$, in the family setting as in \cref{theorem59}. This means that $Z$ is fibered over $\mathbb{N}_c$, the compactification of the integers, and over the family of local fields $H$. For such general presentable $(\infty,1)$-category $X_\mathrm{object}$ of all such Fargues-Fontaine stacks attached to fppf-algebraic stacks over $\mathrm{CAlg}(D(S))$ over $\mathrm{FF}_Z$ via the corresponding condensed analytification in our family situation as in \cref{theorem59}, we can consider the general framework of universal motivic gestalten. When we fix $Z$, we consider the arc toplogy of the solid analytification to define the Fargues-Fontaine stack for $Z$. For such $X$ we consider the coefficient category:
\begin{align}
\mathrm{CoeffCat}_{\infty,\infty}
\end{align}
as above discussion which is the $(\infty,1)$-category of all the presentable $(\infty,\infty)$-categories, which can be degenerated to $(\infty,1)$-groupoid. Now we consider the following category of all the functors from $X$ into the corresponding category:
\begin{align}
\mathrm{CoeffCat}_{\infty,\infty}.
\end{align}
We denote this category as:
\begin{align}
\mathrm{Functor}^U(X_\mathrm{object},\mathrm{CoeffCat}_{\infty,\infty})
\end{align}
since we are working over the sphere spectrum, we have then a promotion lifting to a bigraded motivic spectrum $\mathrm{S}^{1,0}$, in such a way we allow a twice localization (usually the second localization is the loop space functor stabilization with respect to the Tate element)\footnote{We regard this as some $(\infty,2)$-category.}:
\begin{align}
\mathrm{Fun}^U:=\mathrm{Functor}^U(X_\mathrm{object},\mathrm{CoeffCat}_{\infty,\infty})^{\mathrm{localization}, \mathrm{localization}}.
\end{align}
Then for $X_\mathrm{object}$, we have the following universal motivic gestalten:
\begin{align}
(\mathrm{End}_{\mathrm{End}_{\mathrm{Fun}^U}(1)}(1),\mathrm{End}_{\mathrm{Fun}^U}(1),\mathrm{Fun}^U,\mathrm{Mod}_{\mathrm{Fun}^U}(3Pr^L),\mathrm{Mod}_{\mathrm{Mod}_{\mathrm{Fun}^U}(3Pr^L)}(4Pr^L),...).
\end{align}
Now after \cite{5S5} we have the corresponding abstract 6-functors formalism across these universal gelstalten, for $!$-able maps generated from countably presented ones. In the current formalism we consider the corresponding functoriality over $X:=\mathrm{X}_\mathrm{object}$, where a general form of such map takes the form of a countable limit:
\begin{align}
\lim_k X'_k \rightarrow X
\end{align}
where $X'_k$ in the opposite category is countably presented over $X$, in the category of all presentable $\infty$-categories. For this we consider all the relavant such categories $X$ which are assumed to be simply monoidal and $\infty$-category we denote them as:
\begin{align}
\mathrm{Mon}_\infty.
\end{align}
And we denote the corresponding six functors as:
\begin{align}
f_\flat, f^\sharp, g_\flat, g^\sharp, \boxtimes, \mathrm{Hom}.  
\end{align}
However over the higher category level the difference between $f,g$ is disappearing as in \cite{5S5}. We call when $Z$ is fixed the resulting gestalt
\begin{align}
\mathrm{Fun^U}
\end{align}
a \textit{pre-large Fargues-Fontaine gestalt} for $Z$. Now for different $\mathrm{FF}_Z$ we are going to have different such $X_\mathrm{object}(Z)$, then we have the corresponding different pre-large Fargues-Fontaine gestalten. Then we take the corresponding category of all gestalten generated from such pre-large Fargues-Fontaine gestalten via infinite K\"unneth products above degree 3 in the gestalten, we then call the final category \textit{the category of all the large Fargues-Fontaine gestalten}. 
\end{definition}

\begin{definition}(\text{Large Fargues-Fontaine Ultramodel Gestalten})
We now systematically working over the sphere spectrum $\mathbb{S}$. For a general $(\infty,1)$-ultracategory $X_\mathrm{object}$ over the sphere spectrum $\mathbb{S}$ (for instance a big Grothendieck site of the sphere spectrum $\mathbb{S}$, a big Grothendieck topos, or just a pretopos). Note that any such ultracategory can be regarded as a $\infty$-sheaf values in $\infty$-categories over the site of all the compact Hausdorff spaces:
\begin{align}
\mathrm{CHSpace}.
\end{align}
For such $X_\mathrm{object}$ we consider the coefficient ultracategory:
\begin{align}
\mathrm{CoeffCat}_{\infty,\infty}
\end{align}
as above discussion which is the $(\infty,1)$-ultracategory of all the presentable $(\infty,\infty)$-ultracategories, which can be degenerated to $(\infty,1)$-ultragroupoids. Now we consider the following ultracategory of all the ultrafunctors from $X$ into the corresponding ultracategories:
\begin{align}
\mathrm{CoeffCat}_{\infty,\infty}.
\end{align}
We denote this ultracategory as:
\begin{align}
\mathrm{UltraFunctor}^U(X_\mathrm{object},\mathrm{CoeffCat}_{\infty,\infty})
\end{align}
since we are working over the sphere spectrum, we have then a promotion lifting to a bigraded motivic spectrum $\mathrm{S}^{1,0}$, in such a way we allow a twice localization (usually the second localization is the loop space functor stabilization with respect to the Tate element)\footnote{We regard this as some $(\infty,2)$-category.}:
\begin{align}
\mathrm{UltraFun}^U:=\mathrm{UltraFunctor}^U(X_\mathrm{object},\mathrm{CoeffCat}_{\infty,\infty})^{\mathrm{localization}, \mathrm{localization}}.
\end{align}
Then for general $X_\mathrm{object}$, we have the following universal motivic gestalten:
\begin{align}
(\mathrm{End}_{\mathrm{End}_{\mathrm{UltraFun}^U}(1)}(1),\mathrm{End}_{\mathrm{UltraFun}^U}(1),\mathrm{UltraFun}^U,\mathrm{Mod}_{\mathrm{UltraFun}^U}(3Pr^L),\mathrm{Mod}_{\mathrm{Mod}_{\mathrm{UltraFun}^U}(3Pr^L)}(4Pr^L),...).
\end{align}
As in \cite{5Lu3} we can consider the ultramodels (i.e. the functors from $X_\mathrm{object}$) for $X_\mathrm{object}$:
\begin{align}
\mathrm{UltraModel}(X_\mathrm{object}).
\end{align}
In such a way we also have the corresponding ultracategory of ultrafunctors from this ultracategory to the coefficient ultracategories via further motivicalization through double localizations:
\begin{align}
\mathrm{UltraFun}_M^U:=\mathrm{UltraFunctor}^U(\mathrm{UltraModel}(X_\mathrm{object}),\mathrm{CoeffCat}_{\infty,\infty})^{\mathrm{localization}, \mathrm{localization}}.
\end{align}
Therefore after this we then have an associated ultragestalt:
\begin{align}
(\mathrm{End}_{\mathrm{End}_{\mathrm{UltraFun}_M^U}(1)}(1),\mathrm{End}_{\mathrm{UltraFun}_M^U}(1),\mathrm{UltraFun}_M^U,\mathrm{Mod}_{\mathrm{UltraFun}_M^U}(3Pr^L),\mathrm{Mod}_{\mathrm{Mod}_{\mathrm{UltraFun}_M^U}(3Pr^L)}(4Pr^L),...).
\end{align}
Now after \cite{5S5} we have the corresponding abstract 6-functors formalism across these universal gelstalten, for $!$-able maps generated from countably presented ones. In the current formalism we consider the corresponding functoriality over $X:=\mathrm{X}_\mathrm{object}$, where a general form of such map takes the form of a countable limit:
\begin{align}
\lim_k X'_k \rightarrow X
\end{align}
where $X'_k$ in the opposite ultracategory is countably presented over $X$, in the ultracategory of all presentable $(\infty,1)$-ultracategories. For this we consider all the relavant such ultracategories $X$ which are assumed to be simply $\infty$-ultracategory we denote them as:
\begin{align}
\mathrm{Mon}_\infty.
\end{align}
And we denote the corresponding six functors as:
\begin{align}
f_\flat, f^\sharp, g_\flat, g^\sharp, \boxtimes, \mathrm{Hom}.  
\end{align}
However over the higher category level the difference between $f,g$ is disappearing as in \cite{5S5}. Recall that in this setting the corresponding ultracategories are defined to be carrying integration theory which satisfies Fubini theorem:
\begin{align}
&\int_S (\mathrm{End}_{\mathrm{End}_{\mathrm{UltraFun}^U}(1)}(1),\mathrm{End}_{\mathrm{UltraFun}^U}(1),\mathrm{UltraFun}^U,\\
&\mathrm{Mod}_{\mathrm{UltraFun}^U}(3Pr^L),\mathrm{Mod}_{\mathrm{Mod}_{\mathrm{UltraFun}^U}(3Pr^L)}(4Pr^L),...)(s)d(\int_T d\mu_t) =\\ 
&\int_T\int_S (\mathrm{End}_{\mathrm{End}_{\mathrm{UltraFun}^U}(1)}(1),\mathrm{End}_{\mathrm{UltraFun}^U}(1),\mathrm{UltraFun}^U,\\
&\mathrm{Mod}_{\mathrm{UltraFun}^U}(3Pr^L),\mathrm{Mod}_{\mathrm{Mod}_{\mathrm{UltraFun}^U}(3Pr^L)}(4Pr^L),...)(s)d\mu_t,\\
&\int_S (\mathrm{End}_{\mathrm{End}_{\mathrm{UltraFun}_M^U}(1)}(1),\mathrm{End}_{\mathrm{UltraFun}_M^U}(1),\mathrm{UltraFun}_M^U,\\
&\mathrm{Mod}_{\mathrm{UltraFun}_M^U}(3Pr^L),\mathrm{Mod}_{\mathrm{Mod}_{\mathrm{UltraFun}_M^U}(3Pr^L)}(4Pr^L),...)(s)d(\int_T \mu_t) =\\ 
&\int_T\int_S (\mathrm{End}_{\mathrm{End}_{\mathrm{UltraFun}_M^U}(1)}(1),\mathrm{End}_{\mathrm{UltraFun}_M^U}(1),\mathrm{UltraFun}_M^U,\\
&\mathrm{Mod}_{\mathrm{UltraFun}_M^U}(3Pr^L),\mathrm{Mod}_{\mathrm{Mod}_{\mathrm{UltraFun}_M^U}(3Pr^L)}(4Pr^L),...)(s)d\mu_t.
\end{align}
The measure is $(0,1)$-valued and the corresponding $S,T$ are sets and the $\sigma$-algebras are defined to be the ones generated by power sets. As in the usual abstract measure theory here the product measures are the tensor product measures. We call when $Z$ is fixed the resulting gestalt
\begin{align}
\mathrm{UltraFun^U}
\end{align}
a \textit{pre-large Fargues-Fontaine ultramodel gestalt} for $Z$. Now for different $\mathrm{FF}_Z$ we are going to have different such $X_\mathrm{object}(Z)$, then we have the corresponding different pre-large Fargues-Fontaine ultramodel gestalten. Then we take the corresponding category of all gestalten generated from such pre-large Fargues-Fontaine ultramodel gestalten via infinite K\"unneth products above degree 3 in the gestalten, we then call the final category \textit{the category of all the large Fargues-Fontaine ultramodel gestalten}.
\end{definition}

\begin{corollary}
The mathematical theory of pre-large Fargues-Fontaine gestalten satisfies motivic ultra-categorical Vollst\"andigkeitssatz:
\begin{align}
&(...,\mathrm{End}_{\mathrm{End}_{\mathrm{CohShvFun}}(1)}(1),\mathrm{End}_{\mathrm{ConShvFun}}(1),\mathrm{ConShvFun},\mathrm{Mod}_{\mathrm{ConShvFun}}(nPr^L),\mathrm{Mod}_{\mathrm{Mod}_{\mathrm{ConShvFun}}(nPr^L)}((n+1)Pr^L),...)\\
&\overset{\sim}{\rightarrow}\\
&(\mathrm{End}_{\mathrm{End}_{\mathrm{UltraFun}_M^U}(1)}(1),\mathrm{End}_{\mathrm{UltraFun}_M^U}(1),\mathrm{UltraFun}_M^U,\mathrm{Mod}_{\mathrm{UltraFun}_M^U}(3Pr^L),\mathrm{Mod}_{\mathrm{Mod}_{\mathrm{UltraFun}_M^U}(3Pr^L)}(4Pr^L),...).
\end{align}
When we consider functoriality among pre-large Fargues-Fontaine gestalten, this equivalence respects abstract higher gestalten six-functor formalism on the both sides.
\end{corollary}

\begin{remark}
The reason for this definition is that we hope we can take products in the K\"unneth theorem fashion but still within a same category.
\end{remark}

\subsection{The Mathematical Theory of Large $p$-adic Representation Gestalten}

\noindent Now we apply the construction above to $p$-adic representation gestalten, where we expand the definition of usual $p$-adic representation gestalten. This means we consider the universal motivic gestalten associated with the infinite rank $p$-adic smooth representation sheaves over the quotient stacks for $p$-adic Lie groups.

\begin{definition}(\text{General Large $p$-adic Representation Gestalten})
We now systematically working over the sphere spectrum. Let $G$ be a $p$-adic Lie group and we consider the $\mathbb{S}_W$-smooth (the $p$-adically completed Witt vector sphere spectrum. Here we consider the corresponding lifings of $\mathbb{Z}_p$-representations to this spectrum) representations of $G$ which is then a symmetrical monoical $(\infty,1)$-category. For such general presentable $(\infty,1)$-category $X_\mathrm{object}$ of all the such $p$-adic spectrum valued representations of $G$, we can consider the parallel constrution within the framework of all the motivic universal gestalten. For such $X$ we consider the coefficient category:
\begin{align}
\mathrm{CoeffCat}_{\infty,\infty}
\end{align}
as above discussion which is the $(\infty,1)$-category of all the $(\infty,\infty)$-categories, which can be degenerated to $(\infty,1)$-groupoids. Now we consider the following category of all the functors from $X$ into the corresponding category:
\begin{align}
\mathrm{CoeffCat}_{\infty,\infty}.
\end{align}
We denote this category as:
\begin{align}
\mathrm{Functor}^U(X_\mathrm{object},\mathrm{CoeffCat}_{\infty,\infty})
\end{align}
since we are working over the sphere spectrum, we have then a promotion lifting to a bigraded motivic spectrum $\mathrm{S}^{1,0}$, in such a way we allow a twice localization (usually the second localization is the loop space functor stabilization with respect to the Tate element)\footnote{We regard this as some $(\infty,2)$-category.}:
\begin{align}
\mathrm{Fun}^U:=\mathrm{Functor}^U(X_\mathrm{object},\mathrm{CoeffCat}_{\infty,\infty})^{\mathrm{localization}, \mathrm{localization}}.
\end{align}
Then for $X_\mathrm{object}$, we have the following universal motivic gestalten:
\begin{align}
(\mathrm{End}_{\mathrm{End}_{\mathrm{Fun}^U}(1)}(1),\mathrm{End}_{\mathrm{Fun}^U}(1),\mathrm{Fun}^U,\mathrm{Mod}_{\mathrm{Fun}^U}(3Pr^L),\mathrm{Mod}_{\mathrm{Mod}_{\mathrm{Fun}^U}(3Pr^L)}(4Pr^L),...).
\end{align}
Now after \cite{5S5} we have the corresponding abstract 6-functors formalism across these universal gelstalten, for $!$-able maps generated from countably presented ones. In the current formalism we consider the corresponding functoriality over $X:=\mathrm{X}_\mathrm{object}$, where a general form of such map takes the form of a countable limit:
\begin{align}
\lim_k X'_k \rightarrow X
\end{align}
where $X'_k$ in the opposite category is countably presented over $X$, in the category of all presentable $\infty$-categories. For this we consider all the relavant such categories $X$ which are assumed to be simply monoidal and $\infty$-category we denote them as:
\begin{align}
\mathrm{Mon}_\infty.
\end{align}
And we denote the corresponding six functors as:
\begin{align}
f_\flat, f^\sharp, g_\flat, g^\sharp, \boxtimes, \mathrm{Hom}.  
\end{align}
However over the higher category level the difference between $f,g$ is disappearing as in \cite{5S5}. We call when $G$ is fixed the resulting gestalt
\begin{align}
\mathrm{Fun^U}
\end{align}
a \textit{pre-large $p$-adic representation gestalt} for $Z$. Now for different $G$ we are going to have different such $X_\mathrm{object}(G)$, then we have the corresponding different pre-large $p$-adic representation gestalten. Then we take the corresponding category of all gestalten generated from such pre-large $p$-adic representation gestalten via infinite higher K\"unneth products above degree 3 in the gestalten, we then call the final category \textit{the category of all the large $p$-adic representation gestalten}. This generation process can also happen for the ultracategorical setting, when then call the result generated category \textit{the category of all the large $p$-adic representation ultramodel gestalten}.
\end{definition}

\begin{corollary}
The mathematical theory of pre-large $p$-adic representation gestalten satisfies motivic ultra-categorical Vollst\"andigkeitssatz:
\begin{align}
&(...,\mathrm{End}_{\mathrm{End}_{\mathrm{CohShvFun}}(1)}(1),\mathrm{End}_{\mathrm{ConShvFun}}(1),\mathrm{ConShvFun},\mathrm{Mod}_{\mathrm{ConShvFun}}(nPr^L),\mathrm{Mod}_{\mathrm{Mod}_{\mathrm{ConShvFun}}(nPr^L)}((n+1)Pr^L),...)\\
&\overset{\sim}{\rightarrow}\\
&(\mathrm{End}_{\mathrm{End}_{\mathrm{UltraFun}_M^U}(1)}(1),\mathrm{End}_{\mathrm{UltraFun}_M^U}(1),\mathrm{UltraFun}_M^U,\mathrm{Mod}_{\mathrm{UltraFun}_M^U}(3Pr^L),\mathrm{Mod}_{\mathrm{Mod}_{\mathrm{UltraFun}_M^U}(3Pr^L)}(4Pr^L),...).
\end{align}
When we consider functoriality among pre-large $p$-adic representation gestalten from morphisms on $p$-adic Lie groups, this equivalence respects abstract higher gestalten six-functor formalism on the both sides.
\end{corollary}

\newpage
\subsection*{Acknowledgements}

We thank Professor Kedlaya for discussion on Z\'abr\'adi's work. We then get the chance to write this paper down following some recently established foundation on analytic geometry. We thank Professor Kedlaya for learning the motivic cohomology theory related to Robba sheaves (such as the \textit{K\"unneth theorem} in this setting) from him. We in this paper generalize a previous consideration of Sorensen in the field $\mathbb{F}_p$-situation to the $\mathbb{Z}_p$-ring situation closely following the work of Schneider, Sorensen, Schneider-Sorensen and Heyer-Mann. We follow also closely the work \cite{Sa} on derived $E_1$-rings over any commutative rings and derived Hochschild homological consideration. We thank Professor Sorensen for correspondence during my writing of this paper and for inspiration on generalizing Hodge theory along Breuil-Schneider's key construction and observation. We were trying to follow Fontaine many years ago on almost de Rham representations in order to study them in geometric family, however the current presentation is largely following Bhatt-Lurie. We would like to mention that we received very deep inspiration and understanding on the modern theory of motives from Scholze's 2024 talk notes on ring stacks and his lecture notes in progress now on the course of higher category theory, such as the concept of $(\infty,\infty)$-categorical \textit{gestalten} which certainly is a key conceptualization of many significant $p$-adic ring stacks such as those prismatizations in families we considered in this paper.

\end{document}